\documentclass[a4paper,12pt,reqno]{amsart}

\usepackage[nobysame,initials]{amsrefs}

\usepackage[margin=2.5cm]{geometry}
\usepackage{graphicx}
\usepackage{amsfonts}
\usepackage{verbatim}
\usepackage{enumitem}
\usepackage{bm}

\usepackage[]{nomentbl}
\makenomenclature

\usepackage[colorlinks=true]{hyperref}

\usepackage{hyperref}

\title{The exact minimum number of triangles in graphs of given order and size}

\author{Hong Liu}
\author{Oleg Pikhurko}
\address{Mathematics Institute and DIMAP, University of Warwick, Coventry CV4 7AL, UK}
\email{\tt{$\lbrace$h.liu.9,~o.pikhurko$\rbrace$@warwick.ac.uk}}

\author{Katherine Staden}
\address{Mathematical Institute, University of Oxford, Andrew Wiles Building, Radcliffe Observatory Quarter, Woodstock Road, Oxford, OX2 6GG, UK}
\email{\tt{staden@maths.ox.ac.uk}}

\def\COMMENT#1{}
\let\COMMENT=\footnote

\begin{document}

\newtheorem{theorem}{Theorem}[section]
\newtheorem{lemma}[theorem]{Lemma}
\newtheorem{proposition}[theorem]{Proposition}
\newtheorem{corollary}[theorem]{Corollary}
\newtheorem{cor}[theorem]{Corollary}
\newtheorem{conjecture}[theorem]{Conjecture}
\newtheorem{claim}[theorem]{Claim}
\newtheorem{definition}[theorem]{Definition}
\newtheorem{assumption}[theorem]{Assumption}
\newtheorem{fact}[theorem]{Fact}

\numberwithin{equation}{section}

\def\eps{{\varepsilon}}
\newcommand{\cP}{\mathcal{P}}
\newcommand{\cT}{\mathcal{T}}
\newcommand{\cL}{\mathcal{L}}
\newcommand{\ex}{\mathbb{E}}
\newcommand{\eul}{e}
\newcommand{\pr}{\mathbb{P}}
\newcommand{\cH}{\mathcal{H}}
\newcommand{\ep}{\varepsilon}
\newcommand{\al}{\alpha}
\newcommand{\De}{\Delta}
\newcommand{\de}{\delta}
\newcommand{\error}{k}

\newcommand{\Ppartition}{P1}
\newcommand{\Pcomplete}{P2}
\newcommand{\Pmissing}{P5}
\newcommand{\Pbadedges}{P3}
\newcommand{\PZk}{P4}

\newenvironment{remark}[1][Remark.]{\begin{trivlist}
\item[\hskip \labelsep {\bfseries #1}]}{\end{trivlist}}
\newcommand{\hide}[1]{\ \textsf{[[ #1 ]]}\ }
\newcommand{\Hide}[1]{}
\newcommand{\C}[1]{{\protect\mathcal{#1}}}
\newcommand{\B}[1]{{\bf #1}}
\newcommand{\I}[1]{{\mathbb #1}}
\renewcommand{\O}[1]{\overline{#1}}
\newcommand{\M}[1]{\mathrm{#1}}
\newcommand{\G}[1]{\mathfrak{#1}}
\newcommand{\V}[1]{\mathbold{#1}}
\renewcommand{\mid}{:} 

\maketitle


\noindent\textbf{Overview.} Extremal graph theory is a very active area of mathematics, whose development was greatly fuelled by some fundamental questions that, while easy to state, are notoriously difficult. One famous example of a such question is the Erd\H os-Rademacher problem whose key special case is to determine $g(n,e)$, the minimum number of triangles that a graph with $n$ vertices and $e$ edges can have. It goes back to Rademacher (1941, unpublished) who solved the first non-trivial case $e=t_2(n)+1$, where $t_r(n)$ is the size of the complete balanced $r$-partite graph of order~$n$. This problem was revived by Erd\H{o}s~\cite{Erdos55} and has attracted much attention since then. Mathematicians who directly worked on it include Bollob\'as, Fisher, Khadziivanov, Lov\'asz, Moon, Moser, Nikiforov, Nordhaus, Razborov, Reiher, Simonovits, Stewart, and others. Also, this problem motivated many further important developments (stability theorems, the supersaturation problem, Sidorenko's conjecture, subgraph profiles in graphons, etc.). 

One of the difficulties is that $g(\lambda):=\lim_{n\to\infty} g(n,(\lambda+o(1)){n\choose 2})$ happens to be a countable union of concave functions. So novel approaches were needed here. Fisher~\cite{Fisher89} determined $g(\lambda)$ for $\lambda\in [\frac 12,\frac23]$ by using spectral methods. The whole range of $\lambda$ was resolved by Razborov~\cite{Razborov07,Razborov08}, who introduced his powerful flag algebra method for this purpose. Later, Nikiforov~\cite{Nikiforov11} and Reiher~\cite{Reiher16} re-proved Razborov's result by using Lagrange multipliers for some analytic relaxation of the problem.  

Unfortunately, none of the above methods seem to give the exact value of $g(n,e)$, 
except for rather special pairs $(n,e)$. Essentially all other previously known exact results for $n\to\infty$ are superseded by a difficult  paper of Lov\'asz-Simonovits~\cite{LovaszSimonovits83} who completely solved the problem when $t_r(n)\le e\le t_r(n)+o_r(n^2)$, i.e.,\ $e$ is slightly above some $t_r(n)$. 

Our paper determines the exact value of $g(n,e)$ and describes all extremal graphs when the edge density is bounded away from~1, proving
a conjecture of Lov\'asz-Simonovits~\cite{LovaszSimonovits76} from 1975 in this range. An additional difficulty for the exact result is that extremal graphs come from three somewhat similar but different families. Also, the stability approach via local transformations does not seem to yield the result so we had to use some `global' estimates as well.
We hope that the method of this paper, besides solving almost all cases of the triangle minimisation problem, will be useful in other situations where one has to convert asymptotic calculations into exact results.\\[10pt] 

\noindent{\textbf{Abstract.}
What is the minimum number of triangles in a graph of given order and size? Motivated by earlier results of Mantel and Tur\'an, Rademacher solved the first non-trivial case of this problem in 1941. The problem was revived by Erd\H{o}s in 1955; it is now known as the Erd\H{o}s-Rademacher problem. After attracting much attention, it was solved asymptotically in a major breakthrough by Razborov in 2008.
In this paper, we provide an exact solution for all large graphs whose edge density is bounded away from~$1$, which in this range confirms a conjecture  of Lov\'asz and Simonovits from 1975. Furthermore, we
give a description of the extremal graphs.
\\[10pt]

\noindent\textbf{2010 Mathematics Subject Classification:} 05C35

\tableofcontents

\section{Introduction}

The celebrated theorem of Tur\'an~\cite{Turan41} (with the case $r=3$ proved earlier by Mantel~\cite{Mantel07}) states that, among all $K_r$-free graphs with $n\ge r$ vertices, the \emph{Tur\'an graph} $T_{r-1}(n)$, the complete balanced $(r-1)$-partite graph, is the unique graph maximising the number of edges. Here, the \emph{$r$-clique} $K_r$ is the complete graph with $r$ vertices (and ${r\choose 2}$ edges). 

Let $t_r(n):=e(T_r(n))$ denote the number of edges in $T_r(n)$ and let an \emph{$(n,e)$-graph} mean a graph with $n$ vertices and $e$ edges. Thus the above result implies that every $(n,t_2(n)+1)$-graph $H$ contains at least one triangle. Rademacher in 1941 (unpublished, see~\cite{Erdos55}) showed that $H$ must have at least $\lfloor n/2\rfloor$ triangles. This naturally leads to the following general question that first appears in print in a paper of Erd\H{o}s~\cite{Erdos55} and is now called the \emph{Erd\H os-Rademacher problem}: determine
 $$
 g_r(n,e):=\min\{K_r(H)\mid \mbox{$(n,e)$-graph $H$}\},\quad n,e\in\I N,\ e\le {n\choose 2},
 $$
 where $K_r(H)$ denotes the number of $K_r$-subgraphs in a graph $H$ and $\I N:=\{1,2,\dots\}$ consists of natural numbers.
 
Before discussing the history of this problem in some detail, let us present the general upper bound $h^*(n,e)$ on $g_3(n,e)$ which, as far as the authors know, may actually equal $g_3(n,e)$ for all pairs $(n,e)$. In fact, one of the main results of this paper (stated in a stronger form in Theorem~\ref{main}) is that $g_3(n,e)=h^*(n,e)$ if $n$ is large and $e/{n\choose 2}$ is bounded away from~$1$. In order to define $h^*$, we need to introduce some auxiliary parameters. 

\begin{definition}[Parameters $k$, $m^*$ and $h^*$, vector $\bm{a}^*$, and graph $H^*$]\label{astardef}
Let $n,e\in \I N$ satisfy $e \leq \binom{n}{2}$. Define
\begin{equation}\label{eq:k}
k = k(n,e) := \min \lbrace s \in \mathbb{N} : e \le t_{s}(n) \rbrace,
\end{equation}
 that is, $k$ is the unique positive integer with $t_{k-1}(n)< e\le t_k(n)$.

	Next, let $\bm{a}^*=\bm{a}^*(n,e)$ be the unique integer vector $(a^*_1,\ldots,a^*_k)$ such that
	\begin{itemize}
		\item $a^*_k:= \min \lbrace a \in \mathbb{N} : a(n-a)+t_{k-1}(n-a) \geq e \rbrace$;
		\item $a^*_1+\ldots+a^*_{k-1}=n-a^*_k$ and $a^*_1\ge\ldots\ge a^*_{k-1}\ge a^*_1-1$.
	\end{itemize}
	
	Further, define
	\begin{eqnarray}
	m^* \ =\ m^*(n,e)&:=& \sum_{1\le i<j\le k}a_i^*a_j^*-e,\label{eq:mstar}\\
	h^*(n,e) &:=& \sum_{1\le h< i< j\le k}a_h^*a_i^*a_j^* - m^*\sum_{i=1}^{k-2}a_i^*.\nonumber
	\end{eqnarray}
	
Also, let the graph $H^*=H^*(n,e)$ be obtained from $K^k_{a_1^*,\dots,a_k^*}$, the complete $k$-partite graph with part sizes $a_1^*,\dots,a_k^*$, by removing $m^*$ edges between the last two parts (say, for definiteness, all incident to a vertex in the last part).
\end{definition}

Let us rephrase the above definitions and also argue that $H^*$ is well-defined. We look for an upper bound on $g_3(n,e)$, where we take a complete partite graph, say with parts $A_1^*,\dots,A_k^*$, and remove a star incident to a vertex of $A_k^*$. First, we choose the smallest $k$ for which such an $(n,e)$-graph exists and then the smallest possible size $a^*_k$ of $A_k^*$. Then we let the first $k-1$ parts form the Tur\'an graph $T_{k-1}(n-a^*_k)$, that is, their sizes are  $a_1^*,\dots,a_{k-1}^*$. Since $T_{k-1}(n-a^*_k)$ has at least as many edges as any other $(k-1)$-partite graph of order $n-a_k^*$, it holds that $m^*:=e(K^k_{a_1^*,\dots,a_k^*})-e$ is non-negative. Furthermore, we  have that 
 \begin{equation}\label{m*}
 0\le m^*\le a_{k-1}^*-a_k^*
 \end{equation}
 because, if the upper bound fails, then
 $$
 e(K^k_{a_1^*,\dots,a_{k-2}^*,a_{k-1}^*+1,a_k^*-1})=e(K^k_{a_1^*,\dots,a_k^*})- (a_{k-1}^*-a_k^*+1)\ge e,
 $$
 contradicting the minimality of $a_k^*$ (or the minimality of $k$ if $a_k^*=1$). In particular, we have $m^*\le a^*_{k-1}$ so $H^*$ is well-defined. Thus $H^*$ is an $(n,e)$-graph and
  $$
   h^*(n,e):=K_3(H^*)\ge g_3(n,e)
   $$ 
   is indeed an upper bound on $g_3(n,e)$. 

For example, if  $e\le t_2(n)$ then $H^*(n,e)$ is bipartite  and $h^*(n,e)=0$ (here $k=2$).  Also, $H^*(n,t_r(n))=T_r(n)$. If $1\le \ell< \lceil n/r\rceil$, then $H^*(n,t_r(n)+\ell)$ is obtained from the Tur\'an graph $T_r(n)$ by adding the $\ell$-star $K_{1,\ell}$ into a largest part (here $k=r+1$ and $a_k^*=1$), etc.

Let us return to the history of the triangle-minimisation problem. The problem was revived by Erd\H{o}s~\cite{Erdos55} in 1955 who in particular conjectured that for $1\le \ell < \lfloor n/2 \rfloor$ it holds that $g_3(n,t_2(n)+\ell)=\ell\lfloor n/2\rfloor$. This is exactly the $h^*$-bound; also, note that if $n$ is even and $\ell=n/2$ then $h^*(n,t_2(n)+\ell)$ is strictly smaller than $\ell n/2$
(here $k=3$ and $a_3^*=2$).
So Erd\H{o}s's conjecture cannot be extended here.
In the same paper, Erd\H{o}s~\cite{Erdos55} proved the conjecture when $\ell\le 3$; the same result also appears in Nikiforov~\cite{Nikiforov76}.
Erd\H{o}s in~\cite{Erdos62} was able to prove his conjecture when $\ell < \gamma n$ for some positive constant $\gamma$.
The conjecture was eventually proved in totality for large $n$ by Lov\'asz and Simonovits~\cite{LovaszSimonovits76} in 1975, with the proof of the conjecture also announced by Nikiforov and Khadzhiivanov~\cite{NikiforovKhadzhiivanov81}. 


Moon and Moser~\cite[Page~285]{MoonMoser62} and, independently, Nordhaus and Stewart~\cite[Equation~(5)]{NordhausStewart63} proved that
 \begin{equation}\label{"Goodman"}
 g_3(n,e)\ge \frac{e(4e-n^2)}{3n},
 \end{equation}
 with equality achieved if and only if $e=t_k(n)$ with $k$ dividing $n$. The bound in~\eqref{"Goodman"} can be derived by using the triangle counting method from an earlier paper by Goodman~\cite{Goodman59} and is often referred to as the \emph{Goodman bound}.

In order to state some of the following results, it will be convenient to define the asymptotic version of the problem. Namely, given $\lambda \in [0,1]$ take any integer-valued function $0\le e(n)\le {n\choose 2}$ with $e(n)/{n\choose 2}\to \lambda$ as $n\to\infty$ and define
$$
g_r(\lambda) := \lim_{n \rightarrow \infty}\frac{g_r(n,e(n))}{\binom{n}{r}}.
$$
 It is easy to see from basic principles that the limit exists and does not depend on the choice of the function~$e(n)$.

The upper bound on the function $g_3(\lambda)$ given by the graphs $H^*$  from Definition~\ref{astardef} is as follows. Let $n\to\infty$ and $e=\lambda n^2/2+o(n^2)$. It always holds that, for example, $m^*\le n$ and $a_1^*-a_{k-1}^*\le 1$. So these have negligible effect in the limit and one can consider only complete partite graphs with all parts equal, except at most one part of smaller size. Therefore, for $\lambda\in [0,1)$, let us define 
 \begin{equation}\label{eq:k(lambda)}
 k(\lambda):=\min\{k\in\I N\mid \lambda\le 1-1/{k}\}.
 \end{equation}
 Thus if $\lambda\in (0,1)$, then $k(\lambda)$ is the unique integer $k\geq 2$ satisfying $1-\frac{1}{k-1}<\lambda\le 1-\frac1{k}$, while $k(0)=1$. Let $k=k(\lambda)$ and let $c=c(\lambda)$ be the unique root with $c\ge 1/k$ of the quadratic equation
 \begin{equation}\label{eq:NewC}
 {k-1\choose 2} c^2+(1-c') c'=\lambda/2,
 \end{equation}
 where $c':=1-(k-1)c$. The above equation is the limit version of the desired equality $e(K^k_{c n,\dots,c n,c'n})=\lambda{n\choose 2} +o(n^2)$. Explicitly,
\begin{equation}\label{eq:c2'}
 c(\lambda) = \frac{1}{k}\left(\,1 + \sqrt{1-\frac{k}{k-1}\cdot \lambda}\,\right),\quad \lambda\in (0,1),\quad\text{while }c(0)=1.
 \end{equation}
 Thus 
 \begin{equation}\label{eq:UpperAsymptG3}
 g_3(\lambda)\le h^*(\lambda):=3!\,\left({k-1\choose 3} c^3+{k-1\choose 2}c^2 c'\right),\quad \lambda\in [0,1).
 \end{equation}
 (For $\lambda=1$, we just let $h^*(1):=1$.)

The upper bound in~\eqref{eq:UpperAsymptG3} coincides with the lower bound on $g_3(\lambda)$ given by~\eqref{"Goodman"} when $\lambda=1-1/k$ for all integers $k\ge 1$. Thus
 \begin{equation}\label{eq:GoodmanSharp}
 g_3(1-1/k)=\frac{(k-1)(k-2)}{k^2},\quad k\in\I N.
 \end{equation}

Some of early results on $g_3(\lambda)$ concentrated on finding good convex lower bounds. McKay (unpublished, see~\cite[Page~35]{NordhausStewart63}) showed that $g_3(\lambda)\ge \lambda-\frac 12$. Nordhaus and Stewart~\cite{NordhausStewart63}  conjectured that $g_3(\lambda)\ge \frac43(\lambda-\frac12)$ and presented some partial results in this direction. This conjecture was proved by Bollob\'as~\cite{Bollobas76} who in fact established the best possible convex lower bound on $g_3$, namely, the piecewise linear function which concides with $g_3$ at all values in~\eqref{eq:GoodmanSharp}. 

However, the upper bound $h^*(\lambda)$ is a strictly concave function between any two consecutive values in~\eqref{eq:GoodmanSharp} for $\lambda\ge1/2$. This is one of the reasons why the triangle-minimisation problem is so difficult.

After Bollob\'as~\cite{Bollobas76}, the first improvement ``visible in the limit'' was achieved by Fisher~\cite{Fisher89} who showed that $g_3(\lambda)=h^*(\lambda)$ for all $1/2\le \lambda\le 2/3$. (There was a hole in Fisher's proof, which can be fixed using the results of Goldwurm and Santini~\cite{GoldwurmSantini00}, see~\cite[Remark~3.3]{Csikvari13}.)
Then  Razborov used his newly developed theory of \emph{flag algebras} first to give a different proof of Fisher's result in~\cite{Razborov07} and then to determine the whole function $g_3(\lambda)$ in~\cite{Razborov08} (see Figure~\ref{plot} for a plot of the function).

\begin{theorem}[\cite{Razborov08}]\label{razthm}
For all $\lambda \in [0,1]$, we have that $g_3(\lambda) = h^*(\lambda)$. \end{theorem}

Nikiforov~\cite{Nikiforov11} presented a new proof of Razborov's result and also determined $g_4(\lambda)$ for all $\lambda \in [0,1]$.
More recently, Reiher~\cite{Reiher16} determined $g_r(\lambda)$ for all $\lambda \in [0,1]$ and $r \geq 5$ (also reproving the case $r\in \lbrace 3,4\rbrace$).

\begin{center}
\begin{figure}
\includegraphics[scale=1]{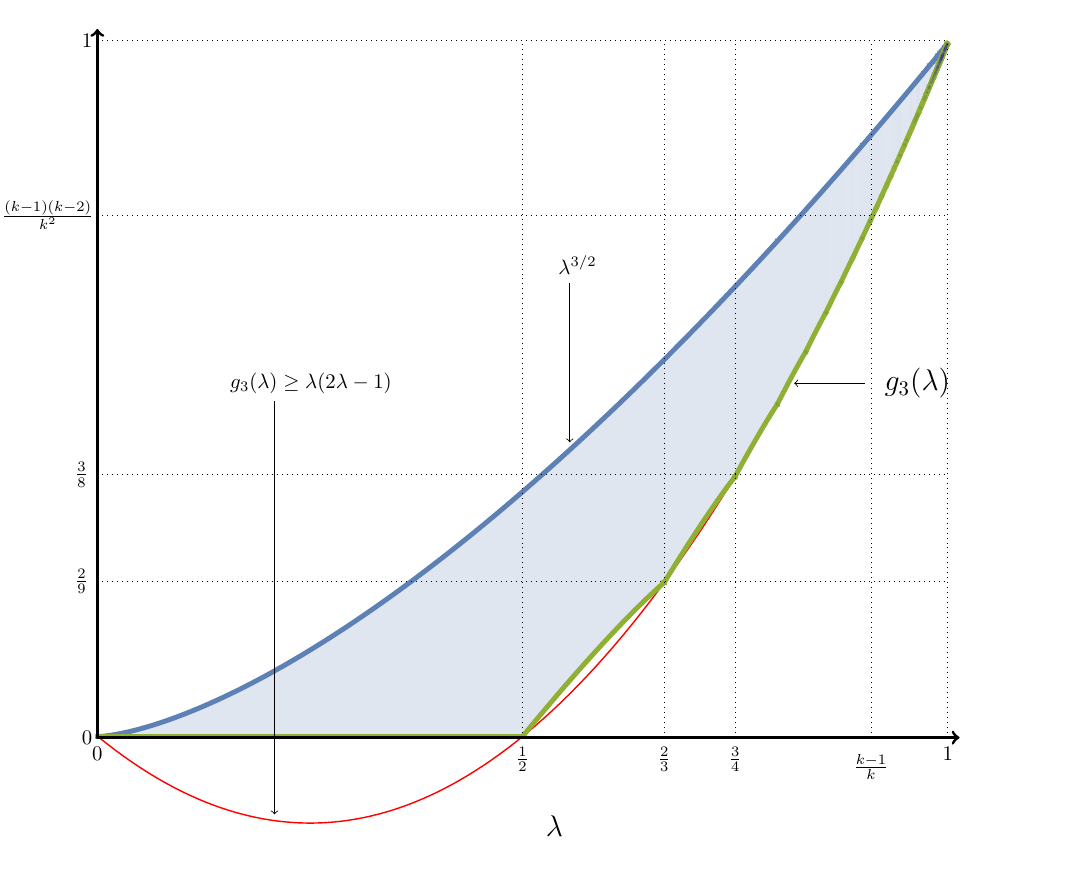}
\caption{The green function is $g_3(\lambda)$, as determined by Theorem~\ref{razthm}. The red curve is Goodman's bound~(\ref{"Goodman"}). The blue curve $\lambda^{3/2}$ is asymptotically the maximum triangle density in a graph of edge density $\lambda$. This follows easily from the Kruskal-Katona theorem~\cite{Katona64,Kruskal63}.
}
\label{plot}
\end{figure}
\end{center}

Another property that makes this problem difficult is that in general there are many asymptotically extremal $(n,e)$-graphs, as the following family demonstrates.

\begin{definition}[Family $\mathcal{H}^*(n,e)$]
\label{df:CH*}
Given $n,e\in\I N$ with $e\le {n\choose 2}$, let $k=k(n,e)$, $\bm{a}^*=(a^*_1,\ldots,a^*_k)$ and $m^*$ be as in Definition~\ref{astardef}.
The family $\mathcal{H}^*(n,e):=\bigcup_{i=0}^2 \mathcal{H}_i^*(n,e)$ is defined as the union of the following three families. Let $T:=K[A_1^*,\dots,A_k^*]$ be the complete partite graph with part sizes $a_1^*\ge \dots\ge a_k^*$ respectively.
\begin{description}
 \item[$\mathcal{H}^*_1(n,e)$] If $m^*=0$, then take all graphs obtained from $T$ by replacing, for some $i\in [k-1]$, $T[A_i^*\cup A_k^*]$ with an arbitrary triangle-free graph with $a_i^*a_k^*$ edges. If $m^*> 0$, take all graphs obtained from $T$ by replacing $T[A_{k-1}^*\cup A_k^*]$ with an arbitrary triangle-free graph with $a_{k-1}^*a_k^*-m^*$ edges.
 
\item[$\mathcal{H}_0^*(n,e)$] Take the family $\mathcal{H}^*_1(n,e)$ and, if $a_k^*=1$, add all graphs obtained from $K_{a_1^*,\dots,a_{k-2}^*,a_{k-1}^*+1}$ by adding a triangle-free graph with $a_{k-1}^*-m^*$ edges such that each added edge lies inside some part of size $a_{k-1}^*+1$.

\item[$\mathcal{H}_2^*(n,e)$] Take those graphs in $\mathcal{H}_1^*(n,e)$ which are $k$-partite, along with the following family. Take disjoint sets $A_1,\dots,A_k$ of sizes $a_1^*,\dots,a_k^*$ respectively and let $m:=m^*$. If $m^*=0$ and $a_1^*\geq a_k^*+2$, then we also allow $(|A_1|,\dots,|A_k|)=(a_2^*,\dots,a_{k-1}^*,a_1^*-1,a_k^*+1)$ and let $m:=a_1^*-a_k^*-1$. Take all graphs obtained from $K[A_1,\ldots,A_k]$ by removing $m$ edges, each connecting $B_i$ to $A_{i}$ for some $i\in I$, where $I:=\{i\in [k-1]: |A_i|=|A_{k-1}|\}$ and $(B_i)_{i\in I}$ are some disjoint subsets of~$A_k$.
\end{description}
\end{definition}

One can check by the definition that every graph in $\mathcal{H^*}(n,e)$ 
has $e$ edges and $h^*(n,e)$ triangles. Also, the graph $H^*(n,e)$ belongs to $\mathcal{H}^*_i(n,e)$ for each $i\in\{0,1,2\}$.  Proposition~\ref{pr:compute} and Conjecture~\ref{conj}, to be stated shortly, will motivate the above  definitions.
Note that every graph in $\mathcal{H}_0^*(n,e)\setminus \mathcal{H}_1^*(n,e)$ has at most $a^*_{k-1}-m^* \leq \frac{n-1}{k-1}$ more edges than the Tur\'an graph $T_{k-1}(n)$. In other words,
\begin{equation}\label{eq:H0=H1}
\mathcal{H}_0^*(n,e)=\mathcal{H}_1^*(n,e),\quad \mbox{for\ \ $t_{k-1}(n)+\frac{n-1}{k-1}<e\le t_k(n)$.}
\end{equation}

In general, $\mathcal{H^*}(n,e)$ contains many non-isomorphic graphs. Nonetheless, a `stability' result was established by Pikhurko and Razborov~\cite{PikhurkoRazborov17} who showed that every almost extremal $(n,e)$-graph is within edit distance $o(n^2)$ from $\mathcal{H}_1^*(n,e)$ (or, equivalently, from $\mathcal{H}^*(n,e)$).

\begin{theorem}[\cite{PikhurkoRazborov17}]
\label{PikhurkoRazborov17approx}
For every $\eps > 0$, there are $\delta,n_0 > 0$ such that, for every $(n,e)$-graph $G$ with $n \geq n_0$ vertices and at most $g_3(n,e)+\delta\binom{n}{3}$ triangles, there exists $H \in \mathcal{H}_1^*(n,e)$ such that $|E(G)\bigtriangleup E(H)| \leq \eps\binom{n}{2}$.
\end{theorem}

While Theorems~\ref{razthm} and~\ref{PikhurkoRazborov17approx} deal only with the asymptotic values, they can also be used to derive some exact results. Namely, if $n=(k-1)a+b$ where $k,a,b\in\I N$ with $a\ge b$ and $e={k-1\choose 2}a^2+(k-1)ab=e(K^k_{a,\dots,a,b})$, then
 \begin{equation}\label{eq:SomeExact}
  g_3(n,e)=K_3(K^k_{a,\dots,a,b})={k-1\choose 3}a^3+{k-1\choose 2} a^2b.
  \end{equation}
  Indeed, if some $(n,e)$-graph $H$ violates the lower bound, then the uniform blow-ups of $H$ violate Theorem~\ref{razthm}; furthermore, every extremal $(n,e)$-graph contains the complete $(k-1)$-partite graph $K^{k-1}_{a,\dots,a,a+b}$ as a spanning subgraph, as otherwise its blow-ups violate Theorem~\ref{PikhurkoRazborov17approx}. 
  
  The above blow-up trick also shows that $g_3(n,e)\ge (n^3/6)\, g_3(2e/n^2)$ for every $(n,e)$. Although, for $e>t_2(n)$, one can show that this bound is tight only when the pair $(n,e)$ is as in~\eqref{eq:SomeExact}, it gives a rather good approximation to $g_3(n,e)$. Namely, calculations based on the explicit formula for $g_3(\lambda)=h^*(\lambda)$ (see e.g.~\cite[Theorem~1.3]{Nikiforov11}) give that
  \begin{equation}\label{asympbound}
 0\ \leq\ 
g_3(n,e) - \frac{n^3}{6}\, g_3\!\left(\frac{2e}{n^2}\right)\  \leq\  \frac{n^3}{n^2-2e},\quad n,e\in\I N,\ e\le {n\choose 2}.
\end{equation}

In a long and difficult paper, Lov\'asz and Simonovits~\cite{LovaszSimonovits83} established the exact result for a large range of parameters. In order to state their main result, we have to define some graph families (which will also appear in our results and proofs).

\begin{definition}[Families $\mathcal{H}_0$, $\mathcal{H}_1$, $\mathcal{H}_2$ and $\mathcal{H}$]\label{df:H}
Given positive integers $e,n$ with $e \leq \binom{n}{2}$, let $k=k(n,e)$ be as in~\eqref{eq:k} and define the following families.
\begin{description}
\item[$\mathcal{H}_0(n,e)$] the family of $(n,e)$-graphs $H$ obtained from adding a triangle-free graph $J$ to  a complete $(k-1)$-partite graph on $n$ vertices.
\item[$\mathcal{H}_1(n,e)$] the family of $(n,e)$-graphs $H$ with a partition $V(H)=A_1\cup \ldots\cup A_{k-2}\cup B$ such that $|A_1| \geq \ldots \geq |A_{k-2}|$; $H[A_1 \cup \ldots \cup A_{k-2}]$ is the complete partite graph $K[A_1,\ldots,A_{k-2}]$; $H[B,V(H)\setminus{B}]$ is complete; and $H[B]$ is a triangle-free graph.

\item[$\mathcal{H}_2(n,e)$] the family of $k$-partite $(n,e)$-graphs $H$ with a partition $A_1,\ldots,A_k$ of $V(H)$ such that $|A_1| \geq \ldots \geq |A_k|$; $H[A_1\cup\dots \cup A_{k-1}]=K[A_1,\dots,A_{k-1}]$,
and for every vertex $x \in A_k$ there is at most one $j \in [k-1]$ such that $x$ is not complete to $A_j$.
\end{description}
 Also, let $\mathcal{H}(n,e) := \mathcal{H}_1(n,e) \cup \mathcal{H}_2(n,e)$ and define
  \begin{equation}\label{eq:h}
  h(n,e) := \min\lbrace K_3(H):H \in \mathcal{H}(n,e)\rbrace.
  \end{equation} 
\end{definition}

Note that $\mathcal{H}_1(n,e)\subseteq\mathcal{H}_0(n,e)$; this inclusion is in general strict as the added edges in the definition of $\mathcal{H}_0(n,e)$ can lie inside different parts.

The main result proved by Lov\'asz and Simonovits~\cite{LovaszSimonovits83} (first announced in their 1975 paper~\cite{LovaszSimonovits76}) is the following.

\begin{theorem}[\cite{LovaszSimonovits76,LovaszSimonovits83}]\label{LS83}
For all integers $k \geq 3$ and $r \geq 3$, there exist $\alpha = \alpha(r,k) > 0$ and $n_0=n_0(r,k)>0$ such that, for all positive integers $(n,e)$ with $n \geq n_0$ and $t_{k-1}(n) < e \leq t_{k-1}(n) + \alpha n^2$, we have that
$$
g_r(n,e)=h_r(n,e) := \min\lbrace K_r(H) : H \in \mathcal{H}(n,e)\rbrace.
$$
If $r=3$, then every extremal graph lies in $\mathcal{H}_0(n,e) \cup \mathcal{H}_2(n,e)$, and there is at least one extremal graph in $\mathcal{H}_1(n,e)$.
If $r \geq 4$, then every extremal graph lies in $\mathcal{H}_1(n,e) \cup \mathcal{H}_2(n,e)$.
\end{theorem}

Although the proof of Theorem~\ref{LS83} does not use the Regularity Lemma, the constant $\alpha(r,k)$ given by it is nonetheless so small that Lov\'asz and Simonovits~\cite[Page 465]{LovaszSimonovits83} write that they ``did not even dare to estimate'' $\alpha(3,3)$. In the same papers~\cite{LovaszSimonovits76,LovaszSimonovits83}, the following bold conjecture was stated.

\begin{conjecture}[\cite{LovaszSimonovits76,LovaszSimonovits83}]
\label{cj:LS}
For all integers $r \geq 3$, there exists $n_0 = n_0(r)>0$ such that $g_r(n,e) = h_r(n,e)$ for all positive integers $n \geq n_0$ and $e \leq \binom{n}{2}$.
\end{conjecture}

Of course, the triangle-minimisation problem for such a restricted class as any of $\mathcal{H}_i(n,e)$ is much easier than the unrestricted function $g_3(n,e)$. In fact, we can solve it exactly.

\begin{proposition}\label{pr:compute} For $i\in \{0,1,2\}$ and all $n,e\in \I N$ with $e\le {n\choose 2}$, we have that
	$\min\lbrace K_3(H):H \in \mathcal{H}_i(n,e)\rbrace=h^*(n,e)$ and $\mathcal{H}_i^*(n,e)$ is the set of graphs in $\mathcal{H}_i(n,e)$ that attain this bound.
	
	In particular, we have that $h(n,e)=h^*(n,e)$.
\end{proposition}

An interesting consequence of Proposition~\ref{pr:compute} that has not been observed before is that, for $r=3$, if Conjecture~\ref{cj:LS} is true, then its conclusion is in fact true for \textbf{all} $n\ge 1$, see Lemma~\ref{lm:AlmostComplete}.

Apart from some cases when $e$ is very close to $\binom{n}{2}$, to the best of the authors' knowledge, all established cases of the conjecture are confined to the direct consequences of Theorem~\ref{razthm} via the blow-up trick and to Theorem~\ref{LS83} (the latter superseding, as $n\to\infty$, all remaining exact results that we mentioned). The main contribution of this paper is to prove the conjecture when $r=3$ and $e/{n\choose 2}$ is bounded away from $1$, and to characterise the extremal graphs in this range.

\begin{theorem}\label{main}
For all $\eps>0$, there exists $n_0>0$ such that for all positive integers $n \geq n_0$ and $e \leq \binom{n}{2}-\eps n^2$, we have that 
$g_3(n,e) = h(n,e)$. Furthermore, the family of extremal $(n,e)$-graphs is precisely $\mathcal{H}_0^*(n,e) \cup \mathcal{H}_2^*(n,e)$.
\end{theorem}

By Theorem~\ref{LS83} and Proposition~\ref{pr:compute}, it is enough to prove Theorem~\ref{main} when $e\ge t_{k-1}(n)+\Omega(n^2)$, where $k=k(n,e)$. This is done in the next theorem. (Note that $\mathcal{H}_0$ is irrelevant in this range by~\eqref{eq:H0=H1}.)

\begin{theorem}\label{strong}
	For all $\eps,\alpha > 0$ and every integer $3 \leq k \leq 1/\eps$, there exists $n_0>0$ such that the following holds.
	For all integers $n,e$ with $n \geq n_0$ and $t_{k-1}(n) + \alpha n^2 \leq e < t_k(n)$, we have $g_3(n,e)=h(n,e)$ and every extremal graph lies in $\mathcal{H}(n,e) =\mathcal{H}_1(n,e) \cup \mathcal{H}_2(n,e)$.
\end{theorem}

We believe that the following strengthening of the case $r=3$ of Conjecture~\ref{cj:LS} holds where, additionally, the exact structure of all extremal graphs is described.

\begin{conjecture}\label{conj}
For all positive integers $n$ and $e \leq \binom{n}{2}$, an $(n,e)$-graph $G$ satisfies $K_3(G) = g_3(n,e)$ if and only if $G \in \mathcal{H}_0^*(n,e) \cup \mathcal{H}_2^*(n,e)$.
\end{conjecture}

\subsection{Organisation of the paper}

We collect some frequently used notation in Section~\ref{notation} (and there is a symbolic glossary at the end of the paper).
Theorem~\ref{main} is formally derived from Theorem~\ref{strong} in Section~\ref{derivation}.
Since the proof of Theorem~\ref{strong} is very involved and long, we provide a sketch in Section~\ref{sketch}, and also try to provide all details in calculations.
In Section~\ref{prelims}, we investigate the function $h(n,e)$ and provide some preliminary tools that will be used later on, in particular we prove Proposition~\ref{pr:compute}. The proof of Theorem~\ref{strong} begins in Section~\ref{beginning}.
Sections~\ref{int1}--\ref{int3} continue the proof in the `intermediate' case, which, roughly speaking, is when $e$ is bounded away from any Tur\'an density. The remaining `boundary' case is dealt with in Section~\ref{bound}. Some concluding remarks can be found in Section~\ref{ConcludingRemarks}.

\section{Notation}\label{notation}

Given a set $X$ and $k \in \mathbb{N}$, let $\binom{X}{k}$ denote the set of $k$-subsets of $X$. Also, $[k]:=\{1,\dots,k\}$. We may abbreviate $\{a,b\}$ to $ab$.
 We write $x = y \pm \eps$ if $y-\eps \leq x \leq y+\eps$.

We use standard graph theoretic notation.
Given a graph $G$ and $A \subseteq V(G)$, we write $\overline{A} := V(G)\setminus A$ for the \emph{complement of $A$ in $G$} and $\overline{G}$ for the graph with vertex set $V(G)$ and edge set $\binom{V(G)}{2}\setminus E(G)$, which we call the \emph{complement of~$G$}.
Further, we write $G[A]$ for the graph induced by $G$ on $A$.
Given disjoint $A,B \subseteq V(G)$, we write $G[A,B]$ for the graph with vertex set $A \cup B$ and edge set $\lbrace ab \in E(G):a \in A, b \in B\rbrace$.
For $x \in V(G)$ and $A \subseteq V(G)$, we set $N_G(x,A) := \lbrace y \in A: xy \in E(G)\rbrace$ and $d_G(x,A) := |N_G(x,A)|$. 
Additionally, we write $N_G(x) := N_G(x,V(G))$ and $d_G(x) := |N_G(x)|$.
Given pairwise-disjoint vertex sets $A_1,\ldots,A_\ell$, we write $K[A_1,\ldots,A_\ell]$ for the complete partite graph with parts $A_1,\ldots,A_\ell$.
When $a_1,\ldots,a_\ell$ are integers, we write $K^\ell_{a_1,\ldots,a_\ell}$ (or $K_{a_1,\ldots,a_\ell}$) for the complete $\ell$-partite graph with parts of sizes $a_1,\ldots,a_\ell$.

A partition of $V(H)$ witnessing that $H\in \mathcal{H}_i(n,e)$ in Definition~\ref{df:H} will be called $\mathcal{H}_i$-\emph{canonical} (or just \emph{canonical}).

Given $x \in V(G)$, we write $K_3(x,G)$ for the number of triangles in $G$ which contain $x$. That is,
$$
K_3(x,G) := e(G[N_G(x)]).
$$
Given $A_1,A_2 \subseteq V(G)\setminus \lbrace x \rbrace$ we write $K_3(x,G;A_1,A_2)$ for the number of triples $\lbrace x,a_1,a_2\rbrace$ that span a triangle in $G$, where $a_i \in A_i$ for $i \in [2]$.
(Note we do not double count when both $a_1,a_2$ lie in $A_1 \cap A_2$.)
If $A_1=A_2=A$ we let $K_3(x,G;A) := K_3(x,G;A,A)$.
Similarly, given $\lbrace x,y \rbrace \in \binom{V(G)}{2}$, let $P_3(xy,G)$ be the number of $3$-vertex paths with endpoints $x$ and $y$; i.e.
$$
P_3(xy,G) := |N_G(x) \cap N_G(y)|.
$$
Let $P_3(xy,G;A) := |N_G(x,A) \cap N_G(y,A)|$.
Given a graph $G$ with vertex partition $A_1,\ldots,A_k$, a \emph{cross-edge} is any edge which lies between parts.
Given two graphs $G,H$ on the same vertex set $V$ and $U \subseteq V$, we say that $G$ and $H$ \emph{only differ at $U$} if $E(G)\bigtriangleup E(H) \subseteq \binom{U}{2}$.

 Given a family $\mathcal{G}(n,e)$ of $(n,e)$-graphs, we write $\mathcal{G}^{\mathrm{min}}(n,e)\subseteq \mathcal{G}(n,e)$ for the subfamily consisting of all graphs with the minimum number of triangles. 

Since we are interested in the case $r=3$, we will say that a pair $(n,e)$ is \emph{valid} if $n,e \in \mathbb{N}$ are such that $\lfloor \frac{n^2}{4}\rfloor < e \leq \binom{n}{2}$ (i.e.~there exist graphs with $n$ vertices and $e$ edges, and every such graph contains at least one triangle).

Given $\ell \in \mathbb{N}$ and $\alpha_1,\ldots,\alpha_\ell \in \mathbb{R}$, for convenience we write 
$$
e(K^\ell_{\alpha_1,\ldots,\alpha_\ell}) := \sum_{ij \in \binom{[\ell]}{2}}\alpha_i\alpha_j\quad\text{ and }\quad K_3(K^\ell_{\alpha_1,\ldots,\alpha_\ell}) := \sum_{hij \in \binom{[\ell]}{3}}\alpha_h\alpha_i\alpha_j
$$
in analogy with the number of edges and triangles in the complete $\ell$-partite graph $K^\ell_{n_1,\ldots,n_\ell}$ which is defined when the $n_i$'s are positive integers.

The \emph{edit distance} between two graphs $G$ and $H$ on the same vertex set is $|E(G) \bigtriangleup E(H)|$, and these graphs are said to be \emph{$d$-close} if $|E(G) \bigtriangleup E(H)| \leq d$.

 Write $x = a \pm \eps$ if $a-\eps \leq x \leq a+\eps$.
 We will ignore floors and ceilings where they do not affect our argument.

\section{Sketch of the proof of Theorem~\ref{strong}}\label{sketch}

The asymptotic results of Fisher~\cite{Fisher89}, Razborov~\cite{Razborov08}, Nikiforov~\cite{Nikiforov11}, Pikhurko-Razborov~\cite{PikhurkoRazborov17} and Reiher~\cite{Reiher16} all use spectral or analytic methods. 
Such techniques do not seem to be helpful for the exact problem, and indeed our proof of Theorem~\ref{strong} uses purely combinatorial methods.
At its heart, our proof uses the well-known stability method: Theorem~\ref{PikhurkoRazborov17approx} implies that any extremal graph $G$ is structurally close to \emph{some} $H$ in $\mathcal{H}^*(n,e)$ and hence some graph in $\mathcal{H}_1(n,e)$.
Then the goal would be to analyse $G$ and show that it cannot contain any imperfections and must in fact lie in $\mathcal{H}_1(n,e)$. 
The stability approach stems from work of Erd\H{o}s~\cite{Erdos67a}
and Simonovits~\cite{Simonovits68} and has been used to solve many major problems in extremal combinatorics.

However, a major obstacle here is the fact that there is a large family of conjectured extremal graphs.
Given any $H \in \mathcal{H}_1(n,e)$ with canonical partition $A_1,\ldots,A_{k-2},B$ as in the definition, one can obtain a different $H' \in \mathcal{H}_1(n,e)$ such that $K_3(H')=K_3(H)$ simply by replacing $H[B]$ with another triangle-free graph containing the same number of edges.
In general, there are many choices for this triangle-free graph.

An additional difficulty is that $\mathcal{H}_1(n,e)$ does not in fact contain every extremal graph, as in Theorem~\ref{LS83}.
So our goal as stated above must be modified.

\medskip
Let us present a brief outline of the proof of Theorem~\ref{strong}.
Suppose that Theorem~\ref{strong} is false. Let $k$ be the minimum integer for which there is an arbitrarily large integer $n$ and some $e$ with $t_{k-1}(n) < e \leq t_k(n)$ such that $\mathcal{H}(n,e)$ does not contain every extremal graph.
Choose a fixed large $n$ and then $e$ as above such that $g_3(n,e)-h(n,e)\leq 0$ is minimal, and let $G \notin \mathcal{H}(n,e)$ be an $(n,e)$-graph with $K_3(G)=g_3(n,e)$.
We call such a $G$ a \emph{worst counterexample}.
One consequence of the choice of $G$ is, for example, that no edge can lie in too many triangles, and the endpoints of every non-edge have many common neighbours.

\medskip
\noindent
\textbf{I: The intermediate case} $t_k(n)-e=\Omega(n^2)$.

\medskip
\noindent\emph{1. Approximate structure (Section ~\ref{int1})}

\medskip
\noindent
Theorem~\ref{PikhurkoRazborov17approx} implies that $G$ is close in edit distance to some graph $H \in \mathcal{H}^*(n,e)$.
Note that $H \in \mathcal{H}_1(n,e')$ for some $e'$ which is close to $e$.
The first step is to show that actually $G$ is close to the specific graph $H^*(n,e)$ (namely, the edit distance is $o(n^2)$; see Lemma~\ref{lem-Almost-k-Partite}).
The $i$th part of $H^*(n,e)$ has size $a_i^*$, which is roughly $cn$ for all $i \in [k-1]$ (Lemma~\ref{lem-slope-c}).
Since $e$ is bounded away from $t_k(n)$, it is not hard to see that $n-(k-1)cn < cn - \Omega(n)$.
So $G$ is close to a complete partite graph with one small part and the other parts equally-sized.
In fact we can show (Lemma~\ref{approx}) that every max-cut partition $A_1,\ldots,A_k$ of $G$ is such that $|\,|A_i|-cn\,|=o(n)$ for $i \in [k-1]$ (and $|\,|A_k|-(n-(k-1)cn)\,| = o(n)$) and $m+h = o(n^2)$ where 
$$
m := \sum_{ij \in \binom{[k]}{2}}e(\overline{G}[A_i,A_j])\quad\text{and}\quad h := \sum_{i \in [k]}e(G[A_i]).
$$
Following~\cite{LovaszSimonovits83}, we say that any pair of vertices in different parts which does not span an edge is a \emph{missing edge}, and any edge inside a part is \emph{bad}.
As usual, we now identify some vertices which are atypical in the sense that they are incident to many missing edges.
Let $Z$ be the set of vertices incident with $\Omega(n)$ missing edges. Thus
\begin{equation}\label{proofZeq}
|Z| = O(m/n) = o(n).
\end{equation}
It turns out that every bad edge is incident to a vertex in $Z$.
Thus, if $Z=\emptyset$, then $G$ is $k$-partite and it is not hard to show (see Corollary~\ref{H2new}(i)) that every extremal $k$-partite $(n,e)$-graph lies in $\mathcal{H}_2(n,e)$, a contradiction.

\medskip
\noindent\emph{2. Transformations (Section ~\ref{sec:trans})}

\medskip
\noindent
Now we would like to make a series of local changes to $G$ to obtain a new $n$-vertex $e$-edge graph $G'$ such that $K_3(G')-K_3(G)=0$ but the structure of $G'$ is much simpler.
Here, `simpler' means `no bad edges', so $G'$ would be $k$-partite, and we would obtain our desired contradiction.
From the property of $Z$ above, these local changes would then only have to be made at $Z$.
Unfortunately this is too ambitious as we do not have fine enough control on the structure of the graph. Therefore we reduce our expectations and aim to find $G'$ such that $K_3(G')-K_3(G)$ is \emph{small} (Lemma~\ref{maintrans}).
That is, we simplify the structure (and thus it is easier to count triangles) at the expense of a few additional triangles.
To be more precise, \emph{small} means $o(m^2/n)$.
Although the transformations themselves are easy to describe, this is the longest and most technical part of the proof.

\begin{itemize}
	\item Transformation 1 (Figure~\ref{partition1}, Lemmas~\ref{subZiedges} and~\ref{Ziedges}): Removing bad edges in the large parts $A_1,\ldots,A_{k-1}$.
	\item Transformation 2 (Figure~\ref{partition2}, Lemmas~\ref{subYiedges} and~\ref{Yiedges}): Reassigning those vertices in $Z \cap A_k$ incident to many missing edges to a large part.
	\item Transformations 3--6 (Figures~\ref{partition3}--\ref{partition6}, Lemmas~\ref{subXiedges}--\ref{G5} and the proof of Lemma~\ref{maintrans}): Dealing with those vertices in $Z \cap A_k$ incident to few missing edges. 
\end{itemize}

\medskip
\noindent\emph{3. Finishing the proof in this case (Section~\ref{int3})}

\medskip
\noindent
\emph{i. Suppose that $m > Cn$ for some large constant $C$ (Section~\ref{sec:superlinear}).}
Write $A_1'',\ldots,A_k''$ for the parts of $G'$.
Keeping track of the transformation $G \rightarrow G'$ allows us to use $G'$ to obtain additional structural information about $G$.
To do this, we apply Lemma~\ref{CompletekPartite}, which measures the difference in the numbers of triangles between a $k$-partite $(n,e)$-graph (such as $G'$) and an `ideal' $k$-partite graph (which is essentially $H^*(n,e)$). 
Because the same is true in $G$ in the intermediate case, the difference in size between the smallest part of $G'$ and the other parts is $\Omega(n)$.
In Lemma~\ref{superlinear1}, this fact and $K_3(G')-K_3(H^*(n,e)) \leq K_3(G')-K_3(G) = o(m^2/n)$ imply via Lemma~\ref{CompletekPartite} that $e(\overline{G'}[A_i'',A_k'']) = \Omega(m)$ for exactly one $i \in [k-1]$, and the other $A_j''$ satisfy $|\,|A_j''|-cn\,| = o(m/n)$ and $|Z \cap A_j''|= o(m/n)$ (which is much stronger than~(\ref{proofZeq})).

Since we had fine control on the transformation $G \rightarrow G'$, similar statements hold in $G$ (Lemma~\ref{lem-size}): $e(\overline{G}[A_i,A_k]) = \Omega(m)$ for exactly one $i \in [k-1]$, and the other $A_j$ satisfy $|\,|A_j|-cn\,|=o(m/n)$ and $|Z \cap A_j| = o(m/n)$. 
This new information about $G$ is substantial enough to show that most of the local changes we did earlier actually \emph{decrease} the number of triangles.
This applies~e.g.~to Transformation~1, and we conclude that $Z \cap A_j = \emptyset$ for all $j \in [k-1]\setminus \lbrace i \rbrace$. So $A_j$ contains no bad edges (Lemma~\ref{empty}).
This analysis requires tight `step-by-step' control on the effect of the transformations, which is what makes the proofs more technical than they would otherwise have to be.
Then a final global change (see Figure~\ref{globaltrans}) brings us to a graph $H \in \mathcal{H}_1(n,e)$ which, if $Z \neq \emptyset$, satisfies $K_3(H)-K_3(G) < 0$, a contradiction.

\medskip
\noindent
\emph{ii. Suppose that $m < Cn$ (Section~\ref{sec:linear}).} This case is different as the errors stemming from $G'$ are too large to allow us to glean any extra information.
Instead, we show directly that most of the transformations we did earlier \emph{do not increase} the number of triangles.
This is possible since we now know that~e.g.~$Z$ has constant size (see~(\ref{proofZeq})).

This case has a different flavour because we may enter the situation where,~e.g.~after performing Transformation~1 to obtain $G_1$, we have $K_3(G_1)=K_3(G)$ and $G_1 \in \mathcal{H}(n,e)$.
Then we have to argue that in fact this must imply $G \in \mathcal{H}(n,e)$, a  contradiction.
This is the only part of the proof where we are not able to obtain a contradiction by strictly decreasing the number of triangles, and must actually analyse the extremal family $\mathcal{H}(n,e)$ (Section~\ref{charext}).

\medskip
\noindent
\textbf{II: The boundary case} $t_k(n)-e=r$ where $r=o(n^2)$ (\emph{Section~\ref{bound}}).

\medskip
\noindent
The proof in this case turns out to be much shorter than the intermediate case.
We now have that $cn = n/k+O(\sqrt{r})$.
A different argument is required to determine the approximate structure of $G$ as we need better bounds in terms of $r$: we use an averaging argument (Lemma~\ref{almostcompleteagain}) which is very similar to Theorem~2 in~\cite{LovaszSimonovits83}.
Thus we obtain a rather strong structure property (Lemma~\ref{approx2}): every max-cut partition $A_1,\ldots,A_k$ of $G$ is such that $|\,|A_i|-n/k|=O(\sqrt{r})$ for \emph{all} $i \in [k]$, and $\sum_{ij \in \binom{[k]}{2}}e(\overline{G}[A_i,A_j])+\sum_{i \in [k]}e(G[A_i])=O(r)$. 

Again, we let $Z$ be the set of vertices with $\Omega(n)$ missing edges, and show that $|Z|=o(n)$ and every bad edge is incident to a vertex in $Z$.
In the intermediate case, the most troublesome vertices were those in $Z \cap A_k$ dealt with in Transformations~3--6.
Now, $A_k$ is not substantially smaller than the other parts, so this is no longer the case and some difficulties from the intermediate case disappear.

We show that, for every $i \in [k]$, the set $A_i \setminus Z$ is `significantly smaller' than $cn$. This then implies that $G[A_1\setminus Z,\ldots,A_k\setminus Z]$ is complete partite (Lemma~\ref{final}). Finally we show that $Z=\emptyset$, completing the proof as before.
For these final steps, we again build up a repository of structural information by performing (much simpler) transformations which strictly decrease the number of triangles unless a desired property holds.

\section{Extremal families and preliminary tools}\label{prelims}


One of the main results of this section is to prove Proposition~\ref{pr:compute} that for all $i=0,1,2$, we have $\mathcal{H}^{\mathrm{min}}_i(n,e) = \mathcal{H}^*_i(n,e)$, and $h(n,e)=h^*(n,e)$ for all valid pairs $(n,e)$. In order to do this, we present some auxiliary definitions and results first.

\subsection{Extremal $k(n,e)$-partite graphs}\label{H2'sec}
The main conclusion of this section will be Corollary~\ref{H2new} which states that all extremal $k(n,e)$-partite $(n,e)$-graphs lie in $\mathcal{H}_2(n,e)$ and at least one such graph is in $\mathcal{H}_1(n,e)$. 

In order to prove it, we need to define a somewhat related family $\mathcal{H}_2'(n,e)$.
Given a valid pair $(n,e)$, let $k := k(n,e)$.
Define $\mathcal{H}_2'(n,e)$ to be the family of $k$-partite $(n,e)$-graphs $H$ with parts $A_1,\ldots,A_k$ of sizes $|A_1|\ge\dots\ge |A_k|$ such that
\begin{enumerate}
\item for all $i \in [k]$ and $x \in A_i$, there is at most one $j \in [k]\setminus \lbrace i \rbrace$ such that $d_{\overline{H}}(x,A_j)>0$;
	\item if $|A_i|+|A_j| > |A_{k-1}|+|A_{k}|$, then 	
	$H[A_i,A_j]$ is complete.
\end{enumerate}
We say that $A_1,\ldots,A_k$ is an \emph{$\mathcal{H}_2'$-canonical partition}.
The above definition is motivated by the following easy lemma.

\begin{lemma}\label{H2}
	Let $(n,e)$ be valid and let $k=k(n,e)$.
	Let $\mathcal{G}(n,e)$ be the set of $k$-partite $(n,e)$-graphs.
	Then $\mathcal{G}^{\mathrm{min}}(n,e) \subseteq \mathcal{H}_2'(n,e)$.
\end{lemma}

\begin{proof}
	Let $G \in \mathcal{G}^{\mathrm{min}}(n,e)$.
	Let $A_1,\ldots,A_k$ be the parts of $G$, where $a_i := |A_i|$ for all $i \in [k]$ and $a_1 \geq \ldots \geq a_k$.
	Let $m:=\sum_{ij \in \binom{[k]}{2}}e(\overline{G}[A_i,A_j])$.
	
	We have that $m\le  a_{k-1}a_k$, for otherwise 
	$$
	e < e(K_{a_1,\dots,a_{k-2},a_{k-1}+a_k})\le t_{k-1}(n)
	$$ and so $k(n,e) \leq k-1$, a contradiction. Consider $G^* := K[A_1,\ldots,A_k]\setminus E^*$, where $E^*$ consists of some $m$ edges of $K[A_{k-1},A_k]$.
	Clearly, $G^*\in \mathcal{G}(n,e)$. Thus, by the minimality of $G\in \mathcal{G}(n,e)$, we have $K_3(G^*)\ge K_3(G)$ .
	On the other hand, since each pair of $E^*$ is in exactly $a_1+\ldots + a_{k-2}$ triangles of $K[A_1,\ldots,A_k]$ and no such triangle is counted more than once, we have
	\begin{align}
	K_3(G^*)-K_3(G) &=(K_3(K[A_1,\ldots,A_k])-K_3(G))-(K_3(K[A_1,\ldots,A_k])-K_3(G^*))\nonumber\\
	&\leq \sum_{ij \in \binom{[k]}{2}}e(\overline{G}[A_i,A_j]) \left( \sum_{h \in [k]\setminus \lbrace i,j\rbrace}a_h \right) - |E^*|(a_1+\ldots + a_{k-2}) \nonumber\\
	&= \sum_{ij \in \binom{[k]}{2}}e(\overline{G}[A_i,A_j]) \left( \sum_{h \in [k]\setminus \lbrace i,j\rbrace}a_h - (a_1+\ldots + a_{k-2}) \right)\nonumber\\
	&= \sum_{ij \in \binom{[k]}{2}}e(\overline{G}[A_i,A_j]) \left((a_{k-1}+a_k)-(a_i+a_j)\right)\ \le\ 0,\label{eq:G*vsG}
	\end{align}
    so we have equality throughout. The sharpness of the first (resp.\ second) inequality in~\eqref{eq:G*vsG} implies the first (resp.\ second) property from the 
    definition of $\mathcal{H}_2'(n,e)$. 
	Thus $G \in \mathcal{H}_2'(n,e)$, as required.
\end{proof}

We also need the following result concerning extremal graphs in $\mathcal{H}_2'(n,e)$.

\begin{lemma}\label{H2missing}
	Let $(n,e)$ be valid with $k=k(n,e)$. Let $H \in (\mathcal{H}_2')^{\mathrm{min}}(n,e)$ with an $\mathcal{H}_2'$-canonical partition $A_1,\ldots,A_k$ having part sizes $a_1\ge \dots\ge a_k$ respectively. Let $m:=\sum_{ij \in \binom{[k]}{2}}e(\overline{H}[A_i,A_j])$. Then the following statements hold.
	\begin{itemize}
	\item[(i)] There exists $G \in \mathcal{H}_1(n,e)\cap \mathcal{H}_2'(n,e)\cap \mathcal{H}_2(n,e)$ with $K_3(G) = K_3(H)$. 
		\item[(ii)] If $a_{k-2}=a_{k-1}$, then $m \leq a_{k-1}-a_k+1$.
	\end{itemize}
\end{lemma}

\begin{proof}
	If $m > a_{k-1}a_k$ then $e < t_{k-1}(n)$, a contradiction.
	Thus there exists $G := K[A_1,\ldots,A_k]\setminus E^*$ where $E^* \subseteq K[A_{k-1},A_k]$ and $|E^*|=m$.
	Clearly, $G \in \mathcal{H}_1(n,e)\cap \mathcal{H}_2'(n,e)\cap \mathcal{H}_2(n,e)$. Also, the calculation as in~\eqref{eq:G*vsG} shows that $K_3(G)\le K_3(H)$. This is equality by the minimality of $H$, proving the first part of the lemma.
	
	Now, let us show~(ii).
  Let $a:=a_{k-2}=a_{k-1}$.
	Suppose on the contrary that $s:= m - a+a_k-1$ is at least~$1$.
	Then $(a+1)(a_k-1)-(a a_k - m) = s \geq 1$. If $s> a(a_k-1)$, then 
	$$
	 e=e(K_{a_1,\dots,a_k})-m=e(K_{a_1,\dots,a_{k-3},a+1,a,a_k-1})-s< e(K_{a_1,\ldots,a_{k-3},a+1,a+a_k-1})\le t_{k-1}(n),
	 $$
	 a contradiction to the definition of $k$.
	Thus there is an $(n,e)$-graph $J$ obtained from the complete $k$-partite graph $K_{a_1,\ldots,a_{k-3},a+1,a,a_k-1}$ by removing $s$ edges between the last two classes (that have sizes $a$ and $a_k-1$).
	Note that $J \in \mathcal{H}_2'(n,e)$.
	But then we have
\begin{align*}
	K_3(H)-K_3(J) 
	&\ge \left(a^2a_k-(s+a-a_k+1) a\right) - \left(a(a+1)(a_k-1) - s(a+1)\right)
	=s>0.
	\end{align*}
	This contradiction completes the proof of the second part.
	\end{proof}

\begin{lemma}\label{lm:H2'min}
	Let $(n,e)$ be valid with $k=k(n,e)$. Then $(\mathcal{H}_2')^{\mathrm{min}}(n,e) = \mathcal{H}_2^{\mathrm{min}}(n,e)$. Moreover, for all graphs in this family, an $\mathcal{H}_2'$-canonical partition is an $\mathcal{H}_2$-canonical partition up to relabelling parts, and vice versa.
\end{lemma}

\begin{proof}
Throughout this proof, we omit $(n,e)$ for brevity. 

We first show that $(\mathcal{H}_2')^{\mathrm{min}} \subseteq \mathcal{H}_2^{\mathrm{min}}$. Take any  $H \in (\mathcal{H}_2')^{\mathrm{min}}$ with
an $\mathcal{H}_2'$-canonical partition $A_1,\dots,A_k$.
We claim that $H\in \mathcal{H}_2$, and some ordering of $\lbrace A_1,\ldots,A_k\rbrace$ is an $\mathcal{H}_2$-canonical partition. Assume that $|A_{k-2}|=|A_{k-1}|=|A_k|$ for otherwise $e(\overline{H}[A_i,A_j])>0$ only if $k \in \lbrace i,j\rbrace$ in which case $H \in \mathcal{H}_2$, as desired. Lemma~\ref{H2missing}(ii) gives that 
	$$
		\sum_{ij \in \binom{[k]}{2}}e(\overline{H}[A_i,A_j]) \leq |A_{k-1}|-|A_k|+1 = 1.
		$$
 Thus $H$ has at most one missing edge, which (if exists) is incident to some part $A_i$ with $|A_i|=|A_k|$. In any case, $H\in\mathcal{H}_2$ with the same canonical partition, up to relabelling, as claimed. If $H$ is not in $\mathcal{H}_2^{\mathrm{min}}$, then  any $H'\in\mathcal{H}_2^{\mathrm{min}}$ has fewer triangles than $H$. However, 
by Lemma~\ref{H2} there is $G\in \mathcal{H}_2'$ with $K_3(G)\le K_3(H')<K_3(H)$, contradicting the extremality of~$H$. In particular, writing $h_2:=K_3(F)$ and $h_2':=K_3(F')$, where $F\in \mathcal{H}_2^{\mathrm{min}}$ and $F'\in (\mathcal{H}_2')^{\mathrm{min}}$, we see that $h_2=h_2'$.

We now show the other direction, i.e.~$(\mathcal{H}_2')^{\mathrm{min}} \supseteq \mathcal{H}_2^{\mathrm{min}}$. Let $\mathcal{G}(n,e)$ be the set of $k$-partite $(n,e)$-graphs. By definition $\mathcal{H}_2\subseteq \mathcal{G}$. As $\mathcal{G}^{\min}\subseteq \mathcal{H}_2'$ due to Lemma~\ref{H2} and $h_2=h_2'$, we have that $\mathcal{H}_2^{\min}\subseteq \mathcal{G}^{\min}\subseteq (\mathcal{H}_2')^{\min}$ as desired.
Furthermore, if $A_1,\ldots,A_k$ is an $\mathcal{H}_2$-canonical partition of $G \in \mathcal{H}_2^{\mathrm{min}}$, some ordering of it is an $\mathcal{H}_2'$-canonical partition.
\end{proof}

For ease of reference, let us summarise some facts that we will need later.

\begin{corollary}\label{H2new}
	Let $(n,e)$ be valid with $k=k(n,e)$. Then the following statements hold.
	\begin{itemize}
		\item[(i)] Every extremal $k$-partite $(n,e)$-graph lies in $\mathcal{H}_2(n,e)$.
		\item[(ii)] At least one extremal $k$-partite $(n,e)$-graph lies in $\mathcal{H}_1(n,e)$.
		\item[(iii)] Let $H \in \mathcal{H}_2^{\min}(n,e)\setminus\mathcal{H}_1(n,e)$ with an $\mathcal{H}_2$-canonical partition $A_1^*,\ldots,A_k^*$. Then
		$$
		\sum_{ij \in \binom{[k]}{2}}e(\overline{H}[A_i^*,A_j^*]) \leq |A_{k-1}^*|-|A_k^*|+1\le n.
		$$
	\end{itemize}
\end{corollary}
 \begin{proof} Part~(i) (resp.~(ii)) is a direct consequence of Lemma~\ref{H2} when combined with Lemma~\ref{lm:H2'min} (resp.\ with Lemma~\ref{H2missing}(i)).
To see~(iii), let $H$ and $A_1^*,\ldots,A_k^*$ be as stated. We claim that $|A_{k-2}^*|=|A_{k-1}^*|$.
Indeed, if $|A_{k-2}^*|\ge |A_{k-1}^*|+1$, then all the missing edges in $H$ should lie in $[A_{k-1}^*,A_k^*]$ as otherwise moving all missing edges to $[A_{k-1}^*,A_k^*]$ would result in a graph still in $\mathcal{H}_2(n,e)$ having strictly fewer triangles than $H$, contradicting to the choice of $H$. But then if all missing edges lie in $[A_{k-1}^*,A_k^*]$ we have $H\in \mathcal{H}_1(n,e)$, a contradiction.
 This together with Lemma~\ref{H2missing}(ii) and Lemma~\ref{lm:H2'min} implies~(iii).
\end{proof}

For future reference, let us state here the following auxiliary lemma, which implies that if the condition on $a$ that defines $a_k^*$ in Definition~\ref{astardef} fails for some $a\le n/k$, then it fails for all smaller values of~$a\in\I N$. 

\begin{lemma}\label{lm:a*k} 
 For any integers $a\ge 1$, $k\ge 2$ and $n\ge ak$, we have
 $$
 a(n-a)+t_{k-1}(n-a)>(a-1)(n-a+1)+t_{k-1}(n-a+1).
 $$
 \end{lemma}
 \begin{proof} Let $a_1\ge \dots\ge a_{k-1}$ be the part sizes of $T_{k-1}(n-a)$. If we increase its order by 1, then the part sizes of the new Tur\'an graph, up to a re-ordering, can be obtained by increasing $a_{k-1}$ by $1$. Thus we need to estimate the following difference:
 \begin{equation}\label{eq:a*k}
 e(K_{a_1,\dots,a_{k-1},a})-e(K_{a_1,\dots,a_{k-2},a_{k-1}+1,a-1})=a_{k-1}a-(a_{k-1}+1)(a-1) = a_{k-1}-a+1,
 \end{equation}
 which is positive, since $a_{k-1}\ge \lfloor(n-a)/(k-1)\rfloor $ is at least $a$ by our assumption $a\le \lfloor n/k\rfloor$.
 \end{proof}

\subsection{Proof of Proposition~\ref{pr:compute}.} First, we describe a transformation that converts an arbitrary $\mathcal{H}_0(n,e)$-extremal graph $G$ into another extremal graph $H'$ of a rather simple structure. Then, we argue in Lemma~\ref{lem-Hlpab} that $H'$ is in fact isomorphic to the special graph $H^*(n,e)$ from Definition~\ref{astardef}. Since $H^*(n,e)\in \mathcal{H}_1(n,e)\subseteq \mathcal{H}_0(n,e)$, this determines the minimum number of triangles for graphs in these two families. From here, it is relatively easy to  derive all remaining claims of Proposition~\ref{pr:compute}.

Let $(n,e)$ be valid and set $k=k(n,e)$. Take an arbitrary graph $G \in \mathcal{H}^{\mathrm{min}}_0(n,e)$ with a vertex partition $B_1,\ldots,B_{k-1}$ such that $G$ consists of the union of $K[B_1,\ldots,B_{k-1}]$ and an edge-disjoint triangle-free graph $J$.
We say that a part $B_j$, $j\in [k-1]$, is \emph{partially full (in $G$)} if $0 < e(G[B_j]) < t_2(b_j)$, where $b_j:=|B_j|$. Since we can move edges in both directions between such parts (keeping the parts triangle-free and thus staying within the family $\mathcal{H}_0(n,e)$), we have by the minimality of $G$ that 
\begin{equation}\label{eq:T2Parts}
 b_i=b_j, \quad \text{for all $i,j\in [k-1]$ such that } B_i \text{ and } B_j \text{ are partially full}.
 \end{equation} 
 
 We construct another graph $H'=H'(G)$ in $\mathcal{H}^{\mathrm{min}}_0(n,e)$ using the following steps.
 \begin{description}
 \item[Step 1] For each partially full part $B_j$, replace $G[B_j]$ by a balanced bipartite graph of the same size (which is possible by Mantel's theorem).
 \item[Step 2] Move edges between partially full parts (keeping them balanced bipatite), until at most one  remains. By~\eqref{eq:T2Parts}, the current graph (denote it by $G_1$) is still in $\mathcal{H}_0^{\mathrm{min}}(n,e)$.
 \item[Step 3] If there is a  part $B_i$ which is partially full in $G_1$, then let $B:=B_i$; otherwise let $B:=B_i$ for some $i\in [k-1]$ with $e(G_1[B_i])=t_2(b_i)$ (such $i$ exists since $e(G_1)=e>t_{k-1}(n)$).
 \item[Step 4] As $V(G)\setminus B$ induces a complete partite graph in $G_1$, let $A_1,\ldots, A_{t-2}$ be its parts of sizes $a_1\ge \dots\ge a_{t-2}$ respectively. Thus each part $B_i$ of $G$ is equal to either $B$, or some $A_j$, or the union of some two parts $A_j\cup A_\ell$. 
 \item[Step 5] Choose integers $a_{t-1} \geq a_t \geq 1$ such that $a_{t-1}+a_t=|B|$ and $(a_{t-1}+1)(a_t-1)<e(G_1[B]) \leq a_{t-1}a_t$, which is possible since $G_1[B]$ is bipartite.
 	Let $A_{t-1},A_t$ be a partition of $B$ with $|A_i|=a_i$ for $i \in \lbrace t-1,t\rbrace$. If $e(G_1[B])=t_2(|B|)$, then we additionally require that the parts $A_{t-1}$ and $A_t$ are given by the bipartition of $G_1[B]\cong T_2(|B|)$.
 \item[Step 6] Let $H'$ be obtained from $K[A_1,\ldots,A_t]$ by removing a star centred at $A_t$ with $m'$ leaves all of which lie in $A_{t-1}$, where $m' := \sum_{ij \in \binom{[t]}{2}}a_ia_j-e=a_{t-1}a_t-e(G_1[B])$. This is possible because, like in~(\ref{m*}), we have
 \begin{equation}\label{eq:m'6}
 0 \leq m' \leq a_{t-1}-a_t.
 \end{equation}
 \end{description}

\begin{lemma}\label{lem-Hlpab}
	For every valid $(n,e)$ and $G\in \mathcal{H}^{\mathrm{min}}_0(n,e)$, the graph $H'$ produced by Steps 1--6 above is isomorphic to  $H^*(n,e)$.\end{lemma}

\begin{proof}  We will use the notation defined in Steps~1--6 (such as the sets $B_i$ and $A_i$, etc). As  $H'$ is obtained from $G_1\in \mathcal{H}^{\mathrm{min}}_0(n,e)$ by replacing a bipartite graph on $B$  with another bipartite graph of the same size (while $B$ is complete to the rest in both graphs), we have that $K_3(H')=K_3(G_1)$ and thus $H'\in \mathcal{H}_0^{\mathrm{min}}(n,e)$.

	\begin{claim}\label{cl:iv}
		If $m'=0$, then $e(H'[A_h \cup A_i \cup A_j])> t_2(|A_h|+|A_i|+|A_j|)$ for all $hij \in \binom{[t]}{3}$. If $m'>0$, then the stated inequality holds for every triple $\{h,t-1,t\}$ with $h\in [t-2]$.	
			\end{claim}
	
	\begin{proof}[Proof of Claim.]
		Let $W := A_{h} \cup A_{i}\cup A_j$.
		Suppose on the contrary that $e(H'[W]) \leq t_2(|W|)$.
Then one can obtain a new graph $G_2$ from $H'$ by replacing $H'[W]$ with a bipartite graph of the same size.  Note that $H'$ is complete between $W$ and $\overline{W}$. (Indeed, this is trivially true if $m'=0$ as then $H'=K[A_1,\dots,A_t]$; otherwise the only non-complete pair is $[A_{t-1},A_t]$ but both of these sets lie inside~$W$.) 

As $H'$ is $t$-partite, the graph $G_2$ is $(t-1)$-partite (with at most one non-complete pair of parts).
By Steps~4--5, we have $t\le 2(k-1)$. So we can represent $G_2$ as the union of complete $(k-1)$-partite and triangle-free graphs, that is, $G_2 \in \mathcal{H}_0(n,e)$.
We have that $K_3(G_2[W])=0$ and $W$ is complete to $\overline{W}$ in both $H'$ and $G_2$. Thus the fact that $H' \in \mathcal{H}_0^{\mathrm{min}}(n,e)$ implies that $K_3(H'[W])=0$.
		However, if $\lbrace t-1,t\rbrace \not\subseteq \lbrace h,i,j\rbrace$, then $H'[W]$ is complete tripartite so clearly contains at least one triangle. Otherwise, if, say, $\lbrace i,j\rbrace = \lbrace t-1,t\rbrace$, then $H'$ spans at least one edge between $A_{t-1},A_t$ (since there are $m' \leq a_{t-1}-a_t < a_{t-1}a_t$ missing edges by~\eqref{eq:m'6}). Each such edge lies in $|A_h|>0$ triangles in~$H'[W]$. 
		So in either case we obtain a contradiction.\end{proof}

	\begin{claim}\label{cl:Last2} If $m'> 0$ then $a_{t-2} \geq a_{t-1}$.\end{claim}
	\begin{proof}[Proof of Claim.]	Suppose the claim is false.  
		Now, make a new graph $G_3$ from $H'$ by replacing $[A_{t-2},A_t]$-edges with $[A_{t-1},A_t]$-edges until this is no longer possible.
		Let $W := A_{t-2} \cup A_{t-1}\cup A_t$.
		If $A_{t-2}\cup A_{t}$ is an independent set in $G_3$ (i.e.\ if $m'\ge a_{t-2}a_t$), then $e(H'[W])=e(G_3[W])\leq t_2(|W|)$, contradicting Claim~\ref{cl:iv} for the triple $\lbrace t-2,t-1,t\rbrace$.
		Thus $G_3[W]$ is obtained from $K[A_{t-2},A_{t-1},A_t]$ by removing $m'$ edges from $K[A_{t-2},A_t]$.
		So $G_3 \in \mathcal{H}_0(n,e)$, and
		$$
		K_3(G_3)-K_3(H') = m'((n-a_{t-1}-a_t)-(n-a_{t-2}-a_t))=m'(a_{t-2}-a_{t-1}) \leq -1,
		$$
		a contradiction proving the claim.
	\end{proof}
		
	If $m'>0$, let $C_i:=A_i$ for $i\in [t]$. If $m'=0$, then let $C_1,\dots,C_t$ be a relabelling of $A_1,\dots,A_t$ so that the sizes of the sets do not increase. 
	Regardless of the value of $m'$, the following statements hold. First, $c_1\ge\dots\ge c_t$, where $c_i:=|C_i|$ for $i\in [t]$. (Indeed, if $m'>0$, this follows from  Steps~4--5 and Claim~\ref{cl:Last2}.)  Also, we have
	\begin{equation}\label{eq:m'C}
	0\le m'\le c_{t-1}-c_t.
	\end{equation}
	(Indeed, if $m'>0$, this is the same as~\eqref{eq:m'6}; otherwise this is a trivial consequence of $m'=0$ and $c_{t-1}\ge c_t$.)  Also, Claim~\ref{cl:iv} applies to any triple $C_i,C_{t-1},C_t$.
	
 The rest of the proof is written so that it works for both $m'=0$ and $m'>0$.
		
\begin{claim}\label{cl:AlmostEqual} We have that $c_1\le c_{t-1}+1$.\end{claim}
\begin{proof}[Proof of Claim.] Suppose that this is false.
	Let  $W := C_1 \cup C_{t-1} \cup C_t$.
	Note that
	$$
	e(K_{c_1-1,c_{t-1}+1,c_t}) - e(H'[W]) = m'-c_{t-1}+c_1-1 =: m''.
	$$
	Now, $m'' \geq m'+1$.
	We claim that additionally $m'' < (c_{t-1}+1)c_t$.
	Suppose that this is not true. Then $e(H'[W]) \leq (c_1-1)(c_{t-1}+c_t+1) \leq t_2(|W|)$,
	contradicting Claim~\ref{cl:iv}.
	Take a partition $C_1',C_{t-1}',C_t'$ of $W$ of sizes $c_1-1,c_{t-1}+1,c_t$ respectively and let a graph $H_W$ be obtained from $K[C_1',C_{t-1}',C_t']$ by removing $m''$ edges between $C_{t-1}'$ and $C_t'$. Then $e(H_W)=e(H'[W])$.
	Obtain $H''$ from $H'$ by replacing $H'[W]$ with $H_W$.
	Note that $H'' \in \mathcal{H}_0(n,e)$.
	By~\eqref{eq:m'C}, we have that
	\begin{align*}
		K_3(H')-K_3(H'')&=K_3(H'[W])-K_3(H_W)\\
		&=\big(c_1c_{t-1}c_t - m'c_1\big) -\big((c_1-1)(c_{t-1}+1)c_t-(m'-c_{t-1}+c_1-1)(c_1-1)\big)\\
		&\geq (c_1-c_t)(c_1-c_{t-1}-2)+1 \geq 1,
	\end{align*}
	a contradiction proving $c_{t-1}+1\ge c_1$.\end{proof}

It follows that $C_1,\dots,C_{t-1}$  induce a Tur\'an graph in $H'$. (Indeed, the sizes of these independent set are almost equal by Claim~\ref{cl:AlmostEqual}; furthermore, if $m'>0$, then all missing edges in $H'$ are between $C_{t-1}=A_{t-1}$ and $C_t=A_t$ while otherwise there are no missing edges at all.)

	Now, we can argue that $t=k$.
	By the definition of $k$, we have to show that $t_{t-1}(n)<e\le t_{t}(n)$.
	Clearly, $H'$ is $t$-partite so $e\leq t_{t}(n)$. So it remains to show $t_{t-1}(n)<e$. Let $T := H'[C_1 \cup \ldots \cup C_{t-1}]\cong T_{t-1}(n-c_t)$. We can obtain both $H'$ and $T_{t-1}(n)$ from $T$ by adding $c_t$ vertices one by one.
	First let us make $H'$ from $T$.
	The number of additional edges is $e-e(T) = c_t(n-c_t) - m'$. 
	If we instead add vertices one by one to $T$ to make $T_{t-1}(n)$, each vertex must miss an entire part of the current graph, so its degree is at most $n-c_{t-1}-1$.
	Thus $t_{t-1}(n)-e(T) \leq c_{t}(n-c_{t-1}-1)$.
	By~\eqref{eq:m'C}, we have
	$$
	e-t_{t-1}(n) \geq c_t(c_{t-1}+1-c_t)-m' \geq (c_t-1)(c_{t-1}-c_t)+c_t>0.
	$$
	Thus $t=k$, as stated.

	Now we can show that $H'$ has part sizes given by the vector $\bm{a}^*=\bm{a}^*(n,e)$ from Definition~\ref{astardef}, finishing the proof of the lemma. By Claim~\ref{cl:AlmostEqual}, we have that
	$\sum_{ij \in \binom{[k-1]}{2}}c_ic_j = t_{k-1}(n-c_k)$.
	Note that $m' = c_{k-1}c_k-e(H'[C_{k-1}\cup C_k])$.
	Thus we have by~\eqref{eq:m'C} that $e-t_{k-1}(n-c_k) = c_k(n-c_k)-m' \leq c_k(n-c_k)$.

	So it remains only to show that $c_{k}$ is the \emph{smallest} natural 
	number $a$ with $f(a):=a(n-a)+t_{k-1}(n-a) \geq e$. Note that $c_k\le n/k$ as it is the smallest among $c_1+\dots+c_k=n$.
	Thus, by Lemma~\ref{lm:a*k}, it is enough to check that
	$c_k-1$ violates this condition. The calculation in~\eqref{eq:a*k}, the estimates that we stated in the previous paragraph and~\eqref{eq:m'C} give that 
	$$f(c_k-1)=f(c_k)-(c_{k-1}-c_k+1)\le e+m'-(m'+1)<e,$$
	as desired. This finishes the proof of the lemma.\end{proof}

\begin{proof}[Proof of Proposition~\ref{pr:compute}.] Let $n,e\in\I N$ with $e\le {n\choose 2}$ and let $k := k(n,e)$.  
	Corollary~\ref{H2new} and Lemma~\ref{lem-Hlpab} show that, for each $i\in \{0,1,2\}$, the minimum number of triangles over the graphs in $\mathcal{H}_i(n,e)\ni H^*(n,e)$ is $K_3(H^*(n,e))=h^*(n,e)$. Thus it remains to describe the extremal graphs. Assume that $k\ge 3$ as otherwise $h(n,e)=h^*(n,e)=0$ and trivially $\mathcal{H}^{\mathrm{min}}_i(n,e)=\mathcal{H}^*_i(n,e)$ for $i=0,1,2$.

	
	First we will prove that $\mathcal{H}^{\mathrm{min}}_i(n,e)=\mathcal{H}^*_i(n,e)$ for $i=0,1$.
	Let $G \in \mathcal{H}_0^{\mathrm{min}}(n,e)$ be arbitrary. 
	Let $G$ have vertex partition $B_1,\ldots,B_{k-1}$ such that $G$ consists of the union of $K[B_1,\ldots,B_{k-1}]$ and an edge-disjoint triangle-free graph $J$.
	Write $b_i := |B_i|$ for all $i \in [k-1]$.
	Apply Steps~1--6 to $G$ to obtain a $t$-partite graph $H'$ with parts $A_1,\ldots,A_t$.
	By Lemma~\ref{lem-Hlpab},  $H'$ is isomorphic to $H^* := H^*(n,e)$.
	Thus $t=k$ and, by relabelling parts, we can assume that $|A_i|=a_i^*$ for all $i \in [k]$ and that all missing edges, if any exist, are in $H'[A_{k-1},A_k]$.

	We will also need the following claim.
	
	\begin{claim}\label{cl:vi} If a part $B_i$ is not partially full in $G$ (that is, if $e(G[B_i])$ is $0$ or $t_2(b_i)$), then
		$G[B_i]= H'[B_i]$ (that is, no adjacency inside $B_i$ is modified).
	\end{claim}
	\begin{proof}[Proof of Claim.] 
		If $e(G[B_i])=0$, then  $B_i=A_j$ for some $j \in [k-2]$ and
		so $e(H'[B_i])=0=e(G[B_i])$, giving the required.
		If  $e(G[B_i])=t_2(b_i)$ then, by construction, $G_1[B_i]=G[B_i]$ are maximum bipartite graphs and so $H'[B_i]=G[B_i]$, as required.
	\end{proof}

Since $t=k$, exactly one part $B_p$ of $G$ is subdivided as $A_q \cup A_r$ in
Steps~4--5 (that is, $B_p=A_q \cup A_r$), while the remaining parts of $G$ correspond to the remaining parts of $H'$.
In particular, $b_p = a_{q}^*+a_r^*$, where, say, $1 \leq q < r \leq k$.

	Let us show that $e(G[B_p])>0$. Indeed, if this is not true, then, by~(\ref{m*}), $H'[B_p]$ contains $a_q^*a_r^*-m^* \geq a_q^*a_r^*-(a_{k-1}^*-a_k^*)>0$ edges, and so is different from the edgeless graph $G[B_p]$. Then Claim~\ref{cl:vi} implies that $B_p$ is partially full, a contradiction.
	
	\medskip
	\noindent
	\textbf{Case 1.} \emph{There exists $h \in [k-1]\setminus \lbrace p \rbrace$ such that $e(G[B_h])>0$.
		In other words, $G \in \mathcal{H}_0^{\mathrm{min}}(n,e)\setminus \mathcal{H}_1(n,e)$.}
	
	\medskip
	\noindent
	We claim that $b_h=b_p$.
	This follows from~\eqref{eq:T2Parts} if $B_h$ and $B_p$ are both partially full.
	Note that $B_h$ is an independent set in $H'$ and so $G[B_h] \neq H'[B_h]$, and Claim~\ref{cl:vi} implies that $B_h$ is partially full.
	So it suffices to show that $B_p$ is partially full.
	If not, then $e(G[B_p])=t_2(b_p)$ (as $e(G[B_p])=0$ is already excluded).
	Since $G[B_i,B_j]=H'[B_i,B_j]$ for all $ij \in \binom{[k-1]}{2}$ and $e(H'[B_h])=0<e(G[B_h])$, there is some $\ell \in [k-1]\setminus \lbrace h\rbrace$ such that $e(H'[B_\ell])>e(G[B_\ell])$.
	Since $H'[B_p]$ is bipartite and $e(G[B_p])=t_2(b_p)\ge e(H'[B_p])$, we have that  $\ell \neq p$ .
	But then $B_\ell=A_j$ for some $j \in [k]$, and so $B_\ell$ is an independent set in $H'$, a contradiction.
	This proves that $b_h=b_p$.
	
   Since $B_p$ is the only part that was subdivided, there is $s \in [k-1]$ such that $A_s=B_h$ and thus $a_s^*=b_h=b_p=a_q^*+a_r^*$.
	Since $a_1^* \geq \ldots \geq a_{k-1}^* \geq \max\lbrace a_1^*-1,a_k^*\rbrace$, we have $a_s^*-a_q^*=1$ and $a_r^*=1$. So $a_k^*=1$ and $a_q^*=a_{k-1}^*$.
	Since $h$ was arbitrary, we conclude that for all $i \in [k-1]$ such that $e(G[B_i])>0$, we have $b_i=a^*_{k-1}+1$.
	So $G \in \mathcal{H}_0^*(n,e)$, as required.
	
	\medskip
	\noindent
	\textbf{Case 2.} \emph{For all $h \in [k-1]\setminus \lbrace p \rbrace$ we have $e(G[B_h])=0$.
		In other words, $G \in \mathcal{H}_1^{\mathrm{min}}(n,e)$.}
	
	\medskip
	\noindent
	Suppose first that $m^*=0$.
	Then $H' = K[A_1,\ldots,A_k]$, and $G$ can be obtained from it by replacing $H'[A_q \cup A_r]$ with $G[B_p]$.
	Moreover, $G[B_p]$ is a triangle-free graph on $a_q^*+a_r^*$ vertices with $a_q^*a_r^*$ edges.
	If $a_r^*=a_k^*$, then $G \in \mathcal{H}^*_1(n,e)$; otherwise $|a_q^*-a_r^*| \leq 1$, so $G[B_p] \cong T_2(a_q^*+a_r^*)$ and thus $G \cong H' \in \mathcal{H}^*_1(n,e)$, getting the required in either case.
	
	Suppose instead that $m^*>0$. Since $G[A_i,A_j]$ is complete for all $\lbrace i,j\rbrace \neq \lbrace q,r\rbrace$, and $H'[A_i,A_j]$ is complete if and only if $\lbrace i,j\rbrace \neq \lbrace k-1,k\rbrace$, we have $\lbrace q,r\rbrace = \lbrace k-1,k\rbrace$.
	Thus $G$ can be obtained from $K[A_1,\ldots,A_k]$ by replacing $K[A_{k-1}\cup A_k]$ with a triangle-free graph with $a_{k-1}^*a_k^*-m^*$ edges.
	This gives that $G \in \mathcal{H}^*_1(n,e)$, as required.
	
	Note that if $G\in \mathcal{H}_1^{\mathrm{min}}(n,e)$ then the above argument always concludes that $G\in \mathcal{H}_1^*(n,e)$, apart from Case~1 (that does not apply here). Thus we have proved that $\mathcal{H}_i^{\mathrm{min}}(n,e) = \mathcal{H}_i^*(n,e)$ for $i =0,1$.

	\medskip
	Now let $G \in \mathcal{H}_2^{\mathrm{min}}(n,e)$ be arbitrary.
	If $G \in \mathcal{H}_1(n,e)$ then, as we have just established, $G \in \mathcal{H}_1^*(n,e)$ (and also $G$ is $k$-partite).
	So $G \in \mathcal{H}_2^*(n,e)$, and thus we may assume that $G \in \mathcal{H}_2^{\mathrm{min}}(n,e)\setminus \mathcal{H}_1(n,e)$.
	
	Let $G$ have $\mathcal{H}_2$-canonical partition $A_1,\ldots,A_k$ with part sizes $a_1 \geq \ldots \geq a_k$ respectively.  
	By Lemma~\ref{lm:H2'min}, we have that $G \in (\mathcal{H}_2')^{\mathrm{min}}(n,e)$, and $A_1,\ldots,A_k$ is an $\mathcal{H}_2'$-canonical partition.
	Since $G \notin \mathcal{H}_1(n,e)$, Corollary~\ref{H2new}(iii) gives that
	\begin{equation}\label{H2m}
	m := \sum_{1\le i<j\le k}e(\overline{G}[A_i,A_j]) \leq a_{k-1}-a_k+1.
	\end{equation}
	Since $\mathcal{H}_2^{\mathrm{min}}(n,e)=(\mathcal{H}_2')^{\mathrm{min}}(n,e)$ by Lemma~\ref{lm:H2'min}, we see that if, for $i$ in $I := \lbrace j \in [k-1]:a_j=a_{k-1} \rbrace$, we let $B_i$ consist of those $x\in A_k$ that have at least one non-neighbour in $A_i$, then these subsets of $A_k$ are disjoint and every missing edge in $G$ intersects one of them. So to prove that $G \in \mathcal{H}_2^*(n,e)$, it suffices to show that
	\begin{itemize}
		\item[(i)] $(a_1,\ldots,a_k)=(a_1^*,\ldots,a_k^*)$; or
		\item[(ii)] $m^*=0$, $a_1^*\ge a_k^*+2$ and $(a_1,\ldots,a_k)=(a_2^*,\ldots,a_{k-1}^*,a_1^*-1,a_k^*+1)$.
	\end{itemize}
	
	By~(\ref{H2m}), we can obtain a graph $G'$ from $G$ by moving all $m$ missing edges between parts $A_{k-1}$ and $A_k$.
	Then $G' \in \mathcal{H}_1^{\mathrm{min}}(n,e)$, which equals $\mathcal{H}^*_1(n,e)$ as we have already shown.
	So $G'$ has a partition $A_1^*,\ldots,A_k^*$ where $|A_i^*|=a_i^*$ and there is some $i \in [k-1]$ such that $G'$ can be obtained from $K[A_1^*,\ldots,A_k^*]$ by replacing $K[A_i^*\cup A_k^*]$ with a triangle-free graph with $a_i^*a_k^*-m^*$ edges. Thus there is a bijection $\sigma : [k-1]\setminus \lbrace i \rbrace \rightarrow [k-2]$ such that
	\begin{equation}\label{sigmaeq}
	A_{\sigma(j)}=A_j^*,\quad\text{for all } j \in [k-1]\setminus \lbrace i \rbrace,
	\end{equation} 
	while $A_{k-1} \cup A_k = A_i^* \cup A_k^*$ and $a_{k-1}+a_k=a_i^*+a_k^*$.  Thus, by the monotonicity of the involved sequences, if we remove the $i$-th and $k$-th
entries from $\bm{a}^*$ then we obtain $(a_1,\dots,a_{k-2})$.

By the minimality of $a_k^*$, we have $a_k^*\le a_{k}$.
	Suppose that $a:=a_k-a_k^*\ge 1$ as otherwise $(a_{k-1},a_k)=(a_i^*,a_k^*)$ and the desired property~(i) follows from~\eqref{sigmaeq}. 
	Since $a_k^*+a= a_k\le a_{k-1}=a_{i}^*-a$, we have
	\begin{equation}\label{eq:m-m*}
	m-m^*=a_{k-1}a_k-a_i^*a_k^*= a_{k-1}a_k-(a_{k-1}+a)(a_k-a)
	\ge a_{k-1}-a_k+a^2.
	\end{equation}
	By~\eqref{H2m} and~\eqref{eq:m-m*}, we have $a=1$, $m^*=0$, $a_k=a_k^*+1$ and $a_{k-1}=a_i^*-1$. Also,
	$a_1^*-1\ge a_i^*-1=a_{k-1}\ge a_k=a_k^*+1$. 
 Recall that $0\le a_1^*-a_i^*\le 1$ by  the definition of~$\bm{a}^*$.
If $a_i^*=a_1^*-1$ then, for all $j \in [k-1]\setminus \lbrace i \rbrace$, by~(\ref{sigmaeq}), we have $a_j=a_{\sigma^{-1}(j)}^* \geq a_1^*-1=a_i^*=a_{k-1}+1$. But then the set $I$ of indices of parts which are not complete to $A_k$ consists only of $k-1$, so $G \in \mathcal{H}_1(n,e)$, a contradiction.
Thus $a_i^*=a_1^*$.
This gives all the statements from~(ii) by~\eqref{sigmaeq}, finishing the proof of the proposition.
\end{proof}

\subsection{Approximating the increment of the function $h^*(n,\cdot)$}

Let a pair $(n,e)$ be valid and let $k=k(2e/n^2)$, where the single-variable function $k$ is defined in~\eqref{eq:k(lambda)}. Also, define $c(n,e):=c(2e/n^2)$ to be the larger root of~\eqref{eq:NewC} for $\lambda=2e/n^2$; this root can be explicitly written as 
\begin{equation}\label{eq:c2}
 c(n,e) := c(2e/n^2) = \frac{1}{k}\left(\,1 + \sqrt{1-\frac{k}{k-1}\cdot \frac{2e}{n^2}}\,\right).
\end{equation} 
Let $c := c(n,e)$.
By definition,
\begin{equation}\label{eq:c}
\binom{k-1}{2}c^2 + (k-1)c(1-(k-1)c) = (k-1)c - \binom{k}{2}c^2 = \frac{e}{n^2}
\end{equation}
and so
\begin{align}
\nonumber e(K^k_{cn,\ldots,cn,n-(k-1)cn}) &= e\quad\text{ and }\\
\nonumber K_3(K^k_{cn,\ldots,cn,n-(k-1)cn}) &= \binom{k-1}{3}c^3n^3 + \binom{k-1}{2}c^2(1-(k-1)c)n^3\\
\label{K3c} &=\binom{k-1}{2}c^2n^3-2{k\choose 3}c^3n^3.
\end{align}
In this section, we show that the increment of the function $h^*(n,\cdot)$ at $e$ is very closely approximated by $(k-2)cn$.

First, we need the following standard estimate of the Tur\'an number.

\begin{lemma}\label{lm:Turan41fact}
Let $s,n$ be integers such that $2 \leq s \leq n$. Then
\begin{equation}\label{eq:Turan41fact}
\left(1-\frac{1}{s}\right)\frac{n^2}{2}-\frac{s}{8} \leq t_s(n) \leq \left(1-\frac{1}{s}\right)\frac{n^2}{2}.
\end{equation}
\end{lemma}
\begin{proof} Divide $n$ by $s$ with remainder: $n=s\ell +r$ with $r\in\{0,\dots,s-1\}$. Then the Tur\'an graph $T_s(n)$
  has $r$ parts of size $\ell+1$ and $s-r$ parts of size $\ell$. Routine calculations show that
   $$
   t_s(n)={r\choose 2}(\ell+1)^2+{s-r\choose 2}\, \ell^2+r(s-r)(\ell+1)\ell = \left(1-\frac{1}{s}\right)\frac{(s\ell+r)^2}2 + \frac{r^2-sr}{2s}.
   $$
   For real $r\in [0,s-1]$, the quadratic function $r^2-rs$ has its minimum at $r=s/2$ and its maximum at $r=0$, giving the required bounds on~$t_s(n)$.
\end{proof}

Because of the gap in~\eqref{eq:Turan41fact}, the values of $k(2e/n^2)$ and $k(n,e)$ may be different when $e$ is  slightly above a Tur\'an number. The following lemma implies that this never occurs inside the proof of Theorem~\ref{strong}, where $t_{k-1}(n)+\Omega(n^2)<e\le t_k(n)$);~(\ref{eq:c2}) then holds with $k(2e/n^2)$ replaced by $k(n,e)$.

\begin{lemma}\label{lm:ks} Let a pair $(n,e)$ be valid. Then
   \begin{itemize}
  \item[(i)] $k(2e/n^2)\le k(n,e)$;
  \item[(ii)]
 if $t_{k-1}(n)+(k-1)/8\leq  e\le t_{k}(n)$, then $k(2e/n^2)=k=k(n,e)$.
  \end{itemize}
 \end{lemma}
  \begin{proof}  Clearly, each of the functions $k(n,e)$ and $k(2e/n^2)$ is non-decreasing in $e$.
  Let $s\in \I N$. Recall that $k(\lambda)$ jumps from $s$ to $s+1$ when $\lambda$ becomes larger than $(s-1)/s$ while $k(n,e)$ jumps from $s$ to $s+1$  when $e$ becomes larger than $t_s(n)$. Now, both of the stated claims follow from Lemma~\ref{lm:Turan41fact}.\end{proof}

\begin{lemma}\label{lm:(k-1)c<1} For every $\lambda\in [0,1)$, we have $(k(\lambda)-1)c(\lambda)<1$.\end{lemma}
 \begin{proof} Assume that $s:=k(\lambda)\ge 2$, as otherwise there is nothing to prove. The
 formula in~\eqref{eq:c2'} shows that $c(x)$ is a strictly decreasing continuous function for $x\in(\frac{s-2}{s-1},\frac{s-1}s]$ and the limit of $c(x)$ as $x$ tends to $\frac{s-2}{s-1}$ from above is $1/(s-1)$. Thus $c(x)<1/(s-1)$ in this half-open interval, as required.
 \end{proof} 

\begin{lemma}\label{lm:integerfact}
	For all valid $(n,e)$, if $c=c(n,e)$ is such that $cn \in \mathbb{N}$, then $k(n,e)=k(2e/n^2)=:k$, and $\bm{a}^*=\bm{a}^*(n,e)$ is equal to $(cn,\ldots,cn,n-(k-1)cn)$. 
	\end{lemma}
 \begin{proof}
 Let $k:=k(2e/n^2)$ and $a:=n-(k-1)cn$. Since $c\ge 1/k$ by definition, we have that $a\le cn$.
 Also, $c<1/(k-1)$ by Lemma~\ref{lm:(k-1)c<1}. Thus $a$ is positive. From $e(K^k_{cn,\dots,cn,a})=e$ we conclude that $k(n,e)\le k$. This must be equality by the first part of Lemma~\ref{lm:ks}.

  	Recall by Definition~\ref{astardef} that $a_k^*$ is the minimum $s\in\I N$ with $s(n-s)+t_{k-1}(n-s)\ge e$, which is satisfied (with equality) for $s=a$.  Thus $a^*_k\le a$.  Now, Lemma~\ref{lm:a*k} implies by the induction on $a-s$ that for every $s=a-1,a-2,\dots,1$ we have $s(n-s)+t_{k-1}(n-s)<e$. Thus indeed $a^*_k=a$. This clearly implies that $a_i^*=cn$ for each $i\in [k-1]$.
 \end{proof}

The following simple lemma describes the change in $H^*(n,e)$ when we increase $e$ by~$1$. Informally speaking, either (i) one missing edge is added, (ii) the smallest part increases by $1$, or (iii) the number of parts increases by $1$. 

\begin{lemma}\label{lm:H^+} Let $e,n\in\I N$ with $e<{n\choose 2}$. Let $k=k(n,e)$, $\bm{a}^*=\bm{a}^*(n,e)$, $m^*=m^*(n,e)$, $k^+=k(n,e+1)$ and $\bm{a}^+=\bm{a}^*(n,e+1)$ be as in Definition~\ref{astardef}. Then the following statements hold.
 \begin{itemize}
 \item[(i)] If $m^*>0$, then $k^+=k$ and $\bm{a}^+=\bm{a}^*$.
 \item[(ii)] If $m^*=0$ and $a^*_1\ge a_k^*+2$, then $k^+=k$, $a_k^+=a_k^*+1$ and $(a_1^+,\dots,a_{k-1}^+)$ is obtained from $(a_1^*-1,a_2^*,\dots,a_{k-1}^*)$ by ordering it non-increasingly.
 \item[(iii)] If $m^*=0$ and $a^*_1\le a_k^*+1$, then $k^+=k^*+1$.
 \end{itemize}
 \end{lemma}
 \begin{proof}
  Let us consider Cases (i) and (ii) together. We can  increase the size of $H^*(n,e)$ without increasing the number of parts: namely, let $H^{\mathrm{(i)}}$ and $H^{\mathrm{(ii)}}$ be obtained from $H^*$ by respectively adding a missing edge or moving a vertex from the first part to the last. Since $k(n,\cdot)$ is a non-decreasing function, we have that $k^+=k$ in both cases. Furthermore, $a_k^*\le n/k$ by~\eqref{m*}. This and the equality $k^+=k$ imply by Lemma~\ref{lm:a*k} that $a_k^+\ge a_k^*$ if $m^*>0$ and $a_k^+\ge a_k^*+1$ if $m^*=0$, with the matching upper bounds on $a_k^+$ witnessed by (the part sizes of) $H^{\mathrm{(i)}}$ and $H^{\mathrm{(ii)}}$, giving the required.
 
The third case is also easy: $k^+>k$ since $H^*(n,e)$ is the Tur\'an graph $T_k(n)$ while $k^+\le k+1$ since $k<n$ and $t_{k+1}(n)\ge t_k(n)+1$.\end{proof}

\begin{lemma}\label{lem-slope-c}
For all valid $(n,e)$, if $e \in [t_{k-1}(n)+k,t_k(n)-1]$, then	with $c=c(n,e)$ we have
\begin{align*}
|(h^*(n,e+1)-h^*(n,e))-(k-2)cn|&\le k\quad\text{and}\\
|(h^*(n,e)-h^*(n,e-1))-(k-2)cn|&\le k.
\end{align*}
Moreover, $|a_i^*-cn| \leq 2$ for all $i \in [k-1]$ where $\bm{a}^*=\bm{a}^*(n,e)$ is defined in Definition~\ref{astardef}.
\end{lemma}

\begin{proof}
For valid $(n,f)$ with $k(n,f)$ equal to $k=k(n,e)$, let
$$
L(n,f) := \sum_{i \in [k-2]}a_i^*(n,f) = n-a^*_{k-1}(n,f) - a^*_k(n,f),
$$
 where $\bm{a}^*(n,f) = (a_1^*(n,f),\ldots,a_k^*(n,f))$ is
 as in Definition~\ref{astardef}.

 Note that if  $f+1\le t_k(n)$ (that is, $k(n,f+1)=k(n,f)=k$), then
  \begin{equation}\label{eq:Lnf+1}
  L(n,f+1)-L(n,f) \in \lbrace -1,0\rbrace.
  \end{equation}
  Indeed, consider how the vector $\bm{a}^*$ changes when we increase $f$ by $1$. Suppose that $m^*(n,f)=0$ as otherwise the vector stays the same by Lemma~\ref{lm:H^+}(i). 
  Note that $a_1^*(n,f)\ge a_k^*(n,f)+2$ since $f<t_k(n)$ so Lemma~\ref{lm:H^+}(ii) applies.
  Here the $k$-th entry of $\bm{a}^*$ increases by 1 while one of the other entries decreases by 1. In any case, $a_{k-1}^*+a_k^*$ stays the same or increases exactly by 1, giving~\eqref{eq:Lnf+1}.

\begin{claim}\label{cl-contralL}
 There exist integers $e^-,e^+$ such that
	\begin{itemize}
		\item[(i)] $e^- \leq e \leq e^+$ and $k(n,e^-)=k(n,e^+)=k$;
		\item[(ii)] $L(n,e^-) \leq (k-2)\lceil cn \rceil$ and $L(n,e^+) \geq (k-2)\lfloor cn \rfloor$.
	\end{itemize}
\end{claim}

\begin{proof}[Proof of Claim.]
	Given some $e^-$ and $e^+$ satisfying (i) we will write $\bm{a}^*(n,e^-) = (a^-_1,\ldots,a^-_k)$ and similarly $\bm{a}^*(n,e^+) = (a^+_1,\ldots,a^+_k)$.
Let us consider $e^-$.
	Suppose first that $\lceil cn \rceil \geq n/(k-1)$.
	Then we let $e^-:=t_{k-1}(n)+1$.
	Now $k(n,e^-)=k$ by definition, and $a^-_k=1$, so $a^-_{k-1}=\lfloor (n-1)/(k-1)\rfloor$. Thus
	$$
	L(n,e^-) = n - \left(\left\lfloor \frac{n-1}{k-1}\right\rfloor + 1\right) \leq \frac{k-2}{k-1} \cdot n \leq (k-2)\lceil cn \rceil,
	$$
	as desired. 
	
	So suppose that $a:=\lceil cn \rceil < n/(k-1)$. Let $e^-$ satisfy $c(n,e^-)=a/n$, that is, $e^-$ is the size of the complete $k$-partite graph $K^k_{a,\dots,a,n-(k-1)a}$. Clearly, $e^-\le t_k(n)$. Since $a<n/(k-1)$, we have that $e^-> t_{k-1}(n)$. Thus $k(n,e^-)=k$. The explicit formula in~\eqref{eq:c2} shows that  $c(n,x)$ is a decreasing function of $x$, even when $k(n,x)$ jumps. Since $c(n,e^-)= a/n$ is at least $c=c(n,e)$, it holds that $e^-\le e$.
 For this $e^-$ we have that $a^-_i=\lceil cn \rceil$ for all $i \in [k-1]$, so Lemma~\ref{lm:integerfact} implies that $L(n,e^-)=(k-2)\lceil cn \rceil$, as required.
	
	It remains to obtain $e^+$.
	Suppose first that $b:=\lfloor cn \rfloor < n/k$.
	Let $e^+ := t_k(n)$.
	Then $k(n,e^+)=k$, $a^+_k=\lfloor n/k \rfloor$ and $a^+_{k-1} = \lfloor (n-a^+_k)/(k-1)\rfloor$.
	Since $b < n/k \leq cn$ by definition, we have that $\lfloor n/k \rfloor = b$.
	Thus
	$$
	L(n,e^+)=n - \left\lfloor \frac{n}{k} \right\rfloor - \left\lfloor \frac{n-\lfloor n/k\rfloor}{k-1}\right\rfloor \geq (n-b)\left(1-\frac{1}{k-1}\right) > (k-2)b,
	$$
	as required.
	
	So suppose that $b \geq n/k$. By our assumption $e\ge t_{k-1}(n)+k$ and Lemma~\ref{lm:ks}, we have that $k(n,e)=k(2e/n^2)$. By  Lemma~\ref{lm:(k-1)c<1}, we have that $(k-1)b\le (k-1)cn<n$. Thus, similarly as above, if we define $e^+=e(K^k_{b,\dots,b,n-(k-1)b})$, then  $k(n,e^+)=k$, $c(n,e^+)=b/n$ is at most $c=c(n,e)$ and thus $e^+\ge e$.
	In this case, $a_i^+=\lfloor cn \rfloor$ for all $i \in [k-1]$, so Lemma~\ref{lm:integerfact} implies that $L(n,e^+)=(k-2)\lfloor cn \rfloor$, as required.\end{proof}

\medskip
\noindent
By~\eqref{eq:Lnf+1}, $L(n,\cdot)$ is a non-increasing function in the range between $t_{k-1}(n)+k$ and $t_k(n)$.
Together with the second part of Claim~\ref{cl-contralL}, this then implies that
\begin{equation}\label{Ldiff}
(k-2)\lfloor cn \rfloor \leq L(n,e^+) \leq L(n,e) \leq L(n,e^-) \leq (k-2)\lceil cn \rceil.
\end{equation}
From this we have that $\lfloor cn \rfloor \leq a_i^* \leq \lceil cn \rceil$ for all $i \in [k-2]$. Since $a_{k-1}^* \geq a_{k-2}^*-1$, the second part of the lemma is proved.

Now, we claim that
\begin{equation}\label{hLdiff}
L(n,e)-1 \leq L(n,e+1) \leq h^*(n,e+1)-h^*(n,e) \leq L(n,e).
\end{equation}
If this holds, then
\begin{eqnarray*}
&&|h^*(n,e+1)-h^*(n,e)-(k-2)cn|\\
&\leq& |h^*(n,e+1)-h^*(n,e)-L(n,e)| + |L(n,e)-(k-2)cn|\\
&\stackrel{(\ref{Ldiff}),(\ref{hLdiff})}{\leq}& 1 + (k-2)\max \lbrace cn-\lfloor cn \rfloor,\lceil cn\rceil-cn \rbrace \leq k-1,
\end{eqnarray*}
proving the first inequality. Similarly, noting that $k(n,e-1)=k(n,e)=k$ by Lemma~\ref{lm:ks} and the fact that $e \geq t_{k-1}(n)+k$, we have that
\begin{eqnarray*}
&& |h^*(n,e)-h^*(n,e-1)-(k-2)cn|\\
&\leq& |h^*(n,e)-h^*(n,e-1)-L(n,e-1)| + |L(n,e-1)-L(n,e)| + |L(n,e)-(k-2)cn|\\
&\le& 1 + 1 + (k-2) = k,
\end{eqnarray*}
where the last inequality follows from~(\ref{eq:Lnf+1}),~(\ref{Ldiff}) and~(\ref{hLdiff}), proving the second.

So it suffices to prove~(\ref{hLdiff}). The first inequality follows from~\eqref{eq:Lnf+1}. If $m^*>0$, then by Lemma~\ref{lm:H^+}(i) the difference $h^*(n,e+1)-h^*(n,e)$ is the number of triangles created by adding one missing edge to $H^*(n,e)$, which is exactly $L(n,e)$. If $m^*=0$, then we are in the second case of Lemma~\ref{lm:H^+}, where we add one more edge into the union of two parts of sizes $a_1^*$ and $a_k^*$, keeping this graph bipartite. Clearly, this new edge creates $n-a_1^*-a_k^*$ triangles. This is $L(n,e)$ if $a_1^*=a_{k-1}^*$ and $L(n,e+1)$ otherwise (i.e.\ if $a_1^*=a_{k-1}^*+1$).
\end{proof}

Lemma~\ref{lem-slope-c} will imply that if there is a counterexample to Theorem~\ref{strong}, then in an appropriately defined `worst counterexample' no edge lies in more than $(k-2)cn+k$ triangles and no non-edge lies in less than $(k-2)cn-k$ copies of $P_3$.
This fact will be extremely useful in our proof of Theorem~\ref{strong}.

\begin{cor}\label{cr:dh} 
	Let $n \in \mathbb{N}$ and $e\in [t_{k-1}(n)+k, t_k(n)-1]$ and let $p>0$ and $c=c(n,e)$.
	Suppose that $g_3(n,e)-h^*(n,e) \leq g_3(n,e^*)-h^*(n,e^*)$ for all $e^*$ with $k(n,e^*)=k$.
	Let $G,G'$ be $(n,e)$-graphs such that $K_3(G)=g_3(n,e) \geq K_3(G')-p$.
Then, for every $\overline{f} \in E(\overline{G})$, $\overline{f'} \in E(\overline{G'})$, $f \in E(G)$ and $f'\in E(G')$, we have that
\begin{itemize}
\item[(i)] $P_3(\overline{f},G) \geq (k-2)cn-k$, and  $P_3(\overline{f'},G') \geq (k-2)cn-k-p$;
\item[(ii)] $P_3(f,G) \leq (k-2)cn+k$, and $P_3(f',G') \leq (k-2)cn+k+p$.
\end{itemize}
\end{cor}
\begin{proof}
Let $\overline{f} \in E(\overline{G})$. 
Then $k(n,e+1)=k$ and 
by the assumption on $G$, for any $(n,e+1)$-graph $G''$, we have that
$$K_3(G)-h^*(n,e)\le g_3(n,e+1)-h^*(n,e+1)\le K_3(G'')-h^*(n,e+1).$$
Thus, by Lemma~\ref{lem-slope-c},
$$
P_3(\overline{f},G) = K_3(G \cup \lbrace \overline{f}\rbrace) - K_3(G) \geq h^*(n,e+1)-h^*(n,e) \geq (k-2)cn-k,
$$
where $G \cup \lbrace \overline{f}\rbrace$ denotes the graph $G$ with the pair $\overline{f}$ added as an edge.
Similarly, for $\overline{f'} \in E(\overline{G'})$, we have
$$
P_3(\overline{f'},G') = K_3(G' \cup \lbrace \overline{f'}\rbrace) - K_3(G')\ge  K_3(G' \cup \lbrace \overline{f'}\rbrace) - K_3(G)-p\geq  (k-2)cn-k-p.
$$

The second part can be proved similarly via the inequality $|h^*(n,e)-h^*(n,e-1)-(k-2)cn| \leq k$ from Lemma~\ref{lem-slope-c}.
\end{proof}

\subsection{Comparing $k$-partite graphs}
The next lemma will be used to compare the number of triangles in two $k$-partite $(n,e)$-graphs $G$ and $F$, in terms of their part sizes and the number of edges missing between parts.
It will later be applied with $\ell := \lfloor cn \rfloor$ and $F$ a graph in $\mathcal{H}_1(n,e)$; and $G$ a graph obtained by switching a small number of adjacencies in a hypothetical counterexample to Theorem~\ref{strong}. Informally speaking, the lemma can be used to derive a quantitative conclusion of the form  that, if the part sizes of $G$ deviate from the almost optimal vector $(\ell,\dots,\ell,n-(k-1)\ell)$, then $K_3(G)$ is larger than $K_3(F)$.

\begin{lemma}\label{CompletekPartite}
Let $n \geq k \geq 3$ and $d>0$ be integers.
Suppose that $G$ and $F$ are $n$-vertex $k$-partite graphs with $e(G) = e(F)$ such that the following hold.
\begin{itemize}
\item[(i)] $G$ has parts $A_1,\ldots,A_k$.
\item[(ii)] $G[A_i,A_j]$ is complete whenever $ij \in \binom{[k-1]}{2}$.
\item[(iii)] $F$ has parts $B_1,\ldots,B_k$ with $\ell_i:=|B_i|$ for $i \in [k]$  
satisfying  $\ell_1 = \ldots = \ell_{k-1} =: \ell > \ell_k>0$.
\item[(iv)] $F[B_i,B_j]$ is complete for all $ij \in \binom{[k]}{2}\setminus\lbrace\lbrace k-1,k\rbrace\rbrace$; also, $e(\overline{F}[B_{k-1},B_k]) \leq d$.
\item[(v)] For all $i \in [k]$ we have that $|d_i| \leq \frac{\ell-\ell_k}{12k^3}$, where $d_i := s_i-\ell_i$ and $s_i:=|A_i|$. 
Moreover, $d_k \geq 0$.
\end{itemize}
Let $m_i := |A_i|\,|A_k|-e(G[A_i,A_k])$ for all $i \in [k-1]$ and $m := m_1+\ldots + m_{k-1}$.
Then
$$
K_3(G)-K_3(F) \geq \sum_{t \in [k-1]}\frac{m_t}{m} \cdot \frac{\ell-\ell_k}{4} \left( (d_t+d_k)^2 +\sum_{\substack{i \in [k-1]\\i \neq t}}d_i^2 \right) - \frac{12d^2}{\ell-\ell_k}.
$$
\end{lemma}

\begin{proof}
Define $d_0:=e(\overline{F}[B_{k-1},B_k]) \leq d$.
Let $H$ be the complete $k$-partite graph with parts $B_1,\ldots,B_k$.
As $\sum_{ij \in \binom{[k]}{2}}s_is_j -m=e(G)=e(F)=  \sum_{ij \in \binom{[k]}{2}}\ell_i\ell_j-d_0$, we have
$$
m' := m - (e(H)-e(F)) = m - d_0 = \sum_{ij \in \binom{[k]}{2}}s_is_j - \sum_{ij \in \binom{[k]}{2}}\ell_i\ell_j.
$$

\begin{claim}
For all $t\in[k-1]$, we have
\begin{equation}\label{eq-L13}
	 \sum_{ijh \in \binom{[k]}{3}}s_is_js_h - m'\sum_{\substack{i \in [k-1]\\i \neq t}}s_i - \sum_{ijh \in \binom{[k]}{3}}\ell_i\ell_j\ell_h\ge \frac{\ell-\ell_k}{3} \left( (d_t+d_k)^2 + \sum_{\substack{i \in [k-1]\\i \neq t}}d_i^2 \right).
\end{equation}
\end{claim}

\begin{proof}
For notational convenience, we prove this for $t=k-1$, and observe that the proof uses only properties~(i)--(iii) and~(v) which are all symmetric in $t \in [k-1]$.

We have that the left-hand side of (\ref{eq-L13}) (with $t=k-1$) is equal to
\begin{align}
\nonumber &\sum_{ijh \in \binom{[k]}{3}}d_id_jd_h + \sum_{ij \in \binom{[k]}{2}}\ell_i\ell_j \sum_{\substack{h \in [k]\\h \neq i,j}}d_h + \sum_{ij \in \binom{[k]}{2}}d_id_j \sum_{\substack{h \in [k]\\h \neq i,j}}\ell_h\\
\label{poly}&\hspace{3cm}- \left( \sum_{i \in [k]}\ell_i  \sum_{\substack{j \in [k]\\j \neq i}}d_j + \sum_{ ij \in \binom{[k]}{2}}d_id_j \right) \sum_{h \in [k-2]}(\ell_h+d_h).
\end{align}
This is a cubic polynomial in $d_1,\ldots,d_k$.
For each $0 \leq t \leq 3$ and $1 \leq i_1 \leq \ldots \leq i_t \leq k$, let $C_{i_1\ldots i_t}$ denote the coefficient of $d_{i_1}\ldots d_{i_t}$.
By a slight abuse of notation, we assume a pair $ij \in \binom{[k]}{2}$ satisfies $i < j$ (and similarly for triples).
Note that $C_{\emptyset}=0$.
Now, for all $i \in [k]$,
$$
C_i = \sum_{hj \in \binom{[k]\setminus\lbrace i \rbrace}{2}}\ell_h\ell_j - \sum_{\substack{j \in [k]\\j \neq i}}\ell_j\sum_{h \in [k-2]}\ell_h.
$$
So $C_1 = \ldots = C_{k-1}$ since $\ell_1=\ldots=\ell_{k-1}$.
Also
\begin{align*}
C_k &= \binom{k-1}{2}\ell^2 - (n-\ell_k)(k-2)\ell = \binom{k-2}{2}\ell^2 + (k-2)\ell\ell_k - (n-\ell)(k-2)\ell = C_1.
\end{align*}
But
 \begin{equation}\label{eq:SumDi=0}
  \sum_{i \in [k]}d_i = 0
  \end{equation} and hence
\begin{equation}\label{l13Ci}
\sum_{i \in [k]}C_id_i = 0,
\end{equation}
that is, the linear part of~(\ref{poly}) is zero.

Next, we simplify the quadratic part.
Suppose that $ij \in \binom{[k-2]}{2}$. Then
 \begin{equation}\label{eq:Cij}
C_{ij} = \sum_{\substack{h \in [k]\\h \neq i,j}}\ell_h - \sum_{\substack{h \in [k]\\h \neq i}}\ell_h - \sum_{\substack{h \in [k]\\h \neq j}}\ell_h - \sum_{h \in [k-2]}\ell_h = \ell+\ell_k-2n.
\end{equation}
Suppose that $i \in [k-2]$. Then
$$
C_{ii} = -\sum_{\substack{h \in [k]\\h \neq i}} \ell_h = \ell-n.
$$
Suppose that $i \in [k-2]$ and $j \in \lbrace k-1,k \rbrace$. Then
\begin{equation}\label{l13Cij}
C_{ij} = \sum_{\substack{h \in [k]\\h \neq i,j}}\ell_h - \sum_{\substack{h \in [k]\\h \neq j}}\ell_h - \sum_{h \in [k-2]}\ell_h = \ell_k-n.
\end{equation}
This implies that
$$
\sum_{\substack{i \in [k-2]\\j \in \lbrace k-1,k \rbrace}}C_{ij}d_id_j = \sum_{i \in [k-2]}(\ell_k-n)(d_{k-1}+d_k)d_i \stackrel{\eqref{eq:SumDi=0}}{=} -(\ell_k-n)\left( \sum_{i \in [k-2]}d_i^2 + 2\sum_{ij \in \binom{[k-2]}{2}}d_id_j \right).
$$
Note that if $i,j \in \lbrace k-1,k \rbrace$, then $C_{ij} = 0$.
So
\begin{equation}\label{l13Cii}
\sum_{i \in [k]}C_{ii}d_i^2 = \sum_{i \in [k-2]}(\ell-n)d_i^2.
\end{equation}
Thus the quadratic terms in~(\ref{poly}) give
\begin{eqnarray}
\nonumber\sum_{\substack{1 \leq i \leq j \leq k}} C_{ij}d_id_j &=& \sum_{i \in [k]}C_{ii}d_i^2 + \sum_{ij \in \binom{[k-2]}{2}}C_{ij}d_{ij} + \sum_{\substack{i \in [k-2]\\j \in \lbrace k-1,k\rbrace}}C_{ij}d_id_j\\
\nonumber &=& \sum_{i \in [k-2]}(\ell-n)d_i^2 + \sum_{ij \in \binom{[k-2]}{2}}d_id_j(\ell+\ell_k-2n)\\
\nonumber &&\quad\quad-(\ell_k-n)\left( \sum_{i \in [k-2]}d_i^2 + 2\sum_{ij \in \binom{[k-2]}{2}}d_id_j \right)\\
\label{quadratic}&=& (\ell-\ell_k)\left(\sum_{ij \in \binom{[k-2]}{2}}d_id_j +\sum_{i \in [k-2]}d_i^2\right).
\end{eqnarray}

Now let us consider the cubic terms in~(\ref{poly}).
We have
\begin{align*}
&\sum_{\substack{ijh \in [k]^3\\ i \leq j \leq h}}C_{ijh}d_id_jd_h = \sum_{ijh\in \binom{[k]}{3}}d_id_jd_h - \sum_{ij \in \binom{[k]}{2}}d_id_j \cdot \sum_{h \in [k-2]}d_h\\
 &= d_{k-1}d_k\sum_{i \in [k-2]}d_i + (d_{k-1}+d_k)\sum_{ij \in \binom{[k-2]}{2}}d_id_j + \sum_{ijh \in \binom{[k-2]}{3}}d_id_jd_h- \sum_{ij \in \binom{[k]}{2}}d_id_j  \sum_{h \in [k-2]}d_h\\
 &= d_{k-1}d_k\sum_{i \in [k-2]}d_i - \sum_{h \in [k-2]}d_h\cdot \sum_{ij \in \binom{[k-2]}{2}}d_id_j + \sum_{ijh \in \binom{[k-2]}{3}}d_id_jd_h- \sum_{ij \in \binom{[k]}{2}}d_id_j  \sum_{h \in [k-2]}d_h.
\end{align*}
Note that, adding the first and the last terms, we get
\begin{align*}
d_{k-1}d_k\sum_{i \in [k-2]}d_i - \sum_{ij \in \binom{[k]}{2}}d_id_j \cdot \sum_{h \in [k-2]}d_h \stackrel{\eqref{eq:SumDi=0}}{=} \left(\sum_{i \in [k-2]}d_i\right)\left(\sum_{i \in [k-2]}d_i^2 + \sum_{jh \in \binom{[k-2]}{2}}d_jd_h\right),
\end{align*}
which gives some cancellations when combined with the second term. Also, for every $\lbrace i,j,h\rbrace\in{[k-2]\choose 3}$,
$$
|d_id_jd_h|\le \max_{s \in [k-2]}|d_s|\cdot \frac{1}{2}(d_j^2+d_h^2)<\max_{s \in [k-2]}|d_s|\cdot \sum_{t\in[k-2]}d_t^2.
$$
These, together with $\max_i|d_i|\le\frac{\ell-\ell_k}{12k^3}$, imply that
\begin{align}\label{eq-cubic}
\left|\sum_{\substack{ijh \in [k]^3\\ i \leq j \leq h}}C_{ijh}d_id_jd_h \right|&\le \left|\left(\sum_{h \in [k-2]}d_h\right)\sum_{i \in [k-2]}d_i^2\right| + \left|\sum_{ijh \in \binom{[k-2]}{3}}d_id_jd_h\right|\le \frac{\ell-\ell_k}{6}\cdot \sum_{i \in [k-2]}d_i^2.
\end{align}

Thus, combining~(\ref{l13Ci}),~(\ref{quadratic}),~(\ref{eq-cubic}), we have that (\ref{poly}) is equal to
\begin{eqnarray*}
&\phantom{=}& \sum_{i \in [k]}C_id_i + \sum_{1 \leq i \leq j \leq k}C_{ij}d_id_j + \sum_{ijh \in [k]^3}C_{ijh}d_id_jd_h\\
&\geq& \frac{\ell-\ell_k}{2}\left((d_{k-1}+d_k)^2+\sum_{i \in [k-2]}d_i^2\right) -\frac{\ell-\ell_k}{6}\cdot \sum_{i \in [k-2]}d_i^2\\
&\geq& \frac{\ell-\ell_k}{3}\left((d_{k-1}+d_k)^2+\sum_{i \in [k-2]}d_i^2\right).
\end{eqnarray*}
This completes the proof of the claim.
\end{proof}

Now,
\begin{eqnarray*}
&& K_3(G)-K_3(F) = K_3(G)-K_3(H) + d_0(\ell_1+\ldots+\ell_{k-2})\\
&\geq& \sum_{ijh \in \binom{[k]}{3}}s_is_js_h - \sum_{h \in [k-1]}m_h\sum_{\substack{i \in [k-1]\\i \neq h}}s_i - \sum_{ijh \in \binom{[k]}{3}}\ell_i\ell_j\ell_h + (k-2)d_0\ell\\
&=& \sum_{t \in [k-1]}\frac{m_{t}}{m}\left( \sum_{ijh \in \binom{[k]}{3}}s_is_js_h - m'\sum_{\substack{i \in [k-1]\\i \neq t}}s_i - \sum_{ijh \in \binom{[k]}{3}}\ell_i\ell_j\ell_h  - d_0\sum_{\substack{i \in [k-1]\\i \neq t}}s_i + (k-2)d_0\ell\right)\\
&=& \sum_{t \in [k-1]}\frac{m_{t}}{m}\left( \sum_{ijh \in \binom{[k]}{3}}s_is_js_h - m'\sum_{\substack{i \in [k-1]\\i \neq t}}s_i - \sum_{ijh \in \binom{[k]}{3}}\ell_i\ell_j\ell_h\right)  - \sum_{t \in [k-1]} \frac{d_0m_t}{m} \sum_{\substack{i \in [k-1]\\i \neq t}}d_i\\
&\stackrel{(\ref{eq-L13})}{\geq}& \sum_{t \in [k-1]}\frac{m_t}{m} \cdot \frac{\ell-\ell_k}{3} \left( (d_t+d_k)^2 + \sum_{\substack{i \in [k-1]\\i \neq t}}d_i^2 \right) + \sum_{t \in [k-1]} \frac{d_0m_t}{m} (d_t+d_k).
\end{eqnarray*} 
Let $\mathcal{I} \subseteq [k-1]$ be such that $t\in \mathcal{I}$ if and only if
$$
\frac{\ell-\ell_k}{3}(d_t+d_k)^2 + d_0(d_t+d_k) > \frac{\ell-\ell_k}{4}(d_t+d_k)^2.
$$
If $s \in [k-1]\setminus \mathcal{I}$, then $|d_s+d_k| \le 12d_0/(\ell-\ell_k)$. Thus
\begin{align*}
&\sum_{t \in [k-1]}\frac{m_t}{m} \left(\frac{\ell-\ell_k}{3} (d_t+d_k)^2 + d_0(d_t+d_k)\right)\\
&\geq \sum_{t \in \mathcal{I}}\frac{m_t}{m} \cdot \frac{\ell-\ell_k}{4}(d_t+d_k)^2 + \sum_{s \in [k-1]\setminus\mathcal{I}}\frac{m_t}{m} \cdot \frac{\ell-\ell_k}{3}(d_t+d_k)^2 - \frac{12d_0^2}{\ell-\ell_k}\\
&\geq \sum_{t \in [k-1]}\frac{m_t}{m} \cdot \frac{\ell-\ell_k}{4}(d_t+d_k)^2  - \frac{12d^2}{\ell-\ell_k}.
\end{align*}
Thus
\begin{align*}
K_3(G)-K_3(F) &\geq \sum_{t \in [k-1]}\frac{m_t}{m} \cdot \frac{\ell-\ell_k}{4} \left( (d_t+d_k)^2 +\sum_{\substack{i \in [k-1]\\i \neq t}}d_i^2 \right) - \frac{12d^2}{\ell-\ell_k},
\end{align*}
as required.
\end{proof}

\subsection{Partitions}\label{partitions}
The structure of the graphs $G$ we will be working with is somewhat complicated and for much of the proof we make a sequence of local changes to $G$ to obtain a collection of new graphs.
Therefore it is useful to define some types of partition to record all the relevant structural information about these graphs.

Let $k,n,e \in \mathbb{N}$ and $\beta>0$ and let $c = c(n,e)$.
We say that an $(n,e)$-graph $H$ \emph{has a $(V_1,\ldots,V_k;\beta)$-partition} if both of the following hold:
\begin{enumerate}
\item[{\bf\Ppartition($H$):}] $V_1 \cup\ldots \cup V_k$ is a partition of $V(H)$ and
$$
\bigg||V_i|-cn\bigg|,\bigg||V_k|-(1-(k-1)c)n\bigg| \leq \beta n
$$
for all $i \in [k-1]$;
\item[{\bf\Pcomplete($H$):}] $H[V_i,V_j]$ is complete for all $ij \in \binom{[k-1]}{2}$.
\end{enumerate}

\medskip
Let $\delta > 0$. We say that $H$ \emph{has a $(V_1,\ldots,V_k;U,\beta,\delta)$-partition} if, in addition to \Ppartition($H$) and \Pcomplete($H$), $U$ is a subset of $V(H)$ such that the following properties hold:

\medskip
\begin{itemize}
\item[{\bf\Pbadedges($H$):}] $|U| \leq \delta n$ and every edge in $\bigcup_{i \in [k]}E(H[V_i])$ is incident with a vertex of $U$; also $\Delta(H[V_i]) \leq \delta n$ for all $i \in [k]$;
\item[{\bf\PZk($H$):}] $U \cap V_k$ has a partition $U_k^1 \cup \ldots \cup U_k^{k-1}$ such that for all $ij \in \binom{[k-1]}{2}$ we have that $G[U_k^i,V_j]$ is complete.
\end{itemize}

If $\gamma_1,\gamma_2>0$ and in addition to \Ppartition($H$)--\PZk($H$), the following property holds, then we say that $H$ \emph{has a $(V_1,\ldots,V_k;U,\beta,\gamma_1,\gamma_2,\delta)$-partition.}
\begin{itemize}
\item[{\bf\Pmissing($H$):}] If $y \in V_i \setminus U$ then $d^m_H(y) := e_{\overline{H}}(y,\overline{V_i}) < \gamma_2 n$ and if $y \in V_i \cap U$ then $d^m_H(y) \geq \gamma_1 n$, for all $i \in [k]$.
\end{itemize}

If \Ppartition($H$), \Pbadedges($H$) and \Pmissing($H$) hold then we say that $H$ \emph{has a weak $(V_1,\ldots,V_k;U,\beta,\gamma_1,\gamma_2,\delta)$-partition}.
Observe that if $\beta^+ \geq \beta$; $\gamma_1^- \leq \gamma_1$; $\gamma_2^+ \geq \gamma_2$ and $\delta^+ \geq \delta$, then a $(V_1,\ldots,V_k;U,\beta,\gamma_1,\gamma_2,\delta)$-partition is also a $(V_1,\ldots,V_k;U,\beta^+,\gamma_1^-,\gamma_2^+,\delta^+)$-partition.
We call $d^m_H(y)$ the \emph{missing degree} of a vertex $y \in V(H)$ with respect to the partition $V_1,\ldots,V_k$.
Let $\underline{m} = (m_1,\ldots,m_{k-1})$ where, for all $i \in [k-1]$ we have $m_i := e(\overline{H}[V_i,V_k])$. We say that $\underline{m}$ is the \emph{missing vector} of $H$ with respect to $(V_1,\ldots,V_k)$.
Observe that, by \Pcomplete($H$),
$$
m_i = \sum_{v \in V_i}d^m_H(v).
$$
An edge is \emph{bad} if both of its endpoints lie in the same $V_i$.
Let $h := \sum_{i \in [k]}e(H[V_i])$ be the total number of bad edges.

\section{Initial steps in the proof of Theorem~\ref{strong}}\label{beginning}

We start by deriving Theorem~\ref{main} from Theorems~\ref{LS83},~\ref{strong} and Proposition~\ref{pr:compute}.
The rest of the paper will concentrate on proving Theorem~\ref{strong}.

\subsection{The proof of Theorem~\ref{main} given Theorem~\ref{strong}}\label{derivation}
Let $\eps>0$.
Assume $\eps < 1/2$.
Theorem~\ref{LS83} gives $\alpha(3,k)>0$ and $n_0(3,k)$ for each integer $3 \leq k \leq 1/\eps$.
Let $\alpha>0$ be the minimum of the above constants $\alpha(3,k)$.
Apply Theorem~\ref{strong} with parameters $\eps,\alpha$ to obtain $n_0(\alpha,k)$ for each integer $3 \leq k \leq 1/\eps$. Let $n_0$ be the maximum of $n_0(3,k)$ and $n_0(\alpha,k)$ over such $k$.

Now let $n \geq n_0$ and $e \leq \binom{n}{2}-\eps n^2$ be positive integers.
Let $k=k(n,e)$, so $t_{k-1}(n) < e \leq t_k(n)$.
If $k \leq 2$ then $g_3(n,e)=0$ and we are done as then
$$
\mathcal{H}_2^*(n,e)\subseteq \mathcal{H}_0^*(n,e)=\{\mbox{$K_3$-free $(n,e)$-graphs}\}.
$$
 So we may assume that $k \geq 3$.
Further, Lemma~\ref{lm:Turan41fact} implies that $k \leq 1/(2\eps)+1 \leq 1/\eps$.
Suppose first that $t_{k-1}(n) < e \leq t_{k-1}(n)+\alpha n^2$.
Then, since $\alpha \leq \alpha(3,k)$, Theorem~\ref{LS83} applied with $r := 3$ implies that $g_3(n,e)=h(n,e)$
 and every extremal graph lies in $\mathcal{H}_0(n,e) \cup \mathcal{H}_2(n,e)$.
Proposition~\ref{pr:compute} then implies that the extremal value is $h^*(n,e)=h(n,e)$ and the family of extremal graphs is precisely $\mathcal{H}_0^*(n,e)\cup\mathcal{H}_2^*(n,e)$.

Suppose instead that $t_{k-1}(n) + \alpha n^2 \leq e \leq t_k(n)$.
Then Theorem~\ref{strong} implies that every extremal graph lies in $\mathcal{H}(n,e)$.
Proposition~\ref{pr:compute} then implies that the family of extremal graph is precisely $\mathcal{H}_1^*(n,e) \cup \mathcal{H}_2^*(n,e)$ (and note that $\mathcal{H}_0^*(n,e)=\mathcal{H}_1^*(n,e)$ for this $e$ by~\eqref{eq:H0=H1}).
So certainly $g_3(n,e)=h(n,e)$.
\hfill$\square$

\subsection{Beginning the proof of Theorem~\ref{strong}}\label{beginningproof}
Let $\eps> 0$.
Suppose that Theorem~\ref{strong} does not hold for this $\eps$.
Then take the minimal integer $k \leq 1/\eps$ such that the conclusion is not true at this $k$ for some $\alpha$, and then choose such an $\alpha$.
By decreasing $\alpha$, we can assume that $\alpha \ll \eps$, and that $\alpha \leq (\alpha_{\ref{LS83}})^5$, where $\alpha_{\ref{LS83}}$ is the minimum constant $\alpha(3,k)$ obtained by applying Theorem~\ref{LS83} with parameters $k$ and $r=3$, for all $3 \leq k \leq 1/\eps$.

By the minimality of $k$, we have that, for all $\ell \in [k-1]$ and all $\alpha'>0$, there exists $n_0(\ell,\alpha')>0$ such that every extremal $(n,e)$-graph with $n \geq n_0(\ell,\alpha')$ and $t_{\ell-1}(n)+\alpha'n^2 \leq e \leq t_\ell(n)$ lies in $\mathcal{H}(n,e)$.

Note that $k\geq 3$ as when $k(n,e)=2$, the family $\mathcal{H}(n,e)$ is the family of $n$-vertex $e$-edge triangle-free graphs, and $g_3(n,e)=0$. (So we can set $n_0(2,\alpha)=1$ for every $\alpha>0$.)

Choose $n_0 = n_0(k) \in \mathbb{N}$ and additional constants such that the dependencies between them are as follows:
\begin{align}\label{hierarchy}
\nonumber &0 < \frac{1}{n_0} \ll \rho_4 \ll \ldots \ll \rho_0 \ll \eta \ll \delta \ll \beta \ll \xi \ll \gamma \ll \alpha \leq (\alpha_{\ref{LS83}})^5\\
&\hspace{1cm} \ll \delta' \ll \xi' \ll \eps \leq \frac{1}{k}.
\end{align}
In particular, we assume that Theorem~\ref{PikhurkoRazborov17approx} holds for $n_0$ with $\rho_4$ playing the role of $\eps$ and that
\begin{equation}\label{n0}
n_0 \geq \max\left\lbrace 2 \cdot n_0(k-1,\alpha/3),~~n_{\ref{PikhurkoRazborov17approx}}(\rho_4),~~2\cdot n_{\ref{LS83}}(k)\right\rbrace,
\end{equation}
where $n_{\ref{PikhurkoRazborov17approx}}(\rho_4)$ is the output of Theorem~\ref{PikhurkoRazborov17approx} applied with parameter $\rho_4$; and $n_{\ref{LS83}}(k)$ is (along with $\alpha_{\ref{LS83}}$) the output of Theorem~\ref{LS83} applied with $k-1$ and $r=3$.
For the reader's convenience, the glossary at the end of the paper gives an informal overview of the roles of the constants in~(\ref{hierarchy}).
We may ignore floors and ceilings where they do not affect our argument.

Now, suppose that Theorem~\ref{strong} fails for this $n_0$, $k$ and $\alpha$. Pick the smallest $n\ge n_0$ such that there is $e$ with
 \begin{equation}\label{cfirst}
t_{k-1}(n)+\alpha n^2 \leq e \leq t_k(n).
\end{equation}
 for which at least one extremal $(n,e)$-graph is not in $\mathcal{H}(n,e)$. If there is more than one choice for $e$ then choose one with
$g_3(n,e)-h(n,e)$ being smallest possible. By Theorem~\ref{LS83}, the inequality
\begin{equation}\label{worst}
g_3(n,e)-h(n,e) \leq g_3(n,e')-h(n,e')
\end{equation}
holds in fact for every $e'$ with $k(n,e')=k$. (Indeed, if $t_{k-1}(n)\le e'< t_{k-1}(n)+\alpha n^2$ then~\eqref{worst} holds as its right-hand side is zero.)

Next, if there is more than one choice of the graph $G$, choose it according to the following criteria in the given order:
\begin{enumerate}[label = (C\arabic{enumi})]
	\item\label{worst-C1} $G \notin \mathcal{H}(n,e)$ has the minimum number of triangles: $K_3(G)=g_3(n,e)$;
	
	\item\label{worst-C2} $G$ has a maximum max-cut $k$-partition: If $A_1^G,\ldots,A^G_k$ is a max-cut partition of $V(G)$, then for every $(n,e)$-graph $J \notin \mathcal{H}(n,e)$ with $K_3(J)=g_3(n,e)$ and every (equivalently, some) max-cut partition $A^J_1,\ldots,A^J_k$ of $V(J)$, we have that
	\begin{equation*}
	\sum_{ij \in \binom{[k]}{2}}e\left(G[A^G_i,A^G_j]\right) \geq \sum_{ij \in \binom{[k]}{2}}e\left(J[A^J_i,A^J_j]\right).
	\end{equation*}
	
	\item\label{worst-C3} There exists a max-cut $k$-partition $A_1^G,\ldots,A^G_k$ of $V(G)$ such that for every $(n,e)$-graph $J$ satisfying~\ref{worst-C1} and~\ref{worst-C2} and every max-cut partition $A^J_1,\ldots,A^J_k$ of $V(J)$, we have
	$$\min_{i\in[k]}\left|A_i^G\right|\le \min_{i\in[k]}\left|A_i^J\right|.$$
\end{enumerate}

We say that such a graph $G$ is a \emph{worst counterexample}.
From now on, $G,n,e$ and all the constants in~(\ref{hierarchy}) are fixed.
Define $c = c(n,e)$.
Corollary~\ref{cr:dh}, Proposition~\ref{pr:compute} and~(\ref{worst}) imply that 
	\begin{equation}\label{2path}
	P_3(wx,G) \geq (k-2)cn-k \quad\text{ and }\quad P_3(yz,G) \leq (k-2)cn+k
	\end{equation}
for all $wx \in E(\overline{G})$ and $yz \in E(G)$.
Since $n$ and $e$ satisfy~(\ref{cfirst}), we have by~\eqref{eq:c2} and Lemma~\ref{lm:Turan41fact} that
\begin{equation}\label{eq:solc2}
\frac{1}{k} \leq c 
\leq \frac{1}{k} + \frac{\sqrt{1-2\alpha k(k-1)}}{k(k-1)} + O(1/n) < \frac{1}{k-1}-\alpha.
\end{equation}
(Here we used $\sqrt{1-x} < 1-x/2$ for $x \in (0,1]$.)
Thus
\begin{equation}\label{ineq:c3}
0 \leq kc - 1 < c-(k-1)\alpha.
\end{equation}
Further, using Theorem~\ref{razthm} and the fact that $e \leq \binom{n}{2}-\eps n^2$, we have
\begin{equation}\label{nikicons}
\left|K_3(K^k_{cn,\ldots,cn,n-(k-1)cn}) - K_3(G)\right| \stackrel{(\ref{eq:UpperAsymptG3}),(\ref{K3c})}{=} \left| \frac{n^3}{6}g_3\left(\frac{2e}{n^2}\right) - g_3(n,e) \right| \stackrel{(\ref{asympbound})}{\leq} \frac{n}{2\eps}.
\end{equation}

Before splitting into cases depending on the size of the difference $t_k(n)-e$, we prove the following useful statement about some structural properties of $G$.

\begin{lemma}\label{baddegclaim}
	Let $0 < 1/n \ll \rho \ll 1/k$, and let $p,d > 0$ be such that
	\begin{equation}\label{blah}
	p^2 \leq d \leq \rho n^2\quad\text{and}\quad 2\rho^{1/6} \leq 1-(k-1)c.
	\end{equation}
	Suppose that there is a partition $V_1,\ldots,V_k$ of $V(G)$ for which \Ppartition($G$) holds with parameter $p/n$ and
	\begin{equation}\label{ed}
	|\,E(G)\bigtriangleup E(K[V_1,\ldots,V_k])\,| \leq d.
	\end{equation}
	Let $A_1,\ldots,A_k$ be a max-cut partition of $G$ where $|A_k| \leq |A_i|$ for all $i \in [k-1]$.
	Then
	\begin{itemize}
		\item[(i)] \Ppartition($G$) holds with respect to $A_1,\ldots,A_k$ with parameter $2k^2\sqrt{d}/n$;
		\item[(ii)] we have
		\begin{equation}\label{hm}
		m := \sum_{ij \in \binom{[k]}{2}}e(\overline{G}[A_i,A_j]) \leq 2k^2\sqrt{d}(kc-1)n+d \leq 3k^2\sqrt{\rho}n^2.
		\end{equation}
	\end{itemize}
	Moreover, for all $i \in [k]$:
	\begin{itemize}
		\item[(iii)] if $xy \in E(G[A_i])$, then $d_{\overline{G}}(x,\overline{A_i})+d_{\overline{G}}(y,\overline{A_i}) \geq (1-(k-1)c)n - 3k^2\sqrt{\rho}n \geq \rho^{1/6}n$;
		\item[(iv)] $\Delta(G[A_i]) \leq \rho^{1/5}n$;
		\item[(v)] $e(G[A_i]) \leq \rho^{1/30}m$.
	\end{itemize}
\end{lemma}

\begin{proof}
	By~(\ref{ed}), there is a partition $V_1,\ldots,V_k$ of $V(G)$ such that, defining $n_i:=|V_i|$ for $i \in [k]$, we have 
	\begin{equation}\label{nicn}
	|n_i-cn| \leq p\quad\text{for all}~i \in [k-1]\quad\text{and}\quad |n_k-(n-(k-1)cn)| \leq p;
	\end{equation}
	and
	$$
	\sum_{i \in [k]}e(G[V_i])+\sum_{ij \in \binom{[k]}{2}}e(\overline{G}[V_i,V_j]) \leq d.
	$$
	The max-cut property implies that
	\begin{equation*}
	\sum_{ij \in \binom{[k]}{2}}e(G[A_i,A_j]) \geq \sum_{ij \in \binom{[k]}{2}}e(G[V_i,V_j]) \geq e-d
	\end{equation*}
	and so
	\begin{equation}\label{heq}
	h := \sum_{i\in[k]}e(G[A_i])=e-\sum_{ij\in{[k]\choose 2}}e(G[A_i,A_j]) \leq d.
	\end{equation}
	For $i \in [k]$, choose $j = j(i) \in [k]$ such that $|A_i \cap V_j|$ is maximal.
	Suppose that there exists $h \in [k]\setminus \lbrace j \rbrace$ such that $|A_i \cap V_h| > \sqrt{2d}$.
	Then
	$$
	e(G[A_i]) \geq |A_i \cap V_j|\,|A_i \cap V_h|-|\,E(G)\bigtriangleup E(K[V_1,\dots,V_k])\,|> (\sqrt{2d})^2-d = d,
	$$
	a contradiction to~\eqref{heq}.
	Thus for each $i \in [k]$ there exists at most one $h \in [k]$ such that $|A_i \cap V_h| > \sqrt{2d}$.
	Suppose that there is some $j \in [k]$ for which no $i \in [k]$ satisfies $j(i)=j$.
	Then, using~(\ref{blah}), we get
	$$
	2k\sqrt{2d} + p \leq 3k\sqrt{2d} \leq 3k\sqrt{2\rho}\,n
	< n-(k-1)cn,
	$$
	and so
	$$
	n_j=\sum_{i \in [k]}|A_i \cap V_j| < k\sqrt{2d} < \frac{n-(k-1)cn-p}{2}.
	$$
	Recall from~(\ref{ineq:c3}) that $c \geq 1-(k-1)c$, so this is
	a contradiction to~(\ref{nicn}).
	Thus, the function $j:[k]\to [k]$ is a bijection and, for each $i \in [k]$,
	$$
	|A_i| \geq |V_{j(i)}| - \sum_{i' \in [k]\setminus \lbrace i \rbrace}|A_{i'} \cap V_{j(i)}| \geq n_{j(i)}-k\sqrt{2d},
	$$
	and similarly $|A_i| \leq n_{j(i)}+k\sqrt{2d}$.
	Suppose first that $j(k)=k$.
	Then
	$$
	\bigg| |A_k|-(n-(k-1)cn) \bigg| \leq |n_k-(n-(k-1)cn)| + k\sqrt{2d} \leq p+k\sqrt{2d} \leq 2k\sqrt{d}
	$$
	and similarly $|\,|A_i|-cn\,| \leq 2k\sqrt{d}$ for all $i \in [k-1]$.
	Suppose instead that $j(k) \neq k$.
	Then $|\,|A_k|-cn\,| \leq k\sqrt{2d}$, and since $A_k$ is the smallest part we have that $n = \sum_{i \in [k]}|A_i| \geq k(cn-k\sqrt{2d})$.
	Thus $cn-k^2\sqrt{2d} \leq n-(k-1)cn \leq cn$, where the last inequality follows from~(\ref{ineq:c3}).
	So
	\begin{align*}
	\bigg|\, |A_k|-(n-(k-1)cn)\, \bigg| &\leq \big|\,|A_k|-n_{j(k)}\big| + |n_{j(k)}-cn| + |cn-(n-(k-1)cn)|\\
	&\leq k\sqrt{2d} + p + k^2\sqrt{2d} \leq 2k^2\sqrt{d},
	\end{align*}
	and similarly $|\,|A_i|-cn\,| \leq 2k^2\sqrt{d}$ for all $i \in [k-1]$.
	Hence \Ppartition($G$) holds with parameter $2k^2\sqrt{d}/n$, proving (i).
	So it also holds with parameter $2k^2\sqrt{\rho} \geq 2k^2\sqrt{d}/n$.
	
	We now prove (ii).
	Write $p_i := cn$ for $i \in [k-1]$ and $p_k := n-(k-1)cn$; and $d_i := p_i-|A_i|$ for all $i \in [k]$. Then $\sum_{i \in [k]}d_i=0$, and we have
	\begin{eqnarray}
	\nonumber m &\stackrel{(\ref{heq})}{=}& \sum_{ij \in \binom{[k]}{2}}|A_i|\,|A_j|-e+h = \frac{1}{2}\left(n^2-\sum_{i \in [k]}p_i^2 + 2\sum_{i\in [k]}p_id_i - \sum_{i \in [k]}d_i^2\right)-e+h\\
	\nonumber &\stackrel{(\ref{heq})}{\leq}& \frac{1}{2}\left(n^2-(k-1)c^2n^2-(n-(k-1)cn)^2\right) + cn\sum_{i \in [k-1]}d_i + (n-(k-1)cn)d_k - e + d\\
	\label{meqagain} &\stackrel{(\ref{eq:c})}{=}& -d_k(kc-1)n +d \stackrel{(i)}{\leq} 2k^2\sqrt{d}(kc-1)n+d \stackrel{(\ref{blah})}{\leq} 3k^2\sqrt{\rho}n^2,
	\end{eqnarray}
	as required.
	
	Next we prove (iii). For any $i \in [k]$, and $xy \in E(G[A_i])$,
	$$
	(k-2)cn+k \stackrel{(\ref{2path})}{\geq} P_3(xy,G) \geq n-|A_i| - (d_{\overline{G}}(x,\overline{A_i})+d_{\overline{G}}(y,\overline{A_i}))
	$$
	and so
	\begin{eqnarray*}
		d_{\overline{G}}(x,\overline{A_i})+d_{\overline{G}}(y,\overline{A_i}) &\stackrel{(i),(\ref{ineq:c3})}{\geq}& n-(k-2)cn-k-cn - 2k^2\sqrt{\rho} n\\
		&\geq&  (1-(k-1)c)n - 3k^2\sqrt{\rho}n \stackrel{(\ref{blah})}{\geq} \rho^{1/6}n,
	\end{eqnarray*}
	as required.
	
	For (iv), suppose on the contrary that there exist $i \in [k]$ and $x \in A_i$ with $d_G(x,A_i) > \rho^{1/5} n$.
	Suppose first that $d_{\overline{G}}(x,\overline{A_i}) \geq k\rho^{1/5} n$.
	By averaging, there is some $\ell \in [k]\setminus \lbrace i \rbrace$ such that $d_{\overline{G}}(x,A_\ell) \geq \rho^{1/5} n$.
	For each $j \in [k]$, let $X_j := N_G(x,A_j)$ and $x_j := |X_j|$.
	By the max-cut property, for any $j \neq i$, we have $x_j \geq x_i \geq \rho^{1/5} n$.
	Let $L$ be the number of triangles containing $x$ and no other vertices from $A_i \cup A_\ell$.
	Part (ii) implies that
	$$
	K_3(x,G) \geq L + x_\ell x_i + (x_i + x_\ell)(n-x_i-x_\ell) - 3k^2\sqrt{\rho} n^2.
	$$
	Obtain a new graph $G'$ by choosing $A_i' \subseteq X_i$ and $A_\ell' \subseteq A_\ell\setminus X_\ell$ with $|A_i'|=|A_\ell'|=\rho^{1/5} n$ and letting $E(G') := (E(G) \cup \lbrace xy:y \in A_\ell'\rbrace) \setminus \lbrace xz:z \in A_i'\rbrace$.
	Now
	$$
	K_3(x,G') \leq L + (x_\ell+\rho^{1/5} n)(x_i-\rho^{1/5} n) + (x_i+x_\ell)(n-x_i-x_\ell).
	$$
	Thus
	$$
	K_3(G')-K_3(G) \leq \rho^{1/5} n(x_i-x_\ell)-\rho^{2/5} n^2 + 3k^2\sqrt{\rho} n^2 < -\rho^{2/5} n^2/2,
	$$
	a contradiction.
	Thus $d_{\overline{G}}(x,\overline{A_i}) < k\rho^{1/5} n$.
	But (ii) also implies that
	$$
	\sum_{y \in X_i}d_{\overline{G}}(y,\overline{A_i}) \leq e(\overline{G}[\,A_i,\overline{A_i}\,]) \leq 3k^2\sqrt{\rho} n^2,
	$$
	so there exists $y \in X_i$ with $d_{\overline{G}}(y,\overline{A_i}) \leq 3k^2\sqrt{\rho} n^2/x_i \leq 3k^2\rho^{3/10}n$.
	But then
	$$
	d_{\overline{G}}(x,\overline{A_i})+d_{\overline{G}}(y,\overline{A_i}) \leq (k\rho^{1/5}+3k^2\rho^{3/10})n < \rho^{1/6}n,$$
	contradicting (iii).  
	
	Finally, we prove (v). Using the previous parts, we have for all $i \in [k]$ that
	\begin{eqnarray*}
		\rho^{1/5} nm &\geq& \rho^{1/5} n \cdot e(\overline{G}[A_i,\overline{A_i}]) \ \stackrel{(iv)}{\geq}\ \sum_{\substack{xy \in E(\overline{G}[\,A_i,\overline{A_i}\,])\\ x \in A_i}}d_G(x,A_i)\\
		&=& \sum_{uv \in E(G[A_i])}(d_{\overline{G}}(u,\overline{A_i})+d_{\overline{G}}(v,\overline{A_i}))\
		\stackrel{(iii)}{\geq}\ e(G[A_i])\rho^{1/6}n,
	\end{eqnarray*}
	giving the required.
\end{proof}


\section{The intermediate case: approximate structure}\label{int1}

We will assume in this section and the succeeding two sections that
\begin{equation}\label{eq:alpha}
t_{k-1}(n) +\alpha n^2 < e < t_k(n) - \alpha n^2
\end{equation}
and say that we are in the \emph{intermediate case}.
(The remaining \emph{boundary} case is treated in Section~\ref{bound}.)
Equations~(\ref{eq:c2'}) and~(\ref{eq:alpha}) imply that
\begin{equation}\label{eq:solc}
c \geq \frac{1}{k}+\sqrt{\frac{2\alpha}{k(k-1)}} > \frac{1 + \sqrt{2\alpha}}{k}.
\end{equation}
Thus we can improve one inequality in~(\ref{ineq:c3}):
\begin{equation}\label{ineq:c}
\quad \sqrt{2\alpha}< kc - 1 \le  c-(k-1)\alpha.
\end{equation}

The aim of this section is to prove the forthcoming lemma about the approximate structure of $G$ in the intermediate case.
One consequence of the statement is that, when $A_1,\ldots,A_k$ is a max-cut partition of $G$, then actually $G$ is close to the complete partite graph $K[A_1,\ldots,A_k]$.
Note that this is not true for an arbitrary extremal graph $H$, so here we crucially use the fact that $G$ is a worst counterexample,~i.e. it satisfies~\ref{worst-C1}--\ref{worst-C3}.

\begin{lemma}[Approximate structure]\label{approx}
Suppose that~(\ref{eq:alpha}) holds.
Let $A_1,\ldots,A_k$ be a max-cut partition of $V(G)$ such that $|A_k| \leq |A_i|$ for all $i \in [k-1]$.
Then there exists $Z \subseteq V(G)$ such that $G$ has an $(A_1,\ldots,A_k;Z,\beta,\xi,\xi,\delta)$-partition with missing vector $\underline{m} =: (m_1,\ldots,m_{k-1})$ such that $m\le \eta n^2$
and  $h \leq \delta m$, where $m := m_1 + \ldots + m_{k-1}$ and $h$ is defined in~\eqref{heq}.
\end{lemma}

To prove the lemma, we will use Theorem~\ref{PikhurkoRazborov17approx} together with a somewhat involved series of deductions.
Define a function $f: V(G) \rightarrow \mathbb{R}$ by setting
\begin{equation}\label{eq-f}
	f(x):=(d_G(x)-(k-2)cn)(k-2)cn+{k-2\choose 2}c^2n^2-K_3(x,G),\quad x \in V(G).
\end{equation}
 The intuition behind this formula is that it becomes the zero function if we apply it to $H:=K^k_{cn,\dots,cn,(1-(k-1)c)n}$ with $c=c(n,e)$:
  \begin{equation}\label{eq-fmotivation}
  (d_H(x)-(k-2)cn)(k-2)cn+{k-2\choose 2}c^2n^2-K_3(x,H)=0\quad\mbox{for all $x\in V(H)$}.
  \end{equation}  
 It turns out that $f(x)$ is small in absolute value for every $x \in V(G)$.

\begin{lemma}\label{lem-f}
	$|f(x)| \leq 6n/\sqrt{\alpha}$ for all $x \in V(G)$.
\end{lemma}

\begin{proof}
    We first give a bound on the gradient of the function $c(n,\cdot)$ that was defined in~\eqref{eq:c2}. We will write $c:=c(n,e)$ as usual. Note that $k(2e/n^2)=k(n,e)$ by Lemma~\ref{lm:ks}.
	Setting $s:=1/\sqrt{\al}$, we have
	\begin{equation}\label{esn}
	e(K^k_{cn,\ldots,cn,cn-s,(1-(k-1)c)n+s})-e=s(kc-1)n-s^2\stackrel{(\ref{ineq:c})}{\ge} \sqrt{2\al}sn-1/\al>\sqrt{\al}sn=n.
	\end{equation}
    Let $p := e(K^k_{cn-\frac{s}{k-1},\ldots,cn-\frac{s}{k-1},(1-(k-1)c)n+s})$ and $c' := c(n,e+n)$.
Then
\begin{align*}
p > e(K^k_{cn,\ldots,cn,cn-s,(1-(k-1)c)n+s}) \stackrel{(\ref{esn})}{\geq} e+n = e(K^k_{c'n,\ldots,c'n,(1-(k-1)c')n}).
\end{align*}
This, together with the fact that $c(n,\cdot)$ is a non-increasing function, implies that $c \geq c'\ge c-\frac{s}{(k-1)n}$, so
    \begin{equation}\label{eq-c-change}
    (k-2)c'n \geq (k-2)\left(cn-\frac{s}{k-1}\right) \geq (k-2)cn-\frac{1}{\sqrt{\alpha}}.
    \end{equation}

Next,~\eqref{eq-fmotivation} (or a direct calculation using~\eqref{eq:NewC},~\eqref{eq:UpperAsymptG3} and~\eqref{eq-f}) shows that
\begin{equation}\label{eq-sum-f}
	\sum_{v\in V(G)}f(v)=3\left(K_3(K^k_{cn,\ldots,cn,n-(k-1)cn}) - K_3(G)\right).
\end{equation}     
Now let $x,y \in V(G)$ be two arbitrary distinct vertices. Let $G'$ be the graph obtained from $G$ by deleting $y$ and cloning $x$.
	(By cloning, we mean adding a new vertex $x'$ whose neighbourhood is identical to $N_G(x)\setminus \lbrace y \rbrace$; so, in particular, $xx' \notin E(G')$.) Then, letting $e' := e(G')-e(G)$, we have that
	\begin{equation*}
e' = \begin{cases} d(x)-d(y), &\quad\mbox{if } xy \notin E(G), \\ 
d(x)-d(y)-1, & \quad\mbox{otherwise. } \end{cases}
\end{equation*}
 Clearly, $|e'|\le n$ and so $k(n,e+e')=k(n,e)$. 
 
 Suppose first that $e' \geq 0$.
 Using Lemma~\ref{lem-slope-c},~(\ref{eq:alpha}) and the facts that $G$ is a worst counterexample and that $c(n,\cdot)$ is a non-increasing function, we have 
	\begin{eqnarray*}
    K_3(G')-K_3(G)&\stackrel{(\ref{worst})}{\ge}& h(n,e+e')-h(n,e) = \sum_{i=1}^{e'}(h(n,e+i)-h(n,e+i-1))\\
     &\ge& \sum_{i=1}^{e'}\left((k-2)\cdot c(n,e+i-1)\cdot n-k\right) \ge e'(k-2)c' n-kn\\
     &\stackrel{(\ref{eq-c-change})}{\ge}& e'(k-2)cn-\frac{2n}{\sqrt{\al}}.
    \end{eqnarray*}
On the other hand, $K_3(G')-K_3(G) \leq K_3(x,G)-K_3(y,G)+(n-2)$. Thus 
	\begin{eqnarray*}
    K_3(x,G)-K_3(y,G)\ge (k-2)cn (d(x)-d(y)-1)-\frac{2n}{\sqrt{\al}} \geq (k-2)cn(d(x)-d(y))-\frac{3n}{\sqrt{\alpha}}.
	\end{eqnarray*}
This implies that 
$$f(x)-f(y)=(d(x)-d(y))(k-2)cn-(K_3(x,G)-K_3(y,G))\le \frac{3n}{\sqrt{\al}}.$$
Using an analogous argument assuming $e' < 0$ and the fact that $x,y$ were arbitrary, we derive that for any $x,y\in V(G)$,
\begin{equation}\label{eq-f-symm}
	|f(x)-f(y)|\le \frac{3n}{\sqrt{\al}}.
\end{equation}

Suppose now for some $x\in V(G)$, we have $|f(x)| \geq 6n/\sqrt{\alpha}$. Then
\begin{eqnarray*}
\frac{3n^2}{\sqrt{\alpha}} &\stackrel{(\ref{eq-f-symm})}{\leq}& \left|\sum_{v \in V(G)}f(v)\right| \stackrel{(\ref{eq-sum-f})}{=} 3\left|K_3(K^k_{cn,\ldots,cn,n-(k-1)cn}) - K_3(G)\right| \stackrel{(\ref{nikicons})}{\leq} \frac{3n}{2\eps},
\end{eqnarray*}
so $1/n_0 \geq 1/n \geq 2\eps/\sqrt{\alpha} \geq \sqrt{\eps}$,
a contradiction to~(\ref{hierarchy}).
\end{proof}

\begin{cor}\label{cr:DeltaG}
	$$\Delta(G)\le (k-1)cn+\frac{42}{\sqrt{\alpha}} \quad  \mbox{ and}\quad \delta(G)\ge (k-2)cn-k.$$
\end{cor}
\begin{proof} 
Let $x \in V(G)$ be arbitrary.
By Lemma~\ref{lem-f},
\begin{eqnarray*}
	&&(d_G(x)-(k-2)cn)(k-2)cn+{k-2\choose 2}c^2n^2=K_3(x,G)+f(x)\\
	&\le& \frac{1}{2}\sum_{y \in N_G(x)}P_3(xy,G) +\frac{6n}{\sqrt{\alpha}} \stackrel{(\ref{2path})}{\le} \frac12 d_G(x) ((k-2)cn+k)+\frac{6n}{\sqrt{\alpha}} \le \frac12 d_G(x) (k-2)cn+\frac{7n}{\sqrt{\alpha}}.
\end{eqnarray*}
Solving for $d_G(x)$, we have, using $c\ge 1/k$, that
$$
d_G(x)\le (k-1)cn+\frac{14}{\sqrt{\alpha}(k-2)c}\le (k-1)cn+\frac{14k}{\sqrt{\alpha}(k-2)}\le (k-1)cn+\frac{42}{\sqrt{\alpha}}.
$$
The claim about minimum degree 
trivially follows from~(\ref{2path}).
\end{proof}



\subsection{$G$ is almost complete $k$-partite}

Theorem~\ref{PikhurkoRazborov17approx} implies that our worst counterexample $G$ is close in edit distance to \emph{some} graph in $\mathcal{H}^*(n,e)$.
In this subsection, we prove that in fact $G$ is close in edit distance to the specific graph $H^*(n,e)$ in $\mathcal{H}^*(n,e)$.
Recall from Definition~\ref{astardef} and~(\ref{m*}) that the edit distance between $H^*(n,e)$ and $K_{a_1^*,\ldots,a_k^*}$ is at most $n$. But Lemma~\ref{lem-slope-c} implies that additionally $|a_i^*-cn| \leq 2$ for all $i \in [k-1]$, so we will in fact show that the edit distance between $G$ and the complete $k$-partite graph with $k-1$ parts of size $\lfloor cn \rfloor$ is $o(n^2)$.

\begin{lemma}\label{lem-Almost-k-Partite} 
$|E(G)\bigtriangleup E(K^k_{\lfloor cn \rfloor,\ldots,\lfloor cn\rfloor,n-(k-1)\lfloor cn\rfloor})| \leq \rho_0 n^2$.
\end{lemma}

\begin{proof}  
Suppose that the statement is not true.
We will first derive some structural properties of $G$ under this assumption.

Let $\mathcal{H}_1(n)$ be the set of $n$-vertex graphs $H$ with vertex partition $A \cup B$ such that $H[A]$ is complete $(k-2)$-partite; $H[A,B]$ is complete, and $H[B]$ is triangle-free.
Pick $H\in\cH_1(n)$ with the minimal edit distance to $G$.
Theorem~\ref{PikhurkoRazborov17approx} and~(\ref{n0}) imply that
\begin{equation}\label{editdistance}
|E(H)\bigtriangleup E(G)|\leq  \rho_4 n^2.
\end{equation}
(Note that $H$ need not have $e$ edges, although we do have $|e-e(H)| \leq \rho_4 n^2$.)
By definition, $H$ comes with a \emph{canonical partition} $A_1,\ldots,A_{k-2},B$ such that each $A_i$ is an independent set and $H[B]$ is triangle free, and $H[A_1,\ldots,A_{k-2},B]$ is complete $(k-1)$-partite.
Now, $G$ is $\rho_4n^2$-close to some graph $H' \in \mathcal{H}_1^*(n,e)$ in which for $i \in [k-2]$ the $i$th part has size $a_i^* = cn\pm 2$ (by Lemma~\ref{lem-slope-c}).
Thus $H$ is $2\rho_4 n^2$-close to $H'$ and consequently 
\begin{equation}\label{Aisize}
\bigg||A_i|-cn\bigg| < \rho_3 n\quad\text{for all}\quad i \in [k-2].
\end{equation}
Let $A := \bigcup_{i \in [k-2]}A_i$.

\begin{claim}\label{cl-AT}
The following hold in $G$:
\begin{itemize}
\item[(i)] for every $x\in A$, $d_G(x,B)>(c+\rho_0)n$ or $d_G(x,A)< ((k-2)c-\rho_0)n$;

\item[(ii)] for any $y\in V(G)$ and $ij \in \binom{[k-2]}{2}$ such that $\min\lbrace d_{\overline{G}}(y,A_i),d_{\overline{G}}(y,A_j)\rbrace\ge\rho_3n$, we have $\min\lbrace d_{G}(y,A_i),d_{G}(y,A_j) \rbrace \leq \rho_3n$;

\item[(iii)] for every $y\in B$, $d_G(y,A)> (k-3)cn+\rho_0n$ or $d_G(y,B)< cn-\rho_0n$.
\end{itemize}
\end{claim}

\begin{proof}[Proof of Claim.]
To prove (i), suppose that there is a vertex $x \in A$ with $d_G(x,B)\le cn+\rho_0n$ and $d_G(x,A)\ge ((k-2)c-\rho_0)n$.
Without loss of generality, we may suppose that $x \in A_1$.
Now modify $H$ to obtain $H'\in \cH_1(n)$ by replacing the neighbourhood of $x$ with $A\setminus\{x\}$.
Then $H'$ has a canonical partition $A_1\setminus \lbrace x \rbrace,A_2,\ldots,A_{k-2},B\cup\lbrace x \rbrace$.
We have that
\begin{eqnarray*}
d_{G\setminus H}(x) + d_{H\setminus G}(x) &\geq& d_G(x,A)-|A\setminus A_1|+|B|-d_G(x,B)\\
&\stackrel{(\ref{Aisize})}{\ge}& ((k-2)c-\rho_0)n-(k-3)(c+\rho_3)n+(1-(k-2)(c+\rho_3))n-(c+\rho_0)n\\
&\geq& (1-(k-2)c-3\rho_0)n,
\end{eqnarray*}
while
\begin{align*}
d_{G\setminus H'}(x) + d_{H'\setminus G}(x) &=  d_G(x,B)+|A|-d_G(x,A)\\
&\leq (c+\rho_0)n+(k-2)(c+\rho_3)n-((k-2)c-\rho_0)n\ \le\  cn+3\rho_0n.
\end{align*} 
Thus
\begin{eqnarray}
\label{bigtriangle}|E(H')\bigtriangleup E(G)| - |E(H)\bigtriangleup E(G)| &=& d_{G\setminus H'}(x) + d_{H'\setminus G}(x) - d_{G\setminus H}(x) - d_{H\setminus G}(x)\\
\nonumber &\leq& (kc-1-c)n + 6\rho_0 n \stackrel{(\ref{ineq:c})}{\leq} -((k-1)\alpha - 6\rho_0) n < -\alpha n,
\end{eqnarray}
contradicting the choice of $H$.

To prove (ii), suppose that there exists $y\in V(G)$ and $ij \in \binom{[k-2]}{2}$ such that $ d_{\overline{G}}(y,A_i),d_{\overline{G}}(y,A_j)\ge\rho_3n$ and $d_{G}(y,A_j)\ge d_{G}(y,A_i)>\rho_3n$.
Then we can obtain a new graph $G'$ by replacing $\rho_3 n$ neighbours of $y$ in $A_i$ with $\rho_3 n$ new neighbours in $A_j$.
There are at most $\rho_4 n^2$ edges missing between $A_i$ and $A_j$ in $G$, so
\begin{align*}
K_3(G)-K_3(G')&=K_3(y,G)-K_3(y,G')\\
&\geq (d_G(y,A_i)d_G(y,A_j)-\rho_4 n^2) - (d_G(y,A_i)-\rho_3 n)(d_G(y,A_j)+\rho_3 n)\\
&\ge \rho_3^2n^2-\rho_4n^2\ge\rho_4n^2.
\end{align*}
This contradicts the fact that $G$ is a worst counterexample (namely,~\ref{worst-C1}).

For (iii), suppose there is some $y\in B$ with $d_G(y,A)\le (k-3)cn+\rho_0n$ and $d_G(y,B)\ge cn-\rho_0n$.
Suppose without loss of generality that $d_{G}(y,A_1)=\min_{j \in [k-2]} \{d_{G}(y,A_j)\}$. We claim that
\begin{equation}\label{claim1}
d_{G}(y,A_1)\le 2\rho_0n.
\end{equation}
Indeed, when $k=3$, we have $A_1=A$ and so $d_{G}(y,A_1)=d_G(y,A)\le \rho_0n$.
Suppose now that $k\ge 4$. If $d_{G}(y,A_1)\ge 2\rho_0n$, then
\begin{eqnarray*}
d_{\overline{G}}(y,A\setminus A_1) &=& |A\setminus A_1|-d_G(y,A)+d_{G}(y,A_1)\\
 &\stackrel{(\ref{Aisize})}{\geq}& (k-3)(c-\rho_3)n - (k-3)cn-\rho_0n+2\rho_0n
\ \geq\ \frac{\rho_0n}{2}.
\end{eqnarray*}
Thus there is some $j \in [k-2]\setminus \lbrace 1 \rbrace$ for which $d_{\overline{G}}(y,A_j) \geq \rho_0 n/(2k) \geq \rho_3 n$.
On the other hand, as $d_{G}(y,A_1)=\min_{j \in [k-2]} \{d_{G}(y,A_j)\}$, we have that $$d_{\overline{G}}(y,A_1)=|A_1|-d_G(y,A_1)\ge |A_1|-d_G(y,A)/(k-2)\ge\rho_3n.$$ 
Then (ii) implies that $d_{G}(y,A_1)\le\rho_3n< 2\rho_0n$, a contradiction.
Thus~(\ref{claim1}) holds.

Obtain $H'$ from $H$ by replacing $N_H(y)$ with $V(H)\setminus A_1$. Then $H' \in \mathcal{H}_1(n)$ has a canonical partition $A_1\cup \lbrace y \rbrace, A_2,\ldots,A_{k-2},B\setminus \lbrace y \rbrace$.
We have
$d_{G\setminus H}(y)+d_{H\setminus G}(y) \geq d_{\overline{G}}(y,A)$, while
\begin{eqnarray*}
	d_{G\setminus H'}(y) + d_{H'\setminus G}(y) &\leq& d_{G}(y,A_1)+d_{\overline{G}}(y,A\setminus A_1)+d_{\overline{G}}(y,B)\\
	&\leq& 2d_{G}(y,A_1)+d_{\overline{G}}(y,A)-|A_1|+|B|-d_G(y,B)\\
	&\le &4\rho_0n+d_{\overline{G}}(y,A)-(c-\rho_3)n+(1-(k-2)(c-\rho_3))n-(c-\rho_0)n\\
	&\le &d_{\overline{G}}(y,A)+(1-kc)n+6\rho_0n\stackrel{(\ref{ineq:c})}{\le} d_{\overline{G}}(y,A)-(\sqrt{2\alpha}-6\rho_0)n.
\end{eqnarray*}
Again, this implies that $|E(H')\bigtriangleup E(G)| < |E(H)\bigtriangleup E(G)|$, contradicting the choice of $H$.
This completes the proof of the claim.
\end{proof}

\medskip
\noindent
The next claim shows that every large enough subset of $B$ must contain many edges.

\begin{claim}\label{cl-Xlarge}
For all $X \subseteq B$ with $|X| \geq (c-\rho_1)n$, we have $E(G[X]) \geq \rho_1 n^2$.
\end{claim}

\begin{proof}[Proof of Claim.] 
Suppose that some $X$ violates the claim. By taking a subset, we
can assume that $|X|=(c-\rho_1)n$. Now~(\ref{eq:solc}) implies that $c\ge 1/k$, and so $|X|\ge n/(2k)$.
Let $\tilde{d}(X,\overline{X}) := \frac{1}{|X|}\sum_{x \in X}d_G(x,\overline{X})$ denote the average degree of vertices in $X$ into $\overline{X}$ in $G$.
Then the average
degree of vertices in $X$ in $G$ is 
\begin{align*}
\frac{1}{|X|}\sum_{x \in X}d_G(x) &=\tilde{d}(X,\overline{X}) +\frac{2e(G[X])}{|X|}\le \tilde{d}(X,\overline{X}) +4k\rho_1n.
\end{align*} 
Let $Y := B\setminus X$.
By Corollary~\ref{cr:DeltaG}, the average degree of vertices in $Y$ is certainly at most
\begin{equation}\label{DeltaG}
\Delta(G)\le (k-1)cn+42/\sqrt{\alpha} \leq (k-1)cn + \rho_3n.
\end{equation}
The average degree of vertices in $A$ in $G[A]$ is 
$$\frac{1}{|A|}\sum_{a \in A}d_G(a,A) \stackrel{(\ref{editdistance})}{\leq} \frac{1}{|A|}\left(\sum_{a \in A}d_H(a,A)+2\rho_4 n^2\right) \stackrel{(\ref{Aisize})}{\leq} (k-3)(c+\rho_3)n+\rho_3n\le (k-3)cn+k\rho_3n.$$ 
Thus the average degree of vertices of $A$ in $G$ is
$$\frac{1}{|A|}\sum_{a \in A}d_G(a) \leq |B|+(k-3)cn+k\rho_3n\stackrel{(\ref{Aisize})}{\le}(1-(k-2)(c-\rho_3))n+(k-3)cn+k\rho_3n\le (1-c+2k\rho_3)n.$$
Hence, by taking the weighted average of these average degrees to obtain the average degree of $G$, we have
\begin{eqnarray*}
&&2\left((k-1)c-\binom{k}{2}c^2\right) \stackrel{(\ref{eq:c})}{=} \frac{2e}{n^2}\\
& \le& \frac{1}{n^2}\left(\,\left(\,\tilde{d}(X,\overline{X})+4k\rho_1 n\right)|X|  +  ((k-1)cn+\rho_3n)|Y|+(1-c+2k\rho_3)n|A|\,\right)\\
& \stackrel{(\ref{Aisize})}{\le}& \left(\frac{\tilde{d}(X,\overline{X})}{n}+4k\rho_1\right)c  + ((k-1)c+\rho_3)(1-(k-1)c+2\rho_1)+(1-c+2k\rho_3)(k-2)(c+\rho_3)\\
	 &\le&2\left((k-1)c-\binom{k}{2}c^2\right) + c\left(\frac{\tilde{d}(X,\overline{X})}{n}-(1-c)\right) +6k\rho_1.
\end{eqnarray*}
Thus
$$
\tilde{d}(X,\overline{X}) \geq \left((1-c)-\frac{6k\rho_1}{c}\right)n \geq |\overline{X}|-\sqrt{\rho_1}n.
$$
In particular, the number of missing edges in $G$ between $X$ and $Y$ is $e(\overline{G}[X,Y])\leq (c-\rho_1)\sqrt{\rho_1}n^2 \leq \sqrt{\rho_1}n^2$.
This further implies that
\begin{eqnarray*}
e(G[Y]) &\leq& |Y|\cdot \Delta(G) - e(G[A,Y]) - e(G[X,Y])\\
&\stackrel{(\ref{editdistance})}{\leq}& |Y|\Delta(G)-(|A||Y|-\rho_4 n^2)-(|X||Y|-\sqrt{\rho_1}n^2)\\
&\stackrel{(\ref{Aisize}),(\ref{DeltaG})}{\le}& |Y|((k-1)cn+42/\sqrt{\alpha}-(k-2)(c-\rho_3)n-(c-\rho_1)n)+\rho_4 n^2 + \sqrt{\rho_1}n^2\\
&\leq& 2\sqrt{\rho_1}n^2.
\end{eqnarray*}

Let $H' \in \mathcal{H}_1(n)$ be the $n$-vertex complete $k$-partite graph with partition $A_1,\ldots,A_{k-2},X,Y$.
Then
\begin{eqnarray*}
|E(G)\bigtriangleup E(H')| &\leq& |E(G)\bigtriangleup E(H)| + e(G[Y])+e(G[X])+e(\overline{G}[X,Y])\\
&\stackrel{(\ref{editdistance})}{\leq}& (\rho_4 + 2\sqrt{\rho_1} + \rho_1+\sqrt{\rho_1})n^2 < 4\sqrt{\rho_1} n^2.
\end{eqnarray*}
But there is a $1$-to-$1$ mapping of parts of $H'$ to parts of $K^k_{\lfloor cn\rfloor,\ldots,\lfloor cn \rfloor,n-(k-1)\lfloor cn \rfloor}$ such that two corresponding parts have size within $2\rho_1$ of one another. Therefore
$$
\left|E(H')\bigtriangleup E(K^k_{\lfloor cn\rfloor,\ldots,\lfloor cn \rfloor,n-(k-1)\lfloor cn \rfloor})\right| \leq \frac{\rho_0 n^2}{2}.
$$
Then $|E(G)\bigtriangleup E(K^k_{\lfloor cn\rfloor,\ldots,\lfloor cn \rfloor,n-(k-1)\lfloor cn \rfloor})| < \rho_0 n^2$,
a contradiction to our initial assumption on~$G$.
\end{proof}



\medskip
\noindent
We are now able to show that vertices in every $A_i$ have small degree in their own part, and further that for distinct $i,j$, the bipartite graph $G[A_i,A_j]$ is complete.

\begin{claim}\label{cl-A-nice}
	For all $i\in[k-2]$ we have $\Delta(G[A_i])< \rho_2 n$.
	Moreover, $G[A] \supseteq  K[A_1,\ldots,A_{k-2}]$.
\end{claim}

\begin{proof}[Proof of Claim.] 
Suppose on the contrary that for some $i \in [k-2]$ there is an $x\in A_i$ with $d_G(x,A_i) \geq \rho_2 n$. Let $Z := N_G(x,A_i)$ and 
$X:=N_G(x,B)$.
We claim that
\begin{equation}\label{xAAi}
d_{\overline{G}}(x,A\setminus A_i) < 6k\rho_3 n.
\end{equation} 
This is vacuously true if $k=3$. 
So suppose that $k\ge 4$.
We will first show that for any $j \in [k-2]\setminus \lbrace i \rbrace$, we have
\begin{equation}\label{almostmaxcut}
d_G(x,A_j) \geq d_G(x,A_i)-\rho_3 n.
\end{equation}
Indeed, let $H' \in \mathcal{H}_1(n)$ have canonical partition obtained from $A_1,\ldots,A_{k-2},B$ by moving $x$ from $A_i$ to $A_j$.
We have that
\begin{eqnarray*}
0 &\leq& |E(G) \bigtriangleup E(H')| - |E(G) \bigtriangleup E(H)|\\
&\leq & d_G(x,A_j)+|A_i|-d_G(x,A_i) - (d_G(x,A_i)+|A_j|-d_G(x,A_j)) \\
&\stackrel{(\ref{Aisize})}{\leq}& 2(d_G(x,A_j)-d_G(x,A_i)) + 2\rho_3,
\end{eqnarray*}
giving~(\ref{almostmaxcut}).
So $d_{G}(x,A_j)\ge |Z|-\rho_3 n \ge (\rho_2-\rho_3)n\ge \rho_3n$. If $d_{\overline{G}}(x,A\setminus A_i)\ge 6k\rho_3n$, then there exists some $j\in[k-2]\setminus \lbrace i \rbrace$ such that $d_{\overline{G}}(x,A_j)\ge 6\rho_3n$.
Then~(\ref{almostmaxcut}) implies that 
$$
|A_i|-1-d_{\overline{G}}(x,A_i)=d_G(x,A_i)\le  d_{G}(x,A_j)+\rho_3 n =|A_j|-d_{\overline{G}}(x,A_j)+\rho_3n$$
and so
\begin{eqnarray*}
d_{\overline{G}}(x,A_i) &\ge& d_{\overline{G}}(x,A_j)+|A_i|-1-|A_j|-2\rho_3 n\\
&\stackrel{(\ref{Aisize})}{\ge}& 6\rho_3n+(c-\rho_3)n-1-(c+\rho_3)n-2\rho_3n
> \rho_3n.
\end{eqnarray*}
Then Claim~\ref{cl-AT}(ii) implies $d_{G}(x,A_i)<\rho_3n<\rho_2n$, a contradiction.
Thus~(\ref{xAAi}) holds.

We have
$$
\sum_{z \in Z}(d_{\overline{G}}(z,X)+d_{\overline{G}}(z,A\setminus A_i)) = e(\overline{G}[Z,X]) + e(\overline{G}[Z,A\setminus A_i]) \leq |E(G)\bigtriangleup E(H)| \stackrel{(\ref{editdistance})}{\leq} \rho_4 n^2.
$$
Thus, by averaging, there is some $z \in Z$ such that
$$
d_{\overline{G}}(z,X)+d_{\overline{G}}(z,A\setminus A_i) \leq \rho_4 n/\rho_2 \leq \rho_3 n.
$$
Then
\begin{eqnarray*}
(k-2)cn+k&\stackrel{(\ref{2path})}{\ge}& P_3(xz,G)\\
&\ge& |X|+|A\setminus A_i|-(d_{\overline{G}}(z,X)+d_{\overline{G}}(z,A\setminus A_i)) - d_{\overline{G}}(x,A\setminus A_i)\\
&\stackrel{(\ref{Aisize}),(\ref{xAAi})}{\geq}& |X| + (k-3)(c-\rho_3)n -\rho_3n - 6k\rho_3 n \geq |X|+(k-3)cn - 7k\rho_3 n.
\end{eqnarray*}
Consequently, 
\begin{equation}\label{eq-X-up}
|X|\le cn+8k\rho_3n. 
\end{equation}
We now bound $d_G(x)$ and $K_3(x,G)$ as follows. We have
\begin{equation}\label{eq-dx}
	d_G(x)\le |X|+|Z|+|A\setminus A_i|\stackrel{(\ref{Aisize})}{\le} |X|+|Z|+(k-3)cn+k\rho_3n.
\end{equation}
We wish to bound $K_3(x,G)$ from below. Let $Y := N_G(x,A\setminus A_i)$. We will need the following lower bound on $|Y|$:
\begin{equation}\label{eq-d-outofAi}
|Y|=|A\setminus A_i|-d_{\overline{G}}(x,A\setminus A_i)\stackrel{(\ref{Aisize}),(\ref{xAAi})}{\ge} (k-3)cn-7k\rho_3n \geq |A\setminus A_i|-8k\rho_3 n.
\end{equation}
Note also that
\begin{eqnarray*}
K_3(x,G;A\setminus A_i) &=& e(G[Y])\ \geq\ e(G[A\setminus A_i])-(|A\setminus A_i|-|Y|)n\\
&\stackrel{(\ref{editdistance}),(\ref{eq-d-outofAi})}{\geq}& e(H[A\setminus A_i])-\rho_4 n^2 - 8k\rho_3 n^2\\
&\geq& \left(\binom{k-3}{2}(c-\rho_3)^2 - \rho_4 - 8k\rho_3\right)n^2 \ \geq\ \binom{k-3}{2}c^2n^2 - \frac{\sqrt{\rho_3}n^2}{2}.
\end{eqnarray*}
Thus
\begin{eqnarray*}
	K_3(x,G)&\stackrel{(\ref{editdistance})}{\geq}& |X||Y| + |Y||Z| + |Z||X|-\rho_4n^2 + e(G[X]) + K_3(x,G;A\setminus A_i)\nonumber\\
	&\stackrel{(\ref{Aisize}),(\ref{eq-d-outofAi})}{\ge}& |X||Z|+(|X|+|Z|)(k-3)cn + e(G[X]) + \binom{k-3}{2}c^2n^2 - \sqrt{\rho_3}n^2.
\end{eqnarray*}

This together with Lemma~\ref{lem-f} implies that,
\begin{eqnarray}\label{eq-f-ineq}
	-\frac{6n}{\sqrt{\alpha}}&\le& f(x)= (d_G(x)-(k-2)cn)(k-2)cn+{k-2\choose 2}c^2n^2-K_3(x,G)\nonumber\\
	&\stackrel{(\ref{Aisize}),\eqref{eq-dx}}{\le}& (|X|+|Z|-cn+k\rho_3n)(k-2)cn+{k-2\choose 2}c^2n^2\nonumber\\
	&&-\left(|X||Z|+(|X|+|Z|)(k-3)cn+e(G[X]) +{k-3\choose 2}c^2n^2-\sqrt{\rho_3}n^2\right)\nonumber\\
	&\le&(|Z|-cn)(cn-|X|)-e(G[X])+\rho_2n^2.
\end{eqnarray}
Then, by considering two cases where the coefficient $cn-|X|$ of $|Z|$ is negative or non-negative and recalling that $\rho_2n \leq |Z| \leq |A_i|$, we have
\begin{eqnarray*}
e(G[X]) &\stackrel{(\ref{eq-X-up})}{\leq}& \frac{6n}{\sqrt{\alpha}} + \rho_2 n^2 +  \max\left\lbrace (\rho_2n-cn)(-8k\rho_3 n), (|A_i|-cn)cn \right\rbrace\\
 &\stackrel{(\ref{Aisize})}{\leq}& 2\rho_2 n^2 + 8k\rho_3 cn^2\
\leq\ 3\rho_2n^2.
\end{eqnarray*}
Thus, by Claim~\ref{cl-Xlarge}, we have $d_G(x,B) = |X|< (c-\rho_1)n$. Claim~\ref{cl-AT}(i) now implies that
$$
((k-2)c-\rho_0)n > d_G(x,A) = |Z| + |Y| \stackrel{(\ref{eq-d-outofAi})}{\geq} |Z| + (k-3)cn-7k\rho_3n,
$$ 
implying that $|Z|\le cn-\rho_0n/2$.
We look again at~(\ref{eq-f-ineq}) to see that
$$
e(G[X]) \leq \frac{6n}{\sqrt{\alpha}} + \rho_2 n^2 -\frac{\rho_0\rho_1n^2}{2}<0,
$$
a contradiction.
This proves the first part of the claim.

For the second part,
let $x\in A_i$ and $y\in A_j$ with $ij \in \binom{[k-2]}{2}$. Then, using the first part,
\begin{eqnarray*}
P_3(xy,G) &\leq& (n-|A_i|-|A_j|)+\Delta(G[A_i])+\Delta(G[A_j]) \stackrel{(\ref{Aisize})}{\leq} (1-2c+2\rho_3)n + 2\rho_2 n\\
&\stackrel{(\ref{ineq:c})}{<}& (k-2)cn - (\sqrt{2\alpha}-2\rho_3-2\rho_2)n < (k-2)cn - \sqrt{\alpha}n.
\end{eqnarray*}
Then~(\ref{2path}) implies that $xy \in E(G)$.
Since $ij$ was arbitrary, we have shown that $K[A_1,\ldots,A_{k-2}] \subseteq G[A]$, as required.
\end{proof}

\medskip
\noindent
We now prove some useful properties of vertices in $B$.

\begin{claim}\label{cl-B}
For every $y \in B$, the following holds:
\begin{itemize}
\item[(i)] If $d_G(y,B) \leq cn +\rho_2 n$, then $A \subseteq N_G(y)$.
\item[(ii)] If $d_G(y,B) > (c-\rho_1/2)n$, then there exists $i \in [k-2]$ such that $d_{\overline{G}}(y,A\setminus A_i) < k\rho_3 n$.
\end{itemize}
\end{claim}

\begin{proof}[Proof of Claim.]
Let $y \in B$ be arbitrary, and let $Y := N_G(y,B)$.
We will first prove (ii). Note that (ii) is vacuously true when $k=3$, so assume $k\ge 4$. Suppose that $d_G(y,B) > (c-\rho_1/2)n$. Claim~\ref{cl-AT}(iii) implies that
\begin{equation}\label{dGA}
d_G(y,A) > (k-3)cn+\rho_0n.
\end{equation}
Let $i \in [k-2]$ be such that $d_{\overline{G}}(y,A_i) = \max_{j \in [k-2]}d_{\overline{G}}(y,A_j)$.

Let us show that this $i$ satisfies~(ii). Suppose on the contrary that $d_{\overline{G}}(y,A\setminus A_i) \geq k\rho_3 n$.
Then there exists $j \in [k-2]\setminus \{i\}$ such that $\rho_3 n \leq d_{\overline{G}}(y,A_j) \leq d_{\overline{G}}(y,A_i)$.
Claim~\ref{cl-AT}(ii) and~(\ref{Aisize}) imply that $d_G(y,A_i \cup A_j) \leq \rho_3 n + (c+\rho_3)n = (c+2\rho_3)n$.
But then
$$
d_G(y,A) \leq d_G(y,A_i \cup A_j) + |A\setminus(A_i \cup A_j)| \stackrel{(\ref{Aisize})}{\leq} (c+2\rho_3)n + (k-4)(c+\rho_3)n \leq (k-3)cn + \rho_2n,
$$
contradicting~(\ref{dGA}).
Thus $d_{\overline{G}}(y,A\setminus A_i) < k\rho_3 n$.
This completes the proof of (ii).

For (i),
suppose now that $|Y| \leq cn + \rho_2 n$.
First consider the case when additionally $|Y| \leq (c-\rho_1/2) n$.
Let $x \in A$ be arbitrary, and let $i \in [k-2]$ be such that $x \in A_i$.
Then Claim~\ref{cl-A-nice} implies that
$$
P_3(xy,G) \leq \Delta(G[A_i]) + |Y| + |A\setminus A_i| \stackrel{(\ref{Aisize})}{\leq} \rho_2 n + (c-\rho_1/2)n + (k-3)(c+\rho_3)n \leq (k-2)cn - \rho_1 n/3. 
$$
Then~(\ref{2path}) implies that $xy \in E(G)$.
Since $x$ was arbitrary, we have proved that $A \subseteq N_G(y)$.
So (i) holds in this case.

Consider the other case when $(c-\rho_1/2)n < |Y| \leq (c+\rho_2)n$.
Part (ii) implies that there exists $i \in [k-2]$ such that $d_{\overline{G}}(y,A\setminus A_i) < k\rho_3 n$.

Let $Z := N_G(y,A\setminus A_i)$.
Then
\begin{eqnarray}
 \nonumber |Z| &=& |A\setminus A_i| - d_{\overline{G}}(y,A\setminus A_i)\geq (k-3)(c-\rho_3)n - k\rho_3n\\
 &\geq& (k-3)cn - 2k\rho_3 n.\label{Zeq}
\end{eqnarray}
Let also $X := N_G(y,A_i)$.
Note that $d_G(y) \leq |X|+|Y|+|A\setminus A_i| \leq |X|+|Y| + (k-3)(c+\rho_3)n$ by~(\ref{Aisize}).
Then~Lemma~\ref{lem-f} implies that
\begin{align}
\nonumber K_3(y,G) &\leq (d_G(y)-(k-2)cn)(k-2)cn+\binom{k-2}{2}c^2n^2 + \frac{6n}{\sqrt{\alpha}}\\
\label{K3yG}&\leq (|X|+|Y|-cn)(k-2)cn + \binom{k-2}{2}c^2n^2 + \rho_2 n^2.
\end{align}
Recall that every pair among $X,Y,Z$ spans a complete bipartite graph in $H$.
Moreover, (ii) implies that
$$
e(G[Z]) \geq e(G[A\setminus A_i])-d_{\overline{G}}(y,A\setminus A_i)n \geq e(G[A\setminus A_i]) - k\rho_3 n^2.
$$
Thus we can use Claim~\ref{cl-A-nice} to lower bound $K_3(y,G)$:
\begin{eqnarray*}
 K_3(y,G) &\geq& e(G[X,Y]) + e(G[Y,Z]) + e(G[Z,X]) + e(G[Z]) + e(G[Y])\\
&\stackrel{(\ref{editdistance})}{\geq}& |X||Y| + |Y||Z| + |Z||X| - \rho_4n^2 +\sum_{hj \in \binom{[k-2]\setminus \lbrace i \rbrace}{2}}|A_h||A_j| - k\rho_3n^2 + e(G[Y])\\
&\stackrel{(\ref{Aisize}),(\ref{Zeq})}{\geq}& |X||Y| + (k-3)cn(|X|+|Y|) + \binom{k-3}{2}c^2n^2 + e(G[Y]) - \sqrt{\rho_3} n^2.
\end{eqnarray*}
This together with~(\ref{K3yG}) implies that
$$
e(G[Y]) \leq (cn-|X|)(|Y|-cn)+2\rho_2 n^2.
$$
As before, considering the two cases when $cn-|X|$ is positive and non-positive and recalling that $(c-\rho_1/2)n < |Y| \leq (c+\rho_2)n$, we have
\begin{eqnarray*}
e(G[Y]) &\leq& \max\big\lbrace cn\cdot\rho_2 n,\, (|A_i|-cn)\cdot\rho_1 n/2  \big\rbrace + 2\rho_2 n^2\\
 &\stackrel{(\ref{Aisize})}{\leq}& \max \left\lbrace c\rho_2 n^2,\, \rho_1\rho_3 n^2/2 \right\rbrace + 2\rho_2 n^2\ <\ \rho_1 n^2.
\end{eqnarray*}
This is a contradiction to Claim~\ref{cl-Xlarge}.
\end{proof}

\begin{claim}\label{cl-yAi}
For every $i \in [k-2]$ and $y \in B$ with $d_{\overline{G}}(y,A\setminus A_i) \leq \rho_2 n/2$, we have that $A_i \subseteq N_G(y)$.
\end{claim}

\begin{proof}[Proof of Claim.]
Choose $i \in [k-2]$ and $y \in B$ with $d_{\overline{G}}(y,A\setminus A_i) \leq \rho_2 n/2$.
Let $X := N_G(y,A_i)$ and $Y := N_G(B,y)$.
Suppose that there exists $x' \in A_i$ such that $x'y \notin E(G)$.
Then Claim~\ref{cl-B}(i) implies that $|Y| > (c+\rho_2)n$.
Claim~\ref{cl-AT}(iii) implies that $d_G(y,A) > (k-3)cn + \rho_0 n$.
Therefore
$$
|X| \geq d_G(y,A)-|A\setminus A_i| \stackrel{(\ref{Aisize})}{>} (k-3)cn + \rho_0n - (k-3)(c+\rho_3)n \geq \rho_0 n/2.
$$
Furthermore,
$$
\sum_{x \in X}\big(d_{\overline{G}}(x,Y) + d_{\overline{G}}(x,A\setminus A_i)\big) = e(\overline{G}[X,Y]) + e(\overline{G}[X,A\setminus A_i]) \stackrel{(\ref{editdistance})}{\leq} \rho_4 n^2,
$$
so there exists $x \in X$ with
$$
d_{\overline{G}}(x,Y)+d_{\overline{G}}(x,A\setminus A_i) \leq \frac{\rho_4n^2}{|X|} \leq \frac{2\rho_4 n}{\rho_0} < \rho_3 n.
$$
Since $d_{\overline{G}}(y,A\setminus A_i) \leq \rho_2 n/2$, we have that
\begin{eqnarray*}
P_3(xy,G) &\geq& (|A\setminus A_i| + |Y|) - d_{\overline{G}}(x,Y)-d_{\overline{G}}(x,A\setminus A_i) - d_{\overline{G}}(y,A\setminus A_i)\\
&\stackrel{(\ref{Aisize})}{\geq}& (k-3)(c-\rho_3)n + (c+\rho_2)n - \rho_3 n - \rho_2 n/2\ \geq\ (k-2)cn + \rho_2 n/3,
\end{eqnarray*}
a contradiction to~(\ref{2path}).
\end{proof}

\medskip
\noindent
We are now able to show that $G$ consists of the complete $(k-1)$-partite graph with parts $A_1,\ldots,A_{k-2},B$, together with some additional edges in $B$.

\begin{claim}\label{cl-A-nice2}
$G\setminus G[B] \cong K[A_1,\ldots,A_{k-2},B]$.
\end{claim}

\begin{proof}[Proof of Claim.]
We will first show that $G[A,B]$ is a complete bipartite graph.
Let $y \in B$ be arbitrary.
It suffices to show that $A \subseteq N_G(y)$.
By Claim~\ref{cl-yAi}, we may assume that $k \geq 4$.
Let $Y := N_G(y,B)$.
By Claim~\ref{cl-B}(i), we may assume that that $|Y| \geq (c+\rho_2)n$, and Claim~\ref{cl-AT}(iii) implies that $d_G(y,A) \geq (k-3)cn + \rho_0n$.
Claim~\ref{cl-B}(ii) implies that there exists $i \in [k-2]$ such that $d_{\overline{G}}(y,A\setminus A_i) < k\rho_3n < \rho_2 n/2$.
Then, by Claim~\ref{cl-yAi}, we have that $A_i \subseteq N_G(y)$.
Thus, for all $j \in [k-2]$, we have $d_{\overline{G}}(y,A\setminus A_j) \leq d_{\overline{G}}(y,A) = d_{\overline{G}}(y,A\setminus A_i) < \rho_2 n/2$.
But Claim~\ref{cl-yAi} now implies that $A_j \subseteq N_G(y)$ for all $j \in [k-2]$.
Thus $A \subseteq N_G(y)$, proving the first part of the claim.

To complete the proof, it suffices by the second assertion of Claim~\ref{cl-A-nice} to show that $e(G[A_i])=0$ for all $i \in [k-2]$.
So let $i \in [k-2]$ and let $x,z \in A_i$ be distinct.
Claim~\ref{cl-A-nice} implies that $A_j \subseteq N_G(x)\cap N_G(z)$ for all $j \in [k-2]$, and since $G[A,B]$ is complete we also have $B \subseteq N_G(x)\cap N_G(z)$.
Thus
$$
P_3(xz,G) \geq n-|A_i| \stackrel{(\ref{Aisize})}{\geq} n - (c+\rho_3)n \stackrel{(\ref{ineq:c})}{\geq} (k-2)cn + ((k-1)\alpha - \rho_3) n.
$$
So~(\ref{2path}) implies that $xz \notin E(G)$.
This completes the proof of the claim.
\end{proof}

\medskip
\noindent
The rigid structural information provided by the last claim allows us to finish the proof by deriving a contradiction to our assumption that $G$ is far in edit distance from $K^k_{\lfloor cn\rfloor,\ldots,\lfloor cn\rfloor,n-(k-1)\lfloor cn\rfloor}$.

\medskip
\noindent
Suppose first that $k=3$.
Claim~\ref{cl-A-nice2} implies that $G[A,B]$ is complete bipartite and $G[A]$ contains no edges.
Thus $G[B]$ exactly minimises the number of triangles given its size,~i.e. $K_3(G[B]) = g_3(n,e(G[B]))$ (otherwise we could replace $G[B]$ in $G$ to obtain an $(n,e)$-graph with fewer triangles).
Now, $K_3(G[B])>0$, otherwise $G \in \mathcal{H}_1(n,e)$, a contradiction.
Therefore
\begin{equation}\label{eq-eB}
	e(G[B]) > t_2(|B|) \stackrel{(\ref{Aisize})}{\geq} \left\lfloor \frac{(1-(c+\rho_3))^2n^2}{4} \right\rfloor \geq \frac{(1-c)^2n^2}{4} - \rho_2 n^2.
\end{equation}
Recalling the definition of $c$ (i.e.~(\ref{eq:c})) in the case $k=3$ and the fact that $c<1/2$ (i.e.~(\ref{eq:solc2})), we have
$$
e(G[B])=e-|A|\,|B| \leq e - (c-\rho_3)(1-(c+\rho_3))n^2 \leq e - c(1-c)n^2 + \rho_2 n^2 \stackrel{(\ref{eq:c})}{=} c(1-2c)n^2 + \rho_2n^2.
$$
This together with~\eqref{eq-eB} implies that $(3c-1)^2 \leq 8\rho_2$ and so
$$
c < \frac{1}{3} + \rho_0 < \frac{1+\sqrt{2\alpha}}{3},
$$
contradicting~(\ref{eq:solc}).

Therefore we may suppose that $k \geq 4$.
Now, by Claim~\ref{cl-A-nice2}, for each $i \in [k-2]$, we have that $A_i$ is an independent set in $G$ and $G[\,A_i,\overline{A_i}\,]$ is a complete bipartite graph.
Let $n_i := |\,\overline{A_i}\,|$ and $e_i := e(G[\,\overline{A_i}\,]) = e-n_i(n-n_i)$ and $G_i := G[\,\overline{A_i}\,]$.
Then $g_3(n,e) = K_3(G) = K_3(G_i) + (n-n_i)e_i$.
Thus $K_3(G_i)=g_3(n_i,e_i)$.
Recall the definition of the function $k(\cdot,\cdot)$ given in~(\ref{eq:k}).

\begin{claim}\label{cl-Gi}
$t_{k-2}(n_i) + \alpha n_i^2/3 \leq e_i \leq t_{k-1}(n_i) - \alpha n_i^2/3$.
\end{claim}

\begin{proof}[Proof of Claim.]
By~(\ref{Aisize}), $|n_i-(1-c)n| \leq \rho_3 n$. 
We then have
\begin{align*}
\frac{e_i}{n_i^2} - \frac{1}{2}\left(1-\frac{1}{k-2}\right) \ge \frac{(1-kc+c)((kc-1)(k-2)+(1-c))}{2(1-c)^2(k-2)} - \rho_2, 
\end{align*}
 where the first term follows by routine calculations with $n_i$ approximated by $(1-c)n$ while the second term $-\rho_2$ absorbs all
 errors. By~(\ref{ineq:c}), the left-hand side is at least
 $$
 \frac{(k-1)\alpha \cdot (1-c)}{2(1-c)^2(k-2)}-\rho_2>\frac{\alpha}{3}
 $$
 and thus $e_i \geq t_{k-2}(n_i)+\alpha n_i^2/3$.
 The other inequality is similar:
 $$
 \frac{e_i}{n_i^2} - \frac{1}{2}\left(1-\frac{1}{k-1}\right)\le -\frac{(k-2)\cdot(kc-1)^2}{2 (k-1)}+\rho_2\stackrel{(\ref{ineq:c})}{\leq}-\frac{(k-2)\cdot 2\alpha}{k-1} +\rho_2<-\frac{\alpha}2
 $$
 and so $e_i \leq t_{k-1}(n_i) - \alpha n_i^2/3$. 
\end{proof}

\medskip
\noindent
But
$$
n_i = n-|A_i| \stackrel{(\ref{Aisize})}{\geq} (1-c-\rho_3)n \stackrel{(\ref{ineq:c})}{\geq} n/2 \geq n_0/2 \stackrel{(\ref{n0})}{\geq} n_0(k-1,\alpha/3)
$$
and so the minimality of $k$ implies that $G_i \in \mathcal{H}(n_i,e_i)$.
Suppose first that $G_i \in \mathcal{H}_1(n_i,e_i)$.
Since $G$ is an $(n,e)$-graph obtained by adding every edge between the independent set $A_i$ and $V(G_i)$, we have that $G \in \mathcal{H}_1(n,e)$, a contradiction to~\ref{worst-C1}.
Suppose instead that $G_i \in \mathcal{H}_2(n_i,e_i)$.
Then $G_i$ is $(k-1)$-partite and so $G$ is $k$-partite.
Corollary~\ref{H2new}(i) then implies that $G \in \mathcal{H}_2(n,e)$, again contradicting~\ref{worst-C1}.
Thus our original assumption was false, and we have shown that 
$|E(G)\bigtriangleup E(K^k_{\lfloor cn \rfloor,\ldots,\lfloor cn\rfloor,n-(k-1)\lfloor cn\rfloor})| \leq \rho_0 n^2$.
This completes the proof of Lemma~\ref{approx}.
\end{proof}

\subsection{Proof of Lemma~\ref{approx}}
Now we are ready to show that every max-cut partition $A_1,\dots,A_k$ of our worst counterexample $G$ has the required approximate structure.

\begin{proof}[Proof of Lemma~\ref{approx}.]
Choose a max-cut $k$-partition $V(G)=A_1 \cup \ldots \cup A_k$. Assume that $|A_k| \leq |A_i|$ for all $i \in [k-1]$.
Define 
 \begin{eqnarray*}
 Z_i &:=& \lbrace z \in A_i:d_{\overline{G}}(z,\overline{A_i}) \geq \xi n\rbrace\quad \mbox{for $i\in [k]$},\\
 Z &:=& Z_1\cup \dots\cup Z_k.
 \end{eqnarray*}
 We need to show that $G$ has an $(A_1,\ldots,A_k;Z,\beta,\xi,\xi,\delta)$-partition,~i.e.,\ that \Ppartition($G$)--\Pmissing($G$) hold with the appropriate parameters.

Let $p:=k$; $d := \rho_0n^2$ and $\rho := \rho_0$.
Then $p^2 \leq d \leq \rho n^2$ and, using~(\ref{ineq:c}), $2\rho^{1/6} \leq (k-1)\alpha \leq 1 -(k-1)c$.
We can apply Lemma~\ref{baddegclaim} with parameters $d$, $p$ and $\rho$, using the $k$-partition returned by Lemma~\ref{lem-Almost-k-Partite} that has $k-1$ parts of size $\lfloor cn\rfloor$.
Lemma~\ref{baddegclaim} implies that \Ppartition($G$) holds for $(A_1,\dots,A_k)$ with parameter $2k^2\sqrt{d}/n \leq 2k^2\sqrt{\rho_0}$ and hence with parameter $\beta$.

For \Pcomplete($G$), let $ij \in \binom{[k-1]}{2}$ and let $x \in A_i$ and $y \in A_j$.
Then Lemma~\ref{baddegclaim}(iv) implies that
\begin{eqnarray*}
P_3(xy,G) &\leq& n - |A_i|-|A_j| + d_G(x,A_i)+d_G(y,A_j) \stackrel{\Ppartition(G)}{\leq} n - 2(c-\beta)n + 2\rho_0^{1/5} n\\
&\stackrel{(\ref{ineq:c})}{\leq}& (k-2)cn - (\sqrt{2\alpha}-2\beta-2\rho_0^{1/5})n < (k-2)cn - \sqrt{\alpha}n.
\end{eqnarray*}
Thus~(\ref{2path}) implies that $xy \in E(G)$.
So \Pcomplete($G$) holds.
Lemma~\ref{baddegclaim}(ii) implies that
\begin{equation}\label{6.2m}
m = \sum_{ij \in \binom{[k]}{2}}e(\overline{G}[A_i,A_j]) \leq 3k^2\sqrt{\rho_0}n^2 < \eta n^2.
\end{equation}

For \Pbadedges($G$), note that $|Z|\le 2m/(\xi n)\le 2\eta n/\xi\le \delta n$.
Furthermore, Lemma~\ref{baddegclaim}(iii) implies that for every $i \in [k]$ and $e \in E(G[A_i])$, there is at least one endpoint $x$ of $e$ with $$
d_{\overline{G}}(x,\overline{A_i}) \geq \frac{1}{2}\left( n-(k-1)cn - 3k^2\sqrt{\rho_0}n\right) \stackrel{(\ref{ineq:c})}{\geq} \frac{(k-1)\alpha n}{3} > \xi n.
$$
Thus $x \in Z$.
The final part of \Pbadedges($G$) follows from Lemma~\ref{baddegclaim}(iv) and the fact that $\rho_0 \ll \delta$.

We now prove \PZk($G$).
Let $z \in Z \cap A_k$ be arbitrary.
By the definition of $Z$, there is some $i \in [k-1]$ such that $d_{\overline{G}}(z,A_i) \geq \xi n/k$.
Let $j \in [k-1]\setminus \lbrace i \rbrace$ and $y \in A_j$ be arbitrary.
We have
\begin{eqnarray*}
P_3(zy,G) &\leq& d_G(y,A_j) + d_G(z,A_k) + d_G(z,A_i) + (n-|A_i|-|A_j|-|A_k|)\\
&\stackrel{\Ppartition(G),\Pbadedges(G)}{\leq}& 2\delta n +  (c+\beta)n - \xi n/k + ((k-3)c+3\beta)n \ \leq\ (k-2)cn - \xi n/(2k).
\end{eqnarray*}
Thus~(\ref{2path}) implies that $xy \in E(G)$.
This proves \PZk($G$).

The property \Pmissing($G$) holds immediately from the definition of $Z$.
The bound on $m$ claimed in the lemma was established in~(\ref{6.2m}).
Finally, Lemma~\ref{baddegclaim}(v) implies that $h \leq k\rho_0^{1/30}m \leq \delta m$.
\end{proof}

\subsection{Applying Lemma~\ref{approx}}\label{applying}

Let $G$ be a worst counterexample, that is, $G$ satisfies~\ref{worst-C1}--\ref{worst-C3}. Let
$A_1,\ldots, A_k$ be a max-cut partition of $G$ satisfying~\ref{worst-C3}. Assume that
$|A_k|=\min_{i\in[k]}|A_i|$. Until the end of Section~\ref{int3}, we fix the $(A_1,\ldots,A_k;Z,\beta,\xi,\xi,\delta)$-partition of $G$ obtained from applying Lemma~\ref{approx} to $G$ and $A_1,\ldots, A_k$ using the parameters in~(\ref{hierarchy}). Let $\underline{m} = (m_1,\ldots,m_{k-1})$ be the missing vector of this partition and let
\begin{equation}\label{m}
m := m_1 + \ldots + m_{k-1} \leq \eta n^2.
\end{equation}
By permuting $A_1,\ldots,A_{k-1}$ if necessary, we may assume that $m_{k-1}=\max_{i\in[k-1]}m_i$.
(This assumption will not be used until the proof of Lemma~\ref{superlinear1}.)
Further,
\begin{equation}\label{h}
	h := \sum_{i \in [k]}e(G[A_i]) \leq \delta m.
\end{equation}

Define
\begin{equation}\label{t}
t := \frac{m}{(kc-1)n} \stackrel{(\ref{ineq:c})}{\geq}\frac{m}{cn}.\quad\text{Then}\quad t^2 \stackrel{(\ref{ineq:c})}{\leq} \frac{m^2}{2\alpha n^2} \stackrel{(\ref{m})}{\leq} \frac{\eta m}{2\alpha} \stackrel{(\ref{hierarchy})}{\leq} \sqrt{\eta}m.
\end{equation}
Since \Pmissing($G$) holds with both $\gamma_1$ and $\gamma_2$ set to the same value $\xi$, this uniquely determines the set $Z$ as
\begin{equation}\label{Zdef}
Z = \bigcup_{i \in [k]}\left\lbrace z \in A_i: d_{\overline{G}}(z,\overline{A_i}) \geq \xi n \right\rbrace.
\end{equation}
For all $i \in [k]$, let
\begin{equation}\label{Z}
Z_i := A_i \cap Z \quad \text{and} \quad R_i := A_i \setminus Z.
\end{equation}

By \Pbadedges($G$), $R_i$ is an independent set for all $i \in [k]$.
By \Pcomplete($G$) and \Pmissing($G$), for each $i \in [k-1]$, every $z \in Z_i$ has $d_{\overline{G}}(z,A_k) \geq \xi n$.
Notice that, by \PZk($G$), the set $Z_k$ has a partition $Z_k^1 \cup \ldots \cup Z_k^{k-1}$ such that, for all $ij \in \binom{[k-1]}{2}$ we have that $G[Z_k^i,A_j]$ is complete. In particular, each vertex in $Z_k^i$ sends
at least $\xi n$ missing edges to $A_i$. Thus we have for all $i \in [k-1]$
\begin{equation}\label{Zsize}
	|Z_i \cup Z_k^i| \leq \frac{2m_i}{\xi n}\quad\text{and}\quad |Z| \leq \frac{2(m_1+\dots+m_{k-1})}{\xi n} = \frac{2m}{\xi n} \stackrel{(\ref{m})}{\leq} \sqrt{\eta}n.
\end{equation}

For each $i \in [k-1]$, let
\begin{equation}\label{XY}
Y_i := \lbrace y \in Z_k^i : d_{G}(y,A_i) \leq \gamma n \rbrace, \quad Y := \bigcup_{i \in [k-1]}Y_i, \quad X_i := Z_k^i \setminus Y_i, \text{ and } X := \bigcup_{i \in [k-1]}X_i.
\end{equation}
See Figure~\ref{partition} for an illustration.

\begin{center}
\begin{figure}
\includegraphics[scale=0.7]{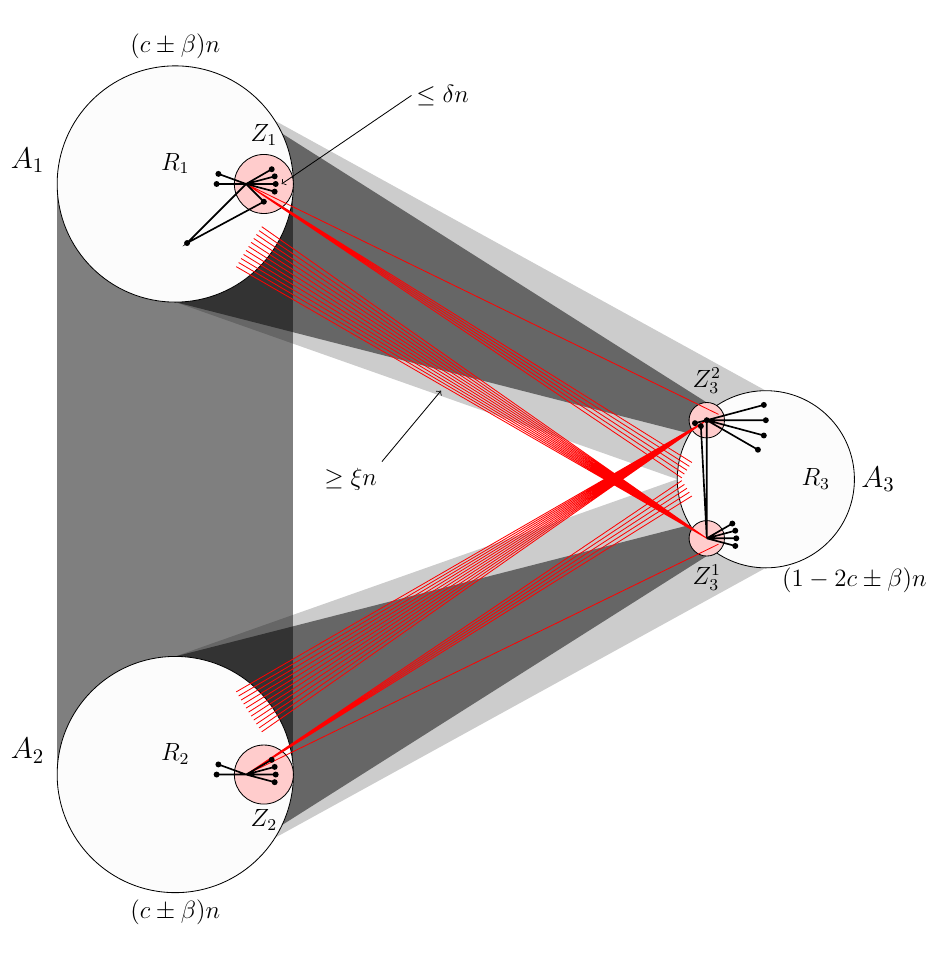}
\caption{An $(A_1,A_2,A_3;Z,\beta,\xi,\xi,\delta)$-partition of $G$ (here $k=3$).
Here and in the other figures, dark grey represents a complete bipartite pair, and light grey represents an `almost complete' bipartite pair, in which each vertex has small missing degree. The red edges are missing edges, and $Z$ is also coloured (light) red.
}
\label{partition}
\end{figure}
\end{center}

In the proof, we will perform various transformations on $G$ which will mainly involve changing adjacencies at vertices in $Y$ and $X$. It turns out that vertices in $X$ are much harder to deal with than those in $Y$, and much of the proof is devoted to these troublesome vertices.

\medskip
We need a simple proposition before we start with the first main ingredient of the proof in Section~\ref{sec:trans}.

\begin{proposition}\label{Gprops}
The following hold in $G$:
\begin{itemize}
\item[(i)] Suppose that $xy \in E(G[A_k])$ and $x \in R_k$. Then $y \in Y$.
\item[(ii)] For all $ij \in \binom{[k-1]}{2}$ we have that $G[Y_i,Y_j]$ is complete.
\end{itemize}
\end{proposition}

\begin{proof}
For (i), first note that  $d_{\overline{G}}(x,\overline{A_k}) < \xi n$ by \Pmissing$(G)$, since $x$ is in 
$R_k=A_k\setminus Z$.
Next,
\Pbadedges$(G)$ implies that $y \in Z_k$.  
By \PZk$(G)$ there is $i \in [k-1]$ such that $y \in Z_k^i$.
Using~(\ref{2path}) and that $G[Z_k^i,A_j]$ is complete for every $j\in[k-1]\setminus\{i\}$, we have that
\begin{eqnarray*}
(k-2)cn + \error \geq P_3(xy,G) &\stackrel{\Ppartition(G),\Pmissing(G)}{\geq}& \sum_{j \in [k-1]\setminus\{i\}}|A_j| + d_G(y,A_i) - \xi n\\
&\stackrel{\Ppartition(G)}{\geq}& (k-2)(c-\beta)n + d_G(y,A_i) - \xi n
\end{eqnarray*}
and so $d_G(y,A_i) \leq (k\beta + \xi)n < \gamma n$. Thus $y \in Y$.

To prove (ii), let $y \in Y_i$ and $x \in Y_j$.
Then
\begin{eqnarray*}
P_3(xy,G) &\leq& \sum_{\stackrel{t \in [k-1]}{t \neq i,j}}|A_t| + d_G(y,A_i) + d_G(x,A_j) + \max_{z \in Y}d_G(z,A_k)\\  &\stackrel{\Ppartition,\Pbadedges(G)}{\leq}& (k-3)(c+\beta)n  + 2\gamma n + \delta n
\ \leq\ (k-2)cn - cn/2. 
\end{eqnarray*}
Thus~(\ref{2path}) implies that $xy \in E(G)$.
\end{proof}

\section{The intermediate case: transformations}\label{sec:trans}

The aim of this section is to prove the following lemma, which enables us to find a $k$-partite $(n,e)$-graph $G'$ which inherits many of the useful properties of $G$ but does not contain many more triangles than $G$ (see Figure~\ref{partition6} for an illustration of $G'$).
Let
\begin{equation}\label{Cvalue}
C := \frac{1}{\sqrt{\delta}}.
\end{equation}

\begin{lemma}\label{maintrans}
	Suppose that $m \geq Cn$.
	Then there exists an $(n,e)$-graph $G'$ with $V(G')=V(G)$ which has the following properties.
	\begin{itemize}
		\item[(i)] For all $i \in [k-1]$ there exists $U_i \subseteq X_i$ such that, letting $A_i'' := A_i \cup Y_i \cup U_i$ and $A_k'' := V(G)\setminus \bigcup_{i \in [k-1]}A_i''$, the graph $G'$ is $k$-partite with partition $A_1'',\ldots,A_k''$, and further has an $(A_1'',\ldots,A_k'';3\beta)$-partition.
		\item[(ii)] The missing vector $\underline{m}' := (m_1',\ldots,m_{k-1}')$ of $G'$ with respect to this partition satisfies $\alpha^2 m_i-2\sqrt{\delta}m \leq m_i' \leq 2m_i + 2\sqrt{\delta}m$ for all $i \in [k-1]$.
		\item[(iii)] $K_3(G') \leq K_3(G) + \delta^{1/4}m^2/(2n)$.
	\end{itemize}
\end{lemma}

It is important to note that we do \emph{not} assume $m \geq Cn$ in any of the lemmas which precede the proof of Lemma~\ref{maintrans} in Section~\ref{maintransproof}.
Indeed, we will require some of these lemmas in both cases $m \geq Cn$ and $m < Cn$.

We will obtain a sequence of $(n,e)$-graphs $G =: G_0, G_1,\ldots,G_6 =: G'$ via a series of transformations such that Transformation $i$ is applied to $G_{i-1}$ to obtain $G_i$ and it preserves the number of edges and vertices: $e(G_{i-1})=e(G_{i})$.
For each $i$, $G_i$ has at most as many bad edges as $G_{i-1}$, and $K_3(G_i)$ is not much larger than $K_3(G_{i-1})$.
The final graph $G'$ is required to have a special partition and a missing vector with the property that each entry is within a constant multiplicative factor of the corresponding entry in $G$.
So each $G_i$ must also have these properties.

Transformation $i$ for $i \in \lbrace 1,2,3\rbrace$ consists of a `local' transformation applied to each of a given set of vertices $U$ in turn, producing graphs $G_{i-1} =: G_{i-1}^0, G_{i-1}^1,\ldots,G_{i-1}^{|U|} =: G_i$.
We first derive some fairly precise properties of the graph $G_{i-1}^j$, and then after that we derive the required less precise properties of the graph $G_i$ obtained after the final step.
The reason for this is that a single step (i.e.,~obtaining $G_{i-1}^1$ only) is also needed at a later stage in the proof to derive a contradiction.

For all $i \in [k-1]$, we will let
\begin{equation}\label{aidef}
a_i := \sum_{j \in [k-1]\setminus\lbrace i \rbrace}|A_j| = n - |A_i| - |A_k|.
\end{equation}

\subsection{Vertices with small missing degree}\label{smalldeg}

In the sequence of transformations described, we will often want to `fill in' some missing edges, and thus we must remove some edges from another part of the graph to compensate.
It will be useful if we have a fairly large stockpile of such edges which somehow exhibit average behaviour, and this property is preserved even after removing many of these well-behaved edges.
For this reason we define $Q_1,\ldots,Q_{k-1}$ and $R_k' \subseteq R_k$ below.

\begin{proposition}\label{Qi}
Let $A_i,R_i,m_i$ for $i \in [k]$ and $Z$ be as in Section~\ref{applying}.
	Let $J$ be an $n$-vertex graph with an $(A_1,\ldots,A_k;Z,2\beta,\xi/4,2\xi,\delta)$-partition and missing vector $\underline{m}^* = (m^*_1,\ldots,m^*_{k-1})$ where $m_i^* \leq m_i$ for all $i \in [k-1]$.
	Then, for all $i \in [k-1]$ there exists $Q_i \subseteq J[R_i,R_k]$ such that
	$Q_i$ is a collection of $2\delta n$ edge-disjoint stars, each with a distinct centre in $A_k$ and with $\delta n$ leaves; and the centre of each star has missing degree at most $2\sqrt{\eta}n$.
	(In particular, for all $e\in Q_i$, we have $P_3(e,J)\ge\sum_{j\in[k-1]\setminus\{i\}}|A_j|-2\sqrt{\eta}n$.)
\end{proposition}
\begin{proof}
	Let $R_k^*\subseteq R_k$ consist of vertices with missing degree at least $2\sqrt{\eta}n$ in $J$. Then 
	$$|R_k^*|\le \frac{\sum_{i\in[k-1]}m_i^*}{2\sqrt{\eta}n}\le \frac{m}{2\sqrt{\eta}n}\stackrel{(\ref{m})}{\le}\frac{\sqrt{\eta}n}{2}.$$
	By \Ppartition,\Pbadedges$(J)$, we have that $|R_i|\ge (c-2\beta) n-|Z| \geq (c-3\beta)n$ for every $i\in[k-1]$ and $|R_k\setminus R_k^*|\ge (1-(k-1)c-4\beta)n \geq 2\delta n\cdot (k-1)$. Thus, each $Q_i$ can be chosen by picking a distinct set of $2\delta n$ vertices in $R_k\setminus R_k^*$ along with $\delta n$ of each one's $R_i$-neighbours (of which there are at least $(c-\beta-2\xi)n$ by \Ppartition,\Pbadedges($J$)).
\end{proof}

Let $R_k' \subseteq R_k$ be such that $|R_k'| = |R_k|-\xi n/2$ and $d_{G}(x',Z_k) \leq d_{G}(x,Z_k)$ for all $x' \in R_k'$ and $x \in R_k\setminus R_k'$.
Let also
\begin{equation}\label{Delta}
\Delta := \max_{x \in R_k'}d_G(x,Z_k) = \max_{x \in R_k'}d_G(x,A_k)
\end{equation}
where the second inequality follows from \Pbadedges($G$).
By \Pbadedges($G$) and~(\ref{h}),
$$
2\delta m \geq 2e(G[A_k]) \geq \sum_{x \in R_k\setminus R_k'}d_G(x,A_k) \geq (|R_k|-|R_k'|)\Delta = \frac{\xi n}{2} \cdot \Delta.
$$
Therefore every $x \in R_k'$ is such that
\begin{equation}\label{Rk'}
d_G(x,A_k) \leq \Delta \leq \frac{4\delta m}{\xi n} \leq \frac{\delta^{1/3}m}{n}.
\end{equation}

\subsection{Transformation $1$: removing bad edges in $A_1,\ldots,A_{k-1}$}

Our first goal is to obtain a graph $G_1$ from $G$ which has the property that $G_1[A_i]$ is independent for all $i \in [k-1]$ and $G_1$ does not contain many more triangles than $G$.
The following lemma concerns the local transformation of removing all bad edges incident to a single $z \in Z\setminus Z_k$ and replacing them with certain missing edges incident to $z$ (see the left-hand image in Figure~\ref{partition1}).

\begin{center}
\begin{figure}
\includegraphics[scale=1]{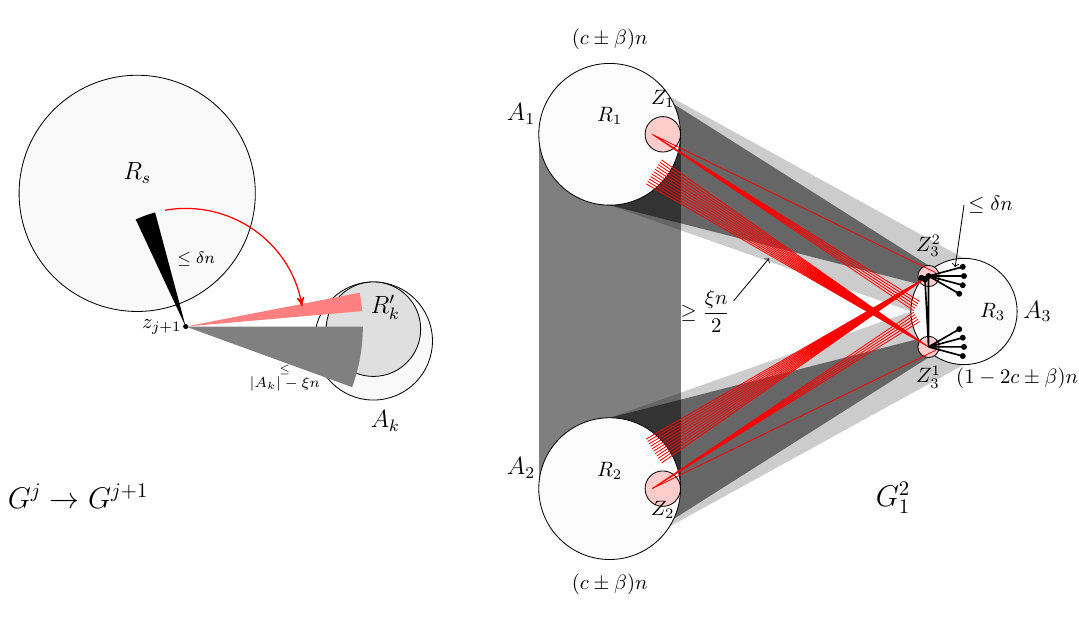}
\caption{Transformation 1: $G \rightarrow G_1^2$ (here $k=3$). Left: A single step $G^j \rightarrow G^{j+1}$ as in Lemma~\ref{subZiedges}, in which the black edges are replaced by the pink edges. Right: The final graph $G_1^2$ obtained in Lemma~\ref{Ziedges}, in which $A_1$ and $A_2$ are now independent sets.}
\label{partition1}
\end{figure}
\end{center}

\begin{lemma}\label{subZiedges}
	Let $p := |Z \setminus Z_k|$ and let $z_1,\ldots,z_p$ be any ordering of $Z\setminus Z_k$.
	For each $r\in [p]$, let $s(r)$ be such that $z_r \in A_{s(r)}$.
	Then there exists a sequence $G=:G^0,G^1,\ldots,G^p =: G_1$ of graphs such that for all $j \in [p]$,
	\begin{enumerate}[label={\emph{J(\arabic*,$j$)}:}]
		\item\label{W1} $G^j$ is an $(n,e)$-graph and has an $(A_1,\ldots,A_k;Z,\beta,\xi/2,\xi,\delta)$-partition.
		\item\label{W2} $E(G^j)\setminus E(G^{j-1}) = \lbrace z_{j}x : x \in R(z_{j}) \rbrace$ for some $R(z_{j}) \subseteq R_k'$, and $E(G^{j-1})\setminus E(G^j)$ is the set of $xz_j \in E(G)$ with $x \in A_{s(j)} \setminus \lbrace z_1,\ldots,z_{j-1} \rbrace$.
		\item\label{W3} $K_3(G^{j})-K_3(G^{j-1}) \leq \sum_{y \in N_{G^{j-1}}(z_{j},A_{s(j)})}\left(\Delta
		- |Z_k\setminus Z_k^{s(j)}| - P_3(yz_{j},G^{j-1};R_k)\right)$. Furthermore, equality holds only if $G^{j-1}[N_{G^j\setminus G^{j-1}}(z_j, R_k), \cup_{i\in[k-1]\setminus\{s(j)\}}A_i]$ is complete.
	\end{enumerate}
\end{lemma}

\begin{remark} The combined properties of Lemma~\ref{subZiedges} state that each $G^j$ is obtained from the previous graph $G^{j-1}$ by replacing all current edges connecting $z_j$ to its part with the same number of new edges between $z_j$ and $R_k'$. Thus
 $d_{G^j}(z_t,A_{s(t)}) = 0$ for all $t \in [j]$; $e(\overline{G^j}[A_i,A_k])=e(\overline{G^{j-1}}[A_i,A_k])$ for all $i \neq s(j)$, and $e(\overline{G^j}[A_{s(j)},A_k]) = e(\overline{G^{j-1}}[A_{s(j)},A_k])-d_{G^{j-1}}(z_j,A_{s(j)})$.\end{remark} 

\begin{proof}[Proof of Lemma~\ref{subZiedges}.]
	Let $G^0 := G$. Suppose we have obtained $G^0,\ldots,G^{j}$ for some $j < p$ such that, for all $r \leq j$, properties J($1,r$)--J($3,r$) hold.
	For $g \in [3]$, let J($g$) denote the conjunction of J($g,1$)$,\ldots,$J($g,j$).
	We obtain $G^{j+1}$ as follows.
	Let $s := s(j+1)$.
	Choose $R(z_{j+1}) \subseteq R_k'\setminus N_{G^j}(z_{j+1})$ such that $|R(z_{j+1})| = d_{G^j}(z_{j+1},A_s)$.
	Let us first see why this is possible.
	One consequence of J($2$) is that the neighbourhood of $z_{j+1}$ in $G^j$ is obtained from its neighbourhood in $G$ by removing its $G$-neighbours among $\lbrace z_1,\ldots, z_j \rbrace \cap A_s$.  
	Thus, as $|R_k'|=|R_k|-\xi n/2$, we have
	\begin{align*}
	d_{\overline{G^{j}}}(z_{j+1},R_k') &\stackrel{J(2)}{=} d_{\overline{G}}(z_{j+1},R_k') \geq d_{\overline{G}}(z_{j+1},A_k)-|Z_k|-\xi n/2 \stackrel{\Pmissing(G)}{\geq} \xi n/2 - \delta n \geq \delta n\\
	&\stackrel{\Pbadedges(G)}{\geq} d_G(z_{j+1},A_s) \stackrel{J(2)}{\geq} d_{G^{j}}(z_{j+1},A_s).
	\end{align*}
	So $R(z_{j+1})$ exists.
	Now define $G^{j+1}$ by setting $V(G^{j+1}) := V(G^{j})$ and
	$$
	E(G^{j+1}) := \left( E(G^{j}) \cup \lbrace z_{j+1}x : x \in R(z_{j+1}) \rbrace \right) \setminus E(G^{j}[z_{j+1},A_s]).
	$$
	Thus $G^{j+1}$ is obtained by replacing all bad edges of $G^j$ which are incident with $z_{j+1}$ by the same number of missing edges of $G^j$ which are incident to $z_{j+1}$.
	The endpoints $x$ of these new edges are chosen in $R_k'$ to ensure that the number of new triangles created is not too large.
	
	We will now show that $G^{j+1}$ satisfies $J(1,j+1),\ldots,J(3,j+1)$, beginning with $J(1,j+1)$.
	By construction, $G^{j+1}$ is an $(n,e)$-graph.
	To show that $G^{j+1}$ has an $(A_1,\ldots,A_k;Z,\beta,\xi/2,\xi,\delta)$-partition, we need to show that \Ppartition($G^{j+1}$)--\Pmissing($G^{j+1}$) hold with the appropriate parameters.
	All properties except \Pmissing($G^{j+1}$) are immediate.
	For \Pmissing, let $i \in [k]$ and let $y \in A_i$ be arbitrary.
	We have that
	\begin{equation}\label{misseqnew}
	d^m_{G^{j+1}}(y) = \begin{cases} d^m_{G^{j}}(y)-1, &\mbox{if } y \in R(z_{j+1}), \\ 
	d^m_{G^{j}}(y)-d_{G^j}(z_{j+1},A_s), &\mbox{if } y=z_{j+1}, \\ 
	d^m_{G^{j}}(y), & \mbox{otherwise. } \end{cases}
	\end{equation}
	Thus if $y \in A_i \setminus Z$, we have $d^m_{G^{j+1}}(y) \leq d^m_{G^{j}}(y) \leq \xi n$ since $G^j$ has an $(A_1,\ldots,A_k;Z,\beta,\xi/2,\xi,\delta)$-partition.
	It remains to consider the case $y=z_{j+1}$ (since missing degree is unchanged for all other vertices in $Z$).
	By the consequence of $J(2)$ stated above,
	\begin{equation}\label{misseq3}
	d^m_{G^{j}}(z_{j+1}) = d^m_G(z_{j+1}) \text{ and } d_{G^j}(z_{j+1},A_s) = d_G(z_{j+1},A_s\setminus \lbrace z_1,\ldots,z_j \rbrace).
	\end{equation}
	Thus, as $G$ has an $(A_1,\ldots,A_k;Z,\beta;\xi,\xi,\delta)$-partition,
	$$
	d^m_{G^{j}}(z_{j+1}) \geq \xi n - d_G(z_{j+1},A_s\setminus \lbrace z_1,\ldots,z_j \rbrace) \stackrel{\Pbadedges(G)}{\geq} (\xi-\delta)n \geq \xi n/2.
	$$
	Thus \Pmissing($G^{j+1}$) holds.
	We have shown that $J(1,j+1)$ holds.
	That $J(2,j+1)$ holds is clear from $J(2)$ and the construction of $G^{j+1}$.
	
	For J(3,$j+1$), observe that a triangle is in $G^{j+1}$ but not $G^j$ if and only if it contains an edge $xz_{j+1}$ where $x \in R(z_{j+1})$; furthermore, no triangle contains two such edges; and a triangle is in $G^j$ but not $G^{j+1}$ if and only if it contains an edge $yz_{j+1}$, where $y \in N_{G^j}(z_{j+1},A_s)$.
	Thus
	\begin{align}
	\label{trans1} K_3(G^{j+1})= K_3(G^{j})&+ \sum_{x \in R(z_{j+1})}P_3(xz_{j+1},G^{j+1}) - \sum_{y \in N_{G^{j}}(z_{j+1},A_s)}P_3(yz_{j+1},G^{j};\overline{A_s})\\
	\nonumber &- K_3(z_{j+1},G^j;A_s).
	\end{align}
	Fix $y \in N_{G^{j}}(z_{j+1},A_s)$.
	By J(1,$j$), \Pcomplete($G^{j}$) holds and, since $y,z_{j+1} \in A_s$, both of these vertices are incident to all of $A_t \cup Z_k^t$ for $t \in [k-1]\setminus \lbrace s \rbrace$.
	Recall the definition of $a_s$ from~(\ref{aidef}).
	So
	$$
	P_3(yz_{j+1},G^{j};\overline{A_s}) = a_{s} + |Z_k \setminus Z^s_k| + P_3(yz_{j+1},G^j;R_k\cup Z_k^s)\ge a_{s} + |Z_k \setminus Z^s_k| + P_3(yz_{j+1},G^j;R_k).
	$$
	Now fix $x \in R(z_{j+1})\subseteq R_k'$.
	Then, by $J(2,j+1)$, we have $d_{G^{j+1}}(z_{j+1},A_s) = 0$, and $d_{G^{j+1}}(x,R_k)=d_G(x,R_k)=0$. So
	\begin{eqnarray}
	P_3(xz_{j+1},G^{j+1}) &=& a_{s} - d_{\overline{G^j}}(x,\cup_{i\in[k-1]\setminus \{s\}}A_i)
	+ P_3(xz_{j+1},G^{j+1};Z_k) \nonumber\\
	&\leq& a_{s} + d_{G^{j+1}}(x,Z_k)\ \stackrel{J(2)}{=}\ a_s + d_G(x,Z_k)\ \stackrel{(\ref{Delta})}{\le}\ a_s+\Delta\label{eq-strong1}.
	\end{eqnarray}
	Therefore,
	\begin{align*}
	K_3(G^{j+1})-K_3(G^{j}) &\stackrel{(\ref{trans1}),\eqref{eq-strong1}}{\leq} \sum_{y \in N_{G^j}(z_{j+1},A_s)}\left(\Delta - |Z_k\setminus Z_k^s| - P_3(yz_{j+1},G^j;R_k)\right),
	\end{align*}
	where equality holds only when equality in~\eqref{eq-strong1} holds for every $x \in R(z_{j+1})$. This happens only if $d_{\overline{G^j}}(x,\cup_{i\in[k-1]\setminus \{s\}}A_i)=0$ for every $x\in R(z_{j+1})$, in other words, $G^{j}[R(z_{j+1}), \cup_{i\in[k-1]\setminus\{s\}}A_i]$ is complete. Recall that $R(z_{j+1})=N_{G^{j+1}\setminus G^{j}}(z_{j+1}, R_k)$. This completes the proof of $J(3,j+1)$.
\end{proof}

We can now derive some properties of $G_1 := G^p$ obtained in Lemma~\ref{subZiedges}, namely that its only bad edges have endpoints in $A_k$, and $G_1$ does not have many more triangles than $G$.
In fact we consider the graph $G_1^\ell$ which is obtained by applying Lemma~\ref{subZiedges} for only vertices $z_j \in Z_1 \cup \ldots \cup Z_\ell$.
See the right-hand side of Figure~\ref{partition1} for an illustration of $G_1^2$ in the case $k=3$.

\begin{lemma}\label{Ziedges}
	Let $\ell \in [k-1]$.
	There exists
	an $(n,e)$-graph $G^\ell_1$ on the same vertex set as $G$ such that
	\begin{itemize}
		\item[(i)] $G^\ell_1$ has an $(A_1,\ldots,A_k;Z,\beta,\xi/2,\xi,\delta)$-partition with missing vector $\underline{m}^{(1,\ell)} := (m^{(1,\ell)}_1,\ldots,m^{(1,\ell)}_{k-1})$ where $m_i/2 \leq m_i^{(1,\ell)} \leq m_i$ for all $i \in [k-1]$.
		\item[(ii)] $E(G^\ell_1[A_i]) = \emptyset$ for all $i \in [\ell]$, and $E(G^\ell_1[A_i])=E(G[A_i])$ otherwise.
		\item[(iii)] $K_3(G^\ell_1) \leq K_3(G) + \delta^{7/8}m^2/n$. 
		\item[(iv)] $N_{G^\ell_1}(z)=N_G(z)$ for all $z \in Z_k$ and $N_{G^\ell_1}(x,A_k)=N_G(x,A_k)$ for all $x \in A_k$.
	\end{itemize}
\end{lemma}

\begin{proof}
	Let $p := |Z\setminus Z_k|$ and let $p' := |Z_1 \cup \ldots \cup Z_\ell|\leq p$.
	Let $z_1,\ldots,z_p$ be an ordering of $Z\setminus Z_k$ such that for $1 \leq i < i' \leq k-1$, every vertex in $Z_i$ appears before any vertex in $Z_{i'}$. 
	Apply Lemma~\ref{subZiedges} to obtain $G^\ell_1 := G^{p'}$ satisfying $J(1,p'),\ldots,J(3,p')$.
	By J(1,$p'$), $G^\ell_1$ has an $(A_1,\ldots,A_k;Z,\beta,\xi/2,\xi,\delta)$-partition. 
	Further, J(2) (defined at the beginning of the proof of Lemma~\ref{subZiedges}) implies that, for $i \in [\ell]$,
	\begin{equation}\label{j2}
	\sum_{\stackrel{j \in [p']}{s(j)=i}}d_{G^{j-1}}(z_j,A_i) = \sum_{\stackrel{j \in [p']}{s(j)=i}} d_G(z_j, A_i\setminus \lbrace z_1,\ldots,z_{j-1}\rbrace) = e(G[A_i]).
	\end{equation}
	If $i \in [k-1]\setminus [\ell]$ then $m^{(1,\ell)}_i = m_i$.
	If $i \in [\ell]$, then
	\begin{eqnarray*}
		m^{(1,\ell)}_i &=& e(\overline{G^{p'}}[A_i,A_k])\ \stackrel{J(2,p')}{=}\ e(\overline{G}[A_i,A_k]) - \sum_{\stackrel{j \in [p']}{s(j)=i}}d_{G^{j-1}}(z_j,A_i) \ \stackrel{(\ref{j2})}{=}\ m_i - e(G[A_i])\\
		&\stackrel{\Pbadedges(G)}{\geq}& m_i - |Z_i|\cdot \delta n\ \geq\ m_i - |Z_i| \cdot \frac{\xi n}{4}\ \stackrel{\Pmissing(G)}{\geq}\ \frac{m_i}{2}
	\end{eqnarray*}
	while clearly $m^{(1,\ell)}_i \leq m_i$, proving (i).
	Part (ii) follows immediately from $J(2)$.
	
	Equation~(\ref{h}) states that $\sum_{i \in [k]}e(G[A_i]) \leq \delta m$.
	Therefore
	\begin{eqnarray*}
		K_3(G_1^\ell) - K_3(G) &=& \sum_{j \in [p']}\left(K_3(G^j)-K_3(G^{j-1})\right) \ \stackrel{J(3)}{\leq}\  \sum_{j \in [p']}d_{G^{j-1}}(z_j,A_{s(j)}) \cdot \Delta\\
		\nonumber &\stackrel{(\ref{j2})}{=}& \sum_{i \in [\ell]}e(G[A_i]) \cdot \Delta \ \stackrel{(\ref{Rk'})}{\leq}\ \delta m \cdot \displaystyle\frac{4\delta m}{\xi n} \leq \displaystyle\frac{\delta^{7/8}m^2}{n}.
	\end{eqnarray*}
	
	Finally, Part (iv) follows from J(2).
\end{proof}

\subsection{Transformation 2: removing $Y_i$-$A_i$ edges}

The next transformation is applied to $G^\ell_1$ to obtain a graph which inherits the properties of $G^\ell_1$ whilst also reassigning $Y_i$ to $A_i$ and removing any edges which are bad relative to this new partition.
The only bad edges which remain are incident to $X$ in $A_k$.
Observe that the $(A_1,\ldots,A_k;Z,\beta,\xi/2,\xi,\delta)$-partition of $G^\ell_1$ is also an $(A_1,\ldots,A_k;Z,2\beta,\xi/4,2\xi,\delta)$-partition.

\begin{center}
\begin{figure}
\includegraphics[scale=0.9]{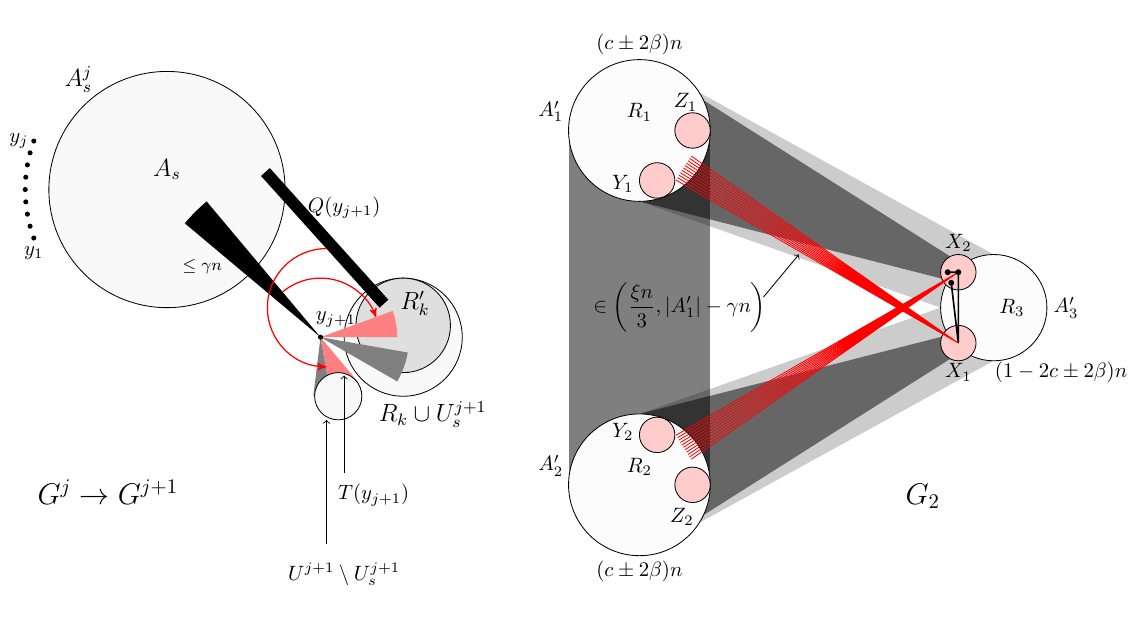}
\caption{Transformation 2: $G_1 \rightarrow G_2$. (here $k=3$). Left: A single step $G^j \rightarrow G^{j+1}$ as in Lemma~\ref{subYiedges}, in which the two sets of black edges are replaced by the corresponding sets of pink edges. Right: The final graph $G_2^2$ obtained in Lemma~\ref{Yiedges}, with the updated partition $A_1',A_2',A_3'$.}
\label{partition2}
\end{figure}
\end{center}

\begin{lemma}\label{subYiedges}
	Let $\ell \in [k-1]$ and let $G^\ell_1$ be any graph satisfying the conclusion of Lemma~\ref{Ziedges} applied with $\ell$.
	Let $q = q(\ell) := |Y_1 \cup \ldots \cup Y_\ell|$ and let $y_1,\ldots,y_{q}$ be an arbitrary ordering of $Y_1 \cup \ldots \cup Y_\ell$.
	For all $j \in [q]$, let $s(j) \in [k-1]$ be such that $y \in Y_{s(j)}$. Let $A_i^0:=A_i$ for $i\in [k]$.
	Let $Q^0_i := Q_{i}$ be obtained by applying Proposition~\ref{Qi} to the graph $J := G^\ell_1$ and the partition $(A_1^0,\dots,A_k^0)$, for all $i \in [k-1]$.
	For all $j \in [q]$, let 
	\begin{equation}\label{Aeq}
	A^j_t := \begin{cases} 
	A^{j-1}_t \cup \lbrace y_j \rbrace, &\mbox{if } t = s(j), \\
	A^{j-1}_t \setminus \lbrace y_j \rbrace, & \mbox{if } t= k,\\
	A^{j-1}_t, & \mbox{otherwise, } \end{cases}
	\end{equation}
	and $U^j := Z_k \cap A^j_k$ and $U^{j,i} := Z_k^i \cap A^j_k$ for every $i \in [k-1]$.
	Then there exists a sequence $G^\ell_1 =: G^0,G^1,\ldots,G^{q} =: G^\ell_2$ of graphs such that for all $j \in [q]$,
	\begin{enumerate}[label={K(\arabic*,$j$):}]
		\item\label{D2} 
		\begin{itemize}
			\item $E(G^j)\setminus E(G^{j-1})$ is a star with centre $y_{j}$, where 
			the set of leaves consists of $T(y_j)$ together with some vertices in $R_k'$, where $T(y_j)$ is the set of non-$G^{j-1}$-neighbours of $y_j$ in $U^{j-1} \setminus U^{j-1,s(j)}$.
			\item $E(G^{j-1})\setminus E(G^j) = \lbrace y_jv \in E(G) : v \in A^{j-1}_{s(j)}\rbrace \cup Q(y_j)$, where $Q(y_j) \subseteq Q^{j-1}_{s(j)}$ and $|Q(y_j)| \leq \delta n$. 
			
			\item If $Z_k=X_{s(j)}\cup Y_{s(j)}$, then $T(y_j)=Q(y_j)=\emptyset$.
			
			\item The total number of cross-edges in $G^j$ is at least that in $G^0$, i.e., 
			$$\sum_{ip\in{[k]\choose 2}}e(G^{j}[A_i^j,A_p^j])\ge \sum_{ip\in{[k]\choose 2}}e(G^{0}[A_i^0,A_p^0]).$$
		\end{itemize}
		Define $Q^j_i := Q^{j-1}_i\setminus Q(y_j)$ for all $i \in [k-1]$. 
		\item\label{D1} $G^j$ is an $(n,e)$-graph and has an $(A_1^j,\ldots,A_k^j;Z,\beta+\frac{j}{n},\frac{\xi}{2}-\frac{j}{n},\xi+2\delta +\frac{j}{n},\delta)$-partition, where $U^{j,1},\ldots,U^{j,k-1}$ is the partition of $U^j := Z \cap A_k^j$ given by \PZk($G^j$).
		\item\label{D3} \begin{align*}K_3(G^{j})-K_3(G^{j-1}) \leq \sum_{y \in N_{G^{j-1}}(y_{j},A_{s(j)}^{j-1})}&\left(\Delta- \frac{\xi}{6\gamma}|U^{j-1}\setminus U^{j-1,s(j)}| -P_3(yy_{j},G^{j-1};R_k)\right).\end{align*}
		Furthermore, equality holds only if $G^{j-1}[N_{G^j\setminus G^{j-1}}(y_j, R_k), \cup_{i\in[k-1]\setminus\{s(j)\}}A_i^{j-1}]$ is complete.
	\end{enumerate}
\end{lemma}

\begin{proof}
	Let $G^0 := G_1^\ell$. Note that $s(r)\le\ell$ for every $r\in[q]$.
	Suppose that we have obtained $G^0,\ldots,G^{j}$ for some $j < q$ such that, for all $r \leq j$, properties $K(1,r)$--$K(3,r)$ hold.
	For $g \in [3]$, let $K(g)$ denote the conjunction of properties $K(g,1),\dots,K(g,j)$.
	Let $s := s(j+1)$. By definition, $U^j\setminus U^{j,s}=(Z_k\setminus Z_k^s)\setminus\{y_1,\ldots,y_j\}$. Recall that 
	$$
	T(y_{j+1}) = N_{\overline{G^j}}(y_{j+1},U^j\setminus U^{j,s}).
	$$
	
	We obtain $G^{j+1}$ as follows.
	Choose a set $R(y_{j+1})$ of $d_{G^{j}}(y_{j+1},A_s^j)$ vertices in $R_k' \setminus N_{G^{j}}(y_{j+1})$.
	Note that $R_i \subseteq A^r_i$ for all $0 \leq r \leq j$ and $i\in [k]$ by~\eqref{Aeq}.
	Choose a set $Q(y_{j+1}) \subseteq Q^j_s$ of size $|T(y_{j+1})|$ with 
	\begin{equation}\label{eq:Q(y)}
	V(Q(y_{j+1}))\cap R_k\subseteq N_{\overline{G^{j}}}(y_{j+1}).
	\end{equation}
	Note that if $Z_k=X_{s}\cup Y_{s}$, then by definition $U^j\setminus U^{j,s}=\emptyset$. Therefore, $T(y_{j+1})=Q(y_{j+1})=\emptyset$.
	Now define $G^{j+1}$ by setting $V(G^{j+1}) := V(G^{j})$ and
	\begin{align*}
	E(G^{j+1}) &:= \left( E(G^{j}) \cup \lbrace y_{j+1}x : x \in R(y_{j+1}) \rbrace \cup \lbrace y_{j+1}z : z \in T(y_{j+1}) \rbrace \right)\setminus \left( E(G^{j}[y_{j+1},A_s^j]) \cup Q(y_{j+1}) \right).
	\end{align*}
	So $G^{j+1}$ is obtained from $G^j$ by replacing every neighbour of $y_{j+1}$ in $A^{j}_s$ with a non-neighbour in $R_k'$; and moving some previously unused edges from $Q_s$ to lie between $y_{j+1}$ and those non-neighbours in $Z_k\setminus Z_k^s$ which lie in $A^j_k$ (see the left-hand side of Figure~\ref{partition2} for an illustration of the transformation $G^j \rightarrow G^{j+1}$).
	
	Let us check that $G^{j+1}$ exists, that is, one can choose the sets $R(y_{j+1})$ and $Q(y_{j+1})$ with the stated properties.
	 Recall that $G$ and $G_1^{\ell}$ agree on $Y$ due to Lemma~\ref{Ziedges}(iv). Thus by Proposition~\ref{Gprops}(ii), $G_1^{\ell}[Y_i,Y_j]$ is complete for all $ij\in{[k-1]\choose 2}$. Consequently $T(y_{r})\cap Y=\emptyset$ for all $1\le r\le j$; in other words, no edge incident to $\{y_{j+1},\dots,y_q\}$ was modified when we passed from $G^0$ to $G^j$. This implies that
	\begin{equation}\label{K2cons}
	N_{G^j}(y_{j+1}) = N_{G^\ell_1}(y_{j+1})\supseteq \bigcup_{i\in[k-1]\setminus \{s\}}(A_i\cup Y_i).
	\end{equation}
	As $A_s^j=A_s\cup\{y_r: r\le j; s(r)=s\}$, together with~\eqref{K2cons}, this implies that $N_{G^j}(y_{j+1},A_s^j)\subseteq N_{G_1^{\ell}}(y_{j+1},A_s)\cup Y$. Since $|Y|\le |Z|\le \delta n$ by~\Pbadedges($G$), we have from
	Lemma~\ref{Ziedges}(iv) that
	\begin{equation}\label{eq-yjplus1}
	d_{G^{j}}(y_{j+1},A^j_s) \leq d_{G^\ell_1}(y_{j+1},A_s)+\delta n = d_G(y_{j+1},A_s)+\delta n \leq (\gamma+\delta)n \leq 2\gamma n.
	\end{equation}
	Thus
	\begin{eqnarray*}
		d_{\overline{G^{j}}}(y_{j+1},R_k') &\stackrel{K(1)}{=}& d_{\overline{G^\ell_1}}(y_{j+1},R_k')\ \geq\ |A_k|-|Z_k|-\xi n/2 - d_{G^\ell_1}(y_{j+1},A_k)\\
		&\stackrel{\Pbadedges(G^\ell_1),\Pmissing(G^\ell_1)}{\geq}& |A_k| - (\xi/2 +2\delta)n \ \stackrel{\Ppartition(G^\ell_1),\eqref{ineq:c}}{\geq}\ 2\gamma n\ \geq\ d_{G^{j}}(y_{j+1},A_s^j).
	\end{eqnarray*}
	So we can choose $R(y_{j+1})$ as required.
	Also, by $K(1)$ and Lemma~\ref{Ziedges}(iv), $N_{G^j}(y_{j+1},R_k)=N_{G}(y_{j+1},R_k)$, which is of size at most $\delta n$ by \Pbadedges($G$). Thus 
	 $$
	 |V(Q_s)\cap N_{\overline{G^{j}}}(y_{j+1},R_k)|
	 \ge |V(Q_s)\cap R_k|-|N_{G^{j}}(y_{j+1},R_k)|\ge 2\delta n-\delta n= \delta n.
	 $$
	Recall that $|Y|\le|Z|\le \sqrt{\eta} n$ by~\eqref{Zsize}, and $Q_s$ consists of $2\delta n$ stars each with $\delta n$ leaves centred at $R_k$. Thus the number of available edges in $Q^j_s$ (i.e.~all edges in $Q_s\setminus \cup_{\ell\in[j]}Q(y_{\ell})$ whose endpoints in $R_k$ are not adjacent to $y_{j+1}$) is at least
	\begin{eqnarray*}
		\delta n (\delta n - |Y|) \ge \delta n\geq |Z|\ge |U^j|
		\geq d_{\overline{G^j}}(y_{j+1},U^j \setminus U^{j,s}) = |T(y_{j+1})|=|Q(y_{j+1})|,
	\end{eqnarray*}
	so we can choose the desired $Q(y_{j+1}) \subseteq Q_s^j$.
	Hence $G^{j+1}$ exists.
	
	Recall that the sets $A^{j+1}_t$, $t\in [k]$, were defined in (\ref{Aeq}).
	It remains to check that $K(1,j+1)$--$K(3,j+1)$ hold. The first three bullet points in Property $K(1,j+1)$ follow immediately from the construction. To see the last bullet point, note that from $G^0$ to $G^{j+1}$, the cross-edges which are no longer present are precisely those in $Q(y_r)$ and $E(G^{r-1}[y_r,A_{s(r)}^{r-1}])$, which are compensated by $\lbrace xy_r : x \in T(y_r)\rbrace$ and $\lbrace xy_r : x \in R(y_r)\rbrace$ respectively for every $1\le r\le j+1$. In fact, $G^{j+1}$ will have more cross-edges than $G^0$ if there are $G[A_k]$-edges incident to $\{y_1,\ldots,y_{j+1}\}$.
	
	To check that $G^{j+1}$ has an
	$$
	(A^{j+1}_1,\ldots,A_k^{j+1};Z,\beta+(j+1)/n,\xi/2-(j+1)/n,\xi+2\delta+(j+1)/n,\delta)
	$$-partition, we need to show that \Ppartition($G^{j+1}$)--\Pmissing($G^{j+1}$) hold with the required parameters.
	For \Ppartition($G^{j+1}$), the part sizes $|A^{j+1}_t|,|A^j_t|$ differ by at most one.
	So for $t \in [k-1]$ we have
	$$
	\bigg| |A^{j+1}_t| - cn \bigg| \leq \bigg| |A^{j+1}_t| - |A^j_t| \bigg| + \bigg| |A^j_t| - cn \bigg| \leq \left(\beta + \frac{j}{n} \right)n + 1 = \left(\beta + \frac{j+1}{n}\right)n, 
	$$
	as required. The case $t=k$ is similar.
	
	By \Pcomplete($G^j$) we have that $G^j[A^j_i,A^j_p]$ is complete for all $ip \in \binom{[k-1]}{2}$.
	Thus, for \Pcomplete($G^{j+1}$), we need only check that $xy_{j+1} \in E(G^{j+1})$ for all $x \in A^{j+1}_i$ with $i \in [k-1]\setminus \lbrace s \rbrace$. Indeed, if $i \in [k-1]\setminus \lbrace s \rbrace$, then $A_i^{j+1}=A_i^j=A_i\cup\{y_r:r\le j; s(r)=i\}$ and, by~\eqref{K2cons} and Lemma~\ref{Ziedges}(iv), $N_{G^j}(y_{j+1})\supseteq A_i^{j+1}$. Finally, note that by construction, $N_{G^{j+1}}(y_{j+1}, A_i^{j+1})=N_{G^{j}}(y_{j+1}, A_i^{j+1})$.

	Note that \Pbadedges($G^{j+1}$) holds by \Pbadedges($G^\ell_1$) and K(1).
	For \PZk($G^{j+1}$), it suffices to show that, for all $ip \in \binom{[k-1]}{2}$, the bipartite graph $G^{j+1}[U^{j+1,i},A^{j+1}_p]$ is complete.
	By \PZk($G^j$) and $K(2,j)$, we have that $G^j[U^{j,i},A^j_p]$ is complete.
	For $i,p \neq s$, this means that $G^j[U^{j+1,i},A^{j+1}_p]$ is complete.
	But $G^j$ and $G^{j+1}$ are identical between these two sets by construction, so we are done in this case.
	Suppose instead that $i=s$.
	Then note that $U^{j+1,s}=U^{j,s}\setminus\{y_{j+1}\}$ and $A_{p}^{j+1}=A_{p}^j$, so we are done as $G^j[U^{j,s},A_{p}^j]$ is complete and $G^{j+1}$ is identical in this part. 
	Suppose finally that $p=s$.
	Then $U^{j+1,i}=U^{j,i}$ and $G^j[U^{j,i},A^{j+1}_s \setminus \lbrace y_{j+1} \rbrace]$ is complete.
	Thus, it suffices to show that $U^{j+1,i} = U^{j,i} \subseteq N_{G^{j+1}}(y_{j+1})$.
	But this is immediate by construction.
	So \PZk($G^{j+1}$) holds with $U^{j+1,i}$ playing the role of $U^i_k$.
	We now turn to \Pmissing($G^{j+1}$). In what follows, $d^m_{G^{r}}$ is the missing degree with respect to the partition $(A_1^{r},\ldots,A_k^{r})$.
	Let $y \in V(G^{j+1})$.
	We have by construction that
	\begin{equation}\label{missingeqYi}
	d^m_{G^{j+1}}(y) = \begin{cases} 
	|A^{j+1}_k|-d_{G^j}(y,A^j_k) - d_{G^j}(y,A^j_s)-|Q(y_{j+1})|, &\mbox{if } y = y_{j+1},\\
	d^m_{G^j}(y) + d_{Q(y_{j+1})}(y) - 1, & \mbox{if } y\in N_{\overline{G^j}}(y_{j+1},A_s^j),
	\\
	d^m_{G^j}(y) + d_{Q(y_{j+1})}(y) + 1, & \mbox{if } y\in N_{\overline{G_j}}(y_{j+1},U^{j+1,s})\setminus R(y_{j+1}),
	\\
	d^m_{G^j}(y) + d_{Q(y_{j+1})}(y), & \mbox{otherwise. } \end{cases}
	\end{equation}
	If $y \in Z \setminus \lbrace y_{j+1}\rbrace$, then $y$ is isolated in $\bigcup_{i \in [k-1]}Q_i$ and hence in $Q(y_{j+1})$.
	So $d^m_{G^{j+1}}(y) \geq d^m_{G^j}(y)$.
	Thus we are done by \Pmissing($G^j$) in this case.
	If $y \notin Z$, then, using $\Delta(\cup_{i\in[k-1]} Q_i)\le 2\delta n$ from Proposition~\ref{Qi} and $Q(y_{1}),\ldots,Q(y_{j+1})$ are edge-disjoint, we have
	$$
	d^m_{G^{j+1}}(y) \stackrel{(\ref{missingeqYi})}{\leq} d^m_{G^\ell_1}(y) + \Delta(\cup_{i\in[k-1]} Q_i) + j+1  \stackrel{\Pmissing(G^\ell_1)}{\leq} (\xi + 2\delta)n +j+1,
	$$
	as required.
	Moreover, by $K(1)$ and \Pbadedges($G_1^{\ell}$), $d_{G^j}(y_{j+1},A^j_k)\le d_{G_1^{\ell}}(y_{j+1},A_k)\le\delta n$. Using~\eqref{eq-yjplus1} and~\eqref{missingeqYi}, we have
	\begin{eqnarray}
	d^m_{G^{j+1}}(y_{j+1}) &=& |A^{j+1}_k| - d_{G^j}(y_{j+1},A^j_k) - d_{G^j}(y_{j+1},A^j_s)-|Q(y_{j+1})| \geq |A_k|-|Y| -2\delta n-2\gamma n\nonumber\\
\nonumber	&\stackrel{\Ppartition(G^\ell_1)}{\geq}& n-(k-1)cn - \beta n - 3\delta n - 2\gamma n\\
\label{eq-ym}	&\stackrel{\eqref{ineq:c}}{\geq}& \alpha n > \xi n/2 - (j+1).
	\end{eqnarray}
	Thus \Pmissing($G^{j+1}$) holds.
	This completes the proof of $K(2,j+1)$.

	Finally, we will show $K(3,j+1)$.
	For every $p\in[k-1]$ and $q\in[j+1]$, let
	$$
	a_p^{q} := \sum_{t \in [k-1]\setminus \lbrace p \rbrace}|A^{q}_t|.
	$$
	Then by~\eqref{Aeq}, $a_s^j=a_s^{j+1}$.
	Observe that a triangle is in $G^{j+1}$ but not $G^j$ if and only if it contains an edge $xy_{j+1}$ where $x \in R(y_{j+1})$ or $x \in (Z_k\setminus Z_k^s)\cap A_k^j$ is a non-neighbour of $y_{j+1}$ in $G^j$ (this is precisely the set $T(y_{j+1})$); and a triangle is in $G^j$ but not $G^{j+1}$ if and only if it contains an edge $uy_{j+1}$, where $u \in N_{G^j}(y_{j+1},A^j_s)$, or an edge $e \in Q(y_{j+1})$.
	Observe that there is no triangle in $G^j$ which contains at least two edges from $E(G^j)\setminus E(G^{j+1})$.
	Indeed, this follows from~\eqref{eq:Q(y)} and the facts that $E(G^j[A^j_s]),E(G^j[R_k]) = \emptyset$ (due to $s\le \ell$, Lemma~\ref{Ziedges}(ii) and $K(1)$).
	Thus
	\begin{align*}
	K_3(G^{j+1})-K_3(G^j) &\leq \sum_{e \in E(G^{j+1})\setminus E(G^j)}P_3(e,G^{j+1}) - \sum_{e \in E(G^j)\setminus E(G^{j+1})}P_3(e,G^j)\\
	&\leq \sum_{x \in R(y_{j+1})}P_3(xy_{j+1},G^{j+1}) - \sum_{y \in N_{G^{j}}(y_{j+1},A_s^j)}P_3(yy_{j+1},G^{j})\\
	&\quad\quad+ \sum_{z \in T(y_{j+1})}P_3(zy_{j+1},G^{j+1}) - \sum_{e \in Q(y_{j+1})} P_3(e,G^j).
	\end{align*}
	We will estimate each summand separately.
	Let $y \in N_{G^{j}}(y_{j+1},A_s^j)$.
	By $K(1,j+1)$ and the definition of $T(y_{j+1})$, we have that
	\begin{align*}
	P_3(yy_{j+1},G^{j}) &\geq a_s^{j} + d_{G^j}(y_{j+1},U^j\setminus U^{j,s}) + P_3(yy_{j+1},G^j;R_k)\\
	&= a_s^{j} + |U^j\setminus U^{j,s}|-|T(y_{j+1})| + P_3(yy_{j+1},G^j;R_k).
	\end{align*}
	Now let $x \in R(y_{j+1})$. 
	Then $d_{G^{j+1}}(y_{j+1},A^j_s) = 0$ and $x \in R_k'$, so
	\begin{eqnarray}\label{eq-st2}
	P_3(xy_{j+1},G^{j+1}) &\leq& a_s^{j}- d_{\overline{G^j}}(x,\cup_{i\in[k-1]\setminus \{s\}}A_i^j) + d_{G^{j+1}}(x,A_k^{j+1})\nonumber\\
	&\leq& a_s^{j} + d_{G}(x,A_k) \leq a_s^j + \Delta,
	\end{eqnarray}
	where we used Lemma~\ref{Ziedges}(iv) to replace $d_{G^\ell_1}(x,A_k)$ by $d_G(x,A_k)$.
	Let $z \in T(y_{j+1})$.
	Let $t \in [k-1]\setminus \lbrace s \rbrace$ be such that $z \in Z_k^t$.
	Then, since $d_{G^{j+1}}(y_{j+1},A^j_s) = 0$ and each of $y_{j+1},z$ has at most $\delta n$ neighbours in $A_k$ and $|A_t^{j+1}| = |A^j_t|\geq |A_t|$,
	\begin{eqnarray*}
		P_3(zy_{j+1},G^{j+1}) &\leq& \sum_{p \in [k-1]\setminus \lbrace s,t\rbrace}|A^{j}_p| + d_{G^{j+1}}(z,A^j_t) + d_{G^{j+1}}(z,A^{j+1}_k)\\
		&\stackrel{\Pbadedges(G^{j+1}),\Pmissing(G^{j+1})}{\leq}& a^{j}_s - \xi n/2 + j+1+ \delta n\ \leq\ a_s^j - \xi n/2+2\delta n.
	\end{eqnarray*}
	Let now $xy \in Q(y_{j+1})$ where $x \in R_s$ and $y \in R_k$. As $Q_s^0\supseteq Q(y_{j+1})$, Proposition~\ref{Qi} implies that $	P_3(xy,G_1^{\ell}) \geq a_s -2\sqrt{\eta}n$. Then by $K(1)$,
	\begin{align*}
	P_3(xy,G^j) \geq P_3(xy,G_1^{\ell})\geq a^j_s -|Y|-2\sqrt{\eta}n \geq a_s^{j} - 2\delta n.
	\end{align*}
Before we upper bound $K_3(G^{j+1})-K_3(G^j)$, we need some preliminary estimates.
Let $a,b,p$ be non-negative integers such that $b \leq a$ and $p \leq 2\gamma n$. We claim that
\begin{equation}\label{abp}
\left(\frac{\xi a}{6\gamma} - b\right)p  \leq \frac{\xi n}{3}(a-b).
\end{equation}
Indeed, if $\frac{\xi a}{6\gamma} - b<0$, then it trivially holds as $a\ge b$. Otherwise $\left(\frac{\xi a}{6\gamma} - b\right)p \le \left(\frac{\xi a}{6\gamma} - b\right)2\gamma n \leq \frac{\xi n}{3}(a-b)$ as desired.

Observe that $|U^j\setminus U^{j,s}|,d_{G^j}(y_{j+1},U^j\setminus U^{j,s}),d_{G^j}(y_{j+1},A^j_s)$ satisfy the conditions on $a,b,p$ respectively.
Indeed, by Lemma~\ref{Ziedges}(iv), $K(2)$ and the definition of $Y$, we have that $$
d_{G^{j}}(y_{j+1},A_s^{j}) \leq d_{G_1^\ell}(y_{j+1},A_s) +|Y|\stackrel{\Pbadedges(G)}{\le} 2\gamma n.
$$ 
Now,
	\begin{eqnarray*}
	K_3(G^{j+1}) - K_3(G^j) &\leq& \sum_{y \in N_{G^j}(y_{j+1},A_s^j)} \left( \Delta - (|U^j\setminus U^{j,s}|-|T(y_{j+1})|) - P_3(yy_{j+1},G^j;R_k) \right)\\
	&&\hspace{2cm}- |T(y_{j+1})|\cdot \xi n/3\\
	&=& d_{G^j}(y_{j+1},A_s^j)\left(\Delta- d_{G^j}(y_{j+1},U^j\setminus U^{j,s})\right) -\sum_{y \in N_{G^j}(y_{j+1},A_s^j)} P_3(yy_{j+1},G^j;R_k)\\
	&&\hspace{2cm}- \left(|U^j\setminus U^{j,s}|-d_{G^j}(y_{j+1},U^j\setminus U^{j,s})\right) \frac{\xi n}{3}\\
	&\stackrel{(\ref{abp})}{\leq}& d_{G^j}(y_{j+1},A_s^j)\left( \Delta - \frac{\xi}{6\gamma}|U^j\setminus U^{j,s}|\right) -\sum_{y \in N_{G^j}(y_{j+1},A_s^j)} P_3(yy_{j+1},G^j;R_k)\\
	&=& \sum_{y \in N_{G^{j}}(y_{j+1},A_s^{j})}\left(\Delta- \frac{\xi}{6\gamma}|U^{j}\setminus U^{j,s}| -P_3(yy_{j+1},G^{j};R_k)\right).
	\end{eqnarray*}
Observe that equality above holds only when equality in~\eqref{eq-st2} holds. This happens only if $d_{\overline{G^j}}(x,\cup_{i\in[k-1]\setminus \{s\}}A_i^j)=0$ for every $x\in R(y_{j+1})$, in other words, $G^{j}[R(y_{j+1}), \cup_{i\in[k-1]\setminus\{s\}}A_i^j]$ is complete. Recall that $R(y_{j+1})=N_{G^{j+1}\setminus G^{j}}(y_{j+1}, R_k)$. So $K(3,j+1)$ holds.
\end{proof}

We can now derive some properties of the graph $G_2 := G_2^{k-1}$ obtained in Lemma~\ref{subYiedges}, namely that its only bad edges have both endpoints in $X$, and $G_2$ does not have many more triangles than $G_1$.
See the right-hand side of Figure~\ref{partition2} for an illustration of $G_2$ in the case $k=3$.
For all $i \in [k-1]$, we will let $A_i':=A_i\cup Y_i$ and
\begin{equation}\label{eq:ai}
a_i' := \sum_{j \in [k-1]\setminus\lbrace i \rbrace}|A_j'| = n - |A_i'|-|A_k'|.
\end{equation}

\begin{lemma}\label{Yiedges}
	There exists an $(n,e)$-graph $G_2$ on the same vertex set as $G_1 := G_1^{k-1}$ such that
	\begin{itemize}
		\item[(i)] $G_2$ has an $(A_1',\ldots,A_k';Z,2\beta,\xi/3,2\xi,\delta)$-partition with missing vector $\underline{m}^{(2)} = (m^{(2)}_1,\ldots,m^{(2)}_{k-1})$ 
		where $A_i' := A_i \cup Y_i$ for $i \in [k-1]$ and $A_k' := A_k \setminus Y = R_k \cup X$; also $\alpha m^{(1)}_i \leq m^{(2)}_i \leq 2m^{(1)}_i$ for all $i \in [k-1]$.
		\item[(ii)] If there are $i \in [k]$ and $xy \in E(G_2[A_i'])$, then $i=k$; furthermore $x,y \in X$ and $xy \in E(G[A_k'])$. 
		\item[(iii)] For every $i \in [k-1]$ and every $z \in X_i$, we have that $d_{G_2}(z,A_i') \geq \gamma n$.
		\item[(iv)] $K_3(G_2) \leq K_3(G_1) + \delta^{1/4}m^2/(3n)$. 
	\end{itemize}
\end{lemma}

\begin{proof}
	Let $q := |Y|$ and apply Lemma~\ref{subYiedges} to obtain $G_2 := G^q=G_2^{k-1}$ satisfying $K(1,q)$--$K(3,q)$.
	Write $\underline{m}^{(1)} = (m^{(1)}_1,\ldots,m^{(1)}_{k-1})$.
	For $g\in[3]$, let $K(g)$ be the conjunction of the properties $K(g,1)$--$K(g,q)$.
	Observe that $A_i' = A^q_i$ for all $i \in [k]$.
	Now $|Y| \leq |Z| \leq \delta n$, and so $q/n \leq \delta$.
	Thus, by $K(1,q)$, $G_2$ has an $(A_1',\ldots,A_k';Z,\beta+\delta,\xi/2-\delta,\xi+3\delta,\delta)$-partition and hence an $(A_1',\ldots,A_k';Z,2\beta,\xi/3,2\xi,\delta)$-partition.
	
	Now, by $K(1)$,
	\begin{align*}
	m^{(2)}_i &= e(\overline{G_2}[A_i',A_k']) = e(\overline{G^q}[A_i\cup Y_i,A_k\setminus Y]) = e(\overline{G^q}[A_i,A_k\setminus Y] + \sum_{y \in Y_i}d_{\overline{G^q}}(y,A_k\setminus Y)\\
	&= e(\overline{G_1}[A_i,A_k\setminus Y]) + \sum_{\stackrel{j \in [q]}{s(j)=i}}|Q(y_j)| + \sum_{y \in Y_i}d_{\overline{G^q}}(y,A_k\setminus Y)\\
	&= m^{(1)}_i - \sum_{y \in Y_i}\left(d_{\overline{G_1}}(y,A_i)-d_{\overline{G^q}}(y,A_k\setminus Y)\right) + \sum_{\stackrel{j \in [q]}{s(j)=i}}|Q(y_j)|.
	\end{align*}
	Note further, using $|Y|\le|Z|\le \delta n$ by~\Pbadedges($G$), that
	\begin{align*}
	&\sum_{y \in Y_i}\left(d_{\overline{G_1}}(y,A_i)-d_{\overline{G^q}}(y,A_k\setminus Y)\right) = \sum_{y \in Y_i}\left(d^m_{G_1}(y)-d^m_{G^q}(y)\right) \leq \sum_{\stackrel{j \in [q]}{s(j)=i}}\left(|A_i| -(d^m_{G^j}(y_j)-|Y|)\right)\\
	&\stackrel{(\ref{eq-ym})}{\leq} \sum_{\stackrel{j \in [q]}{s(j)=i}}(|A_i|-(1-(k-1)c)n+3\gamma n) \stackrel{\Ppartition(G)}{\leq} |Y_i|(kc-1+4\gamma)n \stackrel{(\ref{ineq:c})}{\leq} (c-\alpha)|Y_i|n.
	\end{align*}
	A similar calculation shows that the left-hand side is positive.
	Thus using $K(1)$ for the bound $|Q(y_j)| \leq \delta n$, we have
	$m^{(1)}_i - (c-\alpha)|Y_i|n \leq m^{(2)}_i \leq m^{(1)}_i + \delta n|Y_i|$.
	But the definition of $Y_i$ and Lemma~\ref{Ziedges}(iv) imply that
	$$
	m^{(1)}_i \geq |Y_i| \cdot \min_{y \in Y_i}d_{\overline{G_1}}(y,A_i) = |Y_i| \cdot \min_{y \in Y_i}d_{\overline{G}}(y,A_i) \stackrel{\Ppartition(G)}{\geq} |Y_i| \cdot (c-\beta-\gamma)n \geq |Y_i|\cdot (c-2\gamma)n.
	$$
	Thus, using the fact that $c \leq \frac{1}{k-1} \leq \frac{1}{2}$ from~(\ref{eq:solc2}),
	$$
	\alpha \leq 1-\frac{c-\alpha}{c-2\gamma} \leq \frac{m^{(2)}_i}{m^{(1)}_i} \leq 1+\frac{\delta}{c-2\gamma} \leq 2.
	$$
	This completes the proof of (i).
	
	For (ii), the first part follows from $E(G_1[A_i]) = \emptyset$ due to Lemma~\ref{Ziedges}(ii) and $K(1)$. For the second part,
	suppose $xy \in E(G_2[A'_k])$.
	Now, $Y \cap A'_k = \emptyset$ 
	and $E(G_2[A'_k]) \subseteq E(G_1[A_k])$ so every edge in $E(G_2[A'_k])$ is incident to a vertex of $X$.
	So $x \in X$, say.
	Suppose that $y \notin X$.
	Then $y \in A_k'\setminus X \subseteq R_k$.
	So $xy$ is an edge of $G_1$ and hence of $G$ by Lemma~\ref{Ziedges}(iv).
	This is a contradiction to Proposition~\ref{Gprops}(i).
	This completes the proof of (ii).
	For (iii), note that for any $i\in[k-1]$ and any $z \in X_i$, $G_1$ and $G_2$ are identical in $[z,A_i]$. Thus, by Lemma~\ref{Ziedges}(iv) and the definition of $X$, we have that $d_{G_2}(z,A_i') \geq d_{G_2}(z,A_i)= d_{G}(z,A_i) \ge \gamma n$, as required.
	
	Finally, for (iv),
	\begin{eqnarray*}
	K_3(G_2) - K_3(G_1) &=& \sum_{j \in [q]}\left(K_3(G^j)-K_3(G^{j-1})\right) \ \stackrel{K(3,j)}{\leq}\ \sum_{j \in [q]}d_{G^{j-1}}(y_j,A^{j-1}_{s(j)}) \cdot \Delta\\
	\nonumber & \stackrel{K(1)}{\leq}& \Delta \cdot\sum_{j \in [q]}(d_{G_1}(y_j,A_{s(j)})+|Y|) \ \leq\ \Delta|Z|(\gamma n+|Z|)\\
	& \stackrel{(\ref{Zsize}),(\ref{Rk'})}{\leq}& \frac{\delta^{1/3}m}{n} \cdot \frac{2m}{\xi n} \cdot 2\gamma n\ \leq\ \frac{\delta^{1/4}m^2}{3n},
	\end{eqnarray*}
	as required.
\end{proof}

\subsection{Transformation 3: removing bad $X_i$-$X_i$ edges}\label{sec7.4}

We have obtained a graph $G_2$ from $G$ which has the property that every bad edge has both endpoints in $X$.
In the third transformation, we remove those bad edges whose endpoints both lie in $X_i$ for some $i \in [k-1]$.
The proof is very similar to the proofs of Lemmas~\ref{subZiedges} and~\ref{Ziedges}.

\begin{center}
\begin{figure}
\includegraphics[scale=0.9]{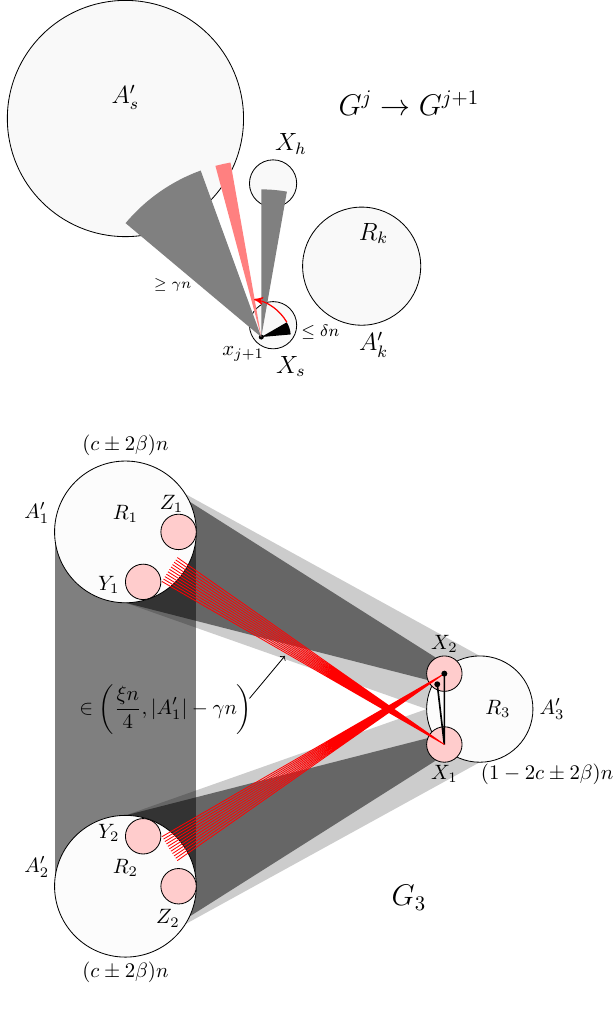}
\caption{Transformation 3: $G_2 \rightarrow G_3$ (here $k=3$). Left: A single step $G^j \rightarrow G^{j+1}$ as in Lemma~\ref{subXiedges}, in which the black edges are replaced by the pink edges. Right: The final graph $G_3$ obtained in Lemma~\ref{Xiedges}, in which $X_1$ and $X_2$ are now independent sets.}
\label{partition3}
\end{figure}
\end{center}

For all $i \in [k-1]$ and $x,y \in X_i$, let
$$
D(x) := d_{G_2}(x,X\setminus X_i) \text{ and } D(x,y) := |N_{G_2}(x,X\setminus X_i) \cap N_{G_2}(y,X\setminus X_i)|
$$
So $D(x)-D(x,y) \geq 0$ with equality if and only if the $G_2$-neighbourhood of $x$ in $X\setminus X_i$ is a subset of $y$'s.

\begin{lemma}\label{subXiedges}
	Let $G_2$ be any graph satisfying the conditions of Lemma~\ref{Yiedges}.
	Let $f := |X|$ and let $x_1,\ldots,x_f$ be any ordering of $X$.
	For each $r \in [f]$, let $s(r)$ be such that $x_r \in X_{s(r)}$.
	Then there exists a sequence $G_2 =: G^0,G^1,\ldots,G^f =: G_3$ of graphs such that for all $j \in [f]$,
	\begin{enumerate}[label={L(\arabic*,$j$):}]
		\item\label{W1'} $G^j$ is an $(n,e)$-graph and has an $(A_1',\ldots,A_k';Z,2\beta,\xi/4,2\xi,\delta)$-partition.
		\item\label{W2'} $E(G^j)\setminus E(G^{j-1}) = \lbrace x_{j}x : x \in R(x_{j}) \rbrace$, where $R(x_{j}) \subseteq R_{s(j)}$, and $E(G^{j-1})\setminus E(G^j)$ is the set of $x_{j'}x_j \in E(G_2)$ with $s(j')=s(j)$ and $j' > j$. 
		Thus $d_{G^j}(x_t,X_{s(t)}) = 0$ for all $t \in [j]$; $e(\overline{G^j}[A_i',A_k'])=e(\overline{G^{j-1}}[A_i',A_k'])$ for all $i \neq s(j)$, and $e(\overline{G^j}[A_{s(j)}',A_k']) = e(\overline{G^{j-1}}[A_{s(j)}',A_k'])-d_{G^{j-1}}(x_j,X_{s(j)})$.
		\item\label{W3'} $K_3(G^{j})-K_3(G^{j-1}) \leq \sum_{y \in N_{G_2}(x_{j},X_{s(j)}\setminus \lbrace x_1,\ldots,x_{j-1} \rbrace)}(D(x_{j})-D(y,x_{j}))$ with equality only if  $K_3(x_j,G^{j-1};X_{s(j)})=0$ and $N_{G^{j-1}}(y,A_{s(j)}')\cap N_{G^{j-1}}(x_j,A_{s(j)}')=\emptyset$ for all $y \in N_{G^{j-1}}(x_j,X_{s(j)})$.
	\end{enumerate}
\end{lemma}

\begin{proof}
	Let $G^0 := G_2$. Suppose we have obtained $G^0,\ldots,G^{j}$ for some $j < f$ such that, for all $r \in [j]$,
	$L(1,r)$--$L(3,r)$ hold.
	Note that $G^0$ has an $(A_1',\ldots,A_k';Z,2\beta,\xi/3,2\xi,\delta)$-partition and hence an $(A_1',\ldots,A_k';Z,2\beta,\xi/4,2\xi,\delta)$-partition.
	For $g \in [3]$, let $L(g)$ denote the conjunction $L(g,1),\ldots,L(g,j)$ of properties.
	We obtain $G^{j+1}$ as follows.
	Let $s := s(j+1)$.
	Choose $R(x_{j+1}) \subseteq R_s\setminus N_{G^j}(x_{j+1}) \subseteq A_s'$ such that $|R(x_{j+1})| = d_{G^j}(x_{j+1},X_s)$.
	Let us first see why this is possible.
	One consequence of $L(2)$ is that the neighbourhood of $x_{j+1}$ in $G^j$ can be obtained from its neighbourhood in $G^0=G_2$ by removing its $G_2$-neighbours among $\lbrace x_r : r \leq j \text{ and } s(r) = s \rbrace$.  
	Thus
	\begin{align*}
	d_{\overline{G^{j}}}(x_{j+1},R_s) &\stackrel{L(2)}{=} d_{\overline{G_2}}(x_{j+1},R_s) \geq d_{\overline{G_2}}(x_{j+1},A_s')-|Z\cap A_s'| \stackrel{\Pmissing(G_2)}{\geq} \xi n/3 - \delta n \geq \delta n\\
	&\stackrel{\Pmissing(G_2)}{\geq} |Z|\ge d_{G^{j}}(x_{j+1},X_s).
	\end{align*}
	So $R(x_{j+1})$ exists.
	Now define $G^{j+1}$ by setting $V(G^{j+1}) := V(G^{j})$ and
	$$
	E(G^{j+1}) := \left( E(G^{j}) \cup \lbrace x_{j+1}x : x \in R(x_{j+1}) \rbrace \right) \setminus E(G^{j}[x_{j+1},X_s]).
	$$
	Thus $G^{j+1}$ is obtained by replacing all bad edges of $G^j$ between $x_{j+1}$ and another vertex in $X_s$ by the same number of missing edges of $G^j$ which are between $x_{j+1}$ and $R_s$.
See the left-hand side of Figure~\ref{partition3} for an illustration of the transformation $G^j \rightarrow G^{j+1}$.
	
	We will now show that $G^{j+1}$ satisfies $L(1,j+1),\ldots,L(3,j+1)$, beginning with $L(1,j+1)$.
	By construction, $G^{j+1}$ is an $(n,e)$-graph.
	To show that $G^{j+1}$ has an $(A_1',\ldots,A_k';Z,2\beta,\xi/4,2\xi,\delta)$-partition, we need to show that \Ppartition($G^{j+1}$)--\Pmissing($G^{j+1}$) hold with the appropriate parameters.
	All properties except \Pmissing($G^{j+1}$) are immediate.
	For \Pmissing, let $i \in [k]$ and let $y \in A_i'$ be arbitrary.
	Let $d^m_{G^j},d^m_{G^{j+1}}$ denote the missing degree in $G^j,G^{j+1}$ with respect to the partition $(A_1',\ldots,A_k')$.
	We have that
	\begin{equation}\label{misseq4}
	d^m_{G^{j+1}}(y) = \begin{cases} d^m_{G^{j}}(y)-1, &\mbox{if } y \in R(x_{j+1}), \\ 
	d^m_{G^{j}}(y)-d_{G^j}(x_{j+1},X_s), &\mbox{if } y=x_{j+1}, \\ 
	d^m_{G^{j}}(y), & \mbox{otherwise. } \end{cases}
	\end{equation}
	Thus if $y \in A_i' \setminus Z$, we have $d^m_{G^{j+1}}(y) \leq d^m_{G^{j}}(y) \leq 2\xi n$ since $G^j$ has an $(A_1',\ldots,A_k';Z,2\beta,\xi/4,2\xi,\delta)$-partition.
	It remains to consider the case $y=x_{j+1}$ (since missing degree is unchanged for all other vertices in $Z$).
	By the consequence of $L(2)$ stated above,
	\begin{equation}\label{misseq32}
	d^m_{G^{j}}(x_{j+1}) = d^m_{G_2}(x_{j+1}) \text{ and } d_{G^j}(x_{j+1},X_s) \le |Z|\le\delta n.
	\end{equation}
	Thus
	$$
	d^m_{G^{j+1}}(x_{j+1}) \stackrel{\Pmissing(G_2)}{\geq} \xi n/3  -\delta n \geq \xi n/4.
	$$
	Thus \Pmissing($G^{j+1}$) holds.
	We have shown that $L(1,j+1)$ holds.
	That $L(2,j+1)$ holds is clear from $L(2)$, the construction of $G^{j+1}$ and ~\eqref{misseq4}.
	
	For $L(3,j+1)$, observe that a triangle is in $G^{j+1}$ but not $G^j$ if and only if it contains an edge $xx_{j+1}$ where $x \in R(x_{j+1})$; and a triangle is in $G^j$ but not $G^{j+1}$ if and only if it contains an edge $yx_{j+1}$, where $y \in N_{G^j}(x_{j+1},X_s)$.
	Observe also that there is no triangle in $G^{j+1}$ which contains more than one vertex in $R(x_{j+1})$.
	Thus
	$$
	K_3(G^{j+1})=K_3(G^{j})+ \sum_{x \in R(x_{j+1})}P_3(xx_{j+1},G^{j+1}) - \sum_{y \in N_{G^{j}}(x_{j+1},X_s)}P_3(yx_{j+1},G^{j};\overline{X_s}) - K_3(x_{j+1},G^j;X_s).
	$$
	We will estimate each summand in turn.
	Fix $y \in N_{G^{j}}(x_{j+1},X_s)$.
	By $L(1,j)$, \Pcomplete($G^{j}$) holds and, since $y,x_{j+1} \in X_s$, both of these vertices are incident to all of $A_t'$ for $t \in [k-1]\setminus \lbrace s \rbrace$.
	So
	\begin{align}
	P_3(yx_{j+1},G^{j};\overline{X_s}) &=  a_{s}' + |N_{G^j}(y,X\setminus X_s) \cap N_{G^j}(x_{j+1},X\setminus X_s)| + |N_{G^{j}}(y,A_s')\cap N_{G^{j}}(x_{j+1},A_s')|\nonumber\\
	&= a_s' + D(y,x_{j+1})+|N_{G^{j}}(y,A_s')\cap N_{G^{j}}(x_{j+1},A_s')|,\label{P3xyj+1}
	\end{align}
	where the last equality uses the fact that $G^j$ and $G_2$ are identical at $[X_s,X\setminus X_s]$ for any $s\in[k-1]$ due to $L(2)$.
	Now fix $x \in R(x_{j+1})$.
	Then $d_{G^{j+1}}(x,A_s') = d_{G_2}(x,A_s') = 0$, also $d_{G^{j+1}}(x_{j+1},X_s) = 0$. By \PZk($G^{j+1}$), $x$ is incident to every vertex in $X_t$ for $t \neq s$. Recall that $d_{G^{j+1}}(x_{j+1},R_k)=0$. Indeed, $E(G[X,R_k])=\emptyset$ due to Proposition~\ref{Gprops}(i), and it remains empty during the transformations $G\rightarrow G_1\rightarrow G_2\rightarrow G^{q}$ for any $q\in[f]$. Thus
	$$
	P_3(xx_{j+1},G^{j+1}) = a_{s}' + P_3(xx_{j+1},G^{j+1};A_k') = a_s' + \sum_{t \in [k-1]\setminus \lbrace s \rbrace}d_{G^{j+1}}(x_{j+1},X_t) \stackrel{L(2)}{=} a_s' + D(x_{j+1}).
	$$
	Therefore
	\begin{eqnarray*}
	&&K_3(G^{j+1})-K_3(G^{j})\\
	&=& \sum_{x \in R(x_{j+1})}P_3(xx_{j+1},G^{j+1}) - \sum_{y \in N_{G^j}(x_{j+1},X_s)} P_3(yx_{j+1},G^j;\overline{X_s}) - K_3(x_{j+1},G^j;X_s)\\
	&\leq& \sum_{y \in N_{G^j}(x_{j+1},X_s)}(D(x_{j+1})-D(y,x_{j+1})-|N_{G^{j}}(y,A_s')\cap N_{G^{j}}(x_{j+1},A_s')|) - K_3(x_{j+1},G^j;X_s)\\
	&=& \sum_{y \in N_{G_2}(x_{j+1},X_s\setminus \lbrace x_1,\ldots,x_j \rbrace)}(D(x_{j+1})-D(y,x_{j+1})-|N_{G^{j}}(y,A_s')\cap N_{G^{j}}(x_{j+1},A_s')|) - K_3(x_{j+1},G^j;X_s),
	\end{eqnarray*}
	proving $L(3,j+1)$.
\end{proof}

Again we are now able to derive some properties of $G_3 := G^f$ obtained in Lemma~\ref{subXiedges}, namely that every bad edge lies between $X_i$ and $X_j$ for some distinct $i,j$; and $G_3$ does not have many more triangles than $G_2$.
The right-hand side of Figure~\ref{partition3} shows $G_3$ in the case when $k=3$.

\begin{lemma}\label{Xiedges}
	There exists an $(n,e)$-graph $G_3$ on the same vertex set as $G_2$ such that
	\begin{itemize}
		\item[(i)] $G_3$ has an $(A_1',\ldots,A_k';Z,2\beta,\xi/4,2\xi,\delta)$-partition with missing vector $\underline{m}^3 := (m^{(3)}_1,\ldots,m^{(3)}_{k-1})$ and $m^{(2)}_i/2 \leq m^{(3)}_i \leq m^{(2)}_i$, where $m^{(3)}_i = m^{(2)}_i$ if and only if $E(G_2[X_i])=\emptyset$.
		\item[(ii)] If there is $i \in [k]$ and $xy \in E(G_3[A_i'])$, then 
		$i=k$ and there exists $\ell \ell' \in \binom{[k-1]}{2}$ such that $x \in X_\ell$ and $y \in X_{\ell'}$. Moreover, for all $st \in \binom{[k-1]}{2}$, we have $E(G_3[X_s,X_{t}]) = E(G_2[X_s,X_{t}])$ and $d_{G_3}(x,A_i') \geq \gamma n$ for all $i \in [k-1]$ and $x \in X_i$.
		\item[(iii)] $K_3(G_3) - K_3(G_2) \leq |Z|^2 \cdot \max_{\stackrel{i \in [k-1]}{x,y \in X_i}} (D(x)-D(x,y))$ with equality only if for all $i\in[k-1]$ we have that $G_2[X_i]$ is triangle-free and $N_{G_2}(x,A_i') \cap N_{G_2}(y,A_i')=\emptyset$ for all $xy \in E(G_2[X_i])$. In particular $K_3(G_3) - K_3(G_2) \leq 
		\sqrt{\delta}m^2/n$.
	\end{itemize}
\end{lemma}

\begin{proof}
	Let $f := |X|$ and apply Lemma~\ref{subXiedges} to $G_2$ to obtain $G_3 := G^f$ satisfying $L(1,f)$--$L(3,f)$. For $g\in[3]$ let $L(g)$ denote the conjunction of properties $L(g,1)$--$L(g,f)$.
	By $L(1,f)$, $G_3$ has an $(A_1',\ldots,A_k';Z,2\beta,\xi/4,2\xi,\delta)$-partition.
	Also, for all $i \in [k-1]$,
	$$
	\sum_{\stackrel{j \in [f]}{s(j)=i}}d_{G^{j-1}}(x_j,X_i) = \sum_{\stackrel{j \in [f]}{s(j)=i}}d_{G_2}(X_i\setminus \lbrace x_1,\ldots,x_{j-1} \rbrace) = e(G_2[X_i]).
	$$
	Thus
	\begin{eqnarray*}
		m^{(3)}_i &=& e(\overline{G^f}[A_i',A_k']) \stackrel{L(2,f)}{=} e(\overline{G_2}[A_i',A_k']) - \sum_{\stackrel{j \in [f]}{s(j)=i}}d_{G^{j-1}}(x_j,X_i) = m^{(2)}_i - e(G_2[X_i])\\
		&\stackrel{\Pbadedges(G_2)}{\geq}& m^{(2)}_i - |X_i|\cdot \delta n \geq m_i^{(2)} - |X_i| \cdot \frac{\xi n}{6} \stackrel{\Pmissing(G_2)}{\geq} \frac{m_i^{(2)}}{2},
	\end{eqnarray*}
	and also $m^{(3)}_i \leq m^{(2)}_i$ with equality holds if and only if $E(G_2[X_i])=\emptyset$.
	This proves (i).

	We now turn to (ii).
	By $L(2)$ and Lemma~\ref{Yiedges}(ii), $E(G_3[A_t']) = E(G_2[A_t']) = \emptyset$ if $t \neq k$.
	Furthermore, $E(G_3[A_k']) \subseteq E(G_2[A_k'])$.
	So if $G_3$ has a bad edge $xy$, both of its endpoints lie in $X$.
	But, for all $r \in [f]$, we have $d_{G^{j}}(x_r,X_{s(r)}) = 0$ for all $j \geq r$.
	So $E(G_3[X_i]) = \emptyset$ for all $i \in [k-1]$. Note that for any $x\in X_i$ with $i\in[k-1]$, after the transformations $G\rightarrow G_1\rightarrow G_2\rightarrow G_3$, we have $N_{G_3}(x,A_i)\supseteq N_G(x,A_i)$. Hence, by the definition of $X$, 
	\begin{equation}\label{eq-xA}
	d_{G_3}(x,A_i')\ge d_{G_3}(x,A_i)\ge d_G(x,A_i)\ge\gamma n.
	\end{equation} 
	This proves (ii).

It remains to establish~(iii). We have that
	\begin{align*}
	K_3(G_3) - K_3(G_2) &= \sum_{j \in [f]}\left(K_3(G^j)-K_3(G^{j-1})\right) \leq \sum_{j \in [f]}\sum_{y \in N_{G_2}(x_{j},X_{s(j)}\setminus \lbrace x_1,\ldots,x_{j-1} \rbrace)}(D(x_{j})-D(y,x_{j}))\\
	&\leq |Z|^2 \cdot \max_{\stackrel{i \in [k-1]}{x,y \in X_i}} (D(x)-D(x,y)) \stackrel{\Pbadedges(G_2)}{\leq} |Z|^2 \cdot \delta n \stackrel{(\ref{Zsize})}{\leq} \frac{4\delta  m^2}{\xi^2 n} \leq \frac{\sqrt{\delta}m^2}{n}.
	\end{align*}
    This together with $L(3)$ implies the inequality in~(iii).
    Further, we have equality only if $K_3(x_j,G^{j-1};X_{s(j)})=0$ for all $j \in [f]$, and $|N_{G^{j-1}}(y,A_s')\cap N_{G^{j-1}}(x_{j},A_s')|$ for all $y \in N_{G^{j-1}}(x_j)$, where $s(j)$ is such that $x_j \in A'_{s(j)}$.
    This occurs if and only if $G_2[X_i]$ is triangle-free for all $i \in [k-1]$, and $N_{G_2}(x,A_i')\cap N_{G_2}(y,A_i')=\emptyset$, as required.
\end{proof}

\subsection{Transformation 4: symmetrising $X_i$-$A_i'$ edges}

Lemma~\ref{Xiedges}(ii) implies that $D(x) = \sum_{t \in [k-1]\setminus \lbrace i \rbrace}d_{G_3}(x,X_t)$ for every $x\in X_i$, $i\in[k-1]$.
Next we obtain an $(n,e)$-graph $G_4$ with the property that, for all $i \in [k-1]$ and all but at most one vertex $x \in X_i$, either $G_4[x,A_i']$ is empty or it is almost complete (see the left-hand side of Figure~\ref{partition4}).

\begin{center}
\begin{figure}
\includegraphics[scale=0.9]{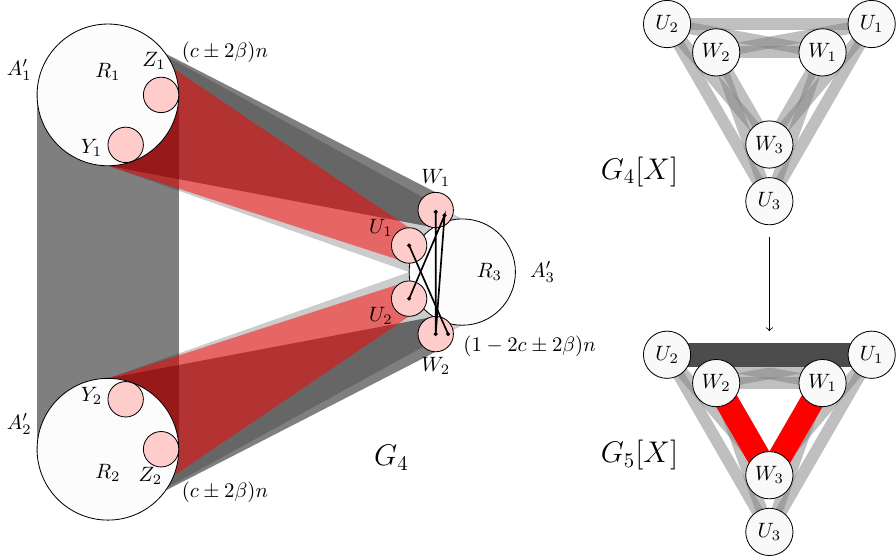}
\caption{Transformations 4 and 5. Dark grey and red represent (almost) complete/empty bipartite pairs respectively. Left: $G_4$ (here $k=3$). The only bad edges lie in $[U_i,U_j \cup W_j]$ for some $ij \in \binom{[k-1]}{2}$. Right: $G_4 \rightarrow G_5$ in the case $k=4$ and $I_1 = \lbrace 12\rbrace$ and $I_2 = \lbrace 13,23\rbrace$.}
\label{partition4}
\end{figure}
\end{center}

\begin{lemma}\label{symmetrise}
	There exists an $(n,e)$-graph $G_4$ on the same vertex set as $G_3$ such that
	\begin{itemize}
		\item[(i)] $G_4$ has an $(A_1',\ldots,A_k';Z,2\beta,\xi/5,3\xi,\delta)$-partition; also $G_3$ and $G_4$ can differ only at the union of $[X_i,A_i']$ for $i\in[k-1]$.
		\item[(ii)] For every $i \in [k-1]$, there exists a partition $X_i = U_i \cup W_i$ (into parts which may be empty) such that
		$d_{G_4}(w,A_i') = |A_i'|-\xi n/5$ for all but at most one $w \in W_i$ which has at least $\xi n/5$ non-neighbours in $A_i'$, and $e(G_4[U_i,A_i']) = 0$. Further, for all $i \in [k-1]$, if $U_i \neq \emptyset$, then $W_i \neq \emptyset$.
		\item[(iii)] If there is $i \in [k]$ and $xy \in E(G_4[A_i'])$, then $i=k$ and there exists $st \in \binom{[k-1]}{2}$ such that $x \in X_s$ and $y \in X_t$, and further, $xy \in E(G_3[A_k'])$.
		\item[(iv)] $K_3(G_4) \leq K_3(G_3)$; and if there exists $i \in [k-1]$ and $x,y \in X_i$ such that $D(x) \neq D(y)$, then $K_3(G_4) \leq K_3(G_3)-\xi n/20$. 
		\item[(v)] Let $\underline{m}^4 = (m^{(4)}_1,\ldots,m^{(4)}_{k-1})$ be the missing vector of $G_4$ with respect to $(A_1',\ldots,A_k')$. Then $m^{(4)}_i = m^{(3)}_i$ and $|U_i|\,|A_i'| \leq m_i^{(4)}$ for all $i \in [k-1]$.
	\end{itemize}
\end{lemma}

\begin{proof}
	Roughly speaking, we will obtain $G_4$ from $G_3$ by, for each $i \in [k-1]$, moving all $X_i$-$A_i$ edges to be incident to vertices $x \in X_i$ such that $D(x)$ is minimal. 
	Let $G^{1,0} := G_3$.
	For each $i \in [k-1]$, let $f_i := |X_i|$.

	Set $i=1$ and perform the following procedure.
	\begin{enumerate}
		\item If $X_i=\emptyset$, then let $t_i:=0$ and go to Step (6). Otherwise, let $x^i_1,\ldots,x^i_{f_i}$ be an ordering of $X_i$ such that $D(x^i_1) \leq \ldots \leq D(x^i_{f_i})$.
		Suppose we have constructed $G^{i,0},\ldots,G^{i,j}$ for some $j \geq 0$.
		\item Let $i^+ = i^+(j)$ be the largest $t \in [f_i]$ such that $d_{G^{i,j}}(x^i_t,A_i') > 0$.
		Let $i^- = i^-(j)$ be the smallest $s \in [f_i]$ such that $d_{\overline{G^{i,j}}}(x^i_s,A_i') > \xi n/5$.
		\item If $i^+\le i^-$, then  set $t_i := j$ and go to Step (6).
		\item Choose $x \in N_{G^{i,j}}(x^i_{i^+},A_i')$ and $y \in N_{\overline{G^{i,j}}}(x^i_{i^-},A_i')$.
		Let $G^{i,j+1}$ be the graph on vertex set $V(G^{i,j})$ with $$
		E(G^{i,j+1}) := E(G^{i,j}) \cup \lbrace x^i_{i^-},y \rbrace \setminus \lbrace x^i_{i^+},x \rbrace.
		$$
		\item Set $j := j+1$ and go to Step (2).
		\item If $i=k-1$, set $G_4 := G^{k-1,t_i}$ and STOP. Otherwise, set $G^{i+1,0} := G^{i,t_i}$, then set $i := i+1$ and go to Step (1).
	\end{enumerate}
	
	Observe that, by~\eqref{eq-xA} and \Pmissing($G_3$), for each $i \in [k-1]$ such that $X_i \neq \emptyset$ and for each $x\in X_i$,
	we have
	\begin{equation}\label{eq-xdegA}
	\gamma n \leq d_{G_3}(x,A_i') \leq |A_i'|-\xi n/4.
	\end{equation}
	Thus in $G^{i,0}$, we have $i^+(0)=f_i\ge 1=i^-(0)$. We need to show that the iteration terminates.
	Indeed, for each fixed $i \in [k-1]$, we have that $i^+-i^-$ is a non-increasing function of $j$, which is bounded above by $f_i$. Note further that $i^+-i^-$ remains constant for at most $n$ instances of Steps~(2)--(4) since $d_{G^{i,j}}(x^i_{i^+},A_i')$ strictly decreases.
	Thus we reach Step (6) in a finite number $t_i$ of steps for each $i\in[k-1]$.
	Thus we obtain the final graph $G_4$ in some finite number $t_1+\ldots+t_{k-1}$ steps, as required.
	
	Recall that $E(G_3[X,R_k])=E(G[X,R_k])=\emptyset$. Then for all $i \in [k-1]$, $0 \leq j \leq t_i$, $x \in X_i$ and $u \in A_i'$, we have that
	$$
	P_3(xu,G^{i,j}) = P_3(xu,G_3) = a_i' + d_{G_3}(x,X\setminus X_i) = a_i' + D(x).
	$$
	This follows from the fact that the only edges which change lie between $A_\ell'$ and $X_\ell$ for some $\ell \in [k-1]$, and no such edge forms a triangle with $xu$. 
	Together with the fact that Step~(4) happens only when $i^+>i^-$, we have
	\begin{align}\label{eq:ijcount}
	K_3(G^{i,j})-K_3(G^{i,j-1}) = P_3(x^i_{i^-}y,G^{i,j})-P_3(x^i_{i^+}x,G^{i,j-1}) = D(x^i_{i^-}) - D(x^i_{i^+}) \leq 0. 
	\end{align}
	
	We will now prove (i)--(v).
	Clearly \Ppartition($G_4$)--\PZk($G_4$) hold with the same parameters.
	For \Pmissing($G_4$), note that the missing degree of any $v \in V(G_4)\setminus Z$ changes by at most $|X|\le\delta n$, so \Pmissing($G_3$) implies that it is at most $3\xi n$, as required. 
	For $i \in [k-1]$ every $v \in A_i'\cap Z$ has gained at most $|X| \leq \delta n$ neighbours in $A_k'$, so, by \Pmissing($G_3$), the missing degree of $v$ in $G_4$ is at least $(\xi/4-\delta)n \geq \xi n/5$.
	For $v\in X_i\subseteq X= A_k'\cap Z$ for some $i\in[k-1]$, it follows from the construction that $d_{\overline{G_4}}(v,A_i')\ge\xi n/5$. The last assertion follows from the construction. This completes the proof of (i).
	
	We now prove (ii).
	If $X_i \neq \emptyset$, let
	\begin{equation}\label{UW}
	W_i =
	\lbrace x^i_1,\ldots,x^i_{i^+(t_i)} \rbrace \quad \text{and}\quad U_i := X_i \setminus W_i.
	\end{equation}
	Then (ii) holds by construction.
	Property (iii) also holds by construction.
	
	For (iv),
	let $\ell := \xi n/20$. Recall that for every $i\in[k-1]$ with $|X_i|\ge 2$, we have $i^+(0)=f_i\ge 2>1=i^-(0)$. Then~\eqref{eq-xdegA} and $\xi\ll\gamma$ imply that $t_i\ge \xi n/4-\xi n/5=\ell$ and for any $0\le j\le \ell-1$, we have $i^+(j)=f_i$ and $i^-(j)=1$. Then \eqref{eq:ijcount} implies that
	\begin{equation*}
	K_3(G_4) - K_3(G_3) = \sum_{\stackrel{i \in [k-1]}{X_i \neq \emptyset}}\sum_{j \in [t_i]}\left(K_3(G^{i,j})-K_3(G^{i,j-1}) \right)  \leq 0.
	\end{equation*}
	Furthermore, if there is $i \in [k-1]$ and $x,y \in X_i$ such that $D(x) \neq D(y)$. Then $D(x^i_1) \leq D(x^i_{f_i})-1$. Then the observation above shows that in fact
	$$K_3(G_4) - K_3(G_3) \le \sum_{0\le j \le \ell-1}\left(K_3(G^{i,j+1})-K_3(G^{i,j}) \right)\le \ell\cdot(D(x^i_1) - D(x^i_{f_i}))\le-\xi n/20.$$
	Finally, (v) is immediate by construction and the definition of $U_i$.
\end{proof}

\subsection{Transformation 5: replacing $[W_i,W_j]$-edges with $[U_i,U_j]$-edges}

The required partition of $G'$ is obtained by moving $U_i$ to $A_i'$ for each $i \in [k-1]$, and for \Pcomplete($G'$) to hold, we need that $G'[U_i,U_j]$ is complete.
Using the next transformation, we obtain $G_5$ from $G_4$ by replacing $[W_i,W_j]$-edges with $[U_i,U_j]$-edges.
Thus either we have the required property, or $G_5[W_i,W_j]$ is empty.
See the right-hand side of Figure~\ref{partition4} for an illustration.

\begin{lemma}\label{G5}
	There exists an $(n,e)$-graph $G_5$ on the same vertex set as $G_4$ such that
	\begin{itemize}
		\item[(i)] $G_5$ has an $(A_1',\ldots,A_k';Z,2\beta,\delta)$-partition.
		\item[(ii)] Every pair $e \in E(G_4)\bigtriangleup E(G_5)$ has endpoints $x_s \in X_s, x_t \in X_t$ for some $st \in \binom{[k-1]}{2}$.
		\item[(iii)] There is a partition $I_1 \cup I_2$ of $\binom{[k-1]}{2}$ such that for each $ij \in I_1$, we have $e(\overline{G_5}[U_i,U_j]) = 0$; and for each $ij \in I_2$ we have $e(G_5[W_i,W_j]) = 0$.
		\item[(iv)] $K_3(G_5) < K_3(G_4) + k^2\delta n+2|Z|^3$.
	\end{itemize}
\end{lemma}

\begin{proof}
	Obtain a graph $G_5$ from $G_4$ as follows.
	For all $ij \in \binom{[k-1]}{2}$, let
	$$
	f_{ij} := \min \lbrace e(G_4[W_i,W_j]),e(\overline{G_4}[U_i,U_j]) \rbrace.
	$$
	Let $F^W_{ij} \subseteq E(G_4[W_i,W_j])$ and $F^U_{ij} \subseteq E(\overline{G_4}[U_i,U_j])$ be such that $|F^W_{ij}| = |F^U_{ij}| = f_{ij}$.
	Let $V(G_5) := V(G_4)$ and
	$$
	E(G_5) := E(G_4) \cup \bigcup_{ij \in \binom{[k-1]}{2}}  F^U_{ij} \setminus \bigcup_{ij \in \binom{[k-1]}{2}}  F^W_{ij}.
	$$
	Clearly $G_5$ is an $(n,e)$-graph.
	Parts (i)--(iii) are also clear by construction (to define the partition in (iii), break ties arbitrarily).
	
	It remains to prove part (iv).
	For this, we need to calculate the $P_3$-counts for those adjacencies that were changed by passing from $G_4$ to $G_5$. Recall from Lemma~\ref{Xiedges}(ii) that for any $i\in[k-1]$, if $U_i\neq \emptyset$, then $W_i\neq\emptyset$. Note also that if $U_i=\emptyset$, then the adjacencies involving $X_i$ are the same in $G_4$ and $G_5$. Thus, for fixed $ij\in{[k-1]\choose 2}$, we may assume that $U_i,U_j\neq\emptyset$. Let $w_i\in W_i$ and $w_j\in W_j$ be arbitrary. Suppose that there exists a vertex $w_i' \in W_i$ with $d_{G_4}(w_i',A_i')\geq|A_i'|-\xi n/5$. Then, by \PZk($G_4$), $w_i,w_i'$ are incident to every vertex in $A_\ell'$ with $\ell \in [k-1]\setminus \lbrace i \rbrace$, and $w_j$ is incident to every vertex in $A_\ell'$ with $\ell \in [k-1]\setminus \lbrace j \rbrace$.
	So
	\begin{equation*}
	P_3(w_i'w_j,G_4) \geq a_j'-\xi n/5.
	\end{equation*}
	Also,
	\begin{equation}\label{wiwj}
	P_3(w_iw_j,G_4) \geq P_3(w_iw_j,G_4;\overline{A_k'})\ge a_i'-|A_j'|\stackrel{(\ref{eq:ai})}{=} a_j'-|A_i'|.
	\end{equation}
	Let $u_i \in U_i$ and $u_j \in U_j$.
	Then $d_{G_5}(u_i,A_i'),d_{G_5}(u_j,A_j') = 0$ (since this holds in $G_4$), so
	\begin{equation}\label{7.26}
	P_3(u_iu_j,G_5) \leq a_i' - |A_j'| + d_{G_4}(u_j,A_k') \stackrel{\Ppartition,\Pbadedges(G_4)}{\leq} a_i' - (c-2\beta-\delta)n \leq a_i' - cn/2.
	\end{equation}
	Similarly $P_3(u_iu_j,G_5)\le a_j' - cn/2$. We have shown, for any $w_i \in W_i$, $w_j \in W_j$, $u_i \in U_i$ and $u_j \in U_j$, such that $d_{G_4}(w_\ell,A_\ell') \ge |A_\ell'|-\xi n/5$ for at least one $\ell \in \lbrace i,j\rbrace$, that
	\begin{equation*}
	P_3(u_iu_j,G_5) - P_3(w_iw_j,G_4) \le -cn/2+\xi n/5 < -cn/3.
	\end{equation*}
	If we arbitrarily order $F_{ij}^U$ as $\overline{e}_1,\ldots,\overline{e}_{f_{ij}}$ and  $F_{ij}^W$ as $e_1,\ldots,e_{f_{ij}}$, then we can write
	\begin{align*}
	K_3(G_5) - K_3(G_4) &\le \sum_{ij \in \binom{[k-1]}{2}} \sum_{\ell \in [f_{ij}]}(P_3(\overline{e}_\ell,G_5) - P_3(e_\ell,G_4))+2|Z|^3,
	\end{align*}
	where $2|Z|^3$ bounds from above the error coming from the triangles in $G_4$ using at least two edges from $\bigcup_{ij \in \binom{[k-1]}{2}}F_{ij}^W$.
	Then the only $\ell$ for which the corresponding summand is potentially greater than $-cn/3$ is such that $e_{\ell}=w_iw_j$ where $w_t \in W_t$ for $t \in \lbrace i,j\rbrace$ and $d_{G_4}(w_t,A_t')<|A_t'|-\xi n/5$.
	Given any $u_i \in U_i$ and $u_j \in U_j$, we have in this case
	$$
	P_3(u_iu_j,G_5)-P_3(w_iw_j,G_4) \stackrel{\eqref{wiwj},\eqref{7.26}}{\le} a_i' - |A_j'| + d_{G_4}(u_j,A_k')-(a_i' - |A_j'|)\le  \delta n.
	$$
	But each $W_t$ contains at most one such vertex by Lemma~\ref{symmetrise}(ii), so the number of such summands is at most $\binom{k-1}{2}$.
	Thus we have
	$$
	K_3(G_5)-K_3(G_4) \leq k^2\delta n+2|Z|^3,
	$$
	proving (iv).
\end{proof}

\subsection{Transformation 6 and the proof of Lemma~\ref{maintrans}}\label{maintransproof}

A final transformation  of $G_5$ gives us the required graph $G'$. The transformation does the following.
Let $I_1,I_2$ be defined as in Lemma~\ref{G5}. 
If $ij$ is a pair in $I_1$, it replaces all $[W_i,W_j]$-edges with some missing edges in $[W_i,R_i]$.
If $ij$ is a pair in $I_2$, then it replaces some edges in $[R_i,R_k]$ with all missing edges in $[U_i,U_j]$.
The resulting graph $G'$ (see Figure~\ref{partition6}) has the following properties: (i) an edge remains inside $A_k'$ if and only if it is in $[U_i, W_j\cup U_j]$ for some $ij\in{[k-1]\choose 2}$; (ii) for any $ij\in{[k-1]\choose 2}$, $G'[U_i,U_j]$ is complete while $G'[W_i,W_j]$ is empty. Thus the new partition obtained by moving $U_i$ to $A_i'$ for all $i \in [k-1]$ satisfies \Pcomplete.

\begin{center}
\begin{figure}
\includegraphics[scale=0.9]{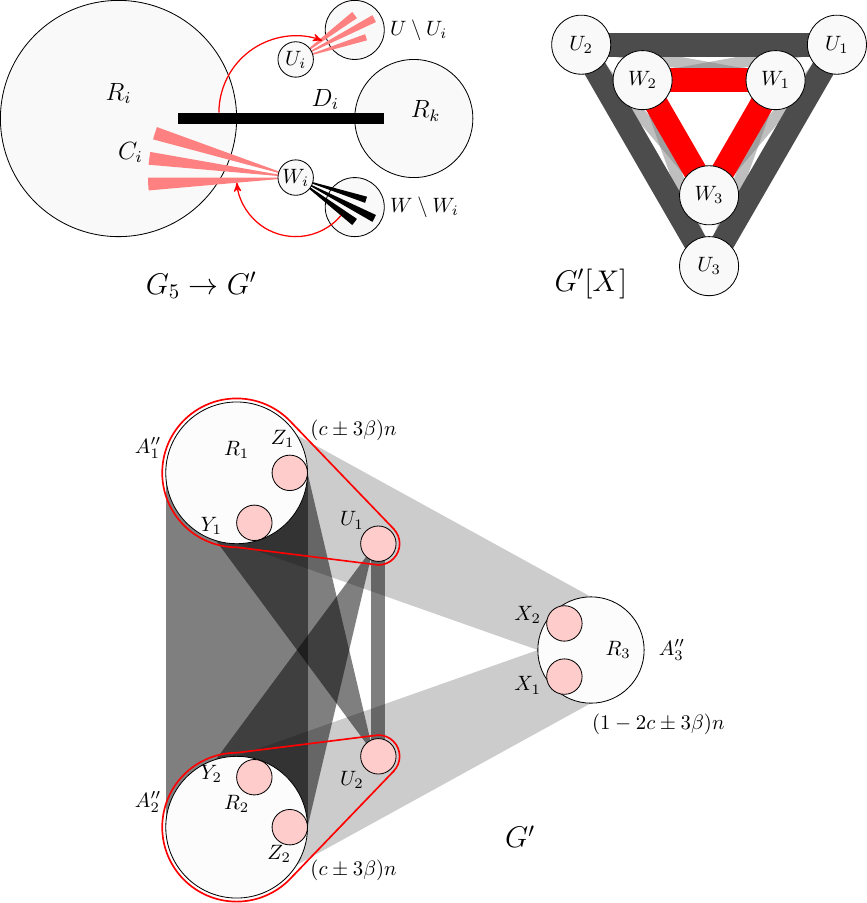}
\caption{Transformation 6. Left: Transformation~6 at $X_i = U_i \cup W_i$. Right: $G'$, in which the redistributed subsets of $X$ are coloured pink. (cf. $G$ in Figure~\ref{partition}).}
\label{partition6}
\end{figure}
\end{center}

\medskip
\noindent
\emph{Proof of Lemma~\ref{maintrans}.}
Apply Lemmas~\ref{subZiedges}--\ref{G5} to obtain $(n,e)$-graphs $G\rightarrow G_1\rightarrow G_2\rightarrow G_3\rightarrow G_4\rightarrow G_5$.
We will obtain $G'$ from $G_5$ as follows.
For each $i \in [k-1]$, choose $C_i \subseteq E(\overline{G_5}[R_i,W_i])$ such that $|C_i| = e(G_5[W_i,\bigcup_{i\ell \in I_1 : \ell > i}W_\ell])$, and $D_i \subseteq E(G_5[R_k,R_i])$ such that $|D_i| = e(\overline{G_5}[U_i,\bigcup_{i\ell \in I_2 : \ell > i}U_\ell])$, each $D_i$ is bipartite, and the collection of sets $V(D_i)\cap R_k$ is pairwise disjoint over $i\in[k-1]$. Let
$$
E(G') := \left( E(G_5) \cup \bigcup_{i \in [k-1]}C_i \cup \bigcup_{ij \in \binom{[k-1]}{2}} E(\overline{G_5}[U_i,U_j]) \right)\setminus \left( \bigcup_{ij \in \binom{[k-1]}{2}}E(G_5[W_i,W_j]) \cup \bigcup_{i \in [k-1]}D_i \right). 
$$
So for each $i \in [k-1]$, we remove all $[W_i,W_j]$-edges with $j > i$ and replace them with missing $[R_i,W_i]$-edges (the set $C_i$); and we add all missing $[U_i,U_j]$-edges with $j > i$ and remove the same number of $[R_k,R_i]$-edges (the set $D_i$) to compensate (see Figure~\ref{partition6}). Write $W=\bigcup_{i\in[k-1]}W_i$ and $U=\bigcup_{i\in[k-1]}U_i$.
Observe that
\begin{align*}
e\left(G_5\left[W_i,\bigcup_{i\ell \in I_1 : \ell > i}W_\ell\right]\right) &\leq e(G_5[W_i,W\setminus W_i]) \stackrel{\Pbadedges(G_5)}{\leq} |W_i|\delta n < |W_i|(\xi/5-\delta)n\\ 
&\stackrel{\Pmissing(G_4)}{\leq} e(\overline{G_4}[W_i,A_i'])-|W_i||Z|
\le e(\overline{G_4}[W_i,R_i])= e(\overline{G_5}[W_i,R_i]),
\end{align*}
where we used Lemma~\ref{G5}(ii) for the last equality.
So $C_i$ exists. 
On the other hand,
\begin{eqnarray*}
	e\left(\overline{G_5}\left[U_i,\bigcup_{i\ell \in I_2:\ell > i}U_\ell\right]\right) &\leq& e(\overline{G_5}[U_i,U\setminus U_i]) \leq |Z|^2 \stackrel{(\ref{Zsize})}{\le} \eta n^2.
\end{eqnarray*}
Note that, for every $v\in R_k$ and $i \in [k-1]$, we have
$$|R_i|\ge d_{G_5}(v,R_i)=d_{G_4}(v,R_i)\stackrel{\Pmissing(G_4)}{\ge} |A_i'|-\frac{\xi n}{5}-|Z| \stackrel{\Ppartition(G_4),\eqref{ineq:c}}{\ge}|R_k|\ge k\sqrt{\eta} n\stackrel{\eqref{Zsize}}{\ge} k|Z|.$$
Thus we can choose $D_i$ to be union of stars with distinct centres at $R_k$ and leaves in $R_i$ such that $V(D_i)\cap R_k$ are pairwise disjoint for all $i\in[k-1]$ as desired.
There is no edge which is both added and removed as $W\cap U=\emptyset$, and
\begin{align}
\label{eW} \sum_{ij \in \binom{[k-1]}{2}}e(G_5[W_i,W_j]) &= \sum_{i \in [k-1]}e(G_5[W_i,\bigcup_{i\ell \in I_1 : \ell > i}W_\ell]) = \sum_{i \in [k-1]}|C_i|,\\
\nonumber \sum_{ij \in \binom{[k-1]}{2}}e(\overline{G_5}[U_i,U_j]) &= \sum_{i \in [k-1]}e(\overline{G_5}[U_i,\bigcup_{i\ell \in I_2 : \ell > i}U_\ell]) = \sum_{i \in [k-1]}|D_i|.
\end{align}
Thus $G'$ is an $(n,e)$-graph.
By construction,
\begin{itemize}
	\item[(1)] every edge in $G'[A_k']$ is in $[U_i, W_j\cup U_j]$ for some $ij\in{[k-1]\choose 2}$; furthermore, $G'[U_1,\ldots,U_{k-1}]$ is complete $(k-1)$-partite.
	\item[(2)] the edge set of $G'[A_i']$ is empty for all $i \in [k-1]$ (this follows from Lemmas~\ref{symmetrise}(iii),~\ref{G5}(ii) and that $G_5$ and $G'$ are identical in $A_i'$ for all $i\in[k-1]$).
	\item[(3)] the edge set of $G'[A_i',U_i]$ is empty for all $i \in [k-1]$ and the edge set of $G'[A_j',U_i]$ is complete for all $j\in[k-1]\setminus \lbrace i \rbrace$ (this follows from Lemmas~\ref{symmetrise}(ii),~\ref{G5}(ii) and that $G_5$ and $G'$ are identical in $[A_i',U_i]$ for all $i\in[k-1]$).
\end{itemize}

With these observations, we can define the required partition of $G'$ and prove (i).
Indeed, let $A_i'' := A_i' \cup U_i$ for all $i \in [k-1]$ and $A_k'' := A_k'\setminus U$.
Properties (1)--(3) imply that $A_i''$ is independent for all $i \in [k]$.

We claim that $G'$ has an $(A_1'',\ldots,A_k'';3\beta)$-partition, i.e.~\Ppartition($G'$) and \Pcomplete($G'$) hold with the appropriate parameters.
For \Ppartition($G'$), clearly $A_1'',\ldots,A_k''$ is a partition of $V(G')$.
Moreover, $\sum_{i \in [k-1]}|U_i|\le|Z| \leq \delta n \leq \beta n$, so \Ppartition($G_5$) implies that \Ppartition($G'$) holds with parameter $3\beta$.

For \Pcomplete($G'$), since $G'[A_i',A_j'] = G_4[A_i',A_j']$ for $ij \in \binom{[k-1]}{2}$, it suffices to check that $G'[U_i,A_j'']$ is complete.
By \PZk($G_4$) we have that $G'[U_i,A_j'] = G_4[U_i,A_j']$ is complete.
But $G'[U_i,U_j]$ is also complete by Property~(1).
This proves \Pcomplete($G'$).
We have shown that $G'$ has an $(A_1'',\ldots,A_k'';3\beta)$-partition.

Our next task is to bound the entries in the missing vector
$\underline{m}' := (m'_1,\ldots,m'_{k-1})$ of $G'$ with respect to $(A_1'',\ldots,A_k'')$.
For each $i \in [k-1]$ we have
\begin{eqnarray}
\nonumber m'_i &=& e(\overline{G'}[A_i'',A_k'']) = e(\overline{G'}[A_i',A'_k\setminus U]) + e(\overline{G'}[U_i,A'_k\setminus U])\\
&=& e(\overline{G'}[A_i',A_k']) + e(\overline{G'}[U_i,A_k'\setminus U]) - e(\overline{G'}[U_i,A_i']),
\label{m''edge}
\end{eqnarray}
where the last equality follows from $e(\overline{G'}[U,A_i'])=e(\overline{G'}[U_i,A_i'])$, a consequence of Property~(3). By Property~(3), $e(\overline{G'}[U_i,A_i'])=|U_i|\,|A_i'|$. Notice also that every transformation from $G$ to $G'$ preserves all adjacencies in $[X,R_k]$ (hence also $[U_i,R_k]$), which is empty in $G$. Together with $A_k'\setminus U=R_k\cup W$, this implies that $$|U_i||R_k|\le e(\overline{G'}[U_i,A_k'\setminus U])\le |U_i||A_k'|.$$
We then derive from~\eqref{m''edge} that
\begin{equation}\label{eq-mi'}
e(\overline{G'}[A_i',A_k'])-|U_i|(|A_i'|-|R_k|)\le m'_i\le e(\overline{G'}[A_i',A_k'])-|U_i|(|A_i'|-|A_k'|).
\end{equation}

Lemma~\ref{G5}(ii) says that $G_5$ has the same number of edges between parts $A_i',A_j'$ as $G_4$ for all $1 \leq i < j \leq k$, and so implies that $e(\overline{G_5}[A_i',A_k']) = m_i^{(4)}$ for all $i \in [k-1]$.
Then
\begin{equation}\label{cidi1}
e(\overline{G'}[A_i',A_k']) = e(\overline{G_5}[A_i',A_k']) - |C_i| + |D_i| = m^{(4)}_i - |C_i| + |D_i|.
\end{equation}
Now, using \Pbadedges($G_5$),
\begin{align}
\label{cidi2} |C_i|+|D_i| &\leq e(G_5[W])+|U_i||Z| \leq e(G_5[A_k']) + |Z|^2 \stackrel{(\ref{Zsize})}{\leq} \frac{2m}{\xi n}(\delta n + \sqrt{\eta}n) \leq 2\sqrt{\delta}m.
\end{align}
Lemma~\ref{symmetrise}(v) implies that $m^{(4)}_i = m^{(3)}_i$ and
$m_i^{(4)} \geq |U_i|\,|A_i'|$
for all $i\in [k-1]$.
Now,
\begin{equation}\label{cidi3}
|A_i'|-|A_k'| = |A_i'|-|R_k| \pm \delta n = |A_i'|-|A_k'|+|Z| \pm \delta n \stackrel{\Pbadedges(G_5),\Ppartition(G_5)}{=} (kc-1)n \pm 5\beta n.
\end{equation}
Thus
\begin{eqnarray*}
m_i' &\stackrel{(\ref{eq-mi'}),(\ref{cidi1})}{\leq}& m_i^{(4)} - |C_i| + |D_i| - |U_i|(|A_i'|-|A_k'|)\\
&\stackrel{(\ref{cidi2}),(\ref{cidi3})}{\leq}& m_i^{(4)} + 2\sqrt{\delta}m - |U_i|(kc-1 \pm 5\beta)n \stackrel{(\ref{ineq:c})}{\leq} m_i^{(4)} + 2\sqrt{\delta}m.
\end{eqnarray*}
In the other direction,
\begin{eqnarray*}
	m'_i &\stackrel{(\ref{eq-mi'}),(\ref{cidi1})}{\geq}& m_i^{(4)} - |C_i|+|D_i|-|U_i|(|A_i'| - |R_k|) \geq m_i^{(4)} - 2\sqrt{\delta}m - \frac{m_i^{(4)}}{|A_i'|}\cdot (kc-1+5\beta)n\\
	&\stackrel{\Ppartition(G_4)}{\geq}& m^{(4)}_i-2\sqrt{\delta}m - \frac{m^{(4)}_i}{(c-2\beta)} \cdot (kc-1+5\beta) = m_i^{(3)} \cdot \frac{1-(k-1)c-7\beta}{c-2\beta} - 2\sqrt{\delta}m\\
	&\stackrel{(\ref{ineq:c})}{\geq}& m_i^{(3)} \cdot \frac{(k-1)\alpha-7\beta}{c-2\beta} - 2\sqrt{\delta}m.
\end{eqnarray*}
Then Lemmas~\ref{Ziedges},~\ref{Yiedges} and~\ref{Xiedges}(i) imply that $\alpha m_i/4\le m_i^{(3)}\le 2m_i$, thus,
$$
\alpha^2 m_i - 2\sqrt{\delta}m \leq \frac{\alpha}{4}\cdot \frac{(k-1)\alpha-7\beta}{c-2\beta}\cdot m_i - 2\sqrt{\delta}m \leq m_i' \leq m_i^{(3)} + 2\sqrt{\delta}m \leq 2m_i + 2\sqrt{\delta}m,
$$
as required.

It remains to bound $K_3(G')-K_3(G)$.
To do so, we will first bound $K_3(G')-K_3(G_5)$.
Let $i \in [k-1]$. Let $x_i \in R_i$ and $w_i \in W_i$ be arbitrary.
Then $d_{G'}(x_i,A_i') = d_{G_5}(x_i,A_i') = 0$ and $d_{G'}(w_i,A_k') \leq |U|$ by Properties~(1) and~(2).
So $P_3(x_iw_i,G') \leq a_i' + |U| \leq a_i' + \delta n$ and hence
\begin{equation}\label{eq-Ci}
\max_{e \in C_i}P_3(e,G') \leq a_i' + \delta n.
\end{equation}
Let $w_j \in W_j$ be arbitrary with $j \in [k-1]\setminus\lbrace i \rbrace$. Recall from Lemma~\ref{symmetrise}(ii) that all vertices in $W_i$ except at most one special vertex have $G_4$-degree in $A_i'$ exactly $|A_i'|-\xi n/5$. Let $W'\subseteq W$ be the set of these special vertices from each $W_i$. Then $|W'|\le k-1$. Further define $E_{W\setminus W'}:=E(G_5[W\setminus W'])$ to be the set of $G_5$-edges in $W\setminus W'$ and $E_{W'}:=E(G_5[W])-E_{W\setminus W'}$ to be the set of $G_5$-edges in $W$ with at least one endpoint in $W'$. Note that 
\begin{equation}\label{eq-Ewp}
	|E_{W'}|\le |W'|\cdot |W|<k|Z|\stackrel{(\ref{Zsize})}{\le} \frac{2km}{\xi n}.
\end{equation}
By \PZk($G_4$) and the definition of $W'$, we see that 
\begin{equation}\label{eq-W1}
P_3(w_iw_j,G_4;\overline{A_k'}) = \sum_{i=1}^{k-1}|A_i'|-2\xi n/5\quad\text{for all}\quad w_iw_j \in E_{W\setminus W'};
\end{equation}
while for any $w_iw_j\in E_{W'}$,~\eqref{wiwj} holds.
By Lemma~\ref{G5}(ii), for every $w \in W'$, we have $N_{G_5}(w,\overline{A_k'})=N_{G_4}(w,\overline{A_k'})$, which in turn implies that the bounds in~\eqref{wiwj} and~\eqref{eq-W1} hold also for $P_3(w_iw_j,G_5)$,~i.e.
\begin{align}\label{repeat}
P_3(w_iw_j,G_5) \geq a_i'-|A_j'| \stackrel{(\ref{eq:ai})}{=} a_j'-|A_i'|\quad\text{and}\\
\nonumber P_3(w_iw_j,G_5;\overline{A_k'}) = \sum_{i=1}^{k-1}|A_i'| - 2\xi n/5\quad\text{for all}\quad w_iw_j \in E_{W\setminus W'}. 
\end{align}

Let $x_k \in R_k$ and $y_i \in R_i$.
By \Pcomplete($G_5$) (i.e.~Lemma~\ref{G5}(i)), $G_5[y_i,A_\ell']$ is complete for all $\ell \in [k-1]\setminus \lbrace i \rbrace$.
Moreover, Lemma~\ref{G5}(ii) implies that $d_{\overline{G_5}}(x_k,\overline{A_k'}) = d_{\overline{G_4}}(x_k,\overline{A_k'})$ which is at most $3\xi n$ by \Pmissing($G_4$).
Thus $P_3(x_ky_i,G_5) \geq a_i'-3\xi n$, and so
\begin{equation}\label{eq-Di}
\min_{e \in D_i}P_3(e,G_5) \geq a_i' - 3\xi n.
\end{equation}
Let $u_i \in U_i$ and $u_j \in U_j$ for $j \in [k-1]\setminus \lbrace i \rbrace$.
Then $d_{G'}(u_i,A_i'),d_{G'}(u_j,A_j') = 0$ by (3) and $d_{G'}(u_i,A_k'),d_{G'}(u_j,A_k') \leq |Z| \leq \delta n$ by (1). So
\begin{equation}\label{eq-Uij}
P_3(u_iu_j,G') \leq a_i' - |A_j'| + \delta n \stackrel{\Ppartition(G_5)}{\leq} a_i' - cn+3\beta n.
\end{equation}
Since for all $i \in [k-1]$ the graph $D_i \subseteq G_5[R_k,R_i]$ is bipartite and the $D_i$ are pairwise vertex-disjoint, any triangle in $G_5$ which contains at least two edges in $\bigcup_{i \in [k-1]}D_i$ also contains an edge in $G_5[R_i]$ or $G_5[R_k]$ for some $i$.
So there are no such triangles.
Since $\bigcup_{i \in [k-1]}D_i \cap W = \emptyset$, the only possible triangles containing at least two edges from $E(G_5)\setminus E(G')$ lie in $W$, and there are at most $|Z|^3$ such triangles.
Thus we can bound $K_3(G')-K_3(G_5)$ as follows.
\begin{eqnarray}
\nonumber &&K_3(G')-K_3(G_5)\le\sum_{i \in [k-1]}\left(\sum_{e \in C_i}P_3(e,G')-\sum_{f \in E(G_5[W_i,\bigcup_{\ell > i}W_\ell])} P_3(f,G_5) \right)\\
&&\quad+ \sum_{i \in [k-1]}\left(\sum_{f \in E(\overline{G_5}[U_i,\bigcup_{\ell>i}U_\ell])}P_3(f,G') - \sum_{e \in D_i}P_3(e,G_5) \right)+2|Z|^3.\label{eq-K3GG'}
\end{eqnarray}
Denote by $\Delta_W$ and $\Delta_U$ the first and second term on the right-hand side of~\eqref{eq-K3GG'} respectively.
If there is at most one non-empty $U_i$ then $\Delta_U=0$.
Otherwise, using~\eqref{eq-Di} and~\eqref{eq-Uij}, we have 
$$\Delta_U\le \sum_{i\in[k-1]}|D_i|\cdot (-cn+3\beta n+3\xi n)<0.$$

\par We claim that $\Delta_W\le \delta^{1/3}m^2/n$. To see this, note that if there is at most one non-empty $W_i$ then $\Delta_W=0$, so assume not.
Suppose first that $e(G_5[W])=\sum_{i\in[k-1]}|C_i|\le \delta^{1/3}m^2/n^2$, where the equality follows from~(\ref{eW}) and the fact that $G_5[W_i,W_j]=\emptyset$. Then by~\eqref{wiwj} and~\eqref{eq-Ci}, 
$$
\Delta_W\stackrel{(\ref{eq-Ci}),(\ref{repeat})}{\le} \sum_{i\in[k-1]}|C_i|\cdot (a_i'+\delta n-a_i'+\max_{j\neq i,k}|A_j'|)\stackrel{\Ppartition(G_5)}{\leq} \sum_{i\in[k-1]}|C_i|\cdot (cn+3\beta n)\le\frac{\delta^{1/3}m^2}{n^2}\cdot 2cn\le \frac{\delta^{1/3}m^2}{n}.
$$
We may then assume
$$
e(G_5[W])\ge \frac{\delta^{1/3}m^2}{n^2}\ge \delta^{1/3}\cdot C\cdot \frac{m}{n}\stackrel{(\ref{Cvalue})}{=}\frac{m}{\delta^{1/6}n}.
$$
In this case, we need to estimate $\Delta_W$ more carefully making use of~\eqref{repeat}:
\begin{eqnarray*}
	\Delta_W&\le&|E_{W\setminus W'}|\cdot \left(\max_{j\neq k}a_j'+\delta n-\sum_{i=1}^{k-1}|A_i'|+\frac{2\xi n}{5}\right)+|E_{W'}|\cdot(\delta n+\max_{j\neq k}|A_j'|)\\
	&\stackrel{\Ppartition(G_5)}{\leq}&|E_{W\setminus W'}|\cdot \left(-\frac{cn}{2}\right)+|E_{W'}|\cdot 2cn=\frac{cn}{2}\cdot (4|E_{W'}|-|E_{W\setminus W'}|)\\
	&=&\frac{cn}{2}\cdot (5|E_{W'}|-e(G_5[W]))\stackrel{(\ref{eq-Ewp})}{\le} \frac{cn}{2}\cdot \left(5\cdot \frac{2km}{\xi n}-\frac{m}{\delta^{1/6}n}\right)<0.
\end{eqnarray*}
Therefore, we have
\begin{eqnarray*}
	K_3(G')-K_3(G_5)&\leq& \Delta_W+\Delta_U+2|Z|^3\le \frac{\delta^{1/3}m^2}{n}+2|Z|^3.
\end{eqnarray*}
Now, letting $G_0 := G$ and $G_6 := G'$ and using Lemmas~\ref{Ziedges}(iii),~\ref{Yiedges}(iv),~\ref{Xiedges}(iii),~\ref{symmetrise}(iv),~\ref{G5}(iv) and the previous inequalities,
\begin{eqnarray*}
	&&K_3(G')-K_3(G) = \sum_{i \in [6]}(K_3(G_{i})-K_3(G_{i-1}))\\ 
	&\leq& \left( \delta^{7/8} + \frac{\delta^{1/4}}{3} + \sqrt{\delta} + 0 + \delta^{1/3} \right)\frac{m^2}{n}+k^2\delta n+4|Z|^3
	\stackrel{(\ref{Zsize})}{\leq} \frac{\delta^{1/4}m^2}{2n},
\end{eqnarray*}
where we use the fact that $m > Cn$ to bound $k^2\delta n \leq k^2\delta m^2/(C^2n) = k^2\delta^2m^2/n$.
This completes the proof of Lemma~\ref{maintrans}.
\hfill$\square$

\section{The intermediate case: finishing the proof}\label{int3}

\subsection{The intermediate case when $m$ is large}\label{sec:superlinear}

In this section we finish the proof of the intermediate case when
\begin{equation}\label{C}
m \geq Cn.
\end{equation}

\subsubsection{Properties of $G$ via $G'$}

We will now use Lemma~\ref{maintrans} to obtain some additional structural information about $G$, which will in turn enable us to redo the transformations in Section~\ref{sec:trans} more carefully.
This will eventually imply that most exceptional sets $X_i,Y_i$ are in fact empty.
After this, one final `global' transformation yields the result.

Apply Lemma~\ref{maintrans} to $G$ to obtain a $k$-partite graph $G'$ with vertex partition $A_1'',\ldots,A_k''$ and missing vector $\underline{m'} = (m_1',\ldots,m_{k-1}')$ satisfying Lemma~\ref{maintrans}(i)--(iii).
Let $m' := \sum_{i \in [k-1]}m_i'$.

The first step is to use Lemma~\ref{CompletekPartite} to show that, in $G'$, the parts $A_1'',\ldots,A_{k-2}''$ all have size within $o(m/n)$ of $cn$, the `expected' size; and that the number of missing edges between these parts and $A_k''$ is $o(m)$.
Roughly speaking, this means that $G'$ has edit distance $o(m)$ from a graph in $\mathcal{H}_1(n,e)$.
Since $m_i' = \Theta(m_i)+o(m)$ for all $i \in [k-1]$, this information about missing edges in $G'$ translates to $G$.
Lemma~\ref{maintrans}(ii) clearly implies that
\begin{equation}\label{eq:m'}
\frac{\alpha^2}{2} \leq \alpha^2-2k\sqrt{\delta} \leq \frac{m'}{m} \leq 2+2k\sqrt{\delta} \leq 3.
\end{equation}

The next proposition shows that the smallest part $A_k''$ of $G'$ has to be noticeably larger than $(1-(k-1)c)n$ since the number of missing edges $m'$ is large.

\begin{proposition}\label{smallpart}
	$|A_k''| \geq (1-(k-1)c)n + \frac{m'}{(kc-1)n}$.
\end{proposition}

\begin{proof}
	Suppose, for a contradiction, that $|A_k''| < n - (k-1)cn+q$, where $q := \frac{m'}{(kc-1)n}$. Let $x:=(k-1)cn-q$.
	Given $|A_k''|$, we certainly have
	$$
	\sum_{ij \in \binom{[k-1]}{2}}|A_i''||A_j''|+(n-|A_k''|)|A_k''| \leq t_{k-1}(n-|A_k''|) + (n-|A_k''|)|A_k''|.
	$$
	Recall that we assume
	$$
	|A_k''|<n-x\stackrel{(\ref{eq:m'})}{\le}(1-(k-1)c)n+\frac{3m}{(kc-1)n} \stackrel{(\ref{m})}{\leq} (1-(k-1)c+\sqrt{\eta})n\stackrel{(\ref{ineq:c})}{\leq}(c-\sqrt{\alpha})n. 
	$$
	As $(1-(k-1)c+\sqrt{\eta})+(k-1)(c-\sqrt{\alpha})<1$, we get from the above inequalities that $|A_k''| <n-x< n/k$.
	We know by Lemma~\ref{lm:a*k} that $t_{k-1}(n-|A_k''|)+(n-|A_k''|)|A_k''|$ is an increasing function of $|A_k''|$ whenever $|A_k''|\le n/k$. Thus we have $t_{k-1}(n-|A_k''|)+(n-|A_k''|)|A_k''| \leq t_{k-1}(x) + x(n-x)$.
	Therefore, since $G'$ has no bad edges,
	\begin{eqnarray*}
		e+m' &=& \sum_{ij \in \binom{[k-1]}{2}}|A_i''||A_j''| + (n-|A_k''|)|A_k''| < t_{k-1}(x)+x(n-x)
		\\
		&\le& {k-1\choose 2}\left(\frac{x}{k-1}\right)^2+x(n-x)\\
		&=&x\left(n-\frac{k}{2(k-1)}x\right)=(k-1)cn^2-\binom{k}{2}c^2n^2 + (kc-1)qn-\frac{kq^2}{2(k-1)}\\
		&\leq& (k-1)cn^2-\binom{k}{2}c^2n^2 + (kc-1)qn \stackrel{(\ref{eq:c})}{=} e + (kc-1)qn = e+m',
	\end{eqnarray*} 
	a contradiction.
\end{proof}

\begin{lemma}\label{superlinear1}
	For all $j \in [k-2]$, the following hold.
	\begin{itemize}
		\item[(i)] $m_j \leq \delta^{1/6}m$.
		\item[(ii)] $|Z_j \cup Z_k^j| \leq \delta^{1/7}m/(2n)$.
		\item[(iii)] $|\,|A_j''|-cn| \leq 6\delta^{1/9}m/n$ and $|A_{k-1}''| \leq cn- \alpha^2 m/(4cn)$.
	\end{itemize}
\end{lemma}

\begin{proof}
	Let $H := K^k_{\lfloor cn \rfloor,\ldots,\lfloor cn\rfloor,n-(k-1)\lfloor cn \rfloor}$ and let $B_1,\ldots,B_k$ be the parts of $H$, where $|B_i|=\lfloor cn\rfloor$ for all $i \in [k-1]$.
	We claim that there is an $(n,e)$-graph $F$ which one can obtain from $H$ by removing at most $(k-1)^2cn$ edges from $H[B_{k-1},B_k]$.
	Inequality~(\ref{ineq:c}) implies rather roughly that $|B_{k-1}|\,|B_k| > (k-1)^2cn$, so it suffices to show that $e \leq E(H) \leq e+(k-1)^2cn$.
	Indeed, by~(\ref{ineq:c}), we have that $\lfloor cn \rfloor > n-(k-1)\lfloor cn \rfloor+k$, so
\begin{align*}
	e&=e(K^k_{cn,\ldots,cn,n-(k-1)cn}) \leq e(K^k_{\lfloor cn \rfloor,\ldots,\lfloor cn \rfloor,n-(k-1)\lfloor cn \rfloor}) = e(H)\\
	&= \binom{k-1}{2}\lfloor cn \rfloor^2 + (k-1)\lfloor cn \rfloor(n-(k-1)\lfloor cn \rfloor) \leq e + (k-1)^2cn,
	\end{align*}
    as required.
	
	We will apply Lemma~\ref{CompletekPartite} with $G',\lbrace A_i''\rbrace_{i\in [k]},F,\lfloor cn \rfloor,(k-1)^2cn$ playing respectively the roles of $G,\lbrace A_i\rbrace_{i\in [k]},F,\ell,d$.
	Let $d_i := |A_i''|-\lfloor cn\rfloor$ for all $i \in [k-1]$ and $d_k := |A_k''| - n+(k-1)\lfloor cn \rfloor$.
	By Proposition~\ref{smallpart}, we have
	\begin{equation}\label{dk}
	d_k \geq \frac{m'}{(kc-1)n}-k \stackrel{(\ref{ineq:c})}{\geq} \frac{m'}{(c-(k-1)\alpha)n}-k \geq \frac{m'}{cn}.
	\end{equation}
	Moreover, for all $i \in [k]$, Lemma~\ref{maintrans}(i) implies that
	$$
	|d_i| \leq 4\beta n < \frac{\sqrt{2\alpha}n}{20k^3} \stackrel{(\ref{ineq:c})}{\leq} \frac{(kc-1)n}{20k^3} \leq \frac{\lfloor cn \rfloor-(n-(k-1)\lfloor cn \rfloor)}{12k^3}. 
	$$
	Then Lemma~\ref{CompletekPartite} can be applied with the parameters above to imply that
	\begin{align*}
	K_3(G) + \frac{\delta^{1/4}m^2}{2n} &\geq K_3(G')\\
	&\geq K_3(F) + \sum_{t \in [k-1]}\frac{m_t'}{m'}\cdot \frac{k\lfloor cn \rfloor-n}{4} \left( (d_t+d_k)^2 + \sum_{i \in [k-1]\setminus \lbrace t \rbrace}d_i^2 \right) - \frac{12(k-1)^4c^2n^2}{k\lfloor cn\rfloor-n}.
	\end{align*}
	Observe that each summand over $t \in [k-1]$ is non-negative by~(\ref{ineq:c}).
	Bounding the last term, we have
	$$
	0 \stackrel{(\ref{ineq:c})}{\leq} \frac{12(k-1)^4c^2n^2}{k\lfloor cn\rfloor-n} \leq \frac{14(k-1)^4c^2n}{kc-1} \stackrel{(\ref{ineq:c}),(\ref{C})}{\leq} \frac{14k^4c^2m^2}{\sqrt{2\alpha}C^2 n} \stackrel{(\ref{Cvalue})}{=} \frac{14k^4c^2\delta m^2}{\sqrt{2\alpha}n} \leq \frac{\delta^{7/8}m^2}{2n}.
	$$
	Furthermore,
	$$
	\frac{k\lfloor cn\rfloor-n}{4} \stackrel{(\ref{ineq:c})}{\geq} \frac{\sqrt{2\alpha}n-k}{4} > \frac{\sqrt{\alpha}n}{4}.
	$$
	Thus, for each $j \in [k-1]$, using the fact that $\delta^{7/8}/2 + \delta^{1/4}/2 \leq \delta^{1/4}$,
	$$
	\frac{m'_j}{m'} \left( (d_j+d_k)^2 + \sum_{i \in [k-1]\setminus \lbrace j \rbrace}d_i^2 \right) \leq \frac{K_3(G) - K_3(F) +\frac{\delta^{1/4} m^2}{n}}{\frac{\sqrt{\alpha}n}{4}} \leq  \frac{4\delta^{1/4}m^2}{\sqrt{\alpha}n^2} \stackrel{(\ref{eq:m'})}{\leq} \frac{16\delta^{1/4}m'^2}{\alpha^{9/2} n^2} \leq \frac{\delta^{2/9}m'^2}{(k-1)n^2}.
	$$
	So for all $ij \in \binom{[k-1]}{2}$, we have that
	\begin{equation}\label{tki}
	|d_j+d_k|,|d_i| \leq \frac{\delta^{1/9} m'}{n} \cdot \sqrt{\frac{m'}{(k-1)m_j'}}.
	\end{equation}
	Suppose that $r \in [k-1]$ is such that $m_r' = \max_{j \in [k-1]}m'_j$.
	Then $m_r' \geq m'/(k-1)$.
	We have
	\begin{equation*}
	|d_r+d_k|\stackrel{(\ref{tki})}{\le} \frac{\delta^{1/9}m'}{n}\quad\text{and}\quad|d_i| \leq \frac{\delta^{1/9}m'}{n}.
	\end{equation*}
	But by~(\ref{dk}), $d_k \geq m'/(cn) > \delta^{1/9}m'/n$. So $d_r < 0$ and in fact 
	$d_r = |A_r''|-\lfloor cn \rfloor \leq \delta^{1/9}m'/n - d_k$. Thus
	\begin{eqnarray}
		|A_r''|  &\stackrel{(\ref{dk})}{<}& \lfloor cn \rfloor -\left(\frac{1}{c}-\delta^{1/9}\right)\frac{m'}{n} \stackrel{(\ref{hierarchy})}{\leq} cn - \frac{m'}{2cn} \stackrel{(\ref{eq:m'})}{\leq} cn - \frac{\alpha^2 m}{4cn}; \quad \text{ and }\label{8.2_1}\\
		\bigg||A_i''|-cn\bigg| &\stackrel{(\ref{tki})}{\leq}& \frac{2\delta^{1/9}m'}{n} \stackrel{(\ref{eq:m'})}{\leq} \frac{6\delta^{1/9}m}{n}\label{8.2_2} 
	\end{eqnarray}
	for all $i \in [k-1]\setminus \lbrace r \rbrace$.
	Suppose now that $m_s' \geq \delta^{1/5} m'$ for some $s \in [k-1]\setminus \lbrace r \rbrace$.
	Then applying~(\ref{tki}) with $ij=rs$, we have
	$$
	|A_r''| \geq \lfloor cn \rfloor - |d_r| \stackrel{(\ref{tki})}{\geq} \lfloor cn\rfloor -\frac{\delta^{1/9} m'}{\sqrt{(k-1)}\delta^{1/10}n} > cn - \frac{4\delta^{1/90}}{\sqrt{(k-1)}n} > cn - \frac{\alpha^2 m}{4cn},
	$$
	a contradiction to~(\ref{8.2_1}).
	Therefore, for all $s \in [k-1]\setminus \lbrace r \rbrace$, we have by Lemma~\ref{maintrans}(ii) that
	$$
	m_s \leq \frac{1}{\alpha^2}\left(m_s'+2\sqrt{\delta}m\right) < \frac{1}{\alpha^2}\left(\delta^{1/5}m'+2\sqrt{\delta}m\right) \stackrel{(\ref{eq:m'})}{\leq} \frac{1}{\alpha^2}\left(3\delta^{1/5}m+2\sqrt{\delta}m\right) \leq \delta^{1/6}m.
	$$
	But $\max_{i \in [k-1]}m_i=m_{k-1} \geq m/(k-1)$, and so $r=k-1$.
	That is, $m_1,\ldots,m_{k-2} \leq \delta^{1/6}m$, as required for~(i).
	By~(\ref{Zsize}), we have for all $s \in [k-2]$ that
	$$
	|Z_s \cup Z_k^s| \leq \frac{2\delta^{1/6}m}{\xi n} \leq \frac{\delta^{1/7}m}{2n},
	$$
	proving~(ii). Part~(iii) follows from~(\ref{8.2_1}) and~(\ref{8.2_2}).
\end{proof}

Since the exceptional sets $Z_1,\ldots,Z_{k-2}$ and $Z_k^1,\ldots,Z_k^{k-2}$ are all small by the previous lemma, it is now easy to show that $G[R_1,R_k], \ldots, G[R_{k-2},R_k]$ are all complete. 
That is, for all $i \in [k-2]$, every missing edge in $G[A_i,A_k]$ is incident to a vertex of $Z$.

\begin{lemma}\label{lem-R1R3}
	For every $i \in [k-2]$, $G[R_i,R_k]$ is complete.
\end{lemma}

\begin{proof}
	Let $x\in R_i$ and $y\in R_k$.
	By Proposition~\ref{Gprops}(i), $N_G(y,A_k) \subseteq Y$.
	By \Pbadedges($G$), $N_G(x,A_i) \subseteq Z_i$.
	Since $A_j'' \supseteq A_j \cup Y_j$ for all $j \in [k-1]$, using Lemma~\ref{superlinear1}(ii) and~(iii) and $m\ge Cn$, we have that
	\begin{eqnarray*}
		P_3(xy,G) &\leq& \sum_{j \in [k-1]\setminus \lbrace i \rbrace}|A_j| + |Z_i| + |Y| \leq \sum_{j \in [k-2]\setminus \lbrace i \rbrace}|A_j''| + |A_{k-1}''| + |Z_i \cup Z_k^i|\\
		&\leq& (k-3)\left(cn+\frac{6\delta^{1/9}m}{n}\right) + cn-\frac{\alpha^2 m}{4cn} + \frac{\delta^{1/7}m}{2n} \leq (k-2)cn - \frac{\alpha^2 m}{5cn}\\
		&\stackrel{(\ref{C})}{\leq}& (k-2)cn - \frac{\alpha^2C}{5c} \stackrel{(\ref{Cvalue})}{\leq} (k-1)cn - 2k.
	\end{eqnarray*}
	Therefore $xy \in E(G)$ by~(\ref{2path}).
\end{proof}

The previous two lemmas now imply very precise information about the sizes of the parts $A_1,\ldots,A_k$ in $G$.
Indeed, we can calculate their sizes up to an $o(m/n)$ error term.
Recall from~(\ref{t}) that $t =\frac{m}{(kc-1)n}$.

\begin{lemma}\label{lem-size}	
	The following hold for parts of $G$.
	\begin{align*}
	|A_1|,\ldots,|A_{k-2}|&=cn\pm \frac{\delta^{1/10}m}{n};\\
	|A_{k-1}|&=cn- t \pm \frac{\delta^{1/11}m}{n} \quad\text{and}\\
	|A_k|&=n-(k-1)cn +t \pm \frac{\delta^{1/11}m}{n}.
	\end{align*}
\end{lemma}

\begin{proof}
	For the first equation, recall that for all $i \in [k-1]$ Lemma~\ref{maintrans}(i) implies that $A_i \subseteq A_i'' \subseteq A_i \cup Z_k^i$.
	If $j \in [k-2]$, then Lemma~\ref{superlinear1}(iii) implies that $|A_j| \leq |A_j''| \leq cn + 6\delta^{1/9}m/n$.
	Using Lemma~\ref{superlinear1}(ii) in addition, we see that also
	$$
	|A_j| \geq |A_j''|-|Z_k^j| \geq cn - \frac{\delta^{1/10}m}{n},
	$$
	as required.
	Therefore there is some $\tau \in \mathbb{R}$ such that
	\begin{align}
	|A_{k-1}| &= cn - \frac{\tau m}{n} \pm \frac{k\delta^{1/10}m}{n}\quad\text{and }\label{lem-size1}\\
	|A_k| &= (1-(k-1)c)n + \frac{\tau m}{n} \pm \frac{k\delta^{1/10}m}{n}\label{lem-size2}.
	\end{align}
	By Proposition~\ref{smallpart}, we have $|A_k|\ge |A_k''|\ge (1-(k-1)c)n+\frac{m'}{(kc-1)n}$. So~(\ref{eq:m'}) implies that $\tau\ge \frac{\alpha^2}{2(kc-1)}$. Let $\tilde{\delta} := k\delta^{1/10}m/n$.
	Then
	\begin{align*}
	e - e(G[A_{k-1},A_k]) &= \sum_{ij \in \binom{[k-2]}{2}}|A_i||A_j| + (|A_{k-1}|+|A_k|)\sum_{i \in [k-2]}|A_i| + \sum_{i \in [k]}e(G[A_i]) - \sum_{i \in [k-2]}m_i\\
	&\stackrel{(\ref{h})}{=} \binom{k-2}{2}(cn\pm\tilde{\delta})^2 + (n-(k-2)cn \pm 2\tilde{\delta})((k-2)cn\pm \tilde{\delta})\\
	&\quad\quad\pm (\delta m + k\delta^{1/6}m)\\
	&= \binom{k-2}{2}c^2n^2 + (n-(k-2)cn)(k-2)cn \pm 3k^2\tilde{\delta}n\\
	&\stackrel{(\ref{eq:c})}{=} e - cn(n-(k-1)cn) \pm 3k^3\delta^{1/10}m.
	\end{align*}
	Here we used Lemma~\ref{superlinear1}(i) to bound $m_i$ for $i \in [k-2]$. We then have
	\begin{equation}\label{eq-tau}
	e(G[A_{k-1},A_k])=cn(n-(k-1)cn) \pm 3k^3\delta^{1/10}m.
	\end{equation}
	
	We claim that $\tau\le 1/\delta$. So suppose for a contradiction that $\tau>1/\delta$.
	Then $|A_{k-1}|+|A_k| = n - \sum_{i \in [k-2]}|A_i| = (1-(k-2)c)n \pm \tilde{\delta}n$. Further,
	 \begin{align*}
	 |A_{k-1}|&\le cn-\frac{\tau m}{n}-\tilde{\delta}\le cn-\frac{m}{\delta n}-\tilde{\delta}\quad\text{and}\\
	 |A_{k}|&\ge (1-(k-1)c)n+ \frac{\tau m}{n}-\tilde{\delta}\ge (1-(k-1)c)n+\frac{m}{\delta n}-\tilde{\delta}.
	 \end{align*}
	 By~(\ref{ineq:c}), we have $|A_{k-1}|> |A_k|$.
	 So the product $|A_{k-1}|\,|A_k|$ is minimised when $|A_k|$ attains the upper bound above. So
	\begin{eqnarray*}
		|A_{k-1}||A_k| &\ge& \left(cn-\frac{m}{\delta n}-\tilde{\delta}\right)\left((1-(k-1)c)n+\frac{m}{\delta n}-\tilde{\delta}\right)\\
		&\ge& cn(n-(k-1)cn) + (kc-1)n\cdot \frac{m}{\delta n}- \frac{m^2}{\delta^2n^2}-\tilde{\delta}n\\
		&\stackrel{(\ref{ineq:c}),(\ref{m})}{\ge} & cn(n-(k-1)cn) + \sqrt{2\alpha}\cdot \frac{m}{\delta}- \frac{\eta m}{\delta^2}-k\delta^{1/10}m\ge cn(n-(k-1)cn)+\frac{m}{\sqrt{\delta}}.
	\end{eqnarray*}
	But then, this implies that
	$$
L
	$$
	contradicting~\eqref{eq-tau}.
	So $\tau \leq 1/\delta$, as claimed.
	
	We now estimate $|A_{k-1}|\,|A_k|$ again more carefully using that $\frac{\alpha^2}{2(kc-1)}\le \tau\le 1/\delta$. We have
	$$
	|A_{k-1}||A_k| = cn(n-(k-1)cn) + (kc-1)\tau m +\left(- \frac{\tau^2m^2}{n^2} \pm  2k\delta^{1/10}m + \frac{2\tau k\delta^{1/10} m^2}{n^2} + \frac{k^2\delta^{1/5}m^2}{n^2} \right).
	$$
	But $m^2/n^2 \leq \eta m$ by~(\ref{m}) and $\tau\le 1/\delta$, so the expression in the final parentheses is at most
	$3k\delta^{1/10}m$. So
	\begin{equation}\label{k-1k}
	|A_{k-1}||A_k| = cn(n-(k-1)cn) + (kc-1)\tau m \pm 3k\delta^{1/10}m.
	\end{equation}
	As $m_{k-1} = (1\pm k\delta^{1/6})m$ due to Lemma~\ref{superlinear1}(i), we have
	$$
	(1\pm k\delta^{1/6})m=m_{k-1}=|A_{k-1}||A_k|-e(G[A_{k-1},A_k])\stackrel{(\ref{eq-tau}),(\ref{k-1k})}{=}(kc-1)\tau m \pm 4k^3\delta^{1/10}m.
	$$
	Solving this for $\tau$, we get
	$$
	\frac{\tau m}{n} = \frac{m}{(kc-1)n} \pm \frac{\delta^{1/11}m}{2n} \stackrel{(\ref{t})}{=} t \pm \frac{\delta^{1/11}m}{2n}.
	$$
	Combined with~(\ref{lem-size1}) and~(\ref{lem-size2}), this completes
	the proof of the lemma.
\end{proof}

The usefulness of $G'$ is now exhausted, and we work only with $G$ for the rest of the proof.
The previous lemma implies that
\begin{equation}\label{aik-2}
a_i = \sum_{j \in [k-1]\setminus \lbrace i \rbrace}|A_j| = (k-2)cn-t \pm \frac{(k-2)\delta^{1/11}m}{n}\quad\text{for all}\quad i \in [k-2].
\end{equation}

Armed with Lemmas~\ref{lem-R1R3} and~\ref{lem-size}, we can now `redo' Transformations~1 and ~2 of Section~\ref{sec:trans}, in a slightly more careful fashion, to imply that $Z_i=Y_i=\emptyset$ for all $i \in [k-2]$.

\begin{proposition}\label{zRinbr}
	Let $i \in [k-2]$ and $z \in Z_i \cup Z_k^i$. Then $d_G(z,R_i) \geq t-\delta^{1/12}m/n>0$.
\end{proposition}

\begin{proof}
	By \Pcomplete($G$), \Pbadedges($G$) and \Pmissing($G$), every such $z$ has at least $\xi n$ non-neighbours in $A_k$.
	Recall the definitions of $R_k'$ and $\Delta$ in Section~\ref{smalldeg}.
	We have
	$$|R_k'\setminus N_G(z)|\ge d_{\overline{G}}(z,A_k)-|R_k\setminus R_k'|-|Z_k| \stackrel{(\ref{Zsize})}{\geq} \xi n/2 - \sqrt{\eta}n \geq \xi n/3.$$
	Thus we can choose $w \in R'_k\setminus N_G(z)$.
	Then $wz \in E(\overline{G})$ and so, by~(\ref{2path}) and \Pcomplete($G$),
	\begin{align*}
	(k-2)cn-\error &\leq P_3(zw,G) \stackrel{(\ref{Delta})}{\leq} a_i + d_G(z,A_i) + \Delta \stackrel{(\ref{Rk'})}{\leq} a_i + |Z_i| + d_G(z,R_i) + \frac{\delta^{1/3} m}{n}\\
	&\leq (k-2)cn - t + d_G(z,R_i) + \frac{k\delta^{1/11}m}{n},
	\end{align*}
	where the last inequality follows from Lemma~\ref{superlinear1}(ii) and~(\ref{aik-2}).
	Hence $d_G(z,R_i) \geq t-\delta^{1/12}m/n$,
	which is positive by~(\ref{t}).
\end{proof}

\begin{lemma}\label{empty}
	$Z_i = Y_i = \emptyset$ for all $i \in [k-2]$.
\end{lemma}

\begin{proof}
	Suppose that there exists $z \in Z_i$ for some $i \in [k-2]$.
	Let $z_1,\ldots,z_p$ be an arbitrary ordering of $Z\setminus Z_k$ such that $z := z_1$. Note that $N_{G}(z,A_{i})\neq\emptyset$ due to Proposition~\ref{zRinbr}.
	Now apply Lemma~\ref{subZiedges} to $G$ and let $F$ be the obtained $(n,e)$-graph $G^1$ which satisfies $J(1,1)$--$J(3,1)$. By $J(3,1)$ we have that
	\begin{eqnarray}
	\label{ineq:ZY} 0 &\leq& K_3(F)-K_3(G) \leq \sum_{y \in N_{G}(z,A_{i})}\left(\Delta
	- |Z_k\setminus Z_k^{i}| - P_3(yz,G;R_k)\right)\\
	\label{bettercount} &\stackrel{(\ref{Rk'})}{\leq}& \sum_{y \in N_G(z,Z_i)} \frac{\delta^{1/3}m}{n} + \sum_{y \in N_{G}(z,R_{i})}\left(\frac{\delta^{1/3}m}{n}
	- |Z_k\setminus Z_k^{i}| - d_G(z,R_k)\right).
	\end{eqnarray}
	Here, for all $y \in R_i$, since Lemma~\ref{lem-R1R3} implies that $R_k \subseteq N_G(y)$, we have $P_3(yz,G;R_k)=d_G(z,R_k)$.
	We must have $|Z_k\setminus Z_k^i| \leq \Delta \leq \delta^{1/3}m/n$, as otherwise the right hand side of~\eqref{ineq:ZY} is negative.
	So Lemma~\ref{superlinear1}(ii) implies that
	\begin{equation}\label{ZiZk}
	|Z_i \cup Z_k| = |Z_k\setminus Z_k^i| + |Z_i\cup Z_k^i| \leq  \frac{\delta^{1/3}m}{n}+\frac{\delta^{1/7}m}{2n}\leq \frac{\delta^{1/7}m}{n}.
	\end{equation}
	We will now bound $d_G(z,R_k)$.
	By \Pbadedges($G$), $z$ has a non-neighbour $u$ in $R_i$.
	Since $u \in R_i$, we have that $N_G(u,A_i) \subseteq Z_i$.
	Thus~(\ref{2path}) then implies that
	$$
	(k-2)cn - \error \leq P_3(uz,G) \leq a_i + d_G(z,R_k) + |Z_i \cup Z_k|.
	$$
	Thus
	\begin{equation}\label{eq-zRk}
	d_G(z,R_k) \geq (k-2)cn-a_i-\frac{2\delta^{1/7}m}{n} \stackrel{(\ref{aik-2})}{\geq} t - \frac{\delta^{1/12}m}{n}.
	\end{equation}
	Using Proposition~\ref{zRinbr},~(\ref{ZiZk}) and~\eqref{eq-zRk}, the final upper bound in~(\ref{bettercount}) is at most
	\begin{equation}\label{421eq}
	\frac{\delta^{1/7}m}{n} \cdot \frac{\delta^{1/3}m}{n} + \left(\frac{\delta^{1/3}m}{n} - 0 -\left( t - \frac{\delta^{1/12}m}{n}\right)  \right)\left(t-\frac{\delta^{1/12}m}{n}\right) \stackrel{(\ref{t})}{\leq} -\frac{t^2}{2},
	\end{equation}
	a contradiction.
	We have proved that $Z_i = \emptyset$, so $A_i=R_i$, for all $i \in [k-2]$.

	\medskip
	Suppose now that there exists $y \in Y_i$ for some $i \in [k-2]$.
	Let $y_1,\ldots,y_q$ be an arbitrary ordering of $Y = \bigcup_{i \in [k-1]}Y_i$ (as in~(\ref{XY})) such that $y := y_1$.
	Observe that, since $Z_1 = \ldots = Z_{k-2} = \emptyset$, the graph $G$ satisfies the conclusions of Lemma~\ref{Ziedges} when $\ell = k-2$.
	Therefore we can apply Lemma~\ref{subYiedges} with $k-2,G$ playing the roles of $\ell,G^\ell_1$. Let $F'$ be the obtained $(n,e)$-graph $G^1$ which satisfies $K(1,1)$--$K(3,1)$.
	Then $K(3,1)$,~\eqref{Rk'} and Lemma~\ref{lem-R1R3} imply that
	\begin{align*}
	0&\le K_3(F')-K_3(G) \leq \sum_{x \in N_G(y,R_i)}\left(\Delta - \frac{\xi}{6\gamma}|Z_k\setminus Z_k^i| - P_3(xy,G;R_k)\right)\\
	&\leq \sum_{x \in N_G(y,R_i)}\left(\frac{\delta^{1/3}m}{n}-\frac{\xi}{6\gamma}\cdot |Z_k\setminus Z_k^i| - d_G(y,R_k) \right).
	\end{align*}
    Again by Proposition~\ref{zRinbr}, $N_G(y,R_i)\neq\emptyset$. Therefore, as in~(\ref{ZiZk}), by Lemma~\ref{superlinear1} we have
	\begin{equation}\label{eq-Zik}
	|Z_i\cup Z_k| = |Z_k| \leq |Z_k^{k-1}| + \frac{(k-2)\delta^{1/7}m}{2n} \leq \frac{6\gamma\delta^{1/3}m}{\xi n} + \frac{(k-2)\delta^{1/7}m}{2n} \leq \frac{k\delta^{1/7}m}{n}.
	\end{equation}
	We will now bound $d_G(y,R_k)$.
	By the definition of $Y$, $y$ has a non-neighbour $u$ in $R_i$.
	 Then~(\ref{2path}) implies that
	$$
	(k-2)cn - \error \leq P_3(uy,G) \leq a_i + d_G(y,R_k) + |Z_k|.
	$$
	Thus
	$$
	d_G(y,R_k) \stackrel{(\ref{eq-Zik})}{\geq} (k-2)cn-k-a_i-\frac{k\delta^{1/7}m}{n} \stackrel{(\ref{aik-2})}{\geq} t - \frac{\delta^{1/12}m}{n}.
	$$
	But then, using Proposition~\ref{zRinbr} to bound $d_G(y,R_i)$, by a similar calculation to~(\ref{421eq}), we have
	$$
	K_3(F')-K_3(G) \leq -t^2/2,
	$$
	a contradiction.
	Thus $Y_i = \emptyset$ for all $i \in [k-2]$.
\end{proof}

We can now use the lemmas in this section to prove the following penultimate ingredient that we require.
Let
\begin{equation}\label{ABX}
A := \bigcup_{i \in [k-2]}R_i;\quad B := A_{k-1} \cup R_k \cup Z_k^{k-1}\quad\text{ and }\quad X' := \bigcup_{i \in [k-2]}X_i.
\end{equation}
Lemma~\ref{superlinear1}(ii) implies that
\begin{equation}\label{X'}
|X'| \leq \frac{\delta^{1/8}m}{n}.
\end{equation}

\begin{lemma}\label{finalGprops}
	The following properties hold for $G$.
	\begin{itemize}
		\item[(i)] $G$ has vertex partition $A \cup B \cup X'$; $G[A]$ is a complete $(k-2)$-partite graph with parts $R_1,\ldots,R_{k-2}$; and $G[A,B]$ is complete.
		\item[(ii)] There exist $b_1\leq b_2 \in \mathbb{N}$ such that $b_1+b_2=|B|$ and $(b_1-1)(b_2+1)< e(G[B]) \leq b_1b_2$. Moreover, for all $x \in X_i$ with $i\in[k-2]$, we have
		\begin{align*}
		K_3(x,G;\overline{X'}) \geq e(G[L_i]) + |L_i|b_1 + d_G(x,R_i)(|L_i|+b_1) + \alpha m,
		\end{align*}
		where $L_i := A\setminus R_i$.
		\item[(iii)] For all $x \in X'$ we have
		$$
		d_G(x,Z_k^{k-1}) = t \pm \frac{2\delta^{1/12}m}{n},\quad\text{and further}\quad b_1 = cn \pm \frac{\delta^{1/13}m}{n}.
		$$
	\end{itemize}
\end{lemma}

\begin{proof}
	The previous lemma implies that $A_i = R_i$ and $Y_i = \emptyset$ for all $i \in [k-2]$.
	So $A \cup B \cup X'$ is a partition of $V(G)$.
	Property~\Pbadedges($G$) implies that $R_i$ is an independent set in $G$ for all $i \in [k-2]$, which, together with~\Pcomplete($G$), implies that $G[A]$ is a complete $(k-2)$-partite graph with parts $R_1,\ldots,R_{k-2}$. 
	Properties~\Pcomplete($G$),~\PZk($G$) and Lemma~\ref{lem-R1R3} imply that $G[A,B]$ is complete.
	This completes the proof of (i).
	
	For (ii) and (iii), let $x \in X_i \subseteq X'$ for some $i \in [k-2]$.
	Proposition~\ref{Gprops}(i) implies that $E(G[X',R_k])=\emptyset$.
	We need to determine $d_G(x,Z_k^{k-1})$ quite precisely.
	For this, let $u \in R_i$ be arbitrary.
	Then
	\begin{equation}\label{Pux}
	P_3(ux,G) = a_i + d_G(x,Z_k^{k-1}) \pm |X'| \stackrel{(\ref{aik-2}),(\ref{X'})}{=} (k-2)cn - t + d_G(x,Z_k^{k-1}) \pm \frac{\delta^{1/12}m}{n}.
	\end{equation}
	Since $x \in X_i$, we have $d_{\overline{G}}(x,R_i)>0$ by definition.
	Also, since $R_i = A_i$, we have $d_G(x,R_i) \geq \gamma n > 0$.
	That is, $N_G(x,R_i),N_{\overline{G}}(x,R_i) \neq \emptyset$.
	So~(\ref{2path}) implies that the right-hand side of~(\ref{Pux}) lies in $[(k-2)cn-k,(k-2)cn+k]$. Thus
	\begin{equation}\label{Zkk-1}
	d_G(x,Z_k^{k-1}) = t \pm \frac{2\delta^{1/12}m}{n}.
	\end{equation}
	Recall that, by \PZk$(G)$, $G[A_{k-1},X']$ is complete. Thus, all of the $m_{k-1} = (1 \pm k\delta^{1/6})m$ missing edges between $A_{k-1}$ and $A_k$ lie in $B$.
	Then Lemma~\ref{superlinear1}(i) and Lemma~\ref{lem-size} imply that
	\begin{eqnarray}
	\nonumber e(G[B]) &=& |A_{k-1}|(|A_k|-|X'|) -m_{k-1} + \left(e(G[A_{k-1}]) + e(G[A_k])\right)\\
	\nonumber &\stackrel{(\ref{h}),(\ref{X'})}{=}& \left(cn-t\pm\frac{\delta^{1/11}m}{n}\right)\left(n-(k-1)cn+t\pm\frac{2\delta^{1/11}m}{n}\right)- m \pm k\delta^{1/6}m \pm \sqrt{\delta}m
	\\
	\label{eGB}&\stackrel{(\ref{t})}{=}& (c-(k-1)c^2)n^2 \pm \delta^{1/12}m.
	\end{eqnarray}
	Also
	\begin{equation}\label{eq-B}
	|B| = |A_{k-1}|+|A_k|-|X'| = n-(k-2)cn \pm \frac{\delta^{1/12}m}{n}.
	\end{equation}
	A simple calculation using~(\ref{ineq:c}),~(\ref{m}) and~(\ref{eGB}) shows that
	$$
	e(G[B]) \leq \frac{1}{4}\left((1-(k-2)c)n - \frac{\delta^{1/12}m}{n}\right)^2 \leq \frac{|B|^2}{4}.
	$$
	Thus there exist $b_1,b_2 \in \mathbb{N}$ such that $b_1 \leq b_2$ and
	$$
	b_1 + b_2 = |B|
	$$
	and
	$$
	(b_1-1)(b_2+1) < e(G[B]) \leq b_1b_2.
	$$
	Suppose, for a contradiction, that $b_1 > cn + q$, where $q := \delta^{1/13}m/n$.
	Since the product $b_1b_2$ is maximised when $b_1,b_2$ are as balanced as possible, while~(\ref{ineq:c}) and~(\ref{eq-B}) imply that $2(cn+q)>|B|$, we have that
	\begin{eqnarray*}
		b_1b_2 &<& (cn+q)(|B|-cn-q) \stackrel{(\ref{eq-B})}{\leq} (cn+q)(n-(k-1)cn-q)+\delta^{1/12}m\\
		&\leq& cn(n-(k-1)cn) - qn(kc-1) + \delta^{1/12}m \stackrel{(\ref{ineq:c}),(\ref{eGB})}{\leq} e(G[B]) - (\sqrt{2\alpha}\delta^{1/13}-3\delta^{1/12})m\\
		&\stackrel{(\ref{C})}{\leq}& e(G[B])-\sqrt{\alpha}\delta^{1/13}Cn \stackrel{(\ref{hierarchy}),(\ref{Cvalue})}{<} e(G[B])-2n,
	\end{eqnarray*}
	a contradiction.
	Similarly, if $b_1<cn-q$, then $b_1b_2 > e(G[B])+2n$, consequently $(b_1-1)(b_2+1)>e(G[B])$, a contradiction.
	Therefore
	\begin{equation}\label{b1b2}
	b_1 = cn \pm \frac{\delta^{1/13}m}{n} \quad\text{ and so }\quad b_2 = n-(k-1)cn \pm \frac{2\delta^{1/13}m}{n}.
	\end{equation}
	So
	\begin{equation}\label{b1}
	b_1-|A_{k-1}| = t \pm \frac{2\delta^{1/13}m}{n}.
	\end{equation}
	Recall from the statement of the lemma that $L_i = A\setminus R_i$.
	Now, $G[x,L_i \cup A_{k-1}]$ is complete by \PZk($G$).
	Also $G[L_i,A_{k-1},R_i]$ is a complete tripartite graph by \PZk($G$).
	Finally, $e(\overline{G}[A_{k-1},Z_k^{k-1}]) \leq e(\overline{G}[A_{k-1},A_k])=m_{k-1}$ by definition.
	Write $\tilde{\delta} := 2\delta^{1/13}m/n$.
	Thus
	\begin{eqnarray*}
		K_3(x,G;\overline{X'}) &\geq& K_3\left(x,G;L_i \cup A_{k-1}\right) + K_3\left(x,G;R_i \cup Z_k^{k-1},L_i \cup A_{k-1}\right) + K_3(x,G;R_i,Z_k^{k-1})\\
		&\geq& e(G[L_i \cup A_{k-1}]) + (d_G(x,R_i)+d_G(x,Z_k^{k-1}))(|L_i|+|A_{k-1}|)\\
		& &\hspace{2.37cm}+ d_G(x,R_i)d_G(x,Z_k^{k-1})-m_{k-1}\\
		&\stackrel{(\ref{Zkk-1}),(\ref{b1})}{\geq}& e(G[L_i]) + |L_i|(b_1-t -\tilde{\delta})+ (d_G(x,R_i)+t-\tilde{\delta})(|L_i|+b_1-t-\tilde{\delta})\\
		& &\hspace{2.37cm}+ d_G(x,R_i)(t-\tilde{\delta})-m_{k-1}\\
		&\stackrel{(\ref{t}),(\ref{b1b2})}{\geq}& e(G[L_i]) + |L_i|b_1 + d_G(x,R_i)(|L_i|+b_1) - m + \frac{cm}{kc-1} -5\tilde{\delta}n-2\sqrt{\eta}m\\
		&\stackrel{(\ref{ineq:c})}{\geq}& e(G[L_i]) + |L_i|b_1 + d_G(x,R_i)(|L_i|+b_1) + \left(\frac{(k-1)\alpha}{c-(k-1)\alpha}-12\delta^{1/13}\right)m\\
		&\geq& e(G[L_i]) + |L_i|b_1 + d_G(x,R_i)(|L_i|+b_1) + \alpha m,
	\end{eqnarray*}
	as required for (ii).
	Part~(iii) follows immediately from~(\ref{Zkk-1}) and~(\ref{b1b2}).
\end{proof}

To complete the proof, we
first observe that if $X'=\emptyset$, then we are done.
Indeed, in this case, Lemma~\ref{finalGprops}(i) and~(ii) imply that $G$ has partition $A,B$ where $G[A]$ is complete $(k-2)$-partite, $G[A,B]$ is complete, and $e(G[B])\leq t_2(|B|)$.
Thus $K_3(G[B])=g_3(|B|,e(B))=0$ and so $G \in \mathcal{H}_1(n,e)$, a contradiction.
So we may assume that $X' \neq \emptyset$.
Now we will
perform a final global transformation on $G$ to obtain an $(n,e)$-graph $H$ which has fewer triangles.

\begin{center}
\begin{figure}
\includegraphics[scale=1.2]{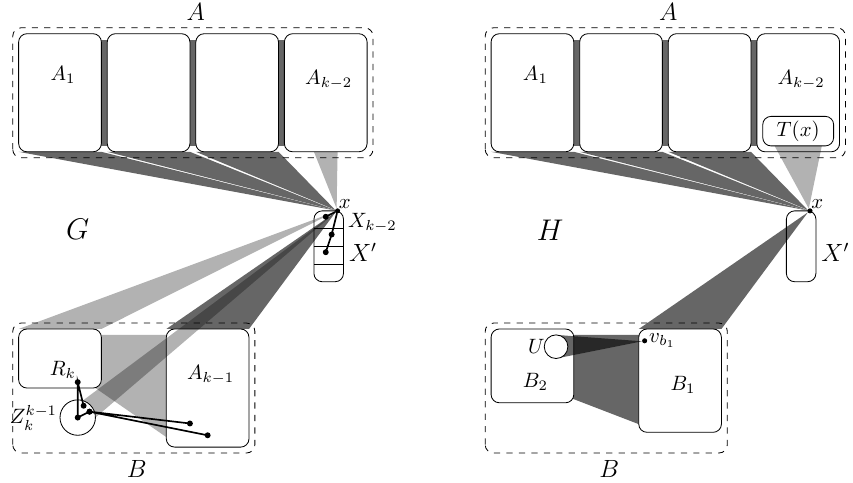}
\caption{$G \rightarrow H$, from the perspective of a single $x \in X_{k-2} \subseteq X'$.}
\label{globaltrans}
\end{figure}
\end{center}

\medskip
\noindent
\emph{Proof of Theorem~\ref{strong} in the intermediate case and when $m \geq Cn$.} We may assume, as observed above, that $X'\neq\emptyset$.
Choose $b_1,b_2$ as in Lemma~\ref{finalGprops}(ii). Let $B_1,B_2$ be an arbitrary partition of $B$ such that $|B_i| = b_i$ for $i \in [2]$.
Let $v_1,\ldots,v_{b_1}$ be an ordering of $B_1$.
Let $U \subseteq B_2$ have size $e(G[B])-(b_1-1)b_2$.
So $0 < |U| \leq b_2$.
Let $x_1,\ldots,x_\ell$ be an arbitrary ordering of $X'$.
For each $g \in [\ell]$, let $s(g) \in [k-2]$ be such that $x_g \in X_{s(g)}$.
Choose an arbitrary set $T(x_g) \subseteq R_{s(g)}$ of size
\begin{eqnarray}
\nonumber |T(x_g)|&=& d_G(x_g,B) + d_G(x_g,\lbrace x_{g+1},\ldots,x_\ell \rbrace)+d_G(x_g,R_{s(g)})-|B_1|\\
\nonumber &=& |A_{k-1}| + d_G(x_g,Z_k^{k-1}) + d_G(x_g,R_{s(g)})-b_1 \pm |X'|\\
\label{dGRi}& \stackrel{(\ref{X'}),(\ref{b1})}{=}& d_G(x_g,R_{s(g)}) \pm \frac{3\delta^{1/13}m}{n}.
\end{eqnarray}
Here we used the facts that $A_{k-1}\subseteq N_G(x_g)$ by \PZk($G$); $R_k \cap N_G(x_g) = \emptyset$ by Proposition~\ref{Gprops}(i); and also Lemma~\ref{finalGprops}(iii).
But by the definition of $X_{s(g)}$, since $R_{s(g)}=A_{s(g)}$ by Lemma~\ref{empty} and using $m \leq \eta n^2$ from~(\ref{m}), the right hand side of~\eqref{dGRi} is at least $\gamma n - 3\delta^{1/13}\eta n \geq \gamma n/2$, and at most $|R_{s(g)}|-\xi n + 3\delta^{1/13}\eta n \leq |R_{s(g)}|-\xi n/2$.
So $T(x_g)$ exists.
Now define a new graph $H$ by setting
\begin{align*}
E(H) &:= \left(E(G) \cup \lbrace v_iy: i \in [b_1-1], y \in B_2 \rbrace \cup \lbrace v_{b_1}y : y \in U \rbrace \cup \bigcup_{x \in X'}\lbrace xy : y \in B_1 \cup T(x) \rbrace\right) \\
&\setminus \left( E(G[B\cup X']) \cup \bigcup_{i \in [k-2]}E(G[X_i,R_i]) \right).
\end{align*}
Thus, informally, $H$ is obtained from $G$ by rearranging the edges in $G[B]$ to form a maximally unbalanced bipartition, and then for each $i \in [k-2]$ replacing the neighbours of $x \in X_i$ which lie in $X' \cup B \cup R_i$ with vertices in $B_1$, and then $R_i$.
See Figure~\ref{globaltrans} for an illustration of $G$ and $H$.

The following claim states some properties of $H$.
\begin{claim}\label{proofclaim}
	\begin{itemize}
		\item[(i)] $H$ is an $(n,e)$-graph such that $H[A,B]$ is complete; $H[A]=G[A]$ and $H[B]$ is bipartite with bipartition $B_1,B_2$, and $X' \neq \emptyset$. Also $E(H[X',B_2])=E(H[X']) = \emptyset$.
		\item[(ii)] Let $T(G)$ be the set of triangles in $G$ containing at least one vertex from $X'$ and define $T(H)$ analogously. Then $|T(H)|\geq|T(G)|$.
	\end{itemize}
\end{claim}

\begin{proof}[Proof of Claim.]
	The first part of (i) follows from Lemma~\ref{finalGprops} and by the construction of $H$.	Since $G[A,B]=H[A,B]$ are both complete, we have that
	\begin{align*}
	K_3(G) &= K_3(G[A]) + K_3(G[B]) + |A|e(G[B]) + |B|e(G[A]) + |T(G)|; \text{ and }\\
	K_3(H) &= K_3(H[A]) + K_3(H[B]) + |A|e(H[B]) + |B|e(H[A]) + |T(H)|.
	\end{align*}
	But $G[A] = H[A]$ and we also have $e(G[B]) = e(H[B])$ from the construction of $H$.
	Moreover $H[B]$ is bipartite so $K_3(H[B])=0$.
	Thus $0\le K_3(H)-K_3(G) = |T(H)|-|T(G)|-K_3(G[B])$.
	Then $|T(H)|\geq |T(G)|$, proving (ii). This completes the proof.
\end{proof}

\medskip
\noindent
In light of the claim, we will obtain a contradiction by showing that in fact $|T(H)|<|T(G)|$. Recall from Claim~\ref{proofclaim}(i) that $X'$ is an independent set in $H$, so there is no triangle in $H$ involving more than one vertex in $X'$, i.e.~$|T(H)|=\sum_{x\in X'}K_3(x,H;\overline{X'})$. By the inclusion-exclusion principle, we have
\begin{equation}\label{THTG}
|T(H)|-|T(G)| \leq \sum_{x \in X'}(K_3(x,H;\overline{X'})-K_3(x,G;\overline{X'})) + |X'|^2\cdot n.
\end{equation}

Let $x \in X'$ and let $i \in [k-2]$ be such that $x \in X_i$.
Let us count the change in triangles involving $x$ and two vertices in $\overline{X'}=A\cup B$.

Define $L_i := A\setminus R_i$ as in the proof of Lemma~\ref{finalGprops}.
By construction, we have that
\begin{itemize}
	\item[(H1)] $R_i \cup L_i \cup B_1 \cup B_2 \cup X'$ is a partition of $V(H)$, and $B_1,B_2,R_i,X'$ are independent sets of $H$.
	\item[(H2)] $L_i \cup B_1 \subseteq N_H(x)$ and $(B_2 \cup X')\cap N_H(x) = \emptyset$ and  $H[R_i,L_i]$, $H[B,L_i]$ are complete bipartite graphs.
\end{itemize}
Thus
\begin{eqnarray*}
	K_3(x,H;\overline{X'}) &\stackrel{(H1),(H2)}{=}& K_3\left(x,H;L_i\right) + K_3\left(x,H;L_i,R_i\right)+K_3(x,H;L_i,B_1)+K_3(x,H;R_i,B_1)\\
	&\stackrel{(H2)}{=}& e(H[L_i]) + |L_i|b_1 + |T(x)|(|L_i|+b_1)\\
	&\stackrel{(\ref{dGRi})}{\leq}& e(G[L_i]) + |L_i|b_1 + d_G(x,R_i)(|L_i|+b_1) + \delta^{1/14}m\\
	&\leq& K_3(x,G;\overline{X'}) - (\alpha -\delta^{1/14})m < K_3(x,G;\overline{X'})-\alpha m/2,
\end{eqnarray*}
where we used Lemma~\ref{finalGprops}(ii) for the penultimate inequality. This together with~(\ref{THTG}) implies that
$$
|T(H)|-|T(G)| \leq -|X'|(\alpha m/2-|X'|n) \stackrel{(\ref{X'})}{\leq} -|X'|\cdot (\alpha/4) \cdot m \leq - \alpha m/4,
$$
a contradiction to Claim~\ref{proofclaim}(ii).
Thus $G$ is not a counterexample to Theorem~\ref{strong}, and we have proved Theorem~\ref{strong} in this case.
\hfill$\square$

\subsection{The intermediate case when $m$ is small}\label{sec:linear}

In this section, we will similarly obtain a contradiction to our assumption that $G$ is a worst counterexample to Theorem~\ref{strong} in the case when
\begin{equation}\label{smallC}
m < Cn.
\end{equation}
This case has a slightly different flavour from the rest of the proof.
Indeed, in all other cases, we are eventually able to obtain from $G$ an $(n,e)$-graph $H$ with strictly fewer triangles than $G$, a contradiction.
However, in the case when $m < Cn$, we can only guarantee an $(n,e)$-graph $H$ with \emph{at most} as many triangles as $G$ but which lies in $\mathcal{H}(n,e)$.
This is enough to prove that $g_3(n,e) = h(n,e)$, but not enough to prove that every extremal graph lies in $\mathcal{H}(n,e)$.
This is not surprising, as when $m < Cn$, our graph $G$ is very close indeed to a graph in $\mathcal{H}(n,e)$.
Recall from the very beginning of the proof in Section~\ref{beginning} that our choice of extremal graph $G$ was not arbitrary: we chose $G$ according to the three criteria~\ref{worst-C1}--\ref{worst-C3}, which ensure that $G$ minimises/maximises certain graph parameters.
Note that~\ref{worst-C3} has not affected the proof until now.
In this part of the proof, we are required to analyse the transformations $G \rightarrow G^1 \rightarrow \ldots \rightarrow G^r \rightarrow H$ that take us from $G$ to $H$. Using $K_3(G) = K_3(G^1) = \ldots = K_3(G^r) = K_3(H)$, we will show that for each $i$ the graph $G^i$ contradicts the choice of $G$ according to~\ref{worst-C1}--\ref{worst-C3}, or $G^i \in \mathcal{H}(n,e)$.
Then some additional work is required to show that this latter consequence implies that actually $G$ itself lies in $\mathcal{H}(n,e)$, also a contradiction.

\medskip
We follow all arguments until the end of Section~\ref{sec:trans}.
In particular, all definitions from Section~\ref{applying} apply.
Now~(\ref{Zsize}) and~(\ref{smallC}) imply that $Z$ has constant size, namely
\begin{equation}\label{smallZ}
|Z| \leq \frac{2C}{\xi}.
\end{equation}
Recall the definition of $R_k'$ in Section~\ref{smalldeg}.
The number of $x \in A_k$ which have at least one neighbour in $Z_k$ is at most $$
\sum_{z \in Z_k}d_G(z,A_k) \stackrel{\Pbadedges(G)}{\leq} |Z_k|\delta n \leq \frac{2C\delta n}{\xi} = \frac{2\sqrt{\delta}n}{\xi} < \frac{\xi n}{2} = |R_k|-|R_k'|,
$$
and so for all $x \in R_k'$ we have $d_G(x,Z_k)=0$. Recalling the definition of $\Delta$ in~(\ref{Delta}), we have
\begin{equation}\label{Delta0}
\Delta = 0.
\end{equation}
This will imply that Transformations~1--3 now do not increase the number of triangles.
Thus by applying Lemmas~\ref{subZiedges},~\ref{subYiedges} and~\ref{Xiedges}, we can easily obtain a graph $G'$ with $K_3(G') = K_3(G)$ in which, for all $i \in [k-1]$, we have $E(G'[A_i]) = \emptyset$ (Lemma~\ref{Aiempty}); $Y_i = \emptyset$ (Lemma~\ref{Yiempty}); and $E(G'[X_i]) = \emptyset$ (Lemma~\ref{GfromG3}).
The final step is to further transform $G'$ to another graph $G'' \in \mathcal{H}(n,e)$ with the same number of triangles.
This proves that $g_3(n,e) = h(n,e)$. 
However, as mentioned above, we must prove that $G' \in \mathcal{H}(n,e)$.
The next subsection contains some auxiliary results which we will need to achieve this.

\subsubsection{Lemmas for characterising extremal graphs}\label{charext}

To compare $G$ to some $H \in \mathcal{H}(n,e)$ which differs slightly from $G$, we need to compare our usual max-cut partition $A_1,\ldots,A_k$ of $G$ with a canonical partition of $H$, which is $A_1^*,\ldots,A_{k-2}^*,B$ when $H \in \mathcal{H}_1(n,e)$, and $A_1^*,\ldots,A_k^*$ when $H \in \mathcal{H}_2(n,e)$.
Recall that, given $U \subseteq V(G)=V(H)$, we say that $G$ and $H$ \emph{only differ at $U$} if $E(G)\bigtriangleup E(H) \subseteq \binom{U}{2}$.
The first lemma will be used in the case when $G' \in \mathcal{H}_1(n,e)$ (this is the easier case).

\begin{lemma}\label{lem-H1}
	Let $H\in\mathcal{H}_1(n,e)$ with $K_3(H)=K_3(G)$, $\Delta(H[A_i])\le 2\gamma n$ for every $i\in[k]$ and $e(H[A_i,A_j])>0$ for every $ij\in{[k]\choose 2}$. Then the following properties hold.
	\begin{itemize}
		\item[(i)] If $A_1^*,\ldots,A_{k-2}^*,B$ is a canonical partition of $H$, then $B=A_p\cup A_q$ for some $pq\in{[k]\choose 2}$
		and there is a permutation $\sigma$ of $[k]$ such that $A_i = A^*_{\sigma(i)}$ for all $i \in [k]\setminus \lbrace p,q\rbrace$.
		Furthermore, $e(\overline{H}[A_s,A_t])>0$ for some $st \in \binom{[k]}{2}$ only if $\lbrace p,q\rbrace = \lbrace s,t \rbrace$.
		\item[(ii)] If $H$ and $G$ only differ at $A_{s'}\cup A_{t'}$, then $H[A_{s'},A_{t'}]$ is complete.
	\end{itemize}
\end{lemma}
\begin{proof}	
	For (i), let $S\in\{A_1^*,\ldots,A_{k-2}^*,B\}$. Suppose for some $i\in[k]$, we have $A_i\cap S,A_i\cap \overline{S}\neq\emptyset$.
	Then, as $H[S,\overline{S}]$ is complete, there exists $v\in A_i$ with 
	$$
	d_H(v,A_i)\ge \frac{|A_i|}{2}\stackrel{\Ppartition(G)}{\geq} \frac{(c-\beta)n}{2} > 2\gamma n\ge\Delta(H[A_i]),
	$$
	a contradiction. So either $A_i\subseteq S$ or $A_i\subseteq\overline{S}$. Since $e(H[A_i,A_j])>0$ for every $ij\in{[k]\choose 2}$, and every $A_p^*$ with $p\in[k-2]$ is an independent set in $H$, $A_p^*$ must contain exactly one $A_i$. This proves the first part of~(i). Suppose now $e(\overline{H}[A_{s},A_t])>0$ for some $st \in \binom{[k]}{2}$. Then the fact that $H[A_1^*,\ldots,A_{k-2}^*,B]$ is complete multipartite implies that $B=A_{s}\cup A_t$.
	
	For (ii), suppose that $H$ and $G$ only differ at $A_{s'}\cup A_{t'}$ and $e(\overline{H}[A_{s'},A_{t'}])>0$. Then by~(i), we have $B=A_{s'}\cup A_{t'}$.
	So $H[\overline{B}]=G[\overline{B}]$ is complete $(k-2)$-partite and $H[B,\overline{B}]=G[B,\overline{B}]$ is complete.
	Since $K_3(H)=K_3(G)$ we have $K_3(G[B])=K_3(H[B])=0$, so $G[B]$ is triangle-free.
	Thus $G\in\mathcal{H}_1(n,e)$ with canonical partition $A_1^*,\ldots,A_{k-2}^*,B$, contradicting the choice of $G$.
\end{proof}

The next lemma analyses a graph $H \in \mathcal{H}_2(n,e)$ obtained by making some small changes to $G$. 

\begin{lemma}\label{lem-H2}
	Let $H\in\mathcal{H}^{\mathrm{min}}_2(n,e)\setminus \mathcal{H}_1(n,e)$ be such that $|E(G)\bigtriangleup E(H)|\le 
	\delta n^2$ and $H[A_i,A_j]$ is complete for every $ij\in{[k-1]\choose 2}$.
	Suppose that
\begin{equation}\label{deq1}
	d := \max_{i \in [k]; v \in V(G)}|d_G(v,A_i)- d_H(v,A_i)| \leq \gamma n.
\end{equation}
	Let $A_1^*,\ldots,A_k^*$ be a canonical partition of $H$.
	Then $R_k\subseteq A_k^*$ and there exists a permutation $\sigma$ of $[k]$ such that $|A_i \bigtriangleup A_{\sigma(i)}^*| \leq k\beta n$ for all $i \in [k]$, and the following properties hold:
	\begin{itemize}
		\item[(i)] If there exists $p \in [k-1]$ for which $Z_k=Z_k^{p}$, then $Z_k \subseteq A_{\sigma(p)}^* \cup A_k^*$. Moreover, there is $j \in [k-1]$ such that $A_{\sigma(i)}^*=A_i$ for all $i \in [k-1]\setminus \lbrace j,p\rbrace$, and if $j \neq p$ then $A_{\sigma(j)}^* \subseteq A_j \subseteq A_{\sigma(j)}^* \cup A_k^*$.
		
		\item[(ii)] If $d \le \delta n$ and $Y=\emptyset$, then $A_k \subseteq A_k^*$, and there is $j \in [k-1]$ such that $A_{\sigma(i)}^*=A_i$ for all $i \in [k-1]\setminus \lbrace j\rbrace$, and $A_{\sigma(j)}^* \subseteq A_j \subseteq A_{\sigma(j)}^* \cup A_k^*$. 	\end{itemize}
\end{lemma}
\begin{proof}
	We require a claim:
	
	\begin{claim}\label{811claim}
		There exists a permutation $\sigma$ of $[k]$ with $\sigma(k)=k$ such that the following hold: 
		\begin{itemize}
			\item[(1)] for all $i \in [k]$ we have $|A_i\bigtriangleup A_{\sigma(i)}^*| \leq k\beta n$;
			\item[(2)] $R_k \subseteq A_k^*$;
			\item[(3)] for all $i \in [k-1]$ we have $A_{\sigma(i)}^*\setminus A_i \subseteq A_k$ and $A_i\setminus A_{\sigma(i)}^* \subseteq A_k^*$;
			\item[(4)] $A_j \subseteq A_{\sigma(j)}^*$ for all but at most one $j \in [k-1]$.
		\end{itemize}
	\end{claim}
	
	\begin{proof}[Proof of Claim.]
    We start with (1).
	Corollary~\ref{H2new}(iii) implies that
	\begin{equation}\label{mH}
		\sum_{ij \in \binom{[k]}{2}}e(\overline{H}[A_i^*,A_j^*]) \leq n.
		\end{equation}
	Further,
	$$
	e(G[A_i]) \stackrel{(\ref{h})}{\leq} \delta m \stackrel{(\ref{smallC})}{\leq} \delta C n \stackrel{(\ref{Cvalue})}{=} \sqrt{\delta}n.
	$$
	Suppose that there exist $i,j \in [k]$ such that $\beta n \leq |A_i \cap A_j^*| \leq |A_i|-\beta n$. 
	Then
	\begin{align*}
	|E(G)\bigtriangleup E(H)| \geq e(H[A_i \cap A_j^*,A_i\setminus A_j^*]) - e(G[A_i]) \stackrel{(\ref{mH})}{\geq} |A_i \cap A_j^*||A_i\setminus A_j^*| - n - \sqrt{\delta}n \geq \frac{\beta^2n^2}{2},
	\end{align*}
	a contradiction.
	Thus, for all $i,j \in [k]$, either $|A_i \cap A_j^*| \leq \beta n$, or $|A_i \cap A_j^*| \geq |A_i|-\beta n$.
	Since for all $i \in [k]$ we have
	$$
	|A_i| \stackrel{\Ppartition(G)}{\geq} n - (k-1)cn-\beta n \stackrel{(\ref{ineq:c})}{\geq} (k-1)\alpha n - \beta n > k\beta n,
	$$
	the first alternative cannot hold for every $j \in [k]$.
	Thus there is exactly one $j \in [k]$ for which $|A_i \cap A_j^*| \geq |A_i|-\beta n$.
	Suppose that there is $j \in [k]$ and $1 \leq i_1 < i_2 \leq k$ such that $|A_{i_p} \cap A_j^*| \geq |A_{i_p}|-\beta n$ for $p \in [2]$.
	Then
	$$
	e(G[A_j^*]) \geq |A_{i_1}\cap A_j^*||A_{i_2} \cap A_j^*|-m \geq (|A_{i_1}|-\beta n)(|A_{i_2}|-\beta n)-\eta n^2 \stackrel{\Ppartition(G)}{>} 2\delta n^2,
	$$
	and so $e(H[A_j^*])>0$, a contradiction.
	That is, there is a permutation $\sigma$ of $[k]$ for which
	\begin{equation}\label{isigmai}
	|A_i \bigtriangleup A_{\sigma(i)}^*| = |A_i \setminus A_{\sigma(i)}^*| + |A_{\sigma(i)}^*\setminus A_i| \leq \beta n + \sum_{j \in [k]\setminus \lbrace \sigma(i)\rbrace} |A_j \cap A^*_{\sigma(i)}| \leq k\beta n.
	\end{equation}
	Since
	$$
	|A_k| \stackrel{\Ppartition(G)}{\leq} n-(k-1)cn+\beta n \stackrel{(\ref{ineq:c})}{\leq} cn - \sqrt{2\alpha}n + \beta n \stackrel{\Ppartition(G)}{<} |A_i| - \sqrt{\alpha} n
	$$
	for all $i \in [k-1]$, we have $|A_{\sigma(k)}^*| \leq |A_{\sigma(i)}^*|-\sqrt{\alpha}n/2$ for all $i \in [k-1]$. Since $|A_1^*| \geq \ldots \geq |A_k^*|$, this implies that $\sigma(k)=k$. This proves (1).
	
	Now, for all $i \in [k]$ and $v \in V(G)$, we have
	\begin{equation}
	\label{deq}
	|d_G(v,A_i)-d_H(v,A_{\sigma(i)}^*)| \stackrel{(\ref{deq1})}{\leq} d + |A_i\bigtriangleup A_{\sigma(i)}^*| \stackrel{(\ref{isigmai})}{\leq} \gamma n + k\beta n \leq 2\gamma n.
	\end{equation}
	For (2), let $x \in R_k$ and $i \in [k-1]$. We have
	$$
	d_H(x,A^*_{\sigma(i)}) \stackrel{(\ref{deq})}{\geq} d_G(x,A_i) - 2\gamma n \stackrel{\Pmissing(G)}{\geq} |A_i| - \xi n - 2\gamma n \stackrel{\Ppartition(G)}{\geq} (c-\beta-\xi-2\gamma)n > 0.
	$$
	Since $A^*_{\sigma(i)}$ is an independent set in $H$, we have that $v \notin A_{\sigma(i)}^*$.
	But $i \in [k-1]$ was arbitrary, so $v \in A_{\sigma(k)}^* = A_k^*$, proving (2).
	
	For (3), suppose that $i \in [k-1]$ and there is some $v \in A_{\sigma(i)}^*\setminus A_i$.
	Then, since $A_{\sigma(i)}^*$ is independent in $H$, we have that
	$$
	d_H(v,A_i) \stackrel{(\ref{deq1})}{\leq} d_G(v,A_i) + d \stackrel{(\ref{deq})}{\leq} d_H(v,A_{\sigma(i)}^*) + 2\gamma n + d \leq 3\gamma n < (c-\beta)n \stackrel{\Ppartition(G)}{<} |A_i|.
	$$
	But $H[A_i,A_j]$ is complete for all $ij \in \binom{[k-1]}{2}$, so $v \notin \bigcup_{j \in [k-1]\setminus \lbrace i \rbrace}A_j$. Thus $v \in A_k$, proving the first part of (3).
	For the second part, suppose that $v \in A_i\setminus A_{\sigma(i)}^*$ and let $j \in [k-1]\setminus \lbrace i \rbrace$.
	Then
	$$
	d_H(v,A_{\sigma(j)}^*) \stackrel{(\ref{deq})}{\geq} d_G(v,A_j) - 2\gamma n \stackrel{\Pcomplete(G)}{=} |A_j| - 2\gamma n \stackrel{\Ppartition(G)}{>} (c-\beta-2\gamma)n > 0.
	$$
	So $u \notin A^*_{\sigma(j)}$ and so $u \in A_k^*$, completing the proof of (3).
	
	\par Finally, for (4), suppose that there is $ij \in \binom{[k-1]}{2}$ for which there exist $v_i \in A_i\setminus A_{\sigma(i)}^*$ and $v_j \in A_j\setminus A_{\sigma(j)}^*$. Since $H[A_i,A_j]$ is complete we have $v_iv_j \in E(H)$. But~(3) implies that $v_i,v_j \in A_k^*$, a contradiction.
	This proves~(4) and completes the proof of the claim.
\end{proof}

\medskip
\noindent
We will now prove Item~(i) of the lemma. So suppose there is $p \in [k-1]$ for which $Z_k = Z_k^p$.
Let $i \in [k-1]\setminus \lbrace p\rbrace$ and $y \in Z_k^p$. Then
	$$
	d_H(y,A^*_{\sigma(i)}) \stackrel{(\ref{deq})}{\geq} d_G(y,A_i) - 2\gamma n \stackrel{\PZk(G)}{=} |A_i| - 2\gamma n \stackrel{\Ppartition(G)}{\geq} (c-\beta-2\gamma)n > 0.
	$$
	Thus $y \notin A_{\sigma(i)}^*$ and so $y \in A_{\sigma(p)}^* \cup A_k^*$.
	Therefore, using Claim~\ref{811claim}(2), $A_k = R_k \cup Z_k^p \subseteq A_{\sigma(p)}^* \cup A_k^*$.
	But, by Claim~\ref{811claim}(3), for all $i \in [k-1]$ we have $A_{\sigma(i)}^* \setminus A_i \subseteq A_k \subseteq A_{\sigma(p)}^* \cup A_k^*$.
	Thus $A_{\sigma(i)}^* \subseteq A_i$ for all $i \in [k-1]\setminus \lbrace p \rbrace$.
	By Claim~\ref{811claim}(4), this implies that there is $j \in [k-1]$ such that $A_i = A_{\sigma(i)}^*$ for all $i \in [k-1]\setminus \lbrace j,p\rbrace$.
	If $j \neq p$, then by Claim~\ref{811claim}(3), $A_{\sigma(j)}^* \subseteq A_j \subseteq A_{\sigma(j)}^* \cup A_k^*$, completing the proof of (i).
	
For (ii), we may now assume that $d \leq \delta n$ and $Y = \emptyset$. Inequality~(\ref{deq}) is replaced by the stronger statement
	\begin{equation}
	\label{deq2}
	|d_G(v,A_i)-d_H(v,A_{\sigma(i)}^*)| \leq d + |A_i\bigtriangleup A_{\sigma(i)}^*| \leq \delta n + k\beta n \leq \sqrt{\beta}n\quad\text{for all }v \in V(G).
	\end{equation}
Let $i \in [k-1]$ and $z \in Z_k = X$. Then, using the definition of $X$,
$$
d_H(z,A^*_{\sigma(i)}) \stackrel{(\ref{deq2})}{\geq} d_G(z,A_i) - \sqrt{\beta} n \geq \gamma n - \sqrt{\beta}n > 0.
$$
Thus $z \notin A_{\sigma(i)}^*$ and so $z \in A_k^*$.
Combining this with Claim~\ref{811claim}(2), we see that again $A_k = R_k \cup X \subseteq A_k^*$.
Then Claim~\ref{811claim}(3) implies that for all $i \in [k-1]$ we have $A_{\sigma(i)}^* \setminus A_i \subseteq A_k \subseteq A_k^*$ and so $A_{\sigma(i)}^* \subseteq A_i$.
By Claim~\ref{811claim}(4), there is $j \in [k-1]$ such that $A_{\sigma(i)}^*=A_i$ for all $i \in [k-1]\setminus \lbrace j \rbrace$; and $A_{\sigma(j)}^* \subseteq A_j \subseteq A_{\sigma(j)}^* \cup A_k^*$. This completes the proof of (ii).
\end{proof}

The final lemma in this subsection will be used to prove that, for all $i \in [k-1]$, we have $E(G[A_i])=\emptyset$ (Lemma~\ref{Aiempty}) and $Y_i=\emptyset$ (Lemma~\ref{Yiempty}).
Its proof uses part~(i) of the previous lemma.

\begin{lemma}\label{moreH2pain}
	Let $p \in [k-1]$ and $z \in A_{p} \cup A_k$ be such that $T := N_G(z,A_{p})$ satisfies $1 \leq |T| \leq \gamma n$; and let $S \subseteq N_{\overline{G}}(z,R_k)$ satisfy $|S|=|T|$.
	Suppose further that $Z_k = Z_k^{p}$ and $P_3(yz,G;R_k)=0$ for all $y \in T$,
	and $G[S,\bigcup_{i \in [k-1]\setminus \lbrace p \rbrace}A_i]$ is complete.
	Obtain $H$ from $G$ by replacing $zy$ for all $y \in T$ with $zx$ for all $x \in S$ and suppose that $K_3(H)=K_3(G)$.
	Then $H$ is an $(n,e)$-graph which does not lie in $\mathcal{H}_2(n,e)\setminus\mathcal{H}_1(n,e)$. 
\end{lemma}

\begin{proof}
	Suppose that the lemma does not hold. Then by definition $H$ is an $(n,e)$-graph and so $H \in \mathcal{H}_2^{\mathrm{min}}(n,e)\setminus\mathcal{H}_1(n,e)$.
	Let $A_1^*,\ldots,A_k^*$ be a canonical partition of $H$.
	Clearly,
	$|E(G)\bigtriangleup E(H)|=|S|+|T| \leq 2\gamma n$ for all $ij \in \binom{[k-1]}{2}$ we have $H[A_i,A_j]=G[A_i,A_j]$ is complete by~\PZk($G$); note also that~(\ref{deq1}) holds.
	So $H$ satisfies the conditions of Lemma~\ref{lem-H2}(i).
	Suppose without loss of generality that the permutation $\sigma$ guaranteed by Lemma~\ref{lem-H2} is the identity permutation.
	By definition, $H$ and $G$ only differ at $A_{p}\cup A_k$.
	We will obtain a contradiction via the next claim.
	
	\begin{claim}
		We have the following properties:
		\begin{itemize}
			\item[(i)] $A_i^*=A_i$ for all $i \in [k-1]\setminus \{p\}$ and $R_k \subseteq A_k^*$;
			\item[(ii)] $z \in A_{p}^*$ and $N_G(z,A_{p})\cap A_{p}^* \neq \emptyset$;
			\item[(iii)] there exists $j \in [k-1]\setminus\{p\}$ for which $G[A_j,R_k]$ is not complete.
		\end{itemize}
	\end{claim}
	
	\begin{proof}[Proof of Claim.]
		We first prove (i). By Lemma~\ref{lem-H2}, $R_k\subseteq A_k^*$ and by Lemma~\ref{lem-H2}(i) there exists $j\in[k-1]$ such that $A_i^* = A_i$ for all $i \in [k-1]\setminus\{j,p\}$ and $A_j^*\subseteq A_j$. We may assume that $j\neq p$ and there is some $v \in A_j\setminus A_j^*\subseteq A_k^*$, for otherwise we are done.
		Further, we have that $S \subseteq R_k \subseteq A_k^*$.
		Thus, recalling that $A_k^*$ is an independent set in $H$ and that $G$ and $H$ only differ at $A_{p}\cup A_k$,
		$$
		0 = d_{H}(v,A_k^*) \geq d_{H}(v,S) = d_G(v,S),
		$$
		a contradiction to the fact that $G[S,\bigcup_{i \in [k-1]\setminus \{p\}}A_i]$ is complete.
		This proves (i).
		
		Thus $A_{p}^* \cup A_k^* = A_{p} \cup A_k =: B$.
		In particular, $H[B]$ is bipartite with bipartition $A_{p}^*,A_k^*$.
		Now, $d_{H}(z,A_k^*) \geq d_{H}(z,R_k) \geq |S| >0$.
		Since $A_k^*$ is an independent set in $H$, we have that $z \in A_{p}^*$.
		Suppose that $T \cap A_{p}^* = \emptyset$.
		Let $G'$ be obtained from $G$ by removing the edges $xz$ for all $x \in T$.
		Then $G' \subseteq H$, and so $G'[B]$ is bipartite with bipartition $A_{p}^*,A_k^*$. Using that $T\cap A_p^*=\emptyset$ and $T\subseteq B$, we see that $T\subseteq A_k^*$. This together with $z \in A_{p}^*$ implies that $G[B]$ is bipartite (with bipartition $A_{p}^*,A_k^*$).
		But the fact that $G$ and $H$ only differ at $A_{p} \cup A_k$ and the definition of $\mathcal{H}_2(n,e)$ imply that $G[\overline{B}]$ is $(k-2)$-partite.
		Thus $G$ is $k$-partite with partition $A_1^*,\ldots, A_k^*$.
		Then Corollary~\ref{H2new}(i) implies that $G \in \mathcal{H}_2(n,e)$, a contradiction.
		Thus $T \cap A_{p}^* \neq \emptyset$.
		This proves~(ii).
		
		For (iii), suppose that $G[A_i,R_k]$ is complete for every $i \in [k-1]\setminus\{p\}$.
		Then, by~\PZk($G$), we have that $A_k = R_k \cup Z_k^{p}$ is complete in $G$ to $\bigcup_{j \in [k-1]\setminus\{p\}}A_j$. Further, \Pcomplete($G$) implies that $G[A_i,A_j]$ is complete for all $ij \in \binom{[k-1]}{2}$.
		Thus $G[B,\overline{B}]$ is complete.
		Now, the facts that $G$ and $H$ only differ at $A_{p} \cup A_k$ and $K_3(H)=K_3(G)$ imply that $K_3(G[B]) = K_3(H[B]) = 0$ since $H[B]$ is bipartite.
		That is, $G[B]$ is triangle-free, so $G \in \mathcal{H}_1(n,e)$, a contradiction to~\ref{worst-C1}.
		This completes the proof of the claim.
	\end{proof}
	
	\medskip\noindent	
	By part~(iii) of the claim, we can choose $j \in [k-1]\setminus\{p\}$; $u \in A_j=A_j^*$ and $v \in R_k \subseteq A_k^*$ such that $uv \notin E(G)$.
	As $G[S,\bigcup_{i \in [k-1]\setminus\{p\}}A_i]$ is complete, we have $v \in R_k\setminus S$ and so $N_G(v)=N_{H}(v)$.
	By part~(ii) of the claim, pick some $y \in T \cap A_{p}^*$.
	Since $H[\lbrace v\rbrace,A_j^*]$ is not complete, we have that $H[\lbrace v\rbrace,A_{p}^*]$ is complete by the definition of $\mathcal{H}_2(n,e)$.
	Thus $\lbrace y,z\rbrace \subseteq N_H(v) = N_G(v)$.
	But then $P_3(yz,G;R_k) \geq 1$, a contradiction.
\end{proof}

\subsubsection{Refining the structure of $G$ via Transformations~1--3}
We now return to our extremal graph $G$ and analyse the effects of Transformations~1--3 on the number of triangles to obtain additional structural information.
To do this, we will apply each `local' transformation once, changing edges at a single vertex to obtain a new graph $G^1$.
This is the part of the proof at which we require the full strength of Lemmas~\ref{subZiedges},~\ref{subYiedges} and~\ref{Xiedges} to carefully analyse $K_3(G^1)-K_3(G)$.
As we mentioned earlier, this turns out to now equal zero, and we show that $G^1 \in \mathcal{H}(n,e)$.

The first step is to apply Transformation~1 (Lemma~\ref{subZiedges}) to show that the only bad edges in $G$ lie in $A_k$.

\begin{lemma}\label{Aiempty}
	$E(G[A_i])=\emptyset$ for all $i \in [k-1]$.
\end{lemma}

\begin{proof}
	Suppose to the contrary that $\bigcup_{i \in [k-1]}E(G[A_i])\neq\emptyset$. Without loss of generality, assume $e(G[A_{k-1}])>0$. Then~\Pbadedges($G$) implies that there is some $z \in Z_{k-1}$ with $d_G(z,A_{k-1}) \ge 1$.
	Let $z =: z_1,\ldots,z_p$ be an ordering of $Z\setminus Z_k$.
	Apply Lemma~\ref{subZiedges} to $G$ to obtain an $(n,e)$-graph $G^1$ which satisfies $J(1,1)$--$J(3,1)$. Then $J(3,1)$ implies that
	$$
	K_3(G^1)-K_3(G) \leq \sum_{y \in N_{G}(z,A_{k-1})}\left(\Delta
	- |Z_k\setminus Z_k^{k-1}| - P_3(yz,G;R_k)\right) \stackrel{(\ref{Delta0})}{\leq} 0.
	$$
	As $K_3(G^1)\ge K_3(G)$, we have equality in the above. Then $J(3,1)$ implies that $G[S,\bigcup_{i\in[k-2]}A_i]$ is complete, where
$S:=N_{G^1\setminus G}(z,R_k) \subseteq  N_{\overline{G}}(z,R_k)$.
	Furthermore, $Z_k=Z_k^{k-1}$ and $P_3(yz,G;R_k)=0$ for all $y\in N_{G}(z,A_{k-1})$.
	
	By $J(2,1)$, for all $i \in [k]$ and $v \in V(G)$, we have 
	$$
	|d_G(v,A_i)-d_{G^1}(v,A_i)| \leq d_G(z,A_{k-1}) \stackrel{\Pbadedges(G)}{\leq} \delta n.
	$$
	We also have that $\Delta(G^1[A_i]) \leq \Delta(G[A_i]) \leq \delta n$.
	Note that
	$$
	\sum_{ij \in \binom{[k]}{2}}e(G^1[A_i,A_j]) = \sum_{ij \in\binom{[k]}{2}}e(G[A_i,A_j]) + d_G(z,A_{k-1}).
	$$
	Since $K_3(G^1)=K_3(G)$, the choice of $G$, in particular~\ref{worst-C2}, implies that we must have $G^1\in\mathcal{H}(n,e)$.
	But $G^1$ satisfies the properties of $H$ in Lemma~\ref{moreH2pain} with $p:=k-1$, so $G^1 \in \mathcal{H}_1^{\mathrm{min}}(n,e)$.
	Then $G^1$ clearly satisfies the hypothesis of Lemma~\ref{lem-H1} and $G^1$ and $G$ only differ at $A_{k-1}\cup A_k$. Lemma~\ref{lem-H1}(ii) implies that $G^1[A_{k-1},A_k]$ is complete. But
	\begin{eqnarray*}
		e(\overline{G^1}[A_{k-1},A_k])&\ge& d_{\overline{G^1}}(z,R_k) =d_{\overline{G}}(z,R_k)-d_G(z,A_{k-1})\ge d_{\overline{G}}(z,A_k)-|Z|-\Delta(G[A_{k-1}])\\
		&\stackrel{\Pmissing(G)}{\ge}& \xi n-\delta n-\delta n\ge \xi n/2,
	\end{eqnarray*}
	a contradiction.
	This completes the proof of the lemma.
\end{proof}

The second step is to apply Transformation~2 (Lemma~\ref{subYiedges}) to show that
$Y$ is empty. Then the only bad edges lie in $A_k$ and by Lemma~\ref{Gprops}, they all have both endpoints in $X$. (By~(\ref{smallZ}) this means that there are only constantly many bad edges.)

\begin{lemma}\label{Yiempty}
	$Y_i = \emptyset$ for all $i \in [k-1]$.
\end{lemma}

\begin{proof}
	Suppose, without loss of generality, that $Y_{k-1}\neq\emptyset$ and fix an arbitrary $y \in Y_{k-1}$. Let $\widehat{A}_i:=A_i$ if $i\in[k-2]$, $\widehat{A}_{k-1}:=A_{k-1}\cup\{y\}$ and $\widehat{A}_k:=A_k\setminus\{y\}$. We may assume that $d_G(y,A_{k-1})\ge 1$, otherwise $\widehat{A}_1,\ldots,\widehat{A}_k$ is a max-cut partition of $G$ which contradicts the choice of $A_1,\ldots,A_k$, in particular~\ref{worst-C3}.
	Let $y=:y_1,y_2,\ldots,y_q$ be an ordering of $Y$.
	Observe that $G$ is a graph which satisfies the conclusions of Lemma~\ref{Ziedges} applied with $\ell := k-1$.
	Thus we can apply Lemma~\ref{subYiedges} to $G$ with $\ell := k-1$
	to obtain a graph $G^1$ satisfying $K(1,1)$--$K(3,1)$.
	By $K(3,1)$,
	$$
	K_3(G^1)-K_3(G) \leq \sum_{x \in N_G(y,A_{k-1})}\left( \Delta - \frac{\xi}{6\gamma}|Z_k\setminus Z_k^{k-1}| - P_3(xy,G;R_k) \right) \leq 0.
	$$
	As $K_3(G^1)\ge K_3(G)$, we have equality in the above. Then $K(3,1)$ implies that $G[S,\cup_{i\in[k-2]}A_i]$ is complete, where $S:=N_{G^1\setminus G}(y,R_k)\subseteq N_{\overline{G}}(y,R_k)$. Furthermore, $Z_k=Z_k^{k-1}=X_{k-1}\cup Y_{k-1}$ and $P_3(xy,G;R_k)=0$ for all $x\in N_{G}(y,A_{k-1})$.
	Since $Z_k=X_{k-1}\cup Y_{k-1}$, by $K(1,1)$, $G^1$ is obtained from $G$ by replacing all edges from $y$ to $A_{k-1}$ with some non-edges from $y$ to $R_k$, i.e.~$T(y)$ and~$R(y)$ are empty. Also by $K(1,1)$, we have that $\sum_{ij\in{[k]\choose 2}}e(G^1[\widehat{A}_i,\widehat{A}_j])\ge \sum_{ij\in{[k]\choose 2}}e(G[A_i,A_j])$.
	Since $K_3(G^1)=K_3(G)$ we must have equality by (C2).
	But for all $i \in [k-1]$ we have $|\widehat{A}_i| \geq |\widehat{A}_{k}|=|A_k|-1$ so (C3) implies that $G^1 \in \mathcal{H}^{\mathrm{min}}(n,e)$.
	Again, $G^1$ satisfies the properties of $H$ in Lemma~\ref{moreH2pain} with $k-1,y,G^1$ playing the roles of $p,z,H$ respectively.
	So we have that $G^1 \in \mathcal{H}_1^{\mathrm{min}}(n,e)$.
	
	Let $A_1^*,\ldots,A_{k-2}^*,B$ be a canonical partition of $G^1$. Note that $G^1$ satisfies the hypothesis of Lemma~\ref{lem-H1}. Indeed,
	$$
	\Delta(G^1[A_k])\le \Delta(G[A_k])+d_G(y,A_{k-1})\stackrel{\Pbadedges(G)}{\le} \delta n+\gamma n\le 2\gamma n.
	$$
	Further, $G^1$ and $G$ only differ on $A_{k-1}\cup A_k$. Thus Lemma~\ref{lem-H1}(ii) implies that $G^1[A_{k-1},A_k]$ is complete. But by construction,
	\begin{eqnarray*}
		e(\overline{G^1}[A_{k-1},A_k])&\ge& d_{\overline{G^1}}(y,A_{k-1})=|A_{k-1}|,
	\end{eqnarray*}
	a contradiction.
	This completes the proof of the lemma.
\end{proof}

\subsection{Obtaining a graph $G_3$.}

We will apply Lemma~\ref{Xiedges} to $G$ to obtain a graph $G_3$ in which $X_i$ is an independent set for all $i \in [k-1]$, but such that $G_3$ may contain constantly many more triangles than $G$.
Then, applying further transformations to $G_3$, we deduce additional information about $G$.

Observe that by Propositions~\ref{Aiempty} and~\ref{Yiempty}, $G$ satisfies all the properties of $G_2$ in Lemma~\ref{Yiedges}, so we can set $G_2 := G$ and, for all $i \in [k-1]$, set $A_i' := A_i$.
Recall from the beginning of Section~\ref{sec7.4} that, for all $i \in [k-1]$ and $x,y \in X_i$, we define
\begin{equation}\label{Dx}
\nonumber D(x) := d_{G}(x,X\setminus X_i) \quad\text{ and }\quad D(x,y) := |N_{G}(x,X\setminus X_i) \cap N_{G}(y,X\setminus X_i)|.
\end{equation}

\begin{lemma}\label{G3prop}
	Let $G_3$ be the $(n,e)$-graph obtained by applying Lemma~\ref{Xiedges} to $G$ playing the role of $G_2$. Then
	\begin{itemize}
		\item[(i)] $G_3$ has an $(A_1,\ldots,A_k;Z,2\beta,\xi/4,2\xi,\delta)$-partition and, for each $i \in [k-1]$, we have $e(\overline{G_3}[A_i,A_k]) \leq m_i$ with equality if and only if $E(G[X_i])=\emptyset$.
		\item[(ii)] For all $i \in [k-1]$, $E(G_3[A_i])=\emptyset$ and $E(G_3[A_k]) = E(G[X_1,\ldots,X_k])$ and $d_{G_3}(x,A_i) \geq \gamma n$ for $x \in X_i$. 
		Further, every pair in $E(G)\setminus E(G_3)$ lies in $X_i$ for some $i \in [k-1]$, and every pair in $E(G_3)\setminus E(G)$ lies in $[X_i,A_i]$ for some $i \in [k-1]$.
		\item[(iii)] For all $i \in [k-1]$ such that $X_i \neq \emptyset$, there exists $D_i \in \mathbb{N}$ such that $D(x)=D_i$ for all $x \in X_i$. Moreover, $P_3(xu,G_3)=a_i+D_i$ for all $x \in X_i$ and $u \in A_i$.
		\item[(iv)] $K_3(G_3) \leq K_3(G) + |Z|^2 \cdot \max_{\stackrel{i \in [k-1]}{x,y \in X_i}}(D_i-D(x,y))$ with equality only if $G[X_i]$ is triangle-free and $N_G(x,A_i) \cap N_G(y,A_i)=\emptyset$ for all $i\in[k-1]$ and $xy \in E(G[X_i])$.
		\item[(v)] Let $G'$ be such that $V(G')=V(G_3)$ and $E(G')\bigtriangleup E(G_3) \subseteq \bigcup_{i \in [k-1]}\lbrace ax: a \in A_i,x \in X_i \rbrace$ and $e(G'[X_i,A_i])=e(G_3[X_i,A_i])$ for all $i \in [k-1]$. Then $K_3(G')=K_3(G_3)$.
	\end{itemize}
\end{lemma}

\begin{proof}
	Parts (i) and (ii) and the fact that
	\begin{equation}\label{LXicons}
	K_3(G_3) \leq K_3(G) + |Z|^2 \cdot \max_{\stackrel{i \in [k-1]}{x,y \in X_i}}(D(x)-D(x,y))
	\end{equation}
	with equality only if $G[X_i]$ is triangle-free and $N_G(x,A_i)\cap N_G(y,A_i)=\emptyset$ for all $i\in[k-1]$ and $xy \in E(G[X_i])$ follow immediately from Lemma~\ref{Xiedges} and Lemma~\ref{subXiedges}~L(2).
	Apply Lemma~\ref{symmetrise} to $G_3$ to obtain an $(n,e)$-graph $G_4$ on the same vertex set satisfying Lemma~\ref{symmetrise}(i)--(v). 
	Then, by Lemma~\ref{symmetrise}(i), for every $xy \in E(G_3)\bigtriangleup E(G_4)$ there exists $i \in [k-1]$ such that $x \in X_i$ and $y \in A_i$.
	Let $i \in [k-1]$, $u \in A_i$ and $x \in X_i$.
	Then by Lemma~\ref{Xiedges}(ii) and~\ref{symmetrise}(i),(iii) we have, for $j\in\{3,4\}$, that $d_{G_j}(u,A_i)=d_{G_j}(x,R_k)=d_{G_j}(x,X_i)=0$ and $X\setminus X_i \subseteq N_{G_j}(u)$. So $P_3(xu,G_j)=a_i+D(x)$.
	Clearly if $G'$ is any graph as in~(v), then these equalities also hold for $G'$, in particular
	\begin{equation}\label{P334}
	P_3(xu,G') = a_i + D(x) = P_3(xu,G_j).
	\end{equation}
	Suppose that there exists $i \in [k-1]$ and $x,y \in X_i$ such that $D(x) \neq D(y)$.
	Then Lemma~\ref{symmetrise}(iv) implies that
	\begin{align}
	\label{K3G4} K_3(G_4) &\leq K_3(G_3) - \frac{\xi n}{20} \stackrel{(\ref{LXicons})}{\leq} K_3(G) + |Z|^2 \cdot \max_{\stackrel{i \in [k-1]}{x,y \in X_i}}(D(x)-D(x,y)) - \frac{\xi n}{20}\\
	\nonumber &\leq K_3(G) + |Z|^3 - \frac{\xi n}{20}
	\stackrel{(\ref{smallZ})}{\leq} K_3(G) + \frac{8C^3}{\xi^3}-\frac{\xi n}{20} < K_3(G)-\frac{\xi n}{30},
	\end{align}
	a contradiction.
	This proves (iii), and together with~(\ref{LXicons}), we also obtain~(iv).
	For (v), observe that there is no triangle in $G_3$ or $G'$ which contains more than one $A_i$-$X_i$ edge, since $A_i$ and $X_i$ are independent sets in both graphs.
	Thus
	$$
	K_3(G')-K_3(G_3) = \sum_{e \in E(G')\setminus E(G_3)}P_3(e,G') - \sum_{e \in E(G_3)\setminus E(G')}P_3(e,G_3) \stackrel{(\ref{P334})}{=} 0,
	$$
	where the last equality follows from the hypotheses on $G'$ in (v) and~(\ref{P334}).
\end{proof}

This allows us to conclude that $G$ and $G_3$ are in fact the same graph.

\begin{lemma}\label{GfromG3}
	The following hold in $G$:
	\begin{itemize}
		\item[(i)] for all $ij \in \binom{[k-1]}{2}$, the graph $G[X_i,X_j]$ is either complete or empty. 
		\item[(ii)] $G=G_3$, so $E(G[X_i])=\emptyset$ for all $i \in [k-1]$.
	\end{itemize}
\end{lemma}

\begin{proof}
	First we will show the following claim:
	\begin{claim}\label{edgesum} If $ij \in \binom{[k-1]}{2}$ is such that $E(G[X_i,X_j]) \neq \emptyset$, then
	\begin{equation}
	e(G_3[X_i,A_i])+e(G_3[X_j,A_j]) \leq cn+\sqrt{\beta}n.
	\end{equation}
	\end{claim}
	
	\begin{proof}[Proof of Claim.]
	To prove the claim, let $x \in X_i$ and $y \in X_j$ such that $xy \in E(G)$.
	Then $xy \in E(G_3)$ by Lemma~\ref{G3prop}(ii).
	By Lemma~\ref{G3prop}(v), we can obtain a graph $G'$ from $G_3$ with the stated properties and such that
	\begin{equation}\label{G'deg}
	d_{G'}(x,A_i) = \min\lbrace |A_i|,e(G_3[X_i,A_i])\rbrace \quad \text{ and }\quad d_{G'}(y,A_j) = \min \lbrace |A_j|,e(G_3[X_j,A_j])\rbrace.
	\end{equation}
	That is, we obtain $G'$ by moving as many $X_i$-$A_i$ edges as possible to $x$, and similarly for $y$ and $X_j$-$A_j$ edges.
	By \PZk($G_3$), $x$ is complete to $\bigcup_{\ell \in [k-1]\setminus \lbrace i \rbrace}A_\ell$ in $G_3$ and $y$ is complete to $\bigcup_{\ell \in [k-1]\setminus \lbrace j \rbrace}A_\ell$ in $G_3$.
	Thus the same is true in $G'$.
	Therefore, using Lemma~\ref{G3prop}(iv) and~(v),
	\begin{equation}\label{eq-K3Gp}
		K_3(G')=K_3(G_3)\le K_3(G)+|Z|^3\stackrel{(\ref{smallZ})}{\le} K_3(G)+\frac{8C^3}{\xi^3}\le K_3(G)+\frac{\beta n}{2}.
	\end{equation}
	Corollary~\ref{cr:dh} applied with $p:=\beta n/2$ implies that
	$$
	(k-2)cn + \beta n \geq P_3(xy,G') \geq \sum_{\ell \in [k-1]\setminus \lbrace i,j\rbrace}|A_\ell| + d_{G'}(x,A_i) + d_{G'}(y,A_j)
	$$ and so
	\begin{equation}\label{sumeq1}
	d_{G'}(x,A_i) + d_{G'}(y,A_j) \stackrel{\Ppartition(G)}{\leq} (k-2)cn+\beta n - (k-3)(cn - \beta n) \leq cn + \sqrt{\beta}n.
	\end{equation}
	Now, \Ppartition($G$) implies that $|A_i|+|A_j| \geq 2cn-2\beta n>cn+\sqrt{\beta}n$, so without loss of generality from~(\ref{G'deg}) we may suppose that $d_{G'}(x,A_i)=e(G_3[X_i,A_i])$.
	If $d_{G'}(y,A_j)=|A_j|$, then 
	$$
	e(G_3[X_i,A_i]) \leq cn + \sqrt{\beta}n - |A_j| \stackrel{\Ppartition(G)}{\leq} cn+\sqrt{\beta}n-(cn-\beta n) \leq 2\sqrt{\beta}n.
	$$
	But this is a contradiction because $e(G_3[X_i,A_i]) \geq d_{G_3}(x,A_i) \geq \gamma n$ by Lemma~\ref{G3prop}(ii).
	Thus $d_{G'}(y,A_j)=e(G_3[X_j,A_j])$, and the claim follows from~(\ref{sumeq1}).
	\end{proof}
	
	\medskip
	\noindent
	Suppose that (i) does not hold.
	Then there exist $ij \in \binom{[k-1]}{2}$; $xy \in E(G[X_i,X_j])$ and $x'y' \in E(\overline{G}[X_i,X_j])$ such that $x,x' \in X_i$ and $y,y' \in X_j$.
	These adjacencies are the same in $G_3$.
	Without loss of generality, we may assume that $x \neq x'$ (but it could be the case that $y=y'$).
	In particular, $|X_i| \geq 2$.

	\begin{claim}\label{G''}
		There exists a graph $G''$ which satisfies Lemma~\ref{G3prop}(v) and such that
		$$
		d_{G''}(x',A_i)+d_{G''}(y',A_j) \leq cn-\xi n/5.
		$$ 
	\end{claim}
	
	\begin{proof}[Proof of Claim.]
			Let
			$$
			p_i :=  e(G_3[X_i,A_i])-2\sqrt{\beta}n.
			$$
		We claim that there is some $G''$ such that $E(G'')\bigtriangleup E(G_3) \subseteq \lbrace av:a \in A_i,v \in X_i\rbrace$ and $e(G''[X_i,A_i])=e(G_3[X_i,A_i])$ in which $d_{G''}(x,A_i)=p_i$.
		To show that $G''$ exists, since $p_i < e(G_3[X_i,A_i])$ and $|X_i| \geq 2$, it suffices to show that $p_i \leq |A_i|$, then we can obtain $G''$ by moving all but $2\sqrt{\beta}n$ $X_i$-$A_i$ edges to $x$.
		But this does indeed hold: Claim~\ref{edgesum} implies that
		$$
		p_i \leq cn - \sqrt{\beta}n \stackrel{\Ppartition(G)}{<} |A_i|,
		$$
		as required. We have
$$
		e(G''[X_i,A_i])=e(G_3[X_i,A_i])=p_i+2\sqrt{\beta}n = d_{G''}(x,A_i)+2\sqrt{\beta}n.
		$$
		Thus $d_{G''}(x',A_i) \leq 2\sqrt{\beta}n$.
		Furthermore,
		\begin{eqnarray*}
			d_{G''}(y',A_j) &=&  d_{G_3}(y',A_j) \stackrel{\Pmissing(G_3)}{\leq} |A_j|-\xi n/4 \stackrel{\Ppartition(G)}{\leq} cn+\beta n - \xi n/4.
		\end{eqnarray*}
		Then $d_{G''}(x',A_i)+d_{G''}(y',A_j) \leq cn + \beta n + 2\sqrt{\beta}n - \xi n/4 \leq cn - \xi n/5$, as required.
	\end{proof}
	
	\medskip
	\noindent
	Apply Claim~\ref{G''} to obtain $G''$.
	Proposition~\ref{Gprops}(i) implies that $N_G(x')$ and $N_G(y')$ are disjoint from $R_k'$. This remains true with $G$ replaced by $G''$, i.e. we have that~$N_{G''}(x')\cap R_k'=\emptyset$ and $N_{G''}(y')\cap R_k'=\emptyset$. Indeed, this follows from Lemma~\ref{G3prop}(ii) and that $G''$ and $G_3$ only differ on $[X_i,A_i]$.
	Thus
	\begin{eqnarray}
	\nonumber P_3(x'y',G'') &\leq& \sum_{\ell \in [k-1]\setminus \lbrace i,j\rbrace}|A_\ell| + d_{G''}(x',A_i) + d_{G''}(y',A_j) + |Z|\\
	\label{G''6}&\stackrel{\Ppartition(G),(\ref{smallZ})}{\leq}& (k-3)(c+\beta)n + cn-\frac{\xi n}{5} + \frac{2C}{\xi} \leq (k-2)cn - \frac{\xi n}{6}. 
	\end{eqnarray}
	On the other hand, by Lemma~\ref{G3prop}(v) and the analogue of~\eqref{eq-K3Gp}, $K_3(G'')=K_3(G_3)\le K_3(G)+8C^3/\xi^3$. As $x'y'\not\in E(G'')$, Corollary~\ref{cr:dh} implies that $P_3(x'y',G'')\ge (k-2)cn-k-8C^3/\xi^3$, contradicting~\eqref{G''6}.
	This completes the proof of (i).
	
	We now turn to (ii).
	We claim first that $K_3(G_3)=K_3(G)$.
	Indeed, for all $i \in [k-1]$ and $x,y\in X_i$, we have
	$$
	D(x,y) = \sum_{\substack{\ell \in [k-1]: \\G_3[X_i,X_\ell]\text{ complete}}}|X_\ell| = D(x) = D(y) = D_i.
	$$
	Then Lemma~\ref{G3prop}(iv) implies that $K_3(G_3)=K_3(G)$.
	
	Recall $m^{(3)} = \sum_{ij \in \binom{[k]}{2}}e(\overline{G_3}[A_i,A_j])$ and Lemma~\ref{G3prop}(i) implies that $m^{(3)} \leq m$ with equality if and only if $E(G[X_i])=\emptyset$ for all $i \in [k-1]$.
	Thus if $m^{(3)}=m$, then Lemma~\ref{G3prop}(ii) implies that $G_3=G$ as desired. We may then assume that $m^{(3)}<m$ and, without loss of generality, that $e(G[X_{k-1}])>0$. By Lemma~\ref{G3prop}(ii), this means that $G_3$ has more cross-edges with respect to $A_1,\ldots,A_k$ than $G$. As $K_3(G_3)=K_3(G)$, by the choice of $G$, in particular~\ref{worst-C2}, we must have
	$G_3\in\mathcal{H}(n,e)$.
	
	For all $i\in[k-1]$ such that $X_{i}\neq\emptyset$, we have
	\begin{equation}\label{eq-Xi}
	e(\overline{G_3}[A_{i},A_k])= e(\overline{G}[A_{i},A_k])-e(G[X_{i}])\stackrel{\Pbadedges(G),\Pmissing(G)}{\ge} |X_{i}|(\xi n-
	\delta n)>0.     
	\end{equation}
	
	Suppose first that $G_3\in\mathcal{H}_1(n,e)$ and $A_1^*,\ldots,A_{k-2}^*,B$ is a canonical partition of $G_3$. By construction, $G_3$ satisfies the hypotheses of Lemma~\ref{lem-H1}.
	Recall that $e(G[X_{k-1}])>0$, in particular, $X_{k-1}\neq\emptyset$. Then~\eqref{eq-Xi} and Lemma~\ref{lem-H1}(i) imply that $B=A_{k-1}\cup A_k$, $G_3[A_i,B]$ is complete and $X_i=\emptyset$ for every $i\in[k-2]$. 
	(There can only be one $i \in [k-1]$ such that $e(\overline{G_3}[A_i,A_k])>0$, so~(\ref{eq-Xi}) and the fact that $X_{k-1}\neq \emptyset$ implies that $X_i=\emptyset$ for all $i \in [k-2]$.)
	But then $G_3$ and $G$ only differ at $A_{k-1}\cup A_k$ and Lemma~\ref{lem-H1}(ii) implies that $G_3[A_{k-1},A_k]$ is  complete, contradicting~\eqref{eq-Xi}.
	
	We may now assume that $G_3\in\mathcal{H}^{\min}_2(n,e)\setminus \mathcal{H}_1(n,e)$ and let $A_1^*,\ldots,A_k^*$ be a canonical partition of $G_3$. We claim that $G_3$ satisfies the hypotheses of Lemma~\ref{lem-H2}(ii).
	Indeed, by Lemma~\ref{G3prop}(ii), \Pmissing($G$) and~(\ref{smallZ}), $|E(G)\bigtriangleup E(G_3)| \leq |Z|^2 \leq \delta n^2$ . Also, $G_3[A_i,A_j]$ is complete for all $ij \in \binom{[k-1]}{2}$ by \Pcomplete($G_3$).
	Finally, $d \leq |Z|^2 < \delta n$ and $Y =\emptyset$ by Proposition~\ref{Yiempty}.
	
	Recall that $X_{k-1}\neq\emptyset$.
	By Lemma~\ref{lem-H2}(ii), 
	\begin{equation}\label{eq-Xkm1}
	X_{k-1}\subseteq A_k\subseteq A_k^*
	\end{equation}
	and there is a bikection $\sigma : [k-1]\rightarrow [k-1]$ and at most one $j \in [k-1]$ such that $A_{\sigma(\ell)}^*=A_\ell$ for all $\ell \in [k-1]\setminus \lbrace j \rbrace$, and $A^*_{\sigma(j)} \subseteq A_j \subseteq A_{\sigma(j)}^* \cup A_k^*$.
	Without loss of generality, assume that $\sigma$ is the identity permutation.
	By \PZk($G_3$), we have that $G_3[X_{k-1},A_\ell]$ is complete for every $\ell \leq k-2$.
	But $X_{k-1} \subseteq A_k^*$ so $A_\ell \cap A_k^*=\emptyset$.
	Thus $A_\ell=A_\ell^*$.
	Therefore $j$ can only be $k-1$ if it exists,~i.e.~$A^*_{k-1} \subseteq A_{k-1} \subseteq A^*_{k-1} \cup A_k^*$.
    But $A_{k-1}^* \cup A_k^* = A_{k-1}\cup A_k$ so $A_k \subseteq A_k^*$. So
	\begin{equation}\label{eq-Astar}
	|A_{k-1}^*|-|A_k^*|\le |A_{k-1}|-|A_k|\stackrel{\Ppartition(G)}{\le} (kc-1+2\beta)n\stackrel{(\ref{ineq:c})}{<} (c-(k-1)\alpha)n + 2\beta n < (c-\alpha)n.
	\end{equation}
	Fix an arbitrary edge $xy\in E(G[X_{k-1}])$.  Note that as $X\subseteq A_k \subseteq A_k^*$ is independent in $G_3$, for every $ij \in \binom{[k-1]}{2}$ we have that $[X_i,X_j]$ is empty in $G_3$, and hence also in $G$  as they are identical at $\bigcup_{ij\in{[k-1]\choose 2}}[X_i,X_j]$.  So $D_i=0$ for all $i \in [k-1]$ . Since $K_3(G_3)=K_3(G)$, by Lemma~\ref{G3prop}(iv) we have that $G[X_i]$ is triangle-free for every $i\in[k-1]$, and $N_G(x,A_i)\cap N_G(y,A_i)=\emptyset$. That is, $x$ and $y$ have no common $A_i$-neighbour in $G$.
	So
	$$
     e(\overline{G}[A_{k-1},\{x,y\}]) \ge |A_{k-1}|\stackrel{\Ppartition(G)}{\ge} (c-\beta)n.
	$$
	By~\eqref{eq-Xkm1}, $\{x,y\}\subseteq X_{k-1}\subseteq A_k^*$, and recall that from $G$ to  $G_3$, at most $|Z|^2$ adjacencies are changed in $[A_{k-1},X]$. Lemma~\ref{lem-H2} implies that $|A_{k-1}\setminus A_{k-1}^*|\le |A_{k-1} \bigtriangleup A_{k-1}^*| \leq k\beta n$. So
	\begin{eqnarray*}
	e(\overline{G_3}[A_{k-1}^*,A_k^*])&\ge& e(\overline{G_3}[A_{k-1}^*,\{x,y\}])\geq  e(\overline{G_3}[A_{k-1},\{x,y\}])-2|A_{k-1}\setminus A_{k-1}^*|\\
	&\ge& e(\overline{G}[A_{k-1},\{x,y\}])-|Z|^2-2|A_{k-1}\setminus A_{k-1}^*|\stackrel{(\ref{smallZ})}{\ge} (c-\beta)n-\frac{4C^2}{\xi^2}-2k\beta n  \\
	&>& (c-\alpha/2)n\stackrel{(\ref{eq-Astar})}{>} |A_{k-1}^*|-|A_k^*|+1,
	\end{eqnarray*} 
	contradicting Corollary~\ref{H2new}(iii).
	This completes the proof of the lemma.
\end{proof}

For $ij \in \binom{[k-1]}{2}$, we write $X_i \sim X_j$ if $G[X_i,X_j]$ is complete and $X_i \not\sim X_j$ if $G[X_i,X_j]$ is empty (recall that exactly one of these holds for every pair $ij$ by Lemma~\ref{GfromG3}(i)).
Thus for all $i \in [k-1]$,
$$
D_i = \sum_{\ell \in [k-1]:X_\ell \sim X_i}|X_\ell|.
$$

\begin{proposition}\label{diffij}
	The following hold.
	\begin{itemize}
		\item[(i)] Let $i,j \in [k-1]$ be such that $X_i,X_j \neq \emptyset$.
		Then $a_i + D_i = a_j + D_j$;
		\item[(ii)] if $G'$ is an $(n,e)$-graph with $E(G')\bigtriangleup E(G) \subseteq \bigcup_{i \in [k-1]}K[X_i,A_i]$ then $K_3(G')=K_3(G)$.
	\end{itemize}
\end{proposition}

\begin{proof}
	Choose arbitrary $i,j \in [k-1]$ and $x \in X_i$ and $x' \in X_j$.
	We obtain (i) by performing a transformation on $G$.
	First observe that, by the definition of $X$ and \Pmissing($G$), we have $\gamma n \leq d(x,A_i) \leq |A_i|-\xi n$.
	So there exist sets $K(x) \subseteq N_{G}(x,A_i)$ and $\overline{K}(x) \subseteq N_{\overline{G}}(x,A_i)$ of size $\xi n$, and equally-sized subsets $K(x') \subseteq N_{G}(x',A_j)$ and $\overline{K}(x') \subseteq N_{\overline{G}}(x',A_j)$.
	Let $J$ be obtained from $G$ by adding $\lbrace xv:v \in \overline{K}(x)\rbrace$ and removing $\lbrace x'u': u' \in K(x')\rbrace$.
	Let $J'$ be obtained from $G$ by adding $\lbrace x'v':v' \in \overline{K}(x')\rbrace$ and removing $\lbrace xu:u \in K(x)\rbrace$.
	For all $a \in A_i$ and $a' \in A_j$ we have by Lemma~\ref{G3prop}(iii), Lemma~\ref{GfromG3}(ii) and the constructions of $J$ and $J'$ that
	\begin{align*}
	P_3(xa,J)&=P_3(xa,J') = P_3(xa,G) = a_i+D_i\quad\text{and}\\
	P_3(x'a',J)&=P_3(x'a',J')=P_3(x'a',G) = a_j+D_j.
	\end{align*}
	
	Since $A_i$, $A_j$ are independent sets in $G$ by Proposition~\ref{Aiempty}, there are no triangles in $J$ containing both edges $xv_1,xv_2$ for distinct $v_1,v_2 \in \overline{K}(x)$; and no triangles in $J$ containing both edges $x'v_1',x'v_2'$ for distinct $v_1',v_2' \in K(x')$.
	Thus
	$$
	K_3(J) - K_3(G) = \sum_{v \in \overline{K}(x)}P_3(xv,J)-\sum_{u \in K(x')}P_3(x'u',G) = \xi n\left( a_i+D_i-a_j-D_j\right)
	$$
	and similarly $K_3(J')-K_3(G)=\xi n(a_j+D_j-a_i-D_i) = -(K_3(J)-K_3(G))$.
	If $a_i+D_i \neq a_j+D_j$, then either $J$ or $J'$ has at least $\xi n$ fewer triangles than $G$, a contradiction.
	Thus $a_i+D_i=a_j+D_j$ for all $i,j \in [k-1]$ for which $X_i,X_j \neq \emptyset$.
	This proves (i).
	
	For (ii), it suffices to show that, for any $i,j \in [k-1]$, if $G'$ is obtained from $G$ by replacing one $X_i$-$A_i$ edge $e_i$ with one $X_j$-$A_j$ edge $e_j$, then $K_3(G)=K_3(G')$. Then this can be iterated to obtain any required $G'$.
	But this follows from (i) since
	$$
	K_3(G')-K_3(G) = P_3(e_j,G')-P_3(e_i,G)=P_3(e_j,G)-P_3(e_i,G)=a_j+D_j - a_i-D_i = 0.
	$$
\end{proof}

It is now easy to complete the proof of Theorem~\ref{strong} in the case under consideration.

\medskip
\noindent
\emph{Proof of Theorem~\ref{strong} in the intermediate case and when $m < Cn$.}
Propositions~\ref{Aiempty} and~\ref{Yiempty} imply that $A_1,\ldots,A_{k-1}$ are independent sets in $G$ and $Y=\emptyset$.
By Proposition~\ref{Gprops}(i), every edge in $G[A_k]$ has both endpoints in $X$.
Now Lemma~\ref{GfromG3} implies that $xy \in E(G[A_k])$ only if there are $ij \in \binom{[k-1]}{2}$ such that $x \in X_i$ and $y \in X_j$.

If $E(G[X])=\emptyset$, then $G$ is $k$-partite.
But then we obtain a contradiction via Corollary~\ref{H2new}(i).
Thus we may
choose $xy \in E(G[X])$ with $x\in X_i$ and $y\in X_j$ for some $ij\in{[k-1]\choose 2}$.
Note that $d_G(x,A_i)>0$ by the definition~(\ref{XY}) of $X_i$.
Let $G'$ be an $(n,e)$-graph obtained from $G$ by successively replacing arbitrary $x$-$A_i$ edges with arbitrary $y$-$A_j$ non-edges until
\begin{itemize}	
	\item[(S1)] $d_{G'}(x,A_i)=1$; or
	
	\item[(S2)] $d_{G'}(y,A_j)=|A_j|$ and $d_{G'}(x,A_i)\ge 1$.
\end{itemize}

We claim that in both cases $d_{G'}(x,A_i)\le \sqrt{\beta}n$. This is clearly true if (S1) holds. If (S2) holds, note that \begin{eqnarray*}
	(k-2)cn+k\stackrel{(\ref{2path})}{\ge} P_3(xy,G)&\ge& \sum_{\ell\in[k-1]\setminus\{i,j\}}|A_{\ell}|+d_{G}(x,A_i)+d_{G}(y,A_j)\\
	&\stackrel{\Ppartition(G)}{\ge}& (k-3)(c-\beta)n+d_{G}(x,A_i)+d_{G}(y,A_j).
\end{eqnarray*}
Thus
$$d_{G'}(x,A_i)=d_{G}(x,A_i)+d_{G}(y,A_j)-d_{G'}(y,A_j)\stackrel{(S2)}{\le} cn+k\beta n-|A_j|\stackrel{\Ppartition(G)}{\le}
\sqrt{\beta} n,$$
as required.
Note that $E(G') \bigtriangleup E(G) \subseteq K[X_i,A_i]\cup K[X_j,A_j]$. So by Proposition~\ref{diffij}(ii), we have $K_3(G')=K_3(G)$. Recall that, by Proposition~\ref{Gprops}(i), in $G$ and also in $G'$, there is no edge between $X$ and $R_k$. Then we can replace all $x$-$A_i$ edges in $G'$ with $x$-$R_k$ non-edges to obtain a new graph $G''$. This is possible as $$
|R_k| \stackrel{\Ppartition(G),\Pbadedges(G)}{\geq} (1-(k-1)c-\beta)n-|Z| \stackrel{(\ref{ineq:c}),(\ref{Zsize})}{\geq} \sqrt{\alpha}n \geq \sqrt{\beta} n\ge d_{G'}(x,A_i).
$$
Fix arbitrary $u\in A_i$ and $u'\in R_k$. 
Note that $\bigcup_{\ell \in [k-1]\setminus \lbrace i \rbrace}A_{\ell} \subseteq N_G(x) \cap N_G(u)$ by \Pcomplete($G$) and \PZk($G$).
Further, $y \in N_G(u) \cap N_G(x)$ by the definition of $X_j \ni y$.
Both of these statements also hold for $G'$.
Thus $P_3(xu,G') \geq a_i+1$.
But $P_3(xu',G'') = a_i$ since $d_{G''}(x,A_i)=0$ and every $X$-$R_k$ edge is incident to $x$ in $G''$.
Thus
$$K_3(G'')-K_3(G)=K_3(G'')-K_3(G')\le -1\cdot d_{G'}(x,A_i)\stackrel{(S1),(S2)}{\le} -1,$$
a contradiction.

This completes the proof of Theorem~\ref{strong} in the intermediate case when $m < Cn$.
\hfill$\square$

\section{The boundary case}\label{bound}
We have shown that no worst counterexample to Theorem~\ref{strong} can satisfy~(\ref{worst}) and~(\ref{eq:alpha}).
That is, we can assume that
\begin{equation}\label{eq:alpha2}
t_k(n) -\alpha n^2 \leq e \leq t_k(n)-1,
\end{equation}
which we refer to as the \emph{boundary case}.
Let
\begin{equation}\label{eq:r}
r := t_k(n)-e \leq \alpha n^2.
\end{equation}
So $r \geq 1$.
Now, Lemmas~\ref{lm:Turan41fact} and~\ref{lm:(k-1)c<1} and~(\ref{eq:c2}) imply that $k(n,e)=k(2e/n^2)$ and
\begin{equation}\label{eq-cn-bT}
\frac{n}{k}+\sqrt{\frac{r}{{k\choose 2}}} \leq cn \leq \frac{n}{k}+\sqrt{\frac{r+k/8}{{k\choose 2}}}\quad\text{and so}\quad\frac{\sqrt{r}}{k} \leq cn-\frac{n}{k} \leq \sqrt{r}.
\end{equation}
Therefore
\begin{equation}\label{ineq:c2}
\frac{n}{k} < cn \leq \frac{n}{k} + \sqrt{\alpha}n.
\end{equation}
A useful consequence of this is that
\begin{equation}\label{betterdiff}
1-(k-1)c \geq \frac{1}{k}-(k-1)\sqrt{\alpha} > \frac{1}{2k}.
\end{equation}

\subsection{The boundary case: approximate structure}

The first step is to obtain an analogue of Lemma~\ref{approx}.
Let
\begin{equation}\label{eq:D}
D := 169k^{k+9}.
\end{equation}

\begin{lemma}[Approximate structure]\label{approx2}
	Suppose that~(\ref{eq:alpha2}) holds.
	Let $G$ be a worst counterexample as defined in Section~\ref{beginningproof} and let $A_1,\ldots,A_k$ be a max-cut partition of $V(G)$. Let $m := \sum_{ij \in \binom{[k]}{2}}e(\overline{G}[A_i,A_j])$ and $h := \sum_{i \in [k]}e(G[A_i])$.
	Then there exists $Z \subseteq V(G)$ such that $G$ has a weak $(A_1,\ldots,A_k;Z,\sqrt{Dr}/n,\xi',\xi',\delta')$-partition in which, for all $i \in [k]$,
	\begin{equation}\label{P1G}
	\bigg| |A_i| - \frac{n}{k} \bigg|, \bigg| |A_i| - cn \bigg| \leq \sqrt{Dr}, \quad m\leq Dr\quad\text{and}\quad h \leq \delta' m.
	\end{equation}
\end{lemma}

Recall from Section~\ref{partitions} that a weak partition requires that \Ppartition, \Pbadedges~and~\Pmissing~all hold with the appropriate parameters.
Note that the partition in Lemma~\ref{approx2} is in terms of primed constants $\xi',\delta'$ which are both large compared to $\alpha$, unlike $\xi,\delta$ in the intermediate case which are small compared to $\alpha$.

We will need the following analogue of Lemma~\ref{lem-Almost-k-Partite}, which is essentially the same as Theorem~2 in~\cite{LovaszSimonovits83}. Since this theorem is not phrased in a way applicable to our situation, we reprove it here.
In fact this lemma applies for all, say, $r \leq \frac{n^2}{2k^2}$, but is only meaningful when $r=o(n^2)$.

\begin{lemma}\label{almostcompleteagain}
There exist integers $n_1,\ldots,n_k$ summing to $n$ with $|n_i-n/k|,|n_i-cn| \leq 6k^{\frac{k+3}{2}}\sqrt{r}$ for all $i \in [k]$ such that
	$|E(G)\bigtriangleup E(K_{n_1,\ldots,n_k})| < 40k^{k+4}r$.
\end{lemma}

\begin{proof}
Define $s \in \mathbb{R}$ by setting
\begin{equation}\label{tnew}
e = \left(1-\frac{1}{s}\right)\frac{n^2}{2}\quad\text{and so}\quad\frac{2r}{n^2} \leq \frac{1}{s}-\frac{1}{k} \leq \frac{2(r+k/8)}{n^2} \stackrel{(\ref{eq:r})}{\leq} 3\alpha.
\end{equation} 
(Here we used Lemma~\ref{lm:Turan41fact}.)
For $0 \leq i \leq 3$, write $N_i$ for the (unique) $3$-vertex graph with exactly $i$ edges, and write $N_i(G)$ for the number of induced copies of $N_i$ in $G$.
So, for example, $N_3(G)=K_3(G)$.
We claim that 
\begin{equation}\label{lseq}
K_3(G) = \binom{s}{3}\left(\frac{n}{s}\right)^3 + \frac{1}{3}\left(\sum_{x \in V(G)}q_G(x)^2 + N_1(G)\right),
\end{equation}
where $q_G(x) := 2e/n-d_G(x)$.
This is a special case of inequality~(14) in~\cite{LovaszSimonovits83}, but we repeat the simple proof of this case here for the reader's convenience.

For each edge $f$ of $G$ and $1 \leq i \leq 3$, let $n_{i,f}$ denote the number of vertices adjacent to exactly $i-1$ vertices of $f$.
Then for all $f \in E(G)$ we have $n_{1,f}+n_{2,f}+n_{3,f}=n-2$, and
$
\sum_{f \in E(G)}n_{i,f} = iN_i(G)
$.
So
\begin{equation}\label{LSproof1}
e(n-2)=3N_3(G) + 2N_2(G)+N_1(G).
\end{equation}
Additionally,
$$
2(N_2(G) + 3N_3(G)) = 2\sum_{v \in V(G)}\binom{d_G(v)}{2} = \sum_{v \in V(G)}2\binom{2e/n-q_G(v)}{2} = \frac{4e^2}{n} + \sum_{v \in V(G)}q_G(v)^2 - 2e,
$$
where we used the fact that $\sum_{v \in V(G)}q_G(v)=0$.
Thus
\begin{eqnarray*}
en &\stackrel{(\ref{LSproof1})}{=}& -3N_3(G) + \frac{4e^2}{n^2} + \sum_{v \in V(G)}q_G(v)^2 +N_1(G).
\end{eqnarray*}
So
\begin{align*}
K_3(G) &= N_3(G) = \frac{1}{3}\left( e\cdot\left(\frac{4e}{n}-n\right) + \sum_{v \in V(G)}q_G(v)^2 + N_1(G)\right)\\
&= \frac{1}{3}\left( \binom{s}{2}\left(\frac{n}{s}\right)^2\left(\frac{s-1}{2}\cdot\frac{4n}{s}-n\right) + \sum_{v \in V(G)}q_G(v)^2 + N_1(G)\right)\\
&= \binom{s}{3}\left(\frac{n}{s}\right)^3 + \frac{1}{3}\left(\sum_{v \in V(G)}q_G(v)^2 + N_1(G)\right),
\end{align*}
as required. 

\medskip
We now consider $G$.
Certainly $G$ has at most as many triangles as the $(n,e)$-graph obtained by deleting $r$ edges between the two smallest classes of $T_k(n)$. By convexity, $K_3(T_k(n))\leq \binom{k}{3}(n/k)^3$, so
\begin{align*}
K_3(G) &\leq K_3(T_k(n))-r\left(n-2\left\lfloor\frac{n}{k}\right\rfloor\right)
\stackrel{(\ref{tnew})}{\leq} \binom{s}{3}\left(\frac{n}{s}\right)^3 + rn+\frac{kn}{8}\le \binom{s}{3}\left(\frac{n}{s}\right)^3 + rkn.
\end{align*}
Thus~(\ref{lseq}) implies that
\begin{equation}\label{props}
\sum_{x \in V(G)}q_G(x)^2 \leq 3rkn\quad\text{and}\quad N_1(G) \leq 3rkn.
\end{equation}
Let $W$ be an arbitrary copy of $K_k$ in $G$.
Let $A_W$ denote the set of vertices adjacent in $G$ to at most $k-2$ vertices in $W$.
Each vertex in $A_W$ lies in at least one copy of $N_1$ (together with any pair of its missing neighbours in $W$).
On the other hand, for every copy of $N_1$, its single edge lies in at most $n^{k-2}$ copies of $K_k$. Thus
$$
\sum_{W \subseteq G: W \cong K_k}|A_W| \leq N_1(G) \cdot n^{k-2} \leq 3rkn^{k-1}.
$$
Denote by $B_W$ the set of $xy \in E(\overline{G})$ such that $d_G(x,V(W))=k-1$ and either (i) $d_G(y,V(W))=k-1$ but $N_G(x,V(W)) \neq N_G(y,V(W))$; or (ii) $d_G(y,V(W))=k$.
Then for every $xy \in B_W$, there is $z \in V(W)$ such that $x,y,z$ span a copy of $N_1$ in $G$, where $x$ plays the role of the isolated vertex.
On the other hand, there are at most $n^{k-1}$ copies of $K_k$ which contain $z$.
Thus
$$
\sum_{W \subseteq G: W \cong K_k}|B_W| \leq N_1(G)\cdot n^{k-1} \leq 3rkn^k.
$$
Let $q_W := \sum_{x \in V(W)}q_G(x)^2$.
Any $x \in V(G)$ lies in at most $n^{k-1}$ copies of $K_k$, so
$$
\sum_{W \subseteq G: W \cong K_k}q_W \leq 3rkn^k.
$$
Thus
$$
\sum_{W \subseteq G: W \cong K_k} \left( n|A_W|+|B_W|+q_W\right) \leq 9rkn^k.
$$
Now, $G$ certainly contains many copies of $K_k$. For example, Theorem~1 in~\cite{LovaszSimonovits83} implies that 
$$
K_k(G) \geq g_k(n,e) \geq \binom{s}{k}\left(\frac{n}{s}\right)^k \stackrel{(\ref{tnew})}{\geq} \frac{1}{2}\left(\frac{n}{k}\right)^k.
$$
Thus, by averaging, there exists a copy $W$ of $K_k$ in $G$ for which
\begin{equation}\label{Wprops}
|A_W| \leq \frac{18rk^{k+1}}{n};\quad |B_W| \leq 18rk^{k+1};\quad\text{and}\quad |q_G(x)| \leq 3\sqrt{2rk^{k+1}}\quad\text{for all }x \in V(W).
\end{equation}
We will use this $W$ to construct a partition of $V(G)$.
Let $w_1,\ldots,w_k$ be the vertices of $W$.
For all $i \in [k]$, let $C_i := \lbrace x \in V(G): N_{\overline{G}}(x,V(W)) = \lbrace w_i \rbrace\rbrace$.
Let also
$$
C_0 := \lbrace x \in V(G): d_G(x,V(W))=k \rbrace\quad\text{and}\quad C_{k+1} :=A_W.
$$
So $C_0,\ldots,C_{k+1}$ is a partition of $V(G)$.

We will now estimate the sizes of each of these sets.
We have that
\begin{equation}\label{Ck+1}
|C_{k+1}|=|A_W| \leq \frac{18rk^{k+1}}{n} \stackrel{(\ref{eq:r})}{\leq} 18k^{k+1}\sqrt{\alpha}\sqrt{r}.
\end{equation}

Now,~\eqref{Wprops} implies that, for all $i \in [k]$,
$$
\bigg| d_G(w_i) - \left(1-\frac{1}{s}\right)n\bigg| = |q_G(w_i)| \leq 3\sqrt{2rk^{k+1}}.
$$
But
$$
d_G(w_i) = |C_0| + \sum_{j \in [k]\setminus \lbrace i \rbrace}|C_j| + d_G(w_i,C_{k+1}) = n - |C_i| \pm |C_{k+1}|,
$$
so
\begin{eqnarray*}
	|C_i| &=& \frac{n}{s} \pm \left(3\sqrt{2rk^{k+1}} + |C_{k+1}|\right) \stackrel{(\ref{tnew}),(\ref{Ck+1})}{=} \frac{n}{k} \pm \left( \frac{2(r+k/8)}{n} + 3\sqrt{2rk^{k+1}} + 18k^{k+1}\sqrt{\alpha}\sqrt{r} \right)\\
	&\stackrel{(\ref{eq:r})}{=}& \frac{n}{k} \pm \left( 3\sqrt{\alpha} + 3\sqrt{2k^{k+1}} + 18k^{k+1}\sqrt{\alpha}\right)\sqrt{r} = \frac{n}{k} \pm 5k^{\frac{k+1}{2}}\sqrt{r}.
\end{eqnarray*}
Thus $|C_0| \leq 5k^{\frac{k+3}{2}}\sqrt{r}$.

For each $i \in \lbrace 2,\ldots,k\rbrace$, let $A_i := C_i$ and let $A_1 := C_0 \cup C_1 \cup C_{k+1}$.
So, for all $i \in [k]$,
\begin{eqnarray}
\nonumber \bigg| |A_i|-cn\bigg|,\bigg| |A_i|-\frac{n}{k}\bigg| &\leq& \bigg||C_i|-\frac{n}{k}\bigg| + |C_0| + |C_{k+1}| + \bigg| \frac{n}{k}-cn\bigg|\\
\nonumber &\stackrel{(\ref{eq-cn-bT})}{\leq}& 5k^{\frac{k+1}{2}}\sqrt{r} + 5k^{\frac{k+3}{2}}\sqrt{r} + 18k^{k+1}\sqrt{\alpha}\sqrt{r} + \sqrt{r}\\
\label{Aidiff} &\leq& 6k^{\frac{k+3}{2}}\sqrt{r}.
\end{eqnarray}
Let $ij \in \binom{[k]}{2}$ and $\overline{e} \in E(\overline{G}[A_i,A_j])$.
Then, by definition, either $\overline{e} \in B_W$ or $\lbrace x,y \rbrace \cap A_W \neq\emptyset$ (note that any such $\overline{e}$ intersecting $C_0$ lies in $B_W$).
Thus by~\eqref{Wprops} and~\eqref{Ck+1}, we have
\begin{equation}\label{missing2}
\sum_{ij \in \binom{[k]}{2}}e(\overline{G}[A_i,A_j]) \leq |B_W|+|C_{k+1}|n \leq 36k^{k+1}r. 
\end{equation}
Let $d_i := n/k-|A_i|$ for all $i \in [k]$.
Now, $\sum_{i \in [k]}d_i=0$ and 
\begin{align*}
\sum_{ij \in \binom{[k]}{2}}|A_i||A_j| = \frac{1}{2}\left(n^2 - \sum_{i \in [k]}\left(\left(\frac{n}{k}\right)^2 - \frac{2d_in}{k} + d_i^2\right)\right)\geq t_k(n) - k \cdot \max_{i \in [k]}\left\lbrace d_i^2\right\rbrace \stackrel{(\ref{Aidiff})}{\geq} e-36k^{k+4}r.
\end{align*}
Thus
$$
\sum_{i \in [k]}e(G[A_i]) = e - \sum_{ij \in \binom{[k]}{2}}(|A_i||A_j|-e(\overline{G}[A_i,A_j])) \stackrel{(\ref{missing2})}{\leq} 36k^{k+4}r + 36k^{k+1}r \leq 38k^{k+4}r
$$
and so, letting $n_i := |A_i|$ for all $i \in [k]$, we have
$$
|E(G)\bigtriangleup E(K_{n_1,\ldots,n_k})| \le 36k^{k+1}r + 38k^{k+4}r < 40k^{k+4}r,
$$
as required.
\end{proof}

The previous lemma together with Lemma~\ref{baddegclaim} combine to prove Lemma~\ref{approx2}.

\medskip
\noindent
\emph{Proof of Lemma~\ref{approx2}.}
Choose a max-cut $k$-partition $V(G)=A_1 \cup \ldots \cup A_k$.
Let 
\begin{equation}\label{Zdefagain}
Z := \bigcup_{i \in [k]}\lbrace z \in A_i:d_G(z,\overline{A_i}) \geq \xi' n\rbrace.
\end{equation}
(In fact there can be no other choice for $Z$.)
We need to show that 
\Ppartition($G$) holds with parameter $\sqrt{Dr}/n$, \Pbadedges($G$) holds with parameter $\delta'$, and \Pmissing($G$) holds with parameter $\xi'$.

Let $p := 6k^{\frac{k+3}{2}}\sqrt{r}$, $d := 40k^{k+4}r$ and $\rho := 40k^{k+4}\alpha$.
Then
$$
p^2 = 36k^{k+3}r < d \leq \rho n^2\quad\text{and}\quad 2\rho^{1/6} \leq 4k^{k/6+4/6}\alpha^{1/6} < \alpha^{1/7} < \frac{1}{2k} \stackrel{(\ref{betterdiff})}{<} 1 - (k-1)c.
$$
Thus, by Lemma~\ref{almostcompleteagain}, we can apply Lemma~\ref{baddegclaim} with parameters $p,d,\rho$ to imply that $A_1,\ldots,A_k$ satisfy conclusions~(i)--(v) of Lemma~\ref{baddegclaim}.

Thus, by~(i), \Ppartition($G$) holds with parameter $2k^2\sqrt{d}/n = 2\sqrt{40}k^{k/2+4}\sqrt{r}/n$. This together with~\eqref{eq-cn-bT} and~\eqref{eq:D} implies the required bound on $\left|\,|A_i|-\frac{n}{k}\right|$ and thus \Ppartition($G$) holds with parameter $\sqrt{Dr}/n$.
Lemma~\ref{baddegclaim}(ii) implies that
\begin{equation}\label{mcalc}
m:=\sum_{ij \in \binom{[k]}{2}}e(\overline{G}[A_i,A_j]) \leq 2k^2\sqrt{d}(kc-1)n + d \stackrel{(\ref{eq-cn-bT})}{\leq} (6\sqrt{40}k^{k/2+5} + 40k^{k+4})r \stackrel{(\ref{eq:D})}{<} Dr.
\end{equation}
For \Pbadedges($G$), as in the intermediate case, every missing edge is incident to at most two vertices in $Z$, so
\begin{equation}\label{Z1}
|Z| \leq \frac{2m}{\xi' n} \leq \frac{2Dr}{\xi'n} < \frac{2D\alpha n}{\xi'} \stackrel{(\ref{eq:D})}{<} \delta'n.
\end{equation}
Furthermore, Lemma~\ref{baddegclaim}(iii) implies that for every $i \in [k]$ and $e \in E(G[A_i])$, there is at least one endpoint $x$ of $e$ with
$$
d_{\overline{G}}(x,\overline{A_i}) \geq \frac{1}{2}\left((1-(k-1)c)n - 3k^2\sqrt{\rho}n\right) \stackrel{(\ref{betterdiff})}{\geq} \frac{1}{2}\left(\frac{1}{2k}-3\sqrt{40}k^{k/2+4}\sqrt{\alpha} \right)n > \frac{n}{5k} > \xi' n.
$$
Thus $x \in Z$.
The final part of \Pbadedges($G$) follows from Lemma~\ref{baddegclaim}(iv) and the fact that $\alpha \ll \delta'$.
Finally, \Pmissing($G$) holds immediately from the definition of $Z$.
The assertion about $m$ was proved in~(\ref{mcalc}) and the assertion about $h$ is an immediate consequence of Lemma~\ref{baddegclaim}(v) and the fact that $\alpha \ll \delta'$.
\hfill$\square$

\subsection{The boundary case: the remainder of the proof}
Apply Lemma~\ref{approx2} to the worst counterexample $G$ as defined in Section~\ref{beginningproof} (so $G$ satisfies~\ref{worst-C1}--\ref{worst-C3}). Now fix a weak $(A_1,\ldots,A_k;Z,\sqrt{Dr}/n,\xi',\xi',\delta')$-partition of $G$ with $Z$ (uniquely) defined as in~\eqref{Zdefagain} and define $m$ as in the statement. For all $i \in [k]$, let
$$
R_i := A_i \setminus Z.
$$
As before, \Pbadedges($G$) implies that $R_i$ is an independent set for all $i \in [k]$.
Suppose first that $Z=\emptyset$. Then $G$ is a $k$-partite graph.
So Corollary~\ref{H2new}(i) implies that $G \in \mathcal{H}_2(n,e)$, a contradiction.
Thus, exactly as in~(\ref{Z1}),
\begin{equation}\label{mbound}
1 \leq |Z| \leq \frac{2m}{\xi'n}\quad\text{and}\quad\xi' \leq \frac{2m}{n}. 
\end{equation}

Given disjoint subsets $A,B \subseteq V(G)$, write $A \sim B$ if $G[A,B]$ is complete.
For any $I \subseteq [k]$, write
$$
R_I := \bigcup_{i \in I}R_i.
$$
We would like to measure quite accurately the difference between $|R_I|/|I|$ and its `expected' size $cn$ for $I \neq \emptyset$ (recalling that $cn$, $n-(k-1)cn$ and $n/k$ are all very close in the boundary case).
Thus we define
\begin{equation}\label{diffeq}
\nonumber \mathrm{diff}(I) := \left(\frac{|R_I|}{|I|}-cn\right)\frac{n}{m},\quad\text{i.e.}\quad|R_I| = \left(cn+\mathrm{diff}(I)\cdot\frac{m}{n}\right)|I|.
\end{equation}
We will write $\mathrm{diff}(i)$ as shorthand for $\mathrm{diff}(\lbrace i \rbrace)$.
A trivial but useful observation is that, for pairwise disjoint $I_1,\ldots,I_p \subseteq [k]$, we have
\begin{equation}\label{diffav}
\min_{i \in [p]}\lbrace \mathrm{diff}(I_i)\rbrace \leq \mathrm{diff}(I_1\cup\ldots \cup I_p) \leq \max_{i \in [p]}\lbrace \mathrm{diff}(I_i)\rbrace.
\end{equation}
Note also that
\begin{equation}\label{large}
\left(cn-\frac{m}{\alpha^{1/3}n}\right)k \stackrel{(\ref{eq-cn-bT}),(\ref{P1G})}{\geq} n + \sqrt{r} - \frac{kDr}{\alpha^{1/3}n} \stackrel{(\ref{eq:r})}{\geq} n+ \sqrt{r}\left(1 - kD\alpha^{1/6}\right) \stackrel{(\ref{eq:D})}{>} n,
\end{equation}
so we have the following:
\begin{itemize}
	\item[($*$)]
	If $I \subseteq [k]$ satisfies $\mathrm{diff}(I) \geq -1/\alpha^{1/3}$, then $|R_I|>|I|n/k$.
\end{itemize}

We cannot guarantee that \Pcomplete($G$) and \PZk($G$) hold in this setting since there is no part which is significantly smaller than the other parts. However, the next lemma shows that an analogue of these properties holds.

\begin{lemma}\label{lem-partition-Z}
	There exists a partition $Z=\bigcup_{I\in {[k]\choose k-2}} Z_I$ of $Z$ such that, for all $ij \in \binom{[k]}{2}$, the following properties hold. We have $Z_{[k]\setminus \lbrace i,j\rbrace}\sim R_{[k]\setminus \lbrace i,j\rbrace}$, $Z_{[k]\setminus \lbrace i,j\rbrace} \subseteq A_i \cup A_j$ and for every $z \in Z_{[k]\setminus \lbrace i,j\rbrace} \cap A_i$ we have that $d_G(z,R_i) \leq \delta'n$ and $d_{\overline{G}}(z,R_j) \geq \xi'n/2$.
\end{lemma}

\begin{proof}
	Let $z \in Z$ be arbitrary, and let $i \in [k]$ be such that $z \in A_i$.
	By the definition of $Z$, there is some $j \in [k]\setminus \lbrace i \rbrace$ such that $d_{\overline{G}}(z,A_j) \geq \xi'n/k$.
	Let $I := [k]\setminus \lbrace i,j\rbrace$ and $x \in R_I$ be arbitrary, and let $h \in I$ be such that $x \in R_h$.
	Then
	\begin{eqnarray*}
		P_3(zx,G) &\leq& d_G(z,A_i)+d_G(z,A_j)+d_G(x,A_h)+(n-|A_i|-|A_j|-|A_h|)\\
		&\stackrel{\Pbadedges(G),(\ref{P1G})}{\leq}& 2\delta'n + n - 2\left(\frac{n}{k}-\sqrt{Dr}\right) - \frac{\xi'n}{k} \stackrel{(\ref{eq:r})}{<} (k-2)\cdot\frac{n}{k}-\frac{\xi'n}{2k} \stackrel{(\ref{ineq:c2})}{<} (k-2)cn - \frac{\xi'n}{3k}.
	\end{eqnarray*}
	Thus~(\ref{2path}) implies that $xz \in E(G)$.
	Since $x$ was arbitrary, we have shown that we can assign $z$ to $Z_{[k]\setminus\lbrace i,j\rbrace}$.
	The second statement follows from \Pbadedges($G$), which says, since $z \in A_i$, that $d_G(z,R_i) \leq d_G(z,A_i) \leq \delta'n$ and \Pmissing($G$), which together with the first statement says that $d_{\overline{G}}(z,R_j) \geq \xi'n-|Z| \geq (\xi'-\delta')n \geq \xi'n/2$.
\end{proof}

The next lemma shows that $\mathrm{diff}(I)$ can only be large when $|I| \leq k-2$.

\begin{lemma}\label{lem-no-large-set}
	If $I \subseteq [k]$ has $\mathrm{diff}(I) \geq -1/\alpha^{1/3}$, then $|I|\le k-2$.
\end{lemma}
\begin{proof}
	Note first that, by~($*$), we have $\mathrm{diff}([k]) < -1/\alpha^{1/3}$.
	Suppose that there exists a set $I \in \binom{[k]}{k-1}$ such that $\mathrm{diff}(I) \geq -1/\alpha^{1/3}$. Without loss of generality, suppose that $I = [k-1]$. Let $q := \frac{m}{\alpha^{1/3}n}$. Then~($*$) implies that $|R_I| \geq (k-1)n/k$.
	Since $\sum_{ij \in \binom{k}{2}}|A_i|\,|A_j|$ is maximised when the parts $A_i$ are as balanced as possible and $cn-q\ge n/k$ due to~\eqref{eq-cn-bT} and~\eqref{P1G}, we have
	\begin{align*}
	e+m-\sum_{i \in [k]}e(G[A_i]) &= \sum_{ij \in \binom{[k]}{2}}|A_i||A_j| \leq e(K^k_{cn-q,\ldots,cn-q,n-(k-1)(cn-q)})\\
	&= e - \binom{k}{2}q^2 + (k-1)q(kc-1)n \stackrel{(\ref{eq-cn-bT})}{\le} e+\frac{(k-1)m}{\alpha^{1/3}n}\cdot k\sqrt{\frac{r+k/8}{\binom{k}{2}}}\\
	&\leq e+\frac{2km\sqrt{r+k}}{\alpha^{1/3}n} \stackrel{(\ref{eq:r})}{\leq} e + 3k\alpha^{1/6}m.
	\end{align*}
	But then
	$$
	\sum_{i \in [k]}e(G[A_i]) \geq (1-3k\alpha^{1/6})m > \sqrt{\delta'}m,
	$$
	a contradiction to Lemma~\ref{approx2}.
\end{proof}

We now show that if there is a missing edge between some $R_i$ and $R_j$, where $i \neq j$, then the union of the other sets $R_\ell$ must be large.

\begin{lemma}\label{lem-large-set-from-missing-edge}
	For all $ij \in \binom{[k]}{2}$, if $R_i\not\sim R_j$, then $\mathrm{diff}([k]\setminus\{i,j\}) \geq -1/(2\alpha^{1/3})$.
\end{lemma}
\begin{proof}
	Set $I:=[k]\setminus\{i,j\}$.
	Since $R_i \not\sim R_j$, there exists $x \in R_i$ and $y \in R_j$ such that $xy \notin E(G)$.
	Then, since $R_i$ and $R_j$ are both independent sets in $G$,
	$$
	(k-2)cn -k \stackrel{(\ref{2path})}{\leq} P_3(xy,G) \leq |Z| + |R_I| \stackrel{(\ref{mbound})}{\leq} \frac{2m}{\xi'n} + |R_I|
	$$
	and so
	$$
	|R_I| \geq (k-2)cn - k - \frac{2m}{\xi'n} \stackrel{(\ref{mbound})}{\geq} (k-2)cn -  \frac{2km}{\xi'n} \geq \left(cn-\frac{m}{2\alpha^{1/3}n}\right)|I|,
	$$
	as required.
\end{proof}

Our next goal is to show that $R_i$ is in fact small for \emph{every} $i \in [k]$, which will in turn imply that $G[R_1,\ldots,R_k]$ is complete $k$-partite. To do this, we need the following lemma.

\begin{lemma}\label{lem-Ri-Z}
	For all $i \in [k]$, if $\mathrm{diff}(i) \geq -1/(2\alpha^{1/3})$, then there exists $j\in [k]\setminus \lbrace i \rbrace$ such that $R_i\not\sim R_j$.
\end{lemma}
\begin{proof}
	Let $i \in [k]$ such that $\mathrm{diff}(i) \geq -1/(2\alpha^{-1/3})$ be arbitrary.
	We begin by proving the following claim:
	
	\begin{claim}
		It suffices to show that $Z_I = \emptyset$ for all $I \in \binom{[k]\setminus \lbrace i \rbrace}{k-2}$.
	\end{claim}
	
	\begin{proof}[Proof of Claim.]
		Suppose that $Z_I = \emptyset$ for all $I \in \binom{[k]\setminus \lbrace i \rbrace}{k-2}$.
		Lemma~\ref{lem-partition-Z} implies that $Z \sim R_i$.
		Suppose now that $R_i \sim R_j$ for all $j \in [k]\setminus \lbrace i \rbrace$.
		Thus $R_i \sim \overline{R_i}$, and $R_i$ is an independent set.
		Let $n' := n-|R_i|$ and $e' := e(G[\overline{R_i}]) = e-n'(n-n')$.
		Note that $J := G[\overline{R_i}]$ satisfies $K_3(J) = g_3(n',e')$ (since otherwise we could replace it in $G$ with an $(n',e')$-graph with fewer triangles to obtain an $(n,e)$-graph with fewer triangles than $G$, contradicting~\ref{worst-C1}).
		Using~\eqref{eq:r},~\eqref{P1G} and~\eqref{mbound}, we have
		\begin{equation}\label{eq-Ris}
		|R_i|=|A_i|\pm|Z|=\frac{n}{k}\pm \sqrt{Dr}\pm \frac{2m}{\xi'n}=\frac{n}{k}\pm \alpha^{1/3}n.
		\end{equation} 
		By~\eqref{eq-Ris}, we have
		\begin{eqnarray*}
		n'(n-n')\ge \left(\frac{n}{k}-\alpha^{1/3}n\right)\left(\frac{k-1}{k}n+\alpha^{1/3}n\right)\ge \frac{k-1}{k^2}n^2-\alpha^{1/3}\frac{k-1}{k}n^2.
		\end{eqnarray*}
Recall from the very beginning of Section~\ref{beginningproof} that $\alpha_{\ref{LS83}}$ is the constant obtained by applying Theorem~\ref{LS83} with parameters $k$ and $r:=3$.
Together with $e<t_k(n)\le (k-1)n^2/(2k)$, we have that		
      \begin{eqnarray}\label{eq-eprime}
		e'&=& e-n'(n-n')\le \frac{k-1}{k}\cdot\frac{n^2}{2}-\left(\frac{k-1}{k^2}n^2-\alpha^{1/3}\frac{k-1}{k}n^2\right)=\frac{k-1}{k}\cdot\frac{n^2}{2}\left(1-\frac{2}{k}+2\alpha^{1/3}\right)\nonumber\\
		&\stackrel{(\ref{eq-Ris})}{\le}&\frac{k-1}{2k}\left(\frac{k-2}{k}+2\alpha^{1/3}\right)\left(\frac{k}{k-1}n'+\alpha^{1/4}n'\right)^2\le t_{k-1}(n')+\alpha^{1/5}(n')^2\nonumber\\
		&\stackrel{(\ref{hierarchy})}{\le}& t_{k-1}(n')+\alpha_{\ref{LS83}}(n')^2
		\end{eqnarray}
		and similarly $e' \geq t_{k-2}(n') + \alpha_{\ref{LS83}}(n')^2$.
		So $k(n',e') \in \lbrace k-1,k\rbrace$. Further,
		$$
n' = n-|R_i| \stackrel{(\ref{eq-Ris})}{\geq} \left(1-\frac{1}{k}\right)n - \alpha^{1/3}n \stackrel{(\ref{ineq:c})}{\geq} n/2 \geq n_0/2 \stackrel{(\ref{n0})}{\geq} \max\lbrace n_0(k-1,\alpha/3), n_{\ref{LS83}}(k)\rbrace.
$$
		Suppose first that $k(n',e')=k-1$.
		Then the minimality of $k$ and the fact that $t_{k-2}(n')+\alpha(n')^2/3 \leq t_{k-2}(n') + \alpha_{\ref{LS83}}(n')^2 \leq e' < t_{k-1}(n')$ implies that Theorem~\ref{strong} holds for $(n',e')$, i.e.~$g_3(n',e')=h(n',e')$, and every extremal graph lies in $\mathcal{H}(n',e')$. So $J \in \mathcal{H}(n',e')$.
		If $J \in \mathcal{H}_1(n',e')$, then since $G$ is obtained by adding an independent set $R_i$ of vertices to $J$ and adding every edge between $R_i$ and $V(J)$, we have that $G \in \mathcal{H}_1(n,e)$, a contradiction to~\ref{worst-C1}.
		Otherwise, $J \in \mathcal{H}_2(n',e')$, and in particular, $J$ is $(k-1)$-partite.
		So $G$ is $k$-partite, and Corollary~\ref{H2new}(i) implies that $G \in \mathcal{H}_2(n,e)$, again contradicting~\ref{worst-C1}.

		Thus we may assume that $k(n',e')=k$.
		Theorem~\ref{LS83} implies that we can obtain a graph $F'\in\mathcal{H}_1(n',e')$ with canonical partition $A_1^{F'},\ldots,A_{k-2}^{F'},B^{F'}$ and $K_3(F')=K_3(G[\overline{R_i}])$.
Let $F$ be the graph obtained from $G$ by replacing $G[\overline{R_i}]$ with $F'$, so $K_3(F)=K_3(G)$. By Corollary~\ref{cr:dh}, for every $xy\in E(F)$,
\begin{equation}\label{eq-FeP3}
P_3(xy,F)\le (k-2)cn+k\stackrel{(\ref{ineq:c2})}{\le} (k-2)\frac{n}{k}+\alpha^{1/3}n.
\end{equation}
For each $j\in[k-2]$ for which $A_j^{F'}$ is non-empty, fix an arbitrary edge $x_jy_j\in F[A_j^{F'},R_i]$, then
$$P_3(x_jy_j,F)\ge n-|A_j^{F'}|-|R_i|,$$
which together with~\eqref{eq-Ris} and~\eqref{eq-FeP3} implies that $|A_j^{F'}|\ge n/k-2\alpha^{1/3}n$. Similarly, for an edge $x_By_B$ in $F[B^{F'}]$ (there must exist one such edge as otherwise $k(n,e)<k$), we have $P_3(x_By_B,F)\ge n-|B^{F'}|$. 
Hence, $|B^{F'}|\ge 2n/k-\alpha^{1/3}n$. But then 
$$n=|R_i|+\sum_{j \in [k-2]}|A_j^{F'}|+|B^{F'}|\ge \frac{k+1}{k}n-\alpha^{1/4}n>n,$$
a contradiction. 
This completes the proof of the claim.
	\end{proof}
	
	\medskip
	\noindent
	Suppose now that there is some $I \in \binom{[k]\setminus \lbrace i \rbrace}{k-2}$ such that $Z_I \neq \emptyset$.
	Let $j\in [k]\setminus \lbrace i \rbrace$ be such that $[k]\setminus \lbrace i,j\rbrace= I$. 
	Let $z \in Z_I$ and let $n_{\ell} := d_G(z,R_{\ell})$ for all $\ell \in [k]$.
	Lemma~\ref{lem-partition-Z} implies that, for some $i',j' \in [k]$ with $\lbrace i',j'\rbrace = \lbrace i,j\rbrace$, we have $d_G(z,R_{i'}) \leq \delta'n$, $d_{\overline{G}}(z,R_{j'}) \geq \xi'n/2$ and, for all $\ell \in I$, we have $n_{\ell}=|R_{\ell}|$.
	Thus
	\begin{eqnarray}\label{eq-Ri-Z}
	\nonumber |R_{\ell}|-n_i-n_j &=& |R_\ell| - n_{i'}-n_{j'} \geq |R_{\ell}|-\delta'n-\left(|R_j|-\frac{\xi'n}{2}\right)\\
	\nonumber &\geq& \left(\frac{\xi'}{2}-\delta'\right)n- \bigg||A_{\ell}|-\frac{n}{k}\bigg| - \bigg||A_j|-\frac{n}{k}\bigg|-|Z|\\
	\nonumber&\stackrel{\Pbadedges(G),(\ref{P1G})}{>}& \left(\frac{\xi'}{2}-2\delta'\right)n - 2\sqrt{Dr} \stackrel{(\ref{eq:r})}{\geq} \left(\frac{\xi'}{2}-2\delta'-2\sqrt{D\alpha}\right)n\\
	&\geq& \frac{\xi'n}{3}.
	\end{eqnarray}
	Lemma~\ref{lem-no-large-set} implies that $\mathrm{diff}([k]\setminus\lbrace j \rbrace) < -1/\alpha^{1/3}$.
	So, using~(\ref{diffav}) and the fact that $\mathrm{diff}(i) \ge -1/(2\alpha^{1/3})$, there exists $\ell \in [k]\setminus \lbrace i,j\rbrace$ such that $\mathrm{diff}(\ell) < -1/\alpha^{1/3}$, so
	\begin{equation}\label{RlRi}
	|R_\ell| < cn-\frac{m}{\alpha^{1/3}n} \leq |R_i|-\frac{m}{2\alpha^{1/3}n}.
	\end{equation}
	
	Let $I' := [k]\setminus\lbrace i,j,\ell\rbrace$ and $W := R_i \cup R_j \cup R_\ell \cup Z$.
	Then
	\begin{equation}\label{dGzW}
	d_G(z,W) = n_i + n_j + |R_\ell| + d_G(z,Z),
	\end{equation}
	$R_{I'} =\overline{W}$ and $\lbrace z \rbrace \sim R_{I'}$.
	Recalling that $n_\ell=|R_\ell|$ for all $\ell\in I$, we have that
	\begin{equation}\label{K3ZG}
	K_3(z,G) \geq e(G[R_{I'}]) + |R_{I'}|(n_i+n_j+|R_\ell|) + |R_\ell|(n_i+n_j)-m.
	\end{equation}
	We have
	\begin{eqnarray*}
		d_G(z,W) &\stackrel{(\ref{dGzW})}{=}& n_i+n_j+n_\ell+d_G(z,Z) \stackrel{(\ref{eq-Ri-Z})}{\leq} 2|R_\ell|-\frac{\xi'n}{3} + |Z|\stackrel{\Pbadedges(G)}{\leq} 2|R_\ell|-\frac{\xi'n}{4}\\
		&\stackrel{(\ref{P1G})}{\leq}& |R_i|+|R_j| + 2D\sqrt{r} + 2|Z| - \frac{\xi'n}{4} \stackrel{\Pbadedges(G),(\ref{eq:r})}{\leq} |R_i|+|R_j| + 2D\sqrt{\alpha}n + 2\delta'n - \frac{\xi'n}{4}\\
		&\leq& |R_i| + |R_j| - \frac{\xi'n}{5}.
	\end{eqnarray*}
	Let $k_i := \min\lbrace d_G(z,W),|R_i|\rbrace$ and $k_j := \max\lbrace d_G(z,W)-k_i,0\rbrace$.
	The previous equation implies that
	\begin{equation}\label{kikj}
	k_i+k_j=d_G(z,W)\quad\text{and}\quad k_ik_j = \begin{cases} 0, &\mbox{if } d_G(z,W)\leq|R_i|, \\ 
	|R_i|(d_G(z,W)-|R_i|) & \mbox{otherwise. } \end{cases}
	\end{equation}
	Obtain a new graph $G'$ from $G$ as follows.
	Let $K_i \subseteq R_i$ with $|K_i|=k_i$ and $K_j \subseteq R_j$ with $|K_j|=k_j$ be arbitrary.
	Note that this is possible as $k_i\leq |R_i|$ and if $k_j>0$, then $k_j \leq d_G(z,W)-|R_i| \leq |R_j|-\xi'n/5$.
	Let $V(G'):=V(G)$ and
	$$
	E(G') := \left(E(G) \cup \lbrace zx: x \in K_i \cup K_j \rbrace\right) \setminus \lbrace zy: y \in N_G(z,W)\rbrace.
	$$
	That is, we obtain $G'$ by changing the $W$-neighbourhood of $z$ to a new neighbourhood of the same size, by adding as many edges as possible to $R_i$, and (if necessary) additional edges to~$R_j$.
	Note that $N_{G'}(z,R_{\ell}\cup Z)=\emptyset$ and $G'$ is an $(n,e)$-graph.
	We have
	\begin{equation}\label{G'G}
	K_3(z,G') \leq e(G[R_{I'}]) + |R_{I'}|d_{G'}(z,W) + k_ik_j.
	\end{equation}
	
	Suppose first that $d_G(z,W)>|R_i|$. Then by~\eqref{kikj}, we have
	\begin{align*}
	K_3(z,G') &\leq e(G[R_{I'}]) + |R_{I'}|d_G(z,W) + |R_i|(d_G(z,W)-|R_i|)
	\end{align*}
	and so
	\begin{eqnarray*}
		&&K_3(G')-K_3(G) = K_3(z,G')-K_3(z,G)\\
		&\stackrel{(\ref{K3ZG})}{\le}& |R_{I'}|(d_G(z,W)-(n_i+n_j+|R_{\ell}|))+|R_i|(d_G(z,W)-|R_i|)-|R_{\ell}|(n_i+n_j)+m\\
		&\stackrel{(\ref{dGzW})}{\le}& |R_{I'}||Z|+|R_i|(n_i+n_j+|R_{\ell}|+|Z|-|R_i|)-|R_{\ell}|(n_i+n_j)+m\\
		&=&|R_{I'}||Z|+ (|R_i|-|R_\ell|)(n_i+n_j-|R_i|)+|Z||R_i| + m\\
		&\stackrel{(\ref{eq-Ri-Z}),(\ref{RlRi})}{\leq}& - (|R_i|-|R_\ell|)\left(|R_i|-|R_\ell|+\frac{\xi'n}{3}\right) + |Z|n+m\\
		&\stackrel{(\ref{mbound}),(\ref{RlRi})}{\leq}& -\frac{m\xi'}{7\alpha^{1/3}} + \frac{2m}{\xi'} + m < -\frac{2m}{\xi'} \stackrel{(\ref{mbound})}{\le} -n ,
	\end{eqnarray*}
	a contradiction.
	
	Therefore we may assume that $d_G(z,W)\leq |R_i|$.
	We need the following claim that $n_j$ is large.
	
	\begin{claim}\label{cl-nj}
		$n_j \geq \frac{km}{4\alpha^{1/3}n}$.
	\end{claim}
	
	\begin{proof}[Proof of Claim.]
		If $\mathrm{diff}(I) \geq -1/\alpha^{1/3}$, then since $\mathrm{diff}(i) \geq -1/(2\alpha^{1/3})$, we also have that $\mathrm{diff}({I\cup\lbrace i \rbrace}) \geq -1/\alpha^{1/3}$, a contradiction to Lemma~\ref{lem-no-large-set}. So $\mathrm{diff}(I) < -1/\alpha^{1/3}$.
		The second part of Lemma~\ref{lem-partition-Z} implies that there is some $u \in N_{\overline{G}}(z,R_i)$.
		Since $R_i$ is an independent set in $G$, we have that
		$$
		(k-2)cn-k \stackrel{(\ref{2path})}{\leq} P_3(zu,G) \leq |Z| + n_j + |R_I|
		$$
		and so, using the fact that $\mathrm{diff}(I) < -1/\alpha^{1/3}$,
		\begin{align*}
		n_j &\geq (k-2)cn-k-|Z|-|R_I| \stackrel{(\ref{mbound})}{\geq} (k-2)cn-k-\frac{2m}{\xi'n} - (k-2)\left(cn-\frac{m}{\alpha^{1/3}n}\right)\\
		&\geq \left(\frac{k-2}{\alpha^{1/3}}-\frac{3k}{\xi'}\right)\frac{m}{n} \geq \frac{km}{4\alpha^{1/3}n},
		\end{align*}
		completing the proof of the claim.
	\end{proof}
	
	\medskip
	\noindent
	Now~(\ref{K3ZG}),~(\ref{kikj}),~(\ref{G'G}) and Claim~\ref{cl-nj} imply that
	\begin{eqnarray*}
	K_3(z,G')-K_3(z,G) &\stackrel{(\ref{dGzW})}{\leq}& |R_{I'}||Z|+m-|R_\ell|(n_i+n_j)\\
	&\stackrel{(\ref{P1G})}{\leq}& n|Z|+m - \left(\frac{n}{k}-\sqrt{Dr}-|Z|\right) \cdot \frac{km}{4\alpha^{1/3}n}\\ 
	&\leq& \frac{2m}{\xi'}+m-\frac{n}{2k}\cdot\frac{km}{4\alpha^{1/3}n}
	\leq -\frac{m}{9\alpha^{1/3}} \stackrel{(\ref{mbound})}{<} 0,
	\end{eqnarray*}
	another contradiction.
	Thus there is no $z \in Z_I$, as required.
\end{proof}

The final ingredient is the following lemma which states that every $R_i$ is small; $G$ induced on the union of the $R_i$ is complete partite; and every $z \in Z$ has large degree into every $R_i$.

\begin{lemma}\label{final}
	The following hold in $G$:
	\begin{itemize}
		\item[(i)] For all $i \in [k]$, we have $\mathrm{diff}(i) < -1/(2\alpha^{1/3})$;
		\item[(ii)] $G[R_1 \cup \ldots \cup R_k]$ is a complete $k$-partite graph (with partition $R_1,\ldots,R_k$);
		\item[(iii)] For all $i \in [k]$ and $z \in Z$, we have $d_G(z,R_i) \geq km/(9\alpha^{1/3}n)$.
	\end{itemize}
\end{lemma}

\begin{proof}
	For (i), suppose that there is some $i \in [k]$ for which $\mathrm{diff}(i) \geq -1/(2\alpha^{-1/3})$.
	Apply Lemma~\ref{lem-Ri-Z} to obtain $j \in [k]\setminus \lbrace i \rbrace$ such that $R_i \not\sim R_j$.
	But Lemma~\ref{lem-large-set-from-missing-edge} implies that $\mathrm{diff}([k]\setminus \lbrace i,j\rbrace) \geq -1/(2\alpha^{1/3})$.
	Thus $\mathrm{diff}([k]\setminus \lbrace j \rbrace) \geq -1/(2\alpha^{1/3})$, a contradiction to Lemma~\ref{lem-no-large-set}.
	
	We now turn to (ii).
	Since $R_i$ is an independent set in $G$ for all $i \in [k]$, it suffices to show that $R_i \sim R_j$ for all $ij \in \binom{[k]}{2}$.
	If there is some $ij \in \binom{[k]}{2}$ for which this does not hold, then Lemma~\ref{lem-large-set-from-missing-edge} implies that $\mathrm{diff}([k]\setminus \lbrace i,j\rbrace) \geq -1/(2\alpha^{1/3})$. Then, by averaging (i.e.~(\ref{diffav})), there is some $\ell \in [k]\setminus \lbrace i,j\rbrace$ for which $\mathrm{diff}(\ell) \geq -1/(2\alpha^{-1/3})$, contradicting (i).
	
	For (iii), let $z \in Z$ be arbitrary.
	Lemma~\ref{lem-partition-Z} implies that there is $I \in \binom{[k]}{k-2}$ such that $z \in Z_I$ (and so $z \sim R_I$).
	Let $ij \in \binom{[k]}{2}$ be such that $I = [k]\setminus \lbrace i,j\rbrace$
	and for all $\ell  \in [k]$ write $n_{\ell} := d_G(z,R_{\ell})$.
	We only need to show that $n_i,n_j \geq (km)/(9\alpha^{1/3}n)$ since for all $\ell  \in I$ we have $$
	n_{\ell} = |R_{\ell}| \stackrel{(\ref{P1G})}{\geq} \frac{n}{k}-\sqrt{Dr} -|Z|> \frac{n}{2k} \stackrel{(\ref{eq:D})}{>} \frac{kD\alpha^{2/3}n}{4} \stackrel{(\ref{eq:r})}{\geq} \frac{kDr}{4\alpha^{1/3}n} \geq \frac{km}{9\alpha^{1/3}n}.
	$$
	The second part of Lemma~\ref{lem-partition-Z} implies that there exist $u_i \in N_{\overline{G}}(z,R_i)$ and $u_j \in N_{\overline{G}}(z,R_j)$.
	Then
	$$
	(k-2)cn-k \stackrel{(\ref{2path})}{\leq} P_3(zu_i,G) \leq |Z| + n_j + |R_I|
	$$
	and so
	$$
	n_j \stackrel{(\ref{mbound})}{\geq} (k-2)cn-k-\frac{2m}{\xi'n}-\sum_{\ell  \in I}|R_{\ell}| \stackrel{(i)}{\geq} (k-2)cn - \frac{2km}{\xi'n} - (k-2)\left(cn-\frac{m}{2\alpha^{1/3}n}\right) \geq \frac{km}{9\alpha^{1/3}n},
	$$
	where we used that fact that $k \geq 3$. An identical proof works for $n_i$.
\end{proof}

\noindent
\emph{Proof of Theorem~\ref{strong} in the boundary case.}
We will show that $Z=\emptyset$, contradicting~(\ref{mbound}).
Suppose not, and let $z \in Z$.
Then Lemma~\ref{lem-partition-Z} implies that there is $I \in \binom{[k]}{k-2}$ for which $z \in Z_I$.
So $z \sim R_I$.
Write $I = [k]\setminus \lbrace i,j\rbrace$ and suppose without loss of generality that $z \in A_i$.
Let $n_{\ell} := d_G(z,R_{\ell})$ for all $\ell  \in [k]$.
Let $F_{Z,j} := G[N_G(z,Z),N_G(z,R_j)]$ and $F_{Z,I} := G[N_G(z,Z),R_I]$.
Then Lemma~\ref{final}(ii) implies that
$$
K_3(z,G) \geq e(G[R_I]) + |R_I|(n_i+n_j)+ n_in_j+e(F_{Z,j}) + e(F_{Z,I}).
$$
We have
$$
N_{\overline{G}}(z,R_j) \stackrel{\Pmissing(G)}{\geq} \xi'n-|Z| \stackrel{\Pbadedges(G)}{>} \delta'n \geq d_G(z,R_i)
$$
and hence we can choose a set $K_j \subseteq N_{\overline{G}}(z,R_j)$ with $|K_j|=d_G(z,R_i)$.
Obtain a graph $G'$ from $G$ as follows.
let $V(G') := V(G)$ and $E(G') := (E(G)\cup\lbrace zx:x \in K_j\rbrace)\setminus \lbrace zy: y \in N_G(z,R_i) \rbrace$.
Clearly $G'$ is an $(n,e)$-graph in which $z$ has no neighbours in $R_i$, so
\begin{align*}
K_3(z,G') &\leq e(G'[R_I]) + |R_I|d_{G'}(z,R_j) + e(G'[N_{G'}(z,Z),R_I]) + e(G'[N_{G'}(z,Z),N_{G'}(z,R_j)]) + |Z|^2\\
&\leq e(G[R_I]) + |R_I|(n_i+n_j)+e(F_{Z,I}) + e(F_{Z,j}) + n_i|Z|+|Z|^2.
\end{align*}
Therefore, using Lemma~\ref{final}(iii), we have
\begin{eqnarray*}
	K_3(G')-K_3(G) &\leq& n_i(|Z|-n_j) + |Z|^2 \stackrel{(\ref{mbound})}{\leq} n_i\left(\frac{2m}{\xi'n}-\frac{km}{9\alpha^{1/3}n}\right) + \frac{4m^2}{(\xi')^2n^2}\\
	&\leq& -\frac{n_im}{10\alpha^{1/3}n} +\frac{4m^2}{(\xi')^2n^2} \leq \left(\frac{4}{(\xi')^2} - \frac{k}{90\alpha^{2/3}}\right)\frac{m^2}{n^2} \stackrel{(\ref{hierarchy})}{<} 0,
\end{eqnarray*}
a contradiction.
Thus $Z = \emptyset$, contradicting~(\ref{mbound}) as required.
\hfill$\square$

\medskip\noindent
This completes the proof of Theorem~\ref{strong}.


\section{Concluding remarks}\label{ConcludingRemarks}

\subsection{Related work}

The more general \emph{supersaturation problem} of determining $g_F(n,e)$, the minimum number of copies of $F$ in an $(n,e)$-edge graph, is also an active area of research.
The range of $e$ for which $g_F(n,e)=0$ is well understood.
Indeed, given a fixed graph $F$, let $\mathrm{ex}(n,F)$ denote the maximum number of edges in an $F$-free $n$-vertex graph,~i.e.~the maximum $e$ for which $g_F(n,e)=0$.
Erd\H{o}s and Stone~\cite{ErdosStone46} proved that $\mathrm{ex}(n,F)=t_{\chi(F)-1}(n)+o(n^2)$, where $\chi(F)$ is the chromatic number of $F$.
The supersaturation phenomenon observed by Erd\H{o}s and Simonovits~\cite{ErdosSimonovits83} asserts that every $(n,e)$-graph $G$ with $e \geq \mathrm{ex}(n,F)+\Omega(n^2)$ contains not just one copy of $F$, but in fact a positive proportion of all $|V(F)|$-sized vertex subsets in $V(G)$ span a copy of $F$. (This also extends to hypergraphs.)

We say that $F$ is \emph{critical} when there is an edge in $F$ whose removal reduces the chromatic number.
Observe that cliques are critical.
Simonovits~\cite{Simonovits68} showed that, for such $F$ and large $n$, we have $\mathrm{ex}(n,F)=t_{\chi(F)-1}(n)$ and $T_{\chi(F)-1}(n)$ is the unique extremal graph.
That is, $g_F(n,e)=0$ if and only if $e \leq t_{\chi(F)-1}(n)$.
Mubayi~\cite{Mubayi10am} showed that there is $c > 0$ such that, for large $n$, and $1 \leq \ell \leq cn$, we have
$$
g_F(n,t_{\chi(F)-1}(n)+\ell) = (1+o(1))\,\ell \cdot \mathrm{copy}(n,F),
$$
where $\mathrm{copy}(n,F)$ is the minimum number of copies of $F$ obtained by adding a single edge to $T_{\chi(F)-1}(n)$.
(This can generally be computed easily for any fixed $F$.)
Notice that this result generalises Erd\H{o}s's result~\cite{Erdos62} from triangles (which are critical) to arbitrary critical $F$.
Further, the error term can be removed in some cases, for example when $F$ is an odd cycle.
Pikhurko and Yilma~\cite{PikhurkoYilma17} generalised Mubayi's result by raising the upper bound $cn$ on $\ell$ to~$o(n^2)$.

The supersaturation problem for non-critical $F$ with $\chi(F)\geq 3$ seems hard; e.g.~even the `simplest' case when $F$ consists of two triangles sharing a vertex poses considerable difficulties, see~\cite{KangMakaiPikhurko17arxiv}.

A famous conjecture of Sidoren\-ko~\cite{Sidorenko93} and Erd\H{o}s-Simonovits~\cite{ErdosSimonovits83} asserts, roughly speaking, that the minimal number of $F$-subgraphs is asymptotically attained by a random graph (we do not give a precise statement of the conjecture here).
The conjecture is known to be true for trees, cycles, complete bipartite graphs, `strongly tree-decomposable graphs' and others, see~\cite{ConlonFoxSudakov10,ConlonKimLeeLee15arxiv,Hatami10,KimLeeLee16,LiSzegedy11arxiv,Szegedy14arxiv:v3}.

\medskip
\noindent
A yet more general problem is the following.
Let $\mathcal{F} := (F_1,\ldots,F_\ell)$ be a tuple of graphs with $v_1,\ldots,v_\ell$ vertices respectively.
Let $F_i(G)$ denote the number of \emph{induced} copies of $F_i$ in a graph $G$, for all $i \in [\ell]$.
To an $n$-vertex graph $G$, associate a vector $f_{\mathcal{F}}(G) := (F_1(G)/\binom{n}{v_1},\ldots,F_\ell(G)/\binom{n}{v_\ell})$ of densities.
What is the set $T(\mathcal{F}) \subseteq \mathbb{R}^\ell$ consisting of the accumulation points of $f_{\mathcal{F}}(G)$?
When $\mathcal{F}=(K_2,K_r)$, it turns out that $T(\mathcal{F})$ has an upper and lower bounding curve. 
The lower bounding curve of $T(\mathcal{F})$ is by definition $y=g_r(x)$, which by Reiher's
clique density theorem~\cite{Reiher16} is a countable union of algebraic curves.
The upper bounding curve is $y=x^{r/2}$, this being a consequence of the Kruskal-Katona theorem~\cite{Katona64,Kruskal63}.
This corresponds to the \emph{maximum} $r$-clique density in a graph with given edge density.
The shaded region in Figure~\ref{plot} is $T(\mathcal{F})$ for $\mathcal{F}=(K_2,K_3)$.

The case $(F_1,F_2)=(K_3,\overline{K_3})$ was solved by Huang, Linial, Naves, Peled and Sudakov~\cite{HuangLinialNavesPeledSudakov14jgt} (here the lower bounding curve is $x+y=1/4$, due to Goodman~\cite{Goodman59}).
Glebov, Grzesik, Hu, Hubai, Kr\'al' and Volec~\cite{GlebovGrzesikHuHubaiKralVolec} study the problem for every remaining pair $(F_1,F_2)$ of three-vertex graphs. For larger graphs the problem becomes extremely challenging.
Some general results on the hardness of determining $T(\mathcal{F})$ were obtained by Hatami and Norine in~\cite{HatamiNorine11,HatamiNorin16arxiv}.%

\subsection{The range $\binom{n}{2}-\eps n^2 < e \leq \binom{n}{2}$.}\label{bigrange}
Our main result, Theorem~\ref{main}, determines $g_3(n,e)$ whenever $2e/n^2$ is bounded away from $1$. There are a few obstacles to extending it to the remaining range $e=\binom{n}{2}-o(n^2)$. One is that Theorem~\ref{PikhurkoRazborov17approx} does not tell us anything meaningful in this range, as the error in its approximation is too large. 

While it is trivial to  determine $g_3(n,e)$ when $e \geq \binom{n}{2}-\lfloor n/2\rfloor$ (with each extremal graph $G$ being the complement of a matching)
and this can extended a bit further with some work, the problem seems to become very difficult in this regime quite quickly. In fact, the following observation shows that, under the assumption that $g_3\equiv h^*$, pushing ${n\choose 2}-e$ beyond $O(n)$ is as difficult as determining $g_3(n,e)$ for all pairs~$(n,e)$.  

\begin{lemma}\label{lm:AlmostComplete} Suppose that for every $C>0$ there is $n_0>0$ such that $g_3(n,e)=h^*(n,e)$ for all $n\ge n_0$ and $e\ge {n\choose 2}-Cn$. Then $g_3(n,e)=h^*(n,e)$ for \textbf{all} $n,e\in \I N$ with $e\le {n\choose 2}$.\end{lemma}
 \begin{proof}
 Suppose on the contrary that some $(n,e)$-graph $G$ satisfies $K_3(G)< h^*(n,e)$. Let $\bm{a}^*=\bm{a}^*(n,e)$. Our assumption for $C:=n/2$ returns some $n_0$. Take $\ell$ such that $n':=\ell a_1^*+n$ is at least~$n_0$. Let $H$ be the complete partite graph with $n'$ vertices, $\ell$ parts of size $a_1^*$ and the last part, call it $A$, of size $n$. Let $G'$ (resp.\ $H'$) be obtained from $H$ by adding a copy of $G$ (resp.\ $H^*(n,e)$) into $A$. Each of these graphs has 
 $
 e':={n'\choose 2}-\ell {a_1^*\choose 2} - {n\choose 2}+e
 $
 edges, which is at least ${n'\choose 2} - \frac{n}2\, n'$ because the maximum degree of the graph complement is at most $n$. Also, $H'$ is isomorphic to $H^*(n',e')$: this follows by induction on $\ell\in\I N$ using the easy claim that if we duplicate a largest part of any $H^*$-graph then we get another $H^*$-graph. However, since $A$ is complete to the rest of $H$, we have $$
 K_3(G')-h^*(n',e')=K_3(G')-K_3(H')=K_3(G)-K_3(H^*(n,e))<0,
 $$ a contradiction to the choice of $n_0$.\end{proof}

An interesting corollary of Proposition~\ref{pr:compute} and Lemma~\ref{lm:AlmostComplete} is that the validity of Conjecture~\ref{cj:LS} will  not be affected if we drop the assumption~$n\ge n_0$.

\subsection{Extensions}

It would be very interesting to extend Theorem~\ref{strong} to the $g_r(n,e)$-problem, as many parts of our proof extend when we minimise the number of $r$-cliques. A structure result for $r$-cliques with $r\ge 4$ (an analogue of Theorem~\ref{PikhurkoRazborov17approx}) was recently proved by Kim, Liu, Pikhurko and Sharifzadeh~\cite{KimLiuPikhurkoSharifzadeh}.

A problem which may be more directly amenable to our method is as follows. Recall that $N_i(G)$ is the number of $3$-subsets of $V(G)$ that induce exactly $i$ edges, $0\le i\le 3$. The question is
to maximise $N_2(G)$ (the number of so-called \emph{cherries}) in an $(n,e)$-graph for $n\ge n_0$. This problem was considered by Harangi~\cite{Harangi13} who obtained some partial results that were enough for his intended application. Note that for every $(n,e)$-graph $G$, we have (see~(\ref{LSproof1}))
 $$
 e(n-2)=3N_3(G)+2N_2(G)+N_1(G).
 $$
 Also, $N_1(H^*(n,e))\le m^*n=o(n^3)$. Since $H^*(n,e)$ asymptotically minimises $N_3$ over $(n,e)$-graphs, it also asymptotically maximises $N_2$. Furthermore, a stronger version of stability (that every almost $N_2$-extremal $(n,e)$-graph is $o(n^2)$-close to $H^*(n,e)$) can be easily derived from Theorem~\ref{PikhurkoRazborov17approx}.

We hope that the method used here will be useful for further instances, where one has to convert an asymptotic result into an exact one.

\section*{Acknowledgements}

Hong Liu was supported by EPSRC grant~EP/K012045/1, ERC
	grant~306493 and the Leverhulme Trust
	Early Career Fellowship~ECF-2016-523. Oleg Pikhur\-ko was
	supported by EPSRC grant EP/K012045/1,  ERC grant 306493
and Leverhulme Research Project Grant RPG-2018-424.
Katherine Staden was supported by ERC grant~306493.

The authors are enormously grateful to the referee of this paper for their careful reading
and helpful suggestions. Oleg Pikhurko would like to thank Alexander Razborov for very useful discussions. 

\bibliography{Triangle_2017_Aug18}

\begin{bibdiv}
\begin{biblist}

\bib{Bollobas76}{article}{
      author={{Bollob\'as}, B.},
       title={On complete subgraphs of different orders},
        date={1976},
     journal={Math.\ Proc.\ Camb.\ Phil.\ Soc.},
      volume={79},
       pages={19\ndash 24},
}

\bib{ConlonFoxSudakov10}{article}{
      author={Conlon, D.},
      author={Fox, J.},
      author={Sudakov, B.},
       title={An approximate version of {S}idorenko's conjecture},
        date={2010},
     journal={Geom.\ Func.\ Analysis},
      volume={20},
       pages={1354\ndash 1366},
}

\bib{ConlonKimLeeLee15arxiv}{unpublished}{
      author={Conlon, D.},
      author={Kim, J.~H.},
      author={Lee, C.},
      author={Lee, J.},
       title={Some advances on {S}idorenko's conjecture},
        date={2015},
        note={E-print arxiv:1510.06533},
}

\bib{Csikvari13}{article}{
      author={Csikv{\'a}ri, P.},
       title={Note on the smallest root of the independence polynomial},
        date={2013},
     journal={Combin.\ Probab.\ Computing},
      volume={22},
       pages={1\ndash 8},
}

\bib{ErdosSimonovits83}{article}{
      author={{Erd\H os}, P.},
      author={Simonovits, M.},
       title={Supersaturated graphs and hypergraphs},
        date={1983},
     journal={Combinatorica},
      volume={3},
       pages={181\ndash 192},
}

\bib{Erdos55}{article}{
      author={Erd{\H{o}}s, P.},
       title={Some theorems on graphs},
        date={1955},
     journal={Riveon Lematematika},
      volume={9},
       pages={13\ndash 17},
}

\bib{Erdos62}{article}{
      author={Erd{\H{o}}s, P.},
       title={On a theorem of {R}ademacher-{T}ur\'an},
        date={1962},
     journal={Illinois J. Math.},
      volume={6},
       pages={122\ndash 127},
}

\bib{Erdos67a}{incollection}{
      author={Erd{\H{o}}s, P.},
       title={Some recent results on extremal problems in graph theory.
  {R}esults},
        date={1967},
   booktitle={Theory of graphs (internat. sympos., rome, 1966)},
   publisher={Gordon and Breach},
     address={New York},
       pages={117\ndash 123 (English); pp. 124\ndash 130 (French)},
}

\bib{ErdosStone46}{article}{
      author={Erd{\H o}s, P.},
      author={Stone, A.~H.},
       title={On the structure of linear graphs},
        date={1946},
     journal={Bull.\ Amer.\ Math.\ Soc.},
      volume={52},
       pages={1087\ndash 1091},
}

\bib{Fisher89}{article}{
      author={Fisher, D.~C.},
       title={Lower bounds on the number of triangles in a graph},
        date={1989},
     journal={J.\ Graph Theory},
      volume={13},
       pages={505\ndash 512},
}

\bib{GlebovGrzesikHuHubaiKralVolec}{unpublished}{
      author={Glebov, R.},
      author={Grzesik, A.},
      author={Hu, P.},
      author={Hubai, T.},
      author={{Kr\'al'}, D.},
      author={Volec, J.},
       title={Densities of 3-vertex graphs},
        date={2016},
        note={E-print arxiv:1610.02446},
}

\bib{GoldwurmSantini00}{article}{
      author={Goldwurm, M.},
      author={Santini, M.},
       title={Clique polynomials have a unique root of smallest modulus},
        date={2000},
     journal={Inform.\ Process.\ Letters},
      volume={75},
       pages={127\ndash 132},
}

\bib{Goodman59}{article}{
      author={Goodman, A.~W.},
       title={On sets of acquaintances and strangers at any party},
        date={1959},
     journal={Amer.\ Math.\ Monthly},
      volume={66},
       pages={778\ndash 783},
}

\bib{Harangi13}{article}{
      author={Harangi, V.},
       title={On the density of triangles and squares in regular finite and
  unimodular random graphs},
        date={2013},
     journal={Combinatorica},
      volume={33},
       pages={531\ndash 548},
}

\bib{Hatami10}{article}{
      author={Hatami, H.},
       title={Graph norms and {S}idorenko's conjecture},
        date={2010},
     journal={Israel J.\ Math.},
      volume={175},
       pages={125\ndash 150},
}

\bib{HatamiNorine11}{article}{
      author={Hatami, H.},
      author={Norine, S.},
       title={Undecidability of linear inequalities in graph homomorphism
  densities},
        date={2011},
     journal={J.\ Amer.\ Math.\ Soc.},
      volume={24},
       pages={547\ndash 565},
}

\bib{HatamiNorin16arxiv}{unpublished}{
      author={Hatami, H.},
      author={Norine, S.},
       title={On the boundary of the region defined by homomorphism densities},
        date={2016},
        note={E-print arxiv:1612.09554},
}

\bib{HuangLinialNavesPeledSudakov14jgt}{article}{
      author={Huang, H.},
      author={Linial, N.},
      author={Naves, H.},
      author={Peled, Y.},
      author={Sudakov, B.},
       title={On the 3-local profiles of graphs},
        date={2014},
     journal={J.\ Graph Theory},
      volume={76},
       pages={236\ndash 248},
}

\bib{KangMakaiPikhurko17arxiv}{unpublished}{
      author={Kang, M.},
      author={Makai, T.},
      author={Pikhurko, O.},
       title={Supersaturation problem for the bowtie},
        date={2017},
        note={E-print arxiv:1710.01471},
}

\bib{Katona64}{article}{
      author={Katona, G. O.~H.},
       title={Intersection theorems for systems of finite sets},
        date={1964},
     journal={Acta Math.\ Acad.\ Sci.\ Hung.},
      volume={15},
       pages={329\ndash 337},
}

\bib{KimLiuPikhurkoSharifzadeh}{unpublished}{
      author={Kim, J.},
      author={Liu, H.},
      author={Pikhurko, O.},
      author={Sharifzadeh, M.},
       title={Asymptotic structure for the clique density theorem},
        date={2019},
        note={E-print arXiv:1906.05942},
}

\bib{KimLeeLee16}{article}{
      author={Kim, J.~H.},
      author={Lee, C.},
      author={Lee, J.},
       title={Two approaches to {Sidorenko's} conjecture},
        date={2016},
     journal={Trans.\ Amer.\ Math.\ Soc.},
      volume={368},
       pages={5057\ndash 5074},
}

\bib{Kruskal63}{incollection}{
      author={Kruskal, J.~B.},
       title={The number of simplices in a complex},
        date={1963},
   booktitle={Mathematical optimization techniques},
      editor={Bellman, R.},
   publisher={Univ.\ California Press, Berkeley},
       pages={251\ndash 278},
}

\bib{LiSzegedy11arxiv}{unpublished}{
      author={Li, J. L.~X.},
      author={Szegedy, B.},
       title={On the logarithmic calculus and {Sidorenko's} conjecture},
        date={2011},
        note={E-print arxiv:1107.1153},
}

\bib{LovaszSimonovits76}{inproceedings}{
      author={Lov{\'a}sz, L.},
      author={Simonovits, M.},
       title={On the number of complete subgraphs of a graph},
        date={1976},
   booktitle={Proceedings of the fifth british combinatorial conference (univ.
  aberdeen, aberdeen, 1975)},
   publisher={Utilitas Math.},
     address={Winnipeg, Man.},
       pages={431\ndash 441. Congressus Numerantium, No. XV},
}

\bib{LovaszSimonovits83}{incollection}{
      author={Lov{\'a}sz, L.},
      author={Simonovits, M.},
       title={On the number of complete subgraphs of a graph. {II}},
        date={1983},
   booktitle={Studies in pure mathematics},
   publisher={Birkh\"auser},
     address={Basel},
       pages={459\ndash 495},
}

\bib{Mantel07}{article}{
      author={Mantel, W.},
       title={Problem 28},
        date={1907},
     journal={Winkundige Opgaven},
      volume={10},
       pages={60\ndash 61},
}

\bib{MoonMoser62}{article}{
      author={Moon, J.~W.},
      author={Moser, L.},
       title={On a problem of {Tur\'an}},
        date={1962},
     journal={Publ.\ Math.\ Inst.\ Hungar.\ Acad. Sci.},
      volume={7},
       pages={283\ndash 287},
}

\bib{Mubayi10am}{article}{
      author={Mubayi, D.},
       title={Counting substructures {I}: Color critical graphs},
        date={2010},
     journal={Adv.\ Math.},
      volume={225},
       pages={2731\ndash 2740},
}

\bib{Nikiforov11}{article}{
      author={Nikiforov, V.},
       title={The number of cliques in graphs of given order and size},
        date={2011},
     journal={Trans.\ Amer.\ Math.\ Soc.},
      volume={363},
       pages={1599\ndash 1618},
}

\bib{Nikiforov76}{article}{
      author={Nikiforov, V.~S.},
       title={On a problem of {P}. {Erd\H os}},
        date={1976/77},
     journal={Annuaire Univ. Sofia Fac. Math. M\'ec.},
      volume={71},
       pages={157\ndash 160},
}

\bib{NikiforovKhadzhiivanov81}{article}{
      author={Nikiforov, V.~S.},
      author={Khadzhiivanov, N.~G.},
       title={Solution of the problem of {P}. {E}rd{\H o}s on the number of
  triangles in graphs with {$n$} vertices and {$[n\sp{2}/4]+l$} edges},
        date={1981},
     journal={C. R. Acad. Bulgare Sci.},
      volume={34},
       pages={969\ndash 970},
}

\bib{NordhausStewart63}{article}{
      author={Nordhaus, E.~A.},
      author={Stewart, B.~M.},
       title={Triangles in an ordinary graph},
        date={1963},
     journal={Can.\ J.\ Math.},
      volume={15},
       pages={33\ndash 41},
}

\bib{PikhurkoRazborov17}{article}{
      author={Pikhurko, O.},
      author={Razborov, A.},
       title={Asymptotic structure of graphs with the minimum number of
  triangles},
        date={2017},
     journal={Combin.\ Probab.\ Computing},
      volume={26},
       pages={138\ndash 160},
}

\bib{PikhurkoYilma17}{article}{
      author={Pikhurko, O.},
      author={Yilma, Z.},
       title={Supersaturation problem for color-critical graphs},
        date={2017},
     journal={J.\ Combin.\ Theory\ {\rm (B)}},
      volume={123},
       pages={148\ndash 185},
}

\bib{Razborov07}{article}{
      author={Razborov, A.},
       title={Flag algebras},
        date={2007},
     journal={J.\ Symb.\ Logic},
      volume={72},
       pages={1239\ndash 1282},
}

\bib{Razborov08}{article}{
      author={Razborov, A.},
       title={On the minimal density of triangles in graphs},
        date={2008},
     journal={Combin.\ Probab.\ Computing},
      volume={17},
       pages={603\ndash 618},
}

\bib{Reiher16}{article}{
      author={Reiher, C.},
       title={The clique density theorem},
        date={2016},
     journal={Annals of Math.},
      volume={184},
       pages={683\ndash 707},
}

\bib{Sidorenko93}{article}{
      author={Sidorenko, A.~F.},
       title={A correlation inequality for bipartite graphs},
        date={1993},
     journal={Graphs Combin.},
      volume={9},
       pages={201\ndash 204},
}

\bib{Simonovits68}{incollection}{
      author={Simonovits, M.},
       title={A method for solving extremal problems in graph theory, stability
  problems},
        date={1968},
   booktitle={Theory of graphs ({P}roc. {C}olloq., {T}ihany, 1966)},
   publisher={Academic Press},
       pages={279\ndash 319},
}

\bib{Szegedy14arxiv:v3}{unpublished}{
      author={Szegedy, B.},
       title={An information theoretic approach to {Sidorenko's} conjecture},
        date={2014},
        note={E-print arxiv:1406.6738v3},
}

\bib{Turan41}{article}{
      author={{Tur\'an}, P.},
       title={On an extremal problem in graph theory (in {Hungarian})},
        date={1941},
     journal={Mat.\ Fiz.\ Lapok},
      volume={48},
       pages={436\ndash 452},
}

\end{biblist}
\end{bibdiv}

\vspace{8mm}

\nomenclature[EF]{$n$}{Number of vertices in $G$}{}{\pageref{hierarchy}}
\nomenclature[EF]{$e$}{Number of edges in $G$}{}{\pageref{hierarchy}}
\nomenclature[EF]{$(n,e)$-graph}{A graph with $n$ vertices and $e$ edges}{}{\pageref{eq:k}}
\nomenclature[EF]{$G$}{`Worst counterexample' graph with $n$ vertices and $e$ edges satisfying (C1)--(C3)}{}{\pageref{worst}}
\nomenclature[EF]{$K_r(H)$}{Number of $r$-cliques in a graph $H$}{}{\pageref{eq:k}}
\nomenclature[EF]{$g_r(n,e)$}{Minimum number of $r$-cliques in an $(n,e)$-graph}{}{\pageref{eq:k}}
\nomenclature[EF]{$T_s(n)$, $t_s(n)$}{$n$-vertex $s$-partite Tur\'an graph and the number of edges it contains}{}{\pageref{eq:k}}
\nomenclature[EF]{$k$}{Minimum $\ell \in \mathbb{N}$ such that $e \leq t_\ell(n)$}{}{\pageref{eq:k}}
\nomenclature[EF]{$H^*(n,e)$}{A conjectured extremal $(n,e)$-graph}{}{\pageref{astardef}}
\nomenclature[EF]{$\bm{a}^*$}{length-$k$ vector whose $i$th entry is the size of the $i$th part in $H^*(n,e)$}{}{\pageref{astardef}}
\nomenclature[EF]{$m^*$}{Number of missing edges in $H^*(n,e)$}{}{\pageref{astardef}}
\nomenclature[EF]{$h^*(n,e)$}{$K_3(H^*(n,e))$}{}{\pageref{astardef}}
\nomenclature[EF]{$\mathcal{H}^*(n,e)$}{Family of $(n,e)$-graphs generated from $H^*(n,e)$}{}{\pageref{razthm}}
\nomenclature[EF]{$\mathcal{H}_0$, $\mathcal{H}_1$, $\mathcal{H}_2$, $\mathcal{H}$}{Auxiliary extremal families of $(n,e)$-graphs. $\mathcal{H}(n,e) = \mathcal{H}_1(n,e) \cup \mathcal{H}_2(n,e)$}{}{\pageref{df:H}}
\nomenclature[EF]{$h(n,e)$}{Minimum number of triangles in graphs in $\mathcal{H}(n,e)$}{}{\pageref{eq:h}}
\nomenclature[EF]{$c$}{Part ratio}{}{\pageref{eq:c}}
\nomenclature[EF]{$e(K^\ell_{\alpha_1,\ldots,\alpha_\ell})$}{Continuous edge count}{}{\pageref{eq:c2}}
\nomenclature[EF]{$K_3(K^\ell_{\alpha_1,\ldots,\alpha_\ell})$}{Continuous triangle count}{}{\pageref{eq:c2}}
\nomenclature[EF]{$K_3(x,G)$}{Number of triangles in $G$ containing vertex $x$}{}{\pageref{notation}}
\nomenclature[EF]{$K_3(x,G;A)$}{Number of triangles in $G$ containing vertex $x$ and at least one other vertex in $A$}{}{\pageref{notation}}
\nomenclature[EF]{$K_3(x,G;A,A)$}{Number of triangles in $G$ containing vertex $x$ and both other vertices in $A$}{}{\pageref{notation}}
\nomenclature[EF]{$P_3(xy,G)$}{Number of common neighbours of $x,y$ in $G$}{}{\pageref{notation}}
\nomenclature[EF]{$P_3(xy,G;A)$}{Number of common neighbours of $x,y$ in $G$ which lie in $A$}{}{\pageref{notation}}
\nomenclature[EF]{$\pm$}{$x = a \pm b$ if $a-b \leq x \leq a+b$, where $b \geq 0$}{}{\pageref{notation}}
\nomenclature[EF]{$\mathcal{G}^{\mathrm{min}}(n,e)$}{Subfamily of a family $\mathcal{G}(n,e)$ of $(n,e)$-graphs which contain the fewest triangles}{}{\pageref{notation}}
\nomenclature[EF]{$n_0$}{Sufficiently large, we require $n \geq n_0$}{}{\pageref{hierarchy}}
\nomenclature[EC]{$\rho_4,\ldots,\rho_0$}{Small constants used exclusively in the proof of Lemma~\ref{approx}}{}{\pageref{hierarchy},\pageref{approx}}
\nomenclature[EC]{$\eta$}{The number of missing edges $m$ in $G$ satisfies $m \leq \eta n^2$}{}{\pageref{hierarchy},\pageref{m}}
\nomenclature[EC]{$\delta$}{(IC) $\vert Z\vert \leq \delta n$; $h \leq \delta m$, $\Delta(G[A_i]) \leq \delta n$ for all $i \in [k]$}{}{\pageref{hierarchy}}
\nomenclature[EC]{$\beta$}{Deviation of $A_1,\ldots,A_{k-1}$ from $cn$ is $\beta n$.}{}{\pageref{hierarchy}}
\nomenclature[EC]{$\xi$}{(IC) $z \in Z$ if and only if it has missing degree at least $\xi n$}{}{\pageref{hierarchy},\pageref{Zdef}}
\nomenclature[EC]{$\gamma$}{$x \in Z_k$ is in $Y$ if and only if it has degree less than $\gamma n$ into its corresponding part}{}{\pageref{hierarchy},\pageref{XY}}
\nomenclature[EC]{$\alpha$}{(IC) $e \leq t_{k}(n)-\alpha n^2$; (BC) $e > t_k(n)-\alpha n^2$}{}{\pageref{hierarchy},\pageref{eq:alpha},\pageref{eq:alpha2}}
\nomenclature[EC]{$\alpha_{1.6}$}{The minimum constant $\alpha(k)$ returned from Theorem~\ref{LS83} applied with $r=3$ and $3 \leq k \leq 1/\eps$}{}{\pageref{hierarchy},\pageref{LS83}}
\nomenclature[EC]{$\delta'$}{(BC) $\vert Z\vert \leq \delta' n$; $h \leq \delta' m$, $\Delta(G[A_i]) \leq \delta' n$ for all $i \in [k]$}{}{\pageref{hierarchy},\pageref{approx2}}
\nomenclature[EC]{$\xi'$}{(BC) $z \in Z$ if and only if it has missing degree at least $\xi' n$}{}{\pageref{hierarchy},\pageref{Zdefagain}}
\nomenclature[EC]{$\eps$}{$e \leq \binom{n}{2} - \eps n^2$}{}{\pageref{hierarchy}}
\nomenclature[EF]{(C1)--(C3)}{Worst counterexample properties}{}{\pageref{worst-C1}}
\nomenclature[EA]{Intermediate case (IC)}{$t_{k-1}(n)+\alpha n^2 \leq e \leq t_k(n)-\alpha n^2$}{}{\pageref{eq:alpha},\pageref{eq:alpha2}}
\nomenclature[EA]{Boundary case (BC)}{$t_k(n)-\alpha n^2 < e \leq t_k(n)-1$}{}{\pageref{eq:alpha2}}
\nomenclature[EF]{$f$}{$f(x)=(d_G(x)-(k-2)cn)(k-2)cn+{k-2\choose 2}c^2n^2-K_3(x,G)$ for $x \in V(G)$}{}{\pageref{eq-f}}
\nomenclature[EF]{$A_1,\ldots,A_k$}{Parts of $G$}{}{\pageref{approx},\pageref{m}}
\nomenclature[EF]{$Z$}{(IC) set of vertices $z$ with $d^m_G(z) \geq \xi n$. Boundary case: $d^m_G(z) \geq \xi'n$}{}{\pageref{Zdef},\pageref{Zdefagain}}
\nomenclature[EF]{$R_1,\ldots,R_k$}{$R_i := A_i\setminus Z$. }{}{\pageref{Z},\pageref{mbound}}
\nomenclature[EF]{$Z_1,\ldots,Z_k$}{(IC) $Z_i := A_i \cap Z$}{}{\pageref{Z}}
\nomenclature[EF]{$Z_k^1,\ldots,Z_k^{k-1}$}{(IC) union is $Z_k$, $G[Z_k^i,A_j]$ complete when $j \in [k-1]\setminus \lbrace i \rbrace$}{}{\pageref{Zsize}}
\nomenclature[EF]{$Z_I$, $I \in \binom{[k]}{k-2}$}{(BC) $G[Z_I,R_I]$ complete}{}{\pageref{lem-partition-Z}}
\nomenclature[EF]{$R_I$, $I \subseteq [k]$}{$\bigcup_{i \in I}R_i$}{}{\pageref{mbound}}
\nomenclature[EF]{$Y_1,\ldots,Y_{k-1}$}{(IC) $Y_i \subseteq Z_k^i$ contains elements with at most $\gamma n$ neighbours in $A_i$}{}{\pageref{XY}}
\nomenclature[EF]{$X_1,\ldots,X_{k-1}$}{(IC) $X_i := Z_k^i\setminus Y_i$}{}{\pageref{XY}}
\nomenclature[EF]{$t$}{(IC) $\frac{m}{(kc-1)n}$}{}{\pageref{t}}
\nomenclature[EF]{$Q_1,\ldots,Q_{k-1}$}{(IC) $Q_i \subseteq E(G[R_i,R_k])$ carefully chosen}{}{\pageref{Qi}}
\nomenclature[EF]{$R_k'$}{(IC) A large subset of $R_k$}{}{\pageref{Delta}}
\nomenclature[EC]{$\Delta$}{(IC) Maximum degree of $x \in R_k'$ into $Z_k$}{}{\pageref{Delta}}
\nomenclature[EB]{$(V_1,\ldots,V_k;U,\beta)$-partition}{(IC) P1($G$), P2($G$)}{}{\pageref{partitions}}
\nomenclature[EB]{$(V_1,\ldots,V_k;U,\beta,\delta)$-partition}{(IC) P1($G$)--P4($G$)}{}{\pageref{partitions}}
\nomenclature[EB]{$(V_1,\ldots,V_k;U,\beta,\gamma_1,\gamma_2,\delta)$-partition}{(IC) P1($G$)--P5($G$)}{}{\pageref{partitions}}
\nomenclature[EB]{weak $(V_1,\ldots,V_k;U,\beta,\gamma_1,\gamma_2,\delta)$-partition}{(IC) P1($G$), P3($G$), P5($G$)}{}{\pageref{partitions}}
\nomenclature[EF]{P1($G$)--P5($G$)}{(IC) Partition properties}{}{\pageref{partitions}}
\nomenclature[EF]{$m$}{Number of missing edges $m = \sum_{i \in [k-1]}m_i$}{}{\pageref{partitions},\pageref{approx},\pageref{approx2}}
\nomenclature[EF]{$\underline{m} = (m_1,\ldots,m_{k-1})$}{Missing vector, $m_i = e(\overline{G}[A_i,A_k])$}{}{\pageref{partitions},\pageref{approx},\pageref{approx2}}
\nomenclature[EF]{$d_H^m(y)$}{(IC) Missing degree of $y$ into corresponding part}{}{\pageref{partitions}}
\nomenclature[EF]{$h$}{Number of bad edges, $\sum_{i \in [k]}e(G[A_i])$}{}{\pageref{heq},\pageref{approx2}}
\nomenclature[EF]{$G_i$, $G_i^\ell$}{(IC) Graph obtained from $G_{i-1}$ after Transformation $i$ (applied with $\ell$)}{}{\pageref{Ziedges},\pageref{Yiedges},\pageref{Xiedges}}
\nomenclature[EF]{$\underline{m}^{(i)}$ (resp. $\underline{m}^{(i,\ell)}$)}{(IC) Missing vector of $G_i$ (resp. of $G_i^\ell$)}{}{\pageref{Ziedges},\pageref{Yiedges},\pageref{Xiedges}}
\nomenclature[EF]{$A_1',\ldots,A_k'$}{(IC) Partition of $G_2$}{}{\pageref{Yiedges}}
\nomenclature[EF]{$A_1'',\ldots,A_k''$}{(IC) Partition of $G'$}{}{\pageref{maintrans},\pageref{eW}}
\nomenclature[EF]{$G''$}{(IC) Graph obtained in Lemma~\ref{maintrans}}{}{\pageref{maintrans}}
\nomenclature[EF]{$\underline{m}'$}{(IC) Missing vector of $G'$}{}{\pageref{maintrans},\pageref{eq:m'}}
\nomenclature[EF]{$D(x), D(x,y)$}{(IC) External $X$-degree of $x \in X$, common external $X$-degree of $x,y \in X$.}{}{\pageref{subXiedges},\pageref{Dx}}
\nomenclature[EF]{$D_i$}{(IC) $D(x)=D_i$ for all $x \in X_i$, proved in Lemma~\ref{G3prop}.}{}{\pageref{G3prop}}
\nomenclature[EF]{$G''$}{(IC) Graph obtained in Lemma~\ref{maintrans}}{}{\pageref{maintrans}}
\nomenclature[EF]{$U_i,W_i$}{(IC) See Lemma~\ref{maintrans}}{}{\pageref{maintrans},\pageref{UW}}
\nomenclature[EF]{$r$}{(BC) $e = t_k(n)-r$}{}{\pageref{eq:r}}
\nomenclature[EF]{$D$}{(BC) $169k^{k+9}$}{}{\pageref{eq:D}}
\nomenclature[EF]{$s$}{(BC) $e = (1-\frac{1}{s})\binom{n}{2}$}{}{\pageref{tnew}}
\nomenclature[EF]{$N_i$}{(BC) $3$-vertex graph with $i$ edges}{}{\pageref{lseq}}
\nomenclature[EF]{$q_G(x)$}{(BC) $\frac{2e}{n}-d_G(x)$}{}{\pageref{lseq}}
\nomenclature[EF]{$\mathrm{diff}(I)$}{$\vert R_I\vert = (cn+\mathrm{diff}(I)\cdot \frac{m}{n})|I|$}{}{\pageref{diffeq}}
\nomenclature[EF]{$a_i$, $i \in [k-1]$}{$a_i = \sum_{j \in [k-1]\setminus \lbrace i \rbrace}\vert A_j\vert$}{}{\pageref{aidef}}

\scriptsize
\printnomenclature

\end{document}